%% file: main.tex
\documentclass[a4paper,12pt,reqno,openany]{amsart} 
\usepackage[left=2cm,right=2cm,top=2.5cm,bottom=2.5cm]{geometry}
\usepackage[vietnamese,english]{babel}

\usepackage[cal=dutchcal, scr=rsfso]{mathalpha}

\usepackage{color} 
\definecolor{Prune}{RGB}{99,0,60}
\usepackage{mdframed}
\usepackage{multirow} 
\usepackage{multicol} 
\usepackage{scrextend} 
\usepackage[absolute]{textpos} 
\usepackage{colortbl}
\usepackage{array}
\usepackage{geometry}

\usepackage{etoolbox}
\usepackage{graphicx}
\usepackage{caption}
\usepackage{booktabs}
\usepackage[dvipsnames]{xcolor}
\definecolor{myprpl}{RGB}{255,0,130}
\usepackage[colorlinks, citecolor=cyan, linkcolor=myprpl]{hyperref}
\usepackage{caption}
\graphicspath{ {./images/} }
\usepackage{scrextend}
\usepackage{fancyhdr}
\usepackage{enumerate}

\usepackage{mathtools}
\usepackage{amsfonts}
\usepackage{amssymb}
\usepackage{amsmath}
\usepackage{amsthm}
\usepackage{float}
\usepackage{tikz-cd}
\usetikzlibrary{babel}
\usetikzlibrary{decorations.pathmorphing,shapes}
\usetikzlibrary{arrows.meta}
\tikzcdset{arrow style=tikz,
    squigarrow/.style={
        decoration={
        snake, 
        amplitude=.4mm,
        segment length=2mm
        }, 
        rounded corners=.2pt,
        decorate
        }
    }

\DeclareFontFamily{U}{wncy}{}
\DeclareFontShape{U}{wncy}{m}{n}{<->wncyr10}{}
\DeclareSymbolFont{mcy}{U}{wncy}{m}{n}
\DeclareMathSymbol{\Sh}{\mathord}{mcy}{"58} 
\DeclareMathSymbol{\Be}{\mathord}{mcy}{"42} 

\newcommand\lcomp{{\scalebox{1}[1.5]{$\big[$}}}
\newcommand\rcomp{{\scalebox{1}[1.5]{$\big]$}}}

\newcommand{\fppf}{\mathrm{fppf}}
\newcommand{\lqf}{\mathrm{lqf}}
\newcommand{\et}{\mathrm{\acute{e}t}}
\newcommand{\Et}{\mathrm{\acute{E}t}}
\newcommand{\Br}{\mathrm{Br}}
\newcommand{\saut}{\mathrm{SAut}}
\newcommand{\sout}{\mathrm{SOut}}
\newcommand{\aut}{\mathrm{Aut}}
\newcommand{\out}{\mathrm{Out}}
\newcommand{\Gal}{\mathrm{Gal}}
\newcommand{\Cb}{{\overline{C}}}

\newcommand{\Gb}{{\hspace{1pt}\overline{\hspace{-1pt}G}}}
\newcommand{\Hb}{{\hspace{1.5pt}\overline{\hspace{-1.5pt}H}}}
\newcommand{\Qb}{{\overline{Q}}}

\newcommand{\Ub}{{\overline{U}}}

\newcommand{\Xb}{{\hspace{1.5pt}\overline{\hspace{-1.5pt}X}}}
\newcommand{\Ca}{{\overline{C'}}}
\newcommand{\Ga}{{\overline{G'}}}

\newcommand{\Caa}{\overline{C'_{k'_s}\hspace{-6pt}}\hspace{6pt}}
\newcommand{\Gaa}{\overline{G'_{k'_s}\hspace{-6pt}}\hspace{6pt}}

\newcommand{\ab}{\mathrm{ab}}
\DeclareMathOperator{\im}{im}
\DeclareMathOperator{\coker}{coker}
\DeclareMathOperator{\Spec}{Spec}
\DeclareMathOperator{\Hom}{Hom}

\newtheoremstyle{myplain}
  {.5em}       
  {.5em}       
  {\itshape}   
  {0pt}        
  {\scshape}   
  {.}          
  {5pt plus 1pt minus 1pt} 
  {}           

\newtheoremstyle{mydefinition}
  {.5em}       
  {.5em}       
  {\normalfont}
  {0pt}        
  {\bfseries}  
  {.}          
  {5pt plus 1pt minus 1pt} 
  {}           

\newtheoremstyle{myremark}
  {.5em}       
  {.5em}       
  {\normalfont}
  {0pt}        
  {\itshape}   
  {.}          
  {5pt plus 1pt minus 1pt} 
  {}           

\numberwithin{equation}{section}

\theoremstyle{myplain}
\newtheorem{thm}{Theorem}[subsection]
\newtheorem{lem}[thm]{Lemma}
\newtheorem{cor}[thm]{Corollary}
\newtheorem{prop}[thm]{Proposition}

\theoremstyle{mydefinition}
\newtheorem{defn}[thm]{Definition}
\newtheorem{exmp}[thm]{Example}

\theoremstyle{myremark}
\newtheorem{rem}[thm]{Remark}

\newtheorem*{intprf}{Proof}
\newtheorem*{intprfsk}{Sketch of Proof}

\newenvironment{prf}[1][]
    {\ifstrempty{#1}{\begin{intprf}}
                    {\begin{intprf}[\textit{#1}]}}
    {\qed\end{intprf}}

\patchcmd{\tocsection}{#1}{{\makebox(0pt,13pt){}}%
\textsection}{}{}
\patchcmd{\tocsubsection}{#1}{\hspace{1cm}
#1}{}{}

\makeatletter
\patchcmd{\@tocline}
  {\hfil}
  {\leaders\hbox{\,.\,}\hfil}
  {}{}
\makeatother

\title[Separable Pseudo-reductive Bands with Applications to Rational Points]
{{\large{\centerline{Separable Pseudo-reductive Bands}\centerline{with Applications to Rational Points}}}
\hspace{0pt}
\\\;
\vspace{-7pt}
\\{
\normalfont{Azur ĐonlagiĆ}
}}

\begin{document}
\selectlanguage{english}

\begin{abstract}
\vspace{-5pt}

We extend the Galois-theoretic Springer-Borovoi interpretation of algebraic bands to a class of \'etale-locally represented bands on the fppf site of an arbitrary field $k$, which we call separable bands. Next, a band represented \'etale-locally over $k$ by a pseudo-reductive group is shown to be globally representable when $[k : k^p] = p$, with counterexamples in general.

When $k$ is a global or local field, we deduce a generalization of Borovoi's abelianization theory to separable bands represented by smooth connected algebraic groups. As an application, we prove that the Brauer-Manin obstruction is the only one to weak approximation on $X$, when $X$~is a homogeneous space of a pseudo-reductive group with smooth connected geometric stabilizer.
\end{abstract}
\maketitle

\vspace{-30pt}

\tableofcontents
\input{part0}

\section{Generalities on Bands}
In this section, we first recall the general concept of a band on a site and its relation~to~inner forms. Then we specialize to the case of the big \'etale or fppf site ($k_\Et$ or $k_\fppf$, respectively) over a field $k$ and compare this general definition with Galois-theoretic definitions given in \cite{Spr66}, \cite{Brv93}, \cite{FSS98} and \cite{DLA19}. This leads naturally to the concept of a separable band on $k_\fppf$, which we study for later use. In the last subsection we show that the $\mathrm{H}^2$ set of locally algebraic bands on $k_\Et$ or $k_\fppf$ admits an equivalent definition in terms of Čech cocycles.

\input{part1-0}
\input{part1-1}

\input{part1-2}

\input{part1-3}

\section{Representability}
Let $k$ be a field which satisfies the condition $[k : k^p] = p$ (the \textit{imperfection degree} of $k$ is $1$). In this section we prove that every \'etale band $(\Gb,\kappa)$ over $k$ represented by a pseudo-reductive group $\Gb$ over $k_s$ is globally representable (Theorem \ref{thmmainpsredrepr}). In particular, we directly deduce the same result for separable bands, as each separable band represented by a smooth algebraic group has at least one \'etale band lying over it by Corollary \ref{corparamlet}.

The condition on the imperfection degree is necessary not only for the exotic pseudo-reductive bands in low characteristic ($p\in\{2,3\}$), but for standard pseudo-reductive constructions in any characteristic. Unlike for reductive bands, global representability fails over general $k$ even for separable bands locally represented by very well-behaved pseudo-reductive groups, such as Weil restrictions of simple and simply connected~groups (Example \ref{exmprepringen}). Such examples exist over any field $k$ over which there is a nontrivial inseparable field extension $k'/k$ not admitting any nontrivial intermediate purely inseparable extensions.

\input{part2-1}

\input{part2-2}
\input{part2-3}

\section{Abelianization}\label{sectabel}
Let $k$ be a global or local field of characteristic $p > 0$. In this section we prove Theorem \ref{thmmainabelh2}, which says that, given a smooth connected affine separable band $L=(\Gb,\kappa)$ over $k$,~the~sequence
$$\mathrm{N}^2(k,L)\xhookrightarrow{\;\;\;\;\;\;}\mathrm{H}^2(k, L)\xrightarrow{\;\;\mathrm{ab}^2\;\;}\mathrm{H}^2(k,L_\ab)$$
is exact, where $L_\ab$ is a commutative affine algebraic group over $k$ (isomorphic over $k_s$ to the group $\Gb/\mathcal D(\Gb)$). This statement directly generalizes the main results of Borovoi's abelianization theory for bands in characteristic $0$ (\cite[Props.\ 6.2, 6.5]{Brv93}), up to a small reformulation whose necessity is explained in the first subsection (and Remark \ref{remmainabel}).

The second subsection treats the essential case when $\Gb$ is pseudo-reductive and, by previous results, $L$ admits a global representative $G$. An essential tool is a generalization of the notion of anisotropicity. In the third subsection, we give a proper definition of the map $\mathrm{ab}^2$ and prove the main statement. Along the way, we also show (in Theorem \ref{thmmainabelh1}) that, when $G$ is a smooth connected affine algebraic group over $k$, the map $\mathrm H^1(k, G)\rightarrow\mathrm H^1(k, G/\mathcal D(G))$ is surjective.

\input{part3-1}
\input{part3-2}
\input{part3-3}

\section{Applications}
In this section we prove (Theorem \ref{thmmainbmobs}) that the Brauer-Manin obstruction is the only one to the Hasse principle and weak approximation on homogeneous spaces of pseudo-reductive~groups (and more generally, of smooth connected groups with split unipotent radical) with smooth connected geometric stabilizer. In the first subsection we recall, in the generality of separable~bands, the theory of Springer bands associated to homogeneous spaces. In the second one, we review some basic facts about the Brauer-Manin obstruction and prove the crucial Lemma \ref{lemphslift} which relates it to homogeneous spaces. The main statement is proven in the third subsection through several reduction steps and approximation lemmas of independent interest (e.g.\ Lemma \ref{lemapproxliftabquot}).\;\;\quad

\input{part4-0}

\input{part4-1}

\input{part4-2}

\appendix
\addcontentsline{toc}{section}{Appendices{\makebox(0pt,18pt){}}}
\section{Pseudo-reductive Groups}
A smooth connected affine algebraic group $G$ over a field $k$ admits a unique largest smooth connected unipotent normal subgroup $\mathscr R_{u,k}(G)$ defined over $k$, called its \textit{unipotent radical}. If $k'/k$ is a separable field extension, then $\mathscr R_{u,k'}(G_{k'}) = \mathscr R_{u,k}(G)_{k'}$ (after reducing to $k'/k$ finite, this follows from the fact that the unipotent radical is preserved by Galois transformations).

When $\mathscr R_{u,k}(G) = 1$, we say that $G$ is \textit{pseudo-reductive}. Moreover, $G$ is reductive if and only if $\mathscr R_{u,\overline k}(G_{\overline k}) = 1$. From the above, it follows that pseudo-reductivity is a distinct notion from reductivity only if $k$ is an imperfect field, i.e. $\smash{\overline k/k}$ is not separable. The converse holds, as over every imperfect field there are examples of non-reductive pseudo-reductive groups: for instance, the Weil restriction $\mathrm R_{k'/k}(\mathbf G_{\mathrm m,\,k'})$ for $k'/k$ finite inseparable (cf.\ Proposition \ref{propweilrestrunip}).

Because $\mathscr R_{u,k}(G/\mathscr R_{u,k}(G)) = 1$, the structure theory of smooth connected affine algebraic groups over a field of positive characteristic rests heavily on that of pseudo-reductive groups, developed mainly in \cite{CGP15} and \cite{CP15}. This appendix gives a short overview of some of the main results in this theory, as well as their simple consequences which are needed in the main body of the paper and that were convenient for us to develop in one separate place. Note that the last two subsections have significant overlap with \cite[\textsection2]{Con12}.

\input{partA-1}
\input{partA-2}
\input{partA-3}

\section{Miscellaneous Useful Facts}
When dealing with (non-smooth) algebraic groups over fields of positive characteristic, it is necessary to work with fppf, as opposed to \'etale (Galois) cohomology. Furthermore, even in cases where Galois cohomology is still useful, some technical statements need to be treated differently than in the more classical, smooth setting. In this section we collect various facts about algebraic groups in positive characteristic and their basic consequences which will be used in the main part of the paper.

The first subsection recalls the properties Weil restrictions along arbitrary finite (flat) maps, their good behavior for smooth groups and the interesting phenomena related to non-smooth groups. The second subsection reviews some basic vanishing statements in Galois cohomology that are relevant only in positive characteristic. Finally, in the third subsection we break down the proof of the well-known statement that all gerbes bound by a semisimple group are trivial, so that parts of it can be generalized in Subsection \ref{ssectgenanis}.

\input{partB-0}
\input{partB-1}
\input{partB-2}

\section{Cohomology on the $\lqf$-site}\label{sectlqf}
This appendix is devoted to proving the following fact, needed in the proof of Lemma \ref{lemphslift}: Let $G$ be a smooth affine algebraic group over a global field $k$ and let $X$ be a homogeneous space of $G$ with smooth connected geometric stabilizer $\Hb$. We denote by $H_\ab$ the maximal Abelian quotient of the Springer band $L_X$, which is defined over $k$ (see Definition \ref{defmainabel}), and by $\widehat{H}_\ab = \mathcal{Hom}(H_\ab,\mathbf G_{\mathrm m})$ its Cartier dual. We claim that there exists a group homomorphism
$$\Sh^1(\widehat{H}_\ab)\longrightarrow\Be(X)$$
which completes a natural commutative triangle involving the Brauer-Manin obstruction on $X$ and the Poitou-Tate pairing by the Springer class $\xi_X$. This homomorphism comes from a functorial map $\mathrm H^1(K, \widehat{H}_\ab)\rightarrow\Br(X_K)/\Br(K)$ (taken over $K = k$ and $K = k_v$ for all $v$), which is induced in fppf hypercohomology by a certain morphism of (derived) complexes of sheaves:
$$\widehat{H}_\ab[-1]\longrightarrow\mathcal{UPic}_X$$
The complex $\mathcal{UPic}_X$ has been well studied as a complex of sheaves on the \'etale site of a field (see references given below), however, \'etale cohomology is not sufficient for us: Our fppf sheaf $\widehat{H}_\ab$ is in general not even representable, let alone by a smooth algebraic group. We are thus lead to consider a replacement complex $\mathcal{UPic}_X$ of sheaves in the fppf topology. For technical reasons (see Lemma \ref{lemr1lqf} below) it is convenient to define it on a ``small site'' equipped with the fppf topology.

The first subsection introduces the $\lqf$-site of a scheme, the complex $\mathcal{UPic}_X$ and its basic properties. The second one gives some general constructions which relate this complex to $\widehat{H}_\ab$. Finally, in the third subsection we show that the required triangle in the proof of Lemma \ref{lemphslift} commutes by an explicit computation in Čech cohomology.

\input{partC-0}
\input{partC-1}
\input{partC-2}

\newpage


\input{bibliography}

\end{document}

%% file: part0.tex
\section{Introduction}

In the 1960s, J.\ Giraud introduced the concept of a ``band'' (fr.\ ``lien'') on a site, an object locally represented by group sheaves with gluing data defined up to inner automorphisms (see \cite{Gir71}). He assigned to each band a reasonably behaved $\mathrm H^2$ set with a distinguished subset of ``neutral elements'', generalizing the second cohomology group of a commutative group sheaf. Around the same time, a similar object (called a ``kernel'' in \cite{Spr66}) was defined in the context of groups and Galois cohomology by T.\ Springer, (almost) corresponding to the notion of a band on the small \'etale site $k_\et$ of a field $k$. 
The difference is in that, while Springer's kernels were defined over the separable closure $k_s$ of $k$ with ``outer'' Galois action, the continuity condition imposed on this action was too restrictive in practice (see Remark \ref{remcontcond}).

A more appropriate continuity condition was recognized by M.\ Borovoi and added in \cite{Brv93} to the definition of ``algebraic kernels'' (those represented by group schemes of finite type over~$k_s$; although not named as such in that paper). The definition of continuity was further refined~by Y.\ Flicker, C.\ Scheiderer and R.\ Sujatha in \cite{FSS98} into a form fully equivalent to Giraud's (algebraic) bands on $k_\et$. This equivalence is proven by C.\ Demarche and G.\ Lucchini Arteche in \cite{DLA19} (and reproven in the course of Subsection \ref{ssectetandsep} of this paper, which also highlights at which point in the proof algebraicity becomes necessary).

Our interest in the study of bands is due to the existence of a band (called the ``Springer band'' after \cite{Spr66} and equipped with a particular $\mathrm H^2$ class $\xi$) associated to a given homogeneous space of a group sheaf, which lifts to a principal homogeneous space if and only if the class $\xi$ is neutral. In \cite{Brv93}, Borovoi developed an abelianization theory for connected affine algebraic bands over number fields and applied it for the class $\xi$ to study the Hasse principle for homogeneous spaces. These ideas are also implicit in \cite{Brv96}, where he proves that the Brauer-Manin obstruction is the only obstruction to the Hasse principle (and weak approximation) for homogeneous spaces of connected affine algebraic groups with connected geometric stabilizers over number fields.

This theorem was extended by C.\ Demarche and D.\ Harari in \cite{DH22} to homogeneous spaces of (connected) reductive groups with (connected) reductive geometric stabilizers over global function fields.
The reductivity hypothesis is difficult to remove in positive characteristic,~due~to the existence of the following types of (affine) algebraic groups:
%
\begin{enumerate}[$\hspace{0.5cm}\bullet$]
    \item Non-split unipotent algebraic groups -- this is not a problem from the perspective of bands, but instead from the perspective of the theory of rational points (and that of Brauer-Manin obstructions), since it is difficult to control the first cohomology set of such groups
    \item Pseudo-reductive groups (i.e.\ smooth connected groups with trivial unipotent radical; see Appendix A) which are not reductive
    \item Non-smooth algebraic groups, which necessitate the use of fppf cohomology, instead~of~\'etale cohomology, especially in statements involving d\'evissage arguments
\end{enumerate}
%
In particular, the third point above shows that the fppf topology is the ``right one'' to consider when working with algebraic groups in positive characteristic. Even when working with smooth groups, the Springer-Borovoi theory does not agree with the notion of a band on the fppf site (in general, it defines multiple bands on $k_\Et$ ``lying over'' the same band on $k_\fppf$, in a sense to be made precise in Definition \ref{deflyingover}) and in particular does not calculate the fppf cohomology~$\mathrm H^2$~set. As a result, it is necessary to carefully distinguish these bands in positive characteristic:

\begin{exmp}
Let $G$ be a smooth algebraic group over a field $k$ with absolute Galois group $\Gamma\coloneqq\Gal(k_s/k)$. Because $G$ is smooth, the points $G(k_s)$ are schematically dense in $G$ and thus $\mathrm Z_G(k_s) = \mathrm Z(G(k_s))$, meaning that the formation of the (pre)sheaf center of $G$ commutes with restriction of the corresponding (pre)sheaf from $\mathrm{Sch}/k$ to $k_\et$ (this is implicit in \cite[1.5]{FSS98}). Let $\mathrm H^2_\et(k,G)$ (resp.\ $\mathrm H^2(k,G)$) be the second cohomology set associated to the unique band $L_\et$ on $k_\et$ (resp. $L$ on $k_\fppf$) represented by $G$. It is in canonical bijection with $\mathrm H^2_\et(k,\mathrm Z_G) = \mathrm H^2(\Gamma,\mathrm Z_G(k_s))$ (resp.\ $\mathrm H^2(k,\mathrm Z_G)$), defined by the global representative $G$.

The map $\mathrm H^2_\et(k,G)\rightarrow\mathrm H^2(k,G)$ is in general neither injective nor surjective, and neither is the map $\mathrm H^1(\Gamma,G(k_s)/\mathrm Z_G(k_s))\rightarrow\mathrm H^1(k,G/\mathrm Z_G)$ (see Proposition \ref{propdiaglet}). These two $H^1$ sets surject onto the sets of \'etale (resp.\ fppf) inner forms of $G$, which are exactly the global representatives of $L_\et$ (resp.\ $L$). Thus in particular, if $G$ is twisted by an fppf-torsor $P$ of $G/\mathrm Z_G$, the resulting inner twist $_PG$ represents the same band $L$ on $L_\fppf$, but the band $L'_\et$ which it represents on~$k_\et$ may be different from $L_\et$ in general!

Another subtle difference is in the notion of maximal Abelian quotient (over $k$) of $L$, resp.\ $L_\et$ (cf.\ Definition \ref{defmainabel}). In the fppf topology, this is the quotient algebraic group $G_\ab\coloneqq G/\mathcal D(G)$, where $\mathcal D(G)$ is the derived subgroup of $G$. However, in the \'etale topology, this is some sheaf $\mathcal F$ (on the big \'etale site $k_\Et$) such that $\mathcal F(k_s) = (G(k_s))_\ab$, but which may be non-representable.~The abelianization theory developed in this paper (Theorem \ref{thmmainabelh2}) works for the cohomology group $\mathrm H^2(k,G_\ab)$ which is, at least a priori, not intrinsically defined using only the \'etale site.
\end{exmp}

To partially extend the Galois-theoretic Springer-Borovoi definition to (algebraic) fppf bands, we propose in this paper an intermediate object called a ``separable band'': This is an algebraic band on the site $k_\fppf$ which admits a local representative over some finite separable extension $k'/k$. We also demand that any two such representatives become isomorphic over a common finite separable extension. This additional condition guarantees that such a band admits a Galois-theoretic description in a unique way (Definition \ref{defetsepband}) and it is satisfied when the local representatives are smooth algebraic groups (see Proposition \ref{propcharsepband}; the main argument here is that then $\mathrm H^1(k_s,G/\mathrm Z_G) = 1$, as suggested by the above example).

After showing some basic properties of separable bands, we devote most of the paper to studying their behavior over local and global fields of positive characteristic. Here, the essential case is that of pseudo-reductive separable bands (which is a well-behaved notion since, if some representative of a band over a finite separable extension is pseudo-reductive, then all such representatives are too). The main results of the paper are:
\begin{enumerate}[$\hspace{0.5cm}\bullet$]
    \item A global representability statement (based on the reductive case, \cite[V, Prop.~3.2]{Dou76})~for \'etale (resp.\ separable) pseudo-reductive bands -- Theorem~\ref{thmmainpsredrepr} (resp.\ Corollary~\ref{cormainpsredrepr})
    \item An $\mathrm H^2$ abelianization theory for smooth connected affine separable bands -- Theorem \ref{thmmainabelh2}\;
    \item A proof that the Brauer-Manin obstruction is the only obstruction to the Hasse principle and to weak approximation on a homogeneous space of a smooth connected affine algebraic group with split unipotent radical (e.g.\ a pseudo-reductive group)
    having smooth connected geometric stabilizer -- Theorem \ref{thmmainbmobs} 
\end{enumerate}
The paper is split into 5 sections (including this introduction) and 3 appendices, all of which have introductions which present their content in more detail. Here is a summary:

In \textit{Appendix A} we review the main aspects of the structure theory of pseudo-reductive groups (mostly following \cite{CGP15}) and clearly formulate them in a way which will allow their repeated use throughout the paper. In \textit{Appendix B} we review some facts about algebraic groups and their cohomology (including Weil restrictions), with an accent on technical statements in positive characteristic that have better-known or classical analogues in characteristic $0$.

\textit{Section 2} of the paper starts with a recollection of basic definitions related to bands on an arbitrary site, which then lead into a comparison with the Springer-Borovoi definition (objects which will further simply be called ``\'etale bands'') over a field. The definition of a separable~band is given in parallel; it is then systematically compared with and related to the \'etale bands lying over it. Last, we introduce the class of ``nicely representable bands with locally~algebraic~center'' on the site $k_\fppf$, which in particular includes (locally) algebraic bands, and we show that the~$\mathrm H^2$ set of such bands admits a description in terms of nonabelian fppf Čech $2$-cocycles.
\footnote{The question of existence of such a Čech cohomology theory in positive characteristic (as well as the remark that we cannot expect it to exist in general) was communicated to the author by prof.\ Borovoi in late 2023, during one of the professor's visits to Orsay. We are happy to give a positive answer in the algebraic case.}

In \textit{Section 3}, we fix a field $k$ of characteristic $p > 0$ such that $[k:k^p] = p$ holds and prove that a given \'etale (or separable) band represented over $k_s$ by a pseudo-reductive group always admits a global representative. We also show that this property fails over general fields even for very well-behaved pseudo-reductive groups (for example, Weil restrictions of semisimple, simply connected groups; see Example \ref{exmprepringen}). These results are further used in \textit{Section 4} to show the main abelianization theorems over a local or global field: The essential case is that of a pseudo-reductive separable band, where we may work with a fixed global representative.~Along~the way, we also prove that the natural map $\mathrm H^1(k,G)\rightarrow\mathrm H^1(k,G/\mathcal D(G))$ is surjective for a smooth connected affine algebraic group $G$ over a local or global field $k$ (Theorem \ref{thmmainabelh1}).

Finally, we apply in \textit{Section 5} the theory developed in previous sections. The main theorem on Brauer-Manin obstructions (generalizing \cite[Thms.\ 2.5, 4.1]{DH22}) is proven after a series of reduction steps and lemmas related to Springer bands and rational points. The key statement which connects these different areas is Lemma \ref{lemphslift}, which says that a smooth homogeneous space with smooth connected stabilizer and no Brauer-Manin obstruction can always be lifted to a principal homogeneous space. The proof of this lemma requires a technical step which is classically, in characteristic $0$ and \'etale topology, a direct consequence of some well-understood statements. In positive characteristic and fppf topology, the analogous statements require a substantial amount of work and difficult calculations using Čech cohomology. Their proof has thus been relegated to \textit{Appendix C}, of a slightly different flavor than the main text, dealing~with the fppf analogue of the $\mathrm{UPic}(\Xb)$ complex (cf.\ \cite{BvH09}) and its connection to Poitou-Tate~theory.

\begin{rem}
In the reductive case, Lemma \ref{lemphslift} reduces to the statement of \cite[Lem.~2.6]{DH22}. In characteristic $0$, it has recently been generalized to the context of ``nonabelian descent types'' by \foreignlanguage{vietnamese}{Nguyễn} M.\ Linh; see \cite[Prop.\ 3.9]{Ngu25}.
\end{rem}

This paper is part of the author's PhD project at the \textit{Laboratoire de Math\'ematiques d'Orsay} of Paris-Saclay University, funded through the PhD Track of \textit{Fondation Math\'ematique Jacques Hadamard}. The author would like to thank his thesis advisor, prof.\ David Harari, for introducing him to the study of rational points and for the very helpful and regular meetings throughout the writing of this work.

\begin{rem}\label{remnotconv}
A few words about notation and conventions: For a field $k$, we write $k_s$ (resp.~$\overline{k}$) for a fixed choice of separable (resp.\ algebraic) closure of $k$. By a ``(locally) algebraic group'' we mean a group scheme (locally) of finite type over a field, not necessarily smooth or affine. Importantly, each ``reductive group'' is assumed connected.

Group schemes are also routinely identified with the corresponding fppf sheaves over the same base scheme $X$ and we always write $\mathrm H^n(X, G)$ for the fppf cohomology of $G$. On~the~other hand, \'etale cohomology (on both the big site $X_\Et$ or small site $X_\et$, which compute the same cohomology) will be denoted by $\mathrm H^n_\et(X,G)$ or more often, when $X = \Spec(k)$ and $G$ is a locally algebraic group, by $\mathrm H^n(\Gamma,G(k_s))$ for $\Gamma = \Gal(k_s/k)$.

We may make this last identification because then the natural map $\varinjlim G(K)\rightarrow G(k_s)$ is an isomorphism, where the direct limit is taken over all finite separable field extensions $K/k$ in $k_s$ (see \cite[IV$_3$, Prop.\ 8.14.2]{EGA}).
\end{rem}

%% file: part1-0.tex
\subsection{Bands, Twisting and Nonabelian $\mathrm{H}^2$ Sets}

Let $\mathcal C$ be a site with final object $S$. For simplicity of notation, assume that every covering of $S$ can be refined by a one-element covering. After this general introduction, we will consider only $S = \Spec(k)$ for a field $k$ and the \'etale or fppf sites $k_\Et$ or $k_\fppf$, respectively.

Given an object $T\in\mathcal C$, we denote by $\mathcal C/T$ the category of morphisms to $T$. For sheaves of groups $\mathcal F, \mathcal G\in\mathrm{Sh}_\mathrm{Grp}(\mathcal C/T)$, there is an obvious sheaf of sets $\mathcal{Isom}_{\mathcal F,\, \mathcal G}$ on $\mathcal C/T$. It is acted on from the right by $\mathcal{Aut}_{\mathcal F}\coloneqq\mathcal{Isom}_{\mathcal F,\, \mathcal F}\in\mathrm{Sh}_\mathrm{Grp}(\mathcal C/T)$, which also naturally acts on $\mathcal F$ from~the~left. The image of the natural map $\mathcal F\rightarrow\mathcal{Aut}_{\mathcal F}$ given by inner automorphisms is by definition the sheaf quotient $\mathcal F/\mathrm{Z}_{\mathcal F}$, where $\mathrm{Z}_{\mathcal F}$ denotes the center of $\mathcal F$.

Let $\mathcal{Out}_{\mathcal F,\,\mathcal G}$ denote the sheafified quotient of $\mathcal{Isom}_{\mathcal F,\, \mathcal G}$ by the above action of $\mathcal F/\mathrm{Z}_{\mathcal F}$. In concrete terms, we have the following direct limit taken over all coverings $T'\rightarrow T$ in $\mathcal C$:
\begin{equation}\label{eqdefshquot}
    \mathcal{Out}_{\mathcal F,\mathcal G}(T)\coloneqq\varinjlim_{T'\rightarrow T}\ker\left(\frac{\mathcal{Isom}_{\mathcal F,\mathcal G}(T')}{(\mathcal F/\mathrm Z_{\mathcal F})(T')}\rightrightarrows\frac{\mathcal{Isom}_{\mathcal F,\mathcal G}(T'\times_T T')}{(\mathcal F/\mathrm Z_{\mathcal F})(T'\times_T T')}\right)
\end{equation}
Note that, since we are considering only isomorphisms and not general morphisms, this has the same effect as taking the left-quotient by $\mathcal G/\mathrm{Z}_{\mathcal G}$. There is in particular a well-defined operation $$\circ\,:\,\mathcal{Out}_{\mathcal G,\,\mathcal H}\times\mathcal{Out}_{\mathcal F,\,\mathcal G}\longrightarrow\mathcal{Out}_{\mathcal F,\,\mathcal H}$$
induced by the usual composition of (iso)morphisms.

\begin{defn}\label{defband}
Consider a triple $(T{\rightarrow} S,\mathcal F,\varphi)$, where $T\rightarrow S$ is a covering, $\mathcal F\in\mathrm{Sh}_\mathrm{Grp}(\mathcal C/T)$ a group sheaf~and $\varphi\in\mathcal{Out}_{\mathrm{pr}_1^*\mathcal F,\,\mathrm{pr}_2^*\mathcal F}(T\times_S T)$. Such a triple is called a \textit{representative triple} if:  $$\mathrm{pr}_{13}^*\varphi = \mathrm{pr}_{23}^*\varphi\circ\mathrm{pr}_{12}^*\varphi
\;\;\textrm{ holds in }\;\;
\mathcal{Out}_{\mathrm{pr}_{13}^*\mathrm{pr}_1^*\mathcal F,\,\mathrm{pr}_{13}^*\mathrm{pr}_2^*\mathcal F}(T\times_S T\times_S T)$$
Here the projections $\mathrm{pr}_I$ have their usual meaning.
Thus $\varphi$ can be seen as a descent datum~of~$\mathcal F$ ``defined up to inner automorphisms'' with respect to the cover $p_0 : T\rightarrow S$.
If $p : T'\rightarrow T$ is a map such that $p_0\circ p$ is a refinement of $p_0$, the triples $(T{\rightarrow} S,\mathcal F,\varphi)$ and $(T'{\rightarrow} S,p^*\mathcal F,(p\times p)^*\varphi)$ will be considered equivalent, and we will say that the second \textit{refines} the first. An equivalence class of triples generated by this relation is called a \textit{band} (fr. \textit{lien}) on $\mathcal C$.

Given bands $L,L'$ represented by $(T{\rightarrow} S,\mathcal F,\varphi)$ and $(T{\rightarrow} S,\mathcal F',\varphi')$, respectively, consider an isomorphism of group sheaves $\alpha : \mathcal F\rightarrow\mathcal F'$ such that $\mathrm{pr}_2^*\alpha\circ\varphi = \varphi'\circ\mathrm{pr}_1^*\alpha$ in $\mathcal{Out}_{\mathrm{pr}_1^*\mathcal F,\,\mathrm{pr}_2^*\mathcal F'}(T\times_S T)$. An equivalence class of such isomorphisms $\alpha$, up to obvious base change, is an \textit{isomorphism of bands} $L\xrightarrow{\sim} L'$. Moreover, when such an $\alpha$ exists over the covering $T\rightarrow S$, we will say that this isomorphism is $T$\textit{-representable}.
\end{defn}

We remark now that there is a more general notion of morphism of bands, defined similarly with respect to the double-sided sheaf quotient $(\mathcal G/\mathrm{Z}_{\mathcal G}){\setminus}\mathcal{Hom}_{\mathcal F,\, \mathcal G}/(\mathcal F/\mathrm{Z}_{\mathcal F})$. We will not introduce such morphisms and will prefer to work more explicitly whenever relating two bands (for instance, in Lemma \ref{lemmainlift}). In such calculations, it will be useful to distinguish the following notion:

\begin{defn}\label{defnice}
A representative triple $(T{\rightarrow} S,\mathcal F,\varphi)$ will be called a \textit{nice triple} if $\varphi$ belongs to the image of the natural map
$$\mathcal{Isom}_{\mathrm{pr}_1^*\mathcal F,\,\mathrm{pr}_2^*\mathcal F}(T\times_S T)\longrightarrow\mathcal{Out}_{\mathrm{pr}_1^*\mathcal F,\,\mathrm{pr}_2^*\mathcal F}(T\times_S T)$$
induced on global sections by the formation of sheaf quotients.

A band defined by a representative triple $(T{\rightarrow} S,\mathcal F,\varphi)$ is also said to be \textit{locally represented} by $\mathcal F$. It is said to be \textit{trivial} or \textit{(globally) representable} if it can be represented a triple of the form $(\mathrm{id}_S,\mathcal F,\mathrm{id})$ for some group sheaf $\mathcal F$ on $S$. 
Note that, if a band is represented by some triple $(T{\rightarrow} S,\mathcal F,\varphi)$, then the cocycle condition on $\varphi$ forces that $\varphi|_T = [\mathrm{id}]\in\mathcal{Out}_{\mathcal F,\,\mathcal F}(T)$.~In~particular, this means that a band is globally representable if and only if it is represented by a nice triple in which the covering $T\rightarrow S$ is an isomorphism.
\end{defn}

\begin{rem}
Our definition of bands via representative triples follows \cite{DM82}, but there the isomorphism $\alpha$ is only required to exist as a section of $\mathcal{Out}_{\mathcal F,\,\mathcal F'}(T)$. The ultimate definition of morphisms of bands is of course completely equivalent, since our stronger condition can always be assumed up to refining the covering $T\rightarrow S$.

On the other hand, it does not follow that every representative triple can be refined by a nice one through refining $T\rightarrow S$, as there does not need to exist any refinement $T'\rightarrow S$ such that the map $T'\times T'\rightarrow T\times T$ factors through a given covering $R\rightarrow T\times T$. Nevertheless, we will see in Proposition \ref{propetrefin} that such a refinement does always exist when $\mathcal C$ is the \'etale site of a field. In particular, every band on that site is represented by a nice triple.

Note that, given a nice triple $(T{\rightarrow} S,\mathcal F,\varphi)$ and a lift $f : \mathrm{pr}_1^*\mathcal F\rightarrow\mathrm{pr}_2^*\mathcal F$ of $\varphi$, the automorphism $(\mathrm{pr}_{13}^*f)^{-1}\circ(\mathrm{pr}_{23}^*f)\circ(\mathrm{pr}_{12}^*f)$ of the sheaf $\mathrm{pr}_{12}^*\mathrm{pr}_1^*\mathcal F = \mathrm{pr}_{13}^*\mathrm{pr}_1^*\mathcal F$ lies in the kernel of the map
$$\mathcal{Aut}_{\mathrm{pr}_{13}^*\mathrm{pr}_1^*\mathcal F,\,\mathrm{pr}_{13}^*\mathrm{pr}_2^*\mathcal F}(T\times_S T\times_S T)\longrightarrow\mathcal{Out}_{\mathrm{pr}_{13}^*\mathrm{pr}_1^*\mathcal F,\,\mathrm{pr}_{13}^*\mathrm{pr}_2^*\mathcal F}(T\times_S T\times_S T)$$
by definition. Thus it is naturally a section in $(\mathcal F/\mathrm Z_{\mathcal F})(T\times_S T\times_S T)$. We will drop the prefix $\mathrm{pr}_{13}^*\mathrm{pr}_1^*$ in front of this expression and similar ones in this paper, always making the choice to favor pullbacks with respect to the smallest-indexed projections, e.g.\ $\mathrm{pr}_1$ and not $\mathrm{pr}_2$ (which we can do since we are considering only isomorphisms and not general morphisms). This consistent and subtle piece of notation should never cause confusion; it features in all of our later cocycle calculations (cf.\ Example~\ref{exmptwisting}, Definition \ref{defh2cech}, Subsections \ref{ssectspringer} and \ref{ssectptcomm}).

We take this opportunity to write down explicitly the data of a $T''$-representable isomorphism of bands defined between (the refinements of) two nice triples $(T{\rightarrow} S,\mathcal F,\varphi)$ and $(T'{\rightarrow} S,\mathcal F',\varphi')$, where $T''$ is a common refinement of $T,T'$. This is a group sheaf isomorphism $\alpha : \mathcal F\rightarrow\mathcal F'$ such that the following diagram (in which we fix arbitrary representatives $f,f'$ of $\varphi,\varphi'$, respectively)
\begin{equation}\label{eqTrepreq}
\begin{tikzcd}[column sep = huge]
\mathrm{pr}_1^*(\mathcal F_{T''})\arrow[d, "\mathrm{pr}_1^*\alpha"]\arrow[r, "f_{\; T''\times_{\!S\,} T''}"] & \mathrm{pr}_2^*(\mathcal F_{T''})\arrow[d, "\mathrm{pr}_2^*\alpha"]\\
\mathrm{pr}_1^*(\mathcal F'_{T''})\arrow[r, "f'_{\; T''\times_{\!S\,} T''}"] & \mathrm{pr}_2^*(\mathcal F'_{T''})
\end{tikzcd}
\end{equation}
of sheaves over $T''\times_S T''$ commutes up to the action of $\mathrm{pr}_1^*(\mathcal F/\mathrm Z_{\mathcal F})(T''\times_S T'')$ on $\mathrm{pr}_1^*(\mathcal F_{T''})$. Again, we simply write that this difference is an element of $(\mathcal F/\mathrm Z_{\mathcal F})(T''\times_S T'')$.
\end{rem}

\begin{exmp}
A band on $k_\Et$ or $k_\fppf$ is called \textit{algebraic} if it is locally represented by a group sheaf $\mathcal F$ which is an algebraic group. To similarly define an algebraic band on the small \'etale site $k_\et$, we must additionally require that the morphism $\varphi$ is represented by a morphism~of~group schemes (which does not hold in general, since group schemes don't embed as a full subcategory of \'etale sheaves on a small site). This condition is necessary to guarantee that a trivial algebraic band on $k_\et$ is represented by an algebraic group on $k$ (by Galois descent, cf.\ \cite[\textsection 2]{Spr66}).

A band is also often called any adjective which describes a geometric property of the group sheaf $\mathcal F$, in particular when that group sheaf is an algebraic group; for example \textit{smooth}, \textit{reductive} or \textit{Abelian} (meaning ``commutative'' in this context).
\end{exmp}

\begin{rem}
A much more conceptual definition of bands is the one given in \cite[IV, 1.1.5]{Gir71}:

Consider the (split) stack of group sheaves on $\mathcal C$ (fr. \textit{champ scind\'e des faisceaux de groups}) $\mathrm{FAGRSC}$, where each fiber $\mathrm{FAGRSC}(S) = \mathrm{Sh}_\mathrm{Grp}(\mathcal C/S)$ is the category of group sheaves over an object $S\in\mathcal C$ (equipped with canonical morphisms given by pullbacks). The stack $\mathrm{LIEN}$ on $\mathcal C$ is constructed by considering the objects of $\mathrm{FAGRSC}$ with morphisms only up to inner automorphisms (two group sheaf morphisms are identified if they can be related by giving~an~inner automorphism on each side), then stackifying the resulting prestack. A \textit{band} (fr. \textit{lien}) on $S$ is an element of the fiber $\mathrm{LIEN}(S)$.

The (underlying groupoid) category of bands on $S$ in this sense clearly agrees with our explicit construction given in Definition \ref{defband}. We do not consider general morphisms between bands.
\end{rem}

\begin{defn}\label{defh2}
A \textit{gerbe} on $\mathcal C$ is a stack $\mathscr X\rightarrow\mathcal C$ fibered in groupoids (meaning that, for each $R\in\mathcal C$, every morphism in the fiber $\mathscr X(R)$ has an inverse) such that:
\begin{itemize}
    \item $\mathscr X$ is locally nonempty: that is, $\mathscr X(T)\neq\varnothing$ for some covering map $T\rightarrow S$
    \item For every $R\in\mathcal C$, any two objects $a,b\in\mathscr X(R)$ are locally isomorphic: $\mathcal{Isom}_{a,b}(T_i)\neq\varnothing$ for some covering $\{T_i\rightarrow R\}_{i\in I}$ and all $i\in I$
\end{itemize}
A gerbe $\mathscr X$ is said to be \textit{trivial} if $\mathscr X(S)\neq\varnothing$.

To each gerbe we may associate a band as follows: Take a covering map $T\rightarrow S$ and an object $a\in\mathscr X(T)$. The objects $\mathrm{pr}_1^*a,\,\mathrm{pr}_2^*a\in\mathscr X(T\times_S T)$ are locally isomorphic, so we may choose isomorphisms between $\mathrm{pr}_1^*\mathcal{Aut}_a$ and $\mathrm{pr}_2^*\mathcal{Aut}_a$ locally on $T\times_S T$ which (trivially) glue up to local actions of $\mathcal{Aut}_a$. This defines a section $\varphi$ in $\mathcal{Out}_{\mathrm{pr}_1^*\mathcal{Aut}_a,\,\mathrm{pr}_2^*\mathcal{Aut}_a}(T\times_S T)$ and a triple $(T{\rightarrow} S, \mathcal{Aut}_a, \varphi)$ which is representative. Note that this section $\varphi$ is not necessarily represented by an isomorphism over $T\times_S T$; this will be discussed in Remark \ref{rembreen}.

The resulting band is unique up to unique isomorphism (induced by a change of choice of $T$ and $a$) and we denote it by $L(\mathscr X)$. Given a band $L$, a pair $(\mathscr X,\,L\simeq L(\mathscr X))$ is called an $L$\textit{-gerbe}. Usually we omit the isomorphism and simply say that $\mathscr X$ is a gerbe \textit{bound} by the band $L$.

We denote by $\mathrm{H}^2(\mathcal C, L)$ the set of gerbes (up to $L$-isomorphism, in a natural sense) bound by a fixed band $L$; also by $\mathrm{H}^2(S, L)$ if the site $\mathcal C$ is understood. The subset of trivial gerbes will be denoted by $\mathrm{N}^2(S, L)$ (the notation $\mathrm{H}^2(S, L)'$ is sometimes used in literature). This is the \textit{subset of neutral elements} in $\mathrm{H}^2(\mathcal C, L)$.
\end{defn}

\begin{exmp}
Given a group sheaf $\mathcal F$ on $S$, denote by $L(\mathcal F)$ the band globally represented by $(\mathrm{id}_S,\mathcal F,\mathrm{id})$. We often write $\mathrm{H}^2(S, \mathcal F)\coloneqq\mathrm{H}^2(S, L(\mathcal F))$. Moreover, there is a distinguished class $1_{\mathcal F}\in\mathrm{H}^2(S, \mathcal F)$, the class of the gerbe $\mathrm{TORS}(S, \mathcal F)$ of torsors of $\mathcal F$ on $\mathcal C$. Clearly, $1_{\mathcal F}\in\mathrm{N}^2(S, \mathcal F)$ since $\mathcal F$ itself is a $\mathcal F$-torsor in $\mathrm{TORS}(S, \mathcal F)(S)$. Conversely, if a band $L$ has $\mathrm{N}^2(S, L)\neq\varnothing$, then it is globally representable by the automorphism sheaf of any element of $\mathscr X(S)$, for any trivial $L$-gerbe $\mathscr X$.

If $\mathcal F$ is commutative, the set $\mathrm{H}^2(S, \mathcal F)$ is in canonical bijection with the ``usual'' $H^2$ group of $\mathcal F$ (and then $\mathrm{N}^2(S, \mathcal F) = \{1_{\mathcal F}\}$). There is a canonical group structure on the set $\mathrm{H}^2(S, \mathcal F)$ of gerbes (which is both defined and generalized by the proposition below) and it agrees with the usual group structure on $\mathrm{H}^2$. We thus keep the same notation for both instances of the group $\mathrm{H}^2(S, \mathcal F)$. See \cite[IV, \textsection2.5]{Mil80} for more details.
\end{exmp}

Starting from a band $L$ on $S$ represented by a triple $(T{\rightarrow} S, \mathcal F, \varphi)$, the restricted triple $(T{\rightarrow} S, \mathrm Z_{\mathcal F}, \varphi|_{\mathrm{pr}_1^*\mathrm Z_{\mathcal F}})$ defines a descent datum on $\mathrm Z_{\mathcal F}$. This datum represents a group sheaf $\mathrm Z_L$ over $S$ independent of any choices. We call it the \textit{center} of $L$.

\begin{prop}\label{propcentacth2}
There is a natural action of the commutative group $\mathrm{H}^2(S, \mathrm Z_L)$ on $\mathrm{H}^2(S, \mathcal F)$. If $\mathrm{H}^2(S, \mathcal F)$ is nonempty, then this action is both free and transitive.
\end{prop}
\begin{prf}
This is shown in \cite[IV, Thm.\ 3.3.3]{Gir71}. The action is given by the contracted product of gerbes (\cite[IV, \textsection2.4]{Gir71}), and we return to it in Proposition \ref{proph2actcomm}.
\end{prf}

Next, we consider twists of group sheaves by cocycles valued in the automorphism sheaf.~This serves to set up some notation for the remainder of the paper.

\begin{exmp}\label{exmptwisting}
Let $\mathcal F,\mathcal N$ be group sheaves on $S$ and consider a morphism $a : \mathcal N\rightarrow\mathcal{Aut}_{\mathcal F}$ of group sheaves. Let $\mathcal P$ be a sheaf of sets on $S$ such that there is a free and (locally) transitive right action $\mathcal P\times\mathcal N\rightarrow\mathcal P$. In other words, $\mathcal P$ is an $\mathcal N$-torsor with class $[\mathcal P]\in\mathrm{H}^1(S,\mathcal N)$. We consider the quotient sheaf ${_\mathcal P}\mathcal F\coloneqq (\mathcal P\times\mathcal F)/\mathcal N$ under the action given by $(x,f)\mapsto (x.n,\, a(n)^{-1}(f))$.

The sheaf ${_\mathcal P}\mathcal F$ has a natural group law given by (the descent of) $[x, f]\cdot[x, f'] = [x, ff']$. In particular, it is an $S$-form of $\mathcal F$ which becomes isomorphic to $\mathcal F$ over any covering $T\rightarrow S$ such that $\mathcal P(T)\neq\varnothing$, via isomorphisms of the form $[x.n, f]\mapsto a(n)(f)$. The group sheaf ${_\mathcal P}\mathcal F$ depends only on the image of the class $[\mathcal P]\in\mathrm{H}^1(S,\mathcal N)$ in $\mathrm{H}^1(S,\mathcal{Aut}_{\mathcal F})$, up to non-unique isomorphism.\quad

Conversely, any $S$-form $\mathcal F'$ of $\mathcal F$ arises in this way: $\mathcal F'\simeq{_\mathcal P}\mathcal F$ for $\mathcal P = \mathcal{Isom}_{\mathcal F,\mathcal F'}$. The right action of $\mathcal{Aut}_{\mathcal F}$ on $\mathcal{Isom}_{\mathcal F,\mathcal F'}$ is always transitive; it is free if and only if the sheaf $\mathcal{Isom}_{\mathcal F,\mathcal F'}$ is locally nonempty, which is the case if and only if $\mathcal F'$ is an $S$-form of $\mathcal F$. We may in fact find the class $[\mathcal P]\in\mathrm{H}^1(S,\mathcal{Aut}_{\mathcal F})$ as a cocycle:

Take some $\alpha\in\mathcal{Isom}_{\mathcal F,\mathcal F'}(T)$ for a covering map $T\rightarrow S$. For descent data $\varphi_{\mathcal F},\varphi_{\mathcal F'}$ associated to $\mathcal F,\mathcal F'$, there is a unique element $\gamma\in\mathcal{Aut}_{\mathcal F}(T\times_S T)$ such that the following square
\begin{center}\begin{tikzcd}[column sep = huge]
\mathrm{pr}_1^*(\mathcal F_T)\arrow[d, "\mathrm{pr}_1^*\alpha"]\arrow[r, "\varphi_{\mathcal F}\;\circ\;\gamma^{-1}"] & \mathrm{pr}_2^*(\mathcal F_T)\arrow[d, "\mathrm{pr}_2^*\alpha"]\\
\mathrm{pr}_1^*(\mathcal F'_T)\arrow[r, "\varphi_{\mathcal F'}"] & \mathrm{pr}_2^*(\mathcal F'_T)
\end{tikzcd}\end{center}
commutes. Pulling back the square in three different ways, we can piece together the following commutative diagram (writing $\mathrm{p}$ for $\mathrm{pr}$):
\begin{center}\begin{tikzcd}[cramped, column sep = tiny]
\mathrm{p}_{12}^*\mathrm{p}_1^*(\mathcal F_T)\arrow[d, "\mathrm{p}_{12}^*\mathrm{p}_1^*\alpha"]\arrow[rrrrrr, "\mathrm{p}_{12}^*(\varphi_{\mathcal F}\,\circ\,\gamma^{-1\!})"]
&&&&&& \mathrm{p}_{12}^*\mathrm{p}_2^*(\mathcal F_T)\arrow[d, "\mathrm{p}_{12}^*\mathrm{p}_2^*\alpha", swap]\arrow[r, equal]
& \mathrm{p}_{23}^*\mathrm{p}_1^*(\mathcal F_T)\arrow[d, "\mathrm{p}_{23}^*\mathrm{p}_1^*\alpha"]\arrow[rrrrrr, "\mathrm{p}_{23}^*(\varphi_{\mathcal F}\,\circ\,\gamma^{-1\!})"]
&&&&&& \mathrm{p}_{23}^*\mathrm{p}_2^*(\mathcal F_T)\arrow[d, "\mathrm{p}_{23}^*\mathrm{p}_2^*\alpha", swap]\arrow[r, equal]
& \mathrm{p}_{13}^*\mathrm{p}_2^*(\mathcal F_T)\arrow[d, "\mathrm{p}_{13}^*\mathrm{p}_2^*\alpha"]\arrow[rrrrrr, "\;\;\mathrm{p}_{13}^*(\varphi_{\mathcal F}\,\circ\,\gamma^{-1\!})^{-1\!}\!\!"]
&&&&&& \mathrm{p}_{13}^*\mathrm{p}_1^*(\mathcal F_T)\arrow[d, "\mathrm{p}_{13}^*\mathrm{p}_1^*\alpha", swap]
\\
\mathrm{p}_{12}^*\mathrm{p}_1^*(\mathcal F'_T)\arrow[rrrrrr, "\mathrm{p}_{12}^*\varphi_{\mathcal F'}"]
&&&&&& \mathrm{p}_{12}^*\mathrm{p}_2^*(\mathcal F'_T)\arrow[r, equal]
& \mathrm{p}_{23}^*\mathrm{p}_1^*(\mathcal F'_T)\arrow[rrrrrr, "\mathrm{p}_{23}^*\varphi_{\mathcal F'}"]
&&&&&& \mathrm{p}_{23}^*\mathrm{p}_2^*(\mathcal F'_T)\arrow[r, equal]
& \mathrm{p}_{13}^*\mathrm{p}_2^*(\mathcal F'_T)\arrow[rrrrrr, "\mathrm{p}_{13}^*\varphi_{\mathcal F'}^{-1}"]
&&&&&& \mathrm{p}_{13}^*\mathrm{p}_1^*(\mathcal F'_T)
\end{tikzcd}\end{center}
\bigskip

\noindent
Since the leftmost and rightmost vertical arrows coincide, and the bottom row is simply the identity map, we conclude that the top row is also the identity. Applying the cocycle property for $\varphi_{\mathcal F}$, we may compute:
$$\mathrm{pr}_{13}^*\gamma = \mathrm{pr}_{12}^*\gamma\circ\mathrm{pr}_{12}^*\varphi_{\mathcal F}^{-1}\circ\mathrm{pr}_{23}^*\gamma\circ\mathrm{pr}_{12}^*\varphi_{\mathcal F} = \mathrm{pr}_{12}^*\gamma\circ(\mathrm{pr}_{12}^*\,\varphi_{\mathcal{Aut}_{\mathcal F}}^{-1})(\mathrm{pr}_{23}^*\gamma)$$
Here we have used that the descent datum $\varphi_{\mathcal{Aut}_{\mathcal F}}$ of the automorphism sheaf is exactly given by $\mathrm{int}(\varphi_{\mathcal F}) : \gamma'\longmapsto\varphi_{\mathcal F}\circ\gamma'\circ\varphi_{\mathcal F}^{-1}$. Finally, this shows that $\gamma$ is a cocycle: $\gamma\in\mathrm{\check Z}^1(T/S,\mathcal{Aut}_{\mathcal F})$

The class $[\gamma]\in\mathrm{H}^1(S,\mathcal{Aut}_{\mathcal F})$ agrees with the class $[\mathcal P]$. If $\mathcal N\hookrightarrow\mathcal{Aut}_{\mathcal F}$ is a subsheaf and $\gamma\in\mathcal N(T\times_S T)$ inside $\mathcal{Aut}_{\mathcal F}(T\times_S T)$, then we see that $\mathcal F'\simeq{_\mathcal P}\mathcal F$ can even be constructed by starting from some $\mathcal N$-torsor $\mathcal P$ such that $[\mathcal P] = [\gamma]\in\mathrm{H}^1(S,\mathcal N)$.
\end{exmp}

\begin{defn}\label{definntwist}
Recall that an $S$-form of $\mathcal F$ is called an \textit{inner form} if it is constructed as above for $\mathcal N = \mathcal F/\mathrm Z_{\mathcal F}\hookrightarrow\mathcal{Aut}_{\mathcal F}$. It is moreover a \textit{pure inner form} if it can be constructed via the map $\mathcal N = \mathcal F\rightarrow\mathcal{Aut}_{\mathcal F}$. Such forms are hence parametrized by the natural image of $\mathrm{H}^1(S,\mathcal F/\mathrm Z_{\mathcal F})$ (resp. $\mathrm{H}^1(S,\mathcal F)$) in $\mathrm{H}^1(S,\mathcal{Aut}_{\mathcal F})$.
\end{defn}

We state the following proposition in a bit more generality than is usual, for later reference:\;\;

\begin{prop}\label{propparambandforms}
Suppose a band $L$ on $S$ is represented by a nice triple $(T{\rightarrow} S, \mathcal F, \varphi)$.~Let $\mathscr S(L,T)$ denote the set of nice triples of the form $(T{\rightarrow} S, \mathcal F', \varphi')$ representing bands isomorphic to $L$, up to $T$-representable equivalence. Then there is a canonical injection of $\mathscr S(L,T)$ into the set $\mathscr{I}(\mathcal F,T)\coloneqq\im(\mathrm{H}^1(T,\mathcal F/\mathrm Z_{\mathcal F})\rightarrow\mathrm{H}^1(T,\mathcal{Aut}_{\mathcal F}))$ of inner forms of $\mathcal F$ over $T$.
\end{prop}
\begin{prf}
For a nice triple $(T{\rightarrow} S, \mathcal F', \varphi')$ representing a band isomorphic to $L$, choose a refinement $T'\rightarrow S$ of $T\rightarrow S$ over which it is $T'$-equivalent to $(T{\rightarrow} S, \mathcal F, \varphi)$ through $\alpha : \mathcal F_{T'}\rightarrow\mathcal F'_{T'}$. We fix isomorphisms $f,f'$ over $T\times_S T$ lifting $\varphi,\varphi'$, respectively. Up to replacing $f$ by $f\circ\mathrm{pr}_1^*(f_T^{-1})$ (which amounts to a $T$-representable equivalence, as $f_T\in(\mathcal F/\mathrm Z_{\mathcal F})(T)$; cf.\ Definition \ref{defnice}), we may assume that $f_T = \mathrm{id}_{\mathcal F}$ and, similarly, $f'_T = \mathrm{id}_{\mathcal F'}$. 
Consider the commutative diagram:
\begin{center}\begin{tikzcd}
    T'\times_T T'\arrow[r]\arrow[d] & T\arrow[d, "\Delta"]\\
    T'\times_S T'\arrow[r] & T\times_S T\\
\end{tikzcd}\end{center}
By base-changing the defining diagram \eqref{eqTrepreq} of $T'$-equivalence from $T'\times_S T'$ to $T'\times_T T'$, we get the following square which commutes up to the action of $(\mathcal F/\mathrm Z_{\mathcal F})(T'\times_T T')$:
\begin{center}\begin{tikzcd}[column sep = large]
\mathrm{pr}_1^*(\mathcal F_{T'})\arrow[d, "\mathrm{pr}_1^*\alpha"]\arrow[rr, "\big(f_{\; T'\times_{\!S\,} T'}\big)_{\; T'\times_{\!T\,} T'}"] & & \mathrm{pr}_2^*(\mathcal F_{T'})\arrow[d, "\mathrm{pr}_2^*\alpha"]\\
\mathrm{pr}_1^*(\mathcal F'_{T'})\arrow[rr, "\big(f'_{\; T'\times_{\!S\,} T'}\big)_{\; T'\times_{\!T\,} T'}"] & & \mathrm{pr}_2^*(\mathcal F'_{T'})
\end{tikzcd}\end{center}
The top arrow agrees with $(f_T)_{ T'\times_{\!T\,} T'} = (\mathrm{id}_{\mathcal F})_{ T'\times_{\!T\,} T'} = \widetilde{\varphi}_{\mathcal F}$, where $\widetilde{\varphi}_{\mathcal F}$ denotes the descent data of $\mathcal F$ with respect to $T'\rightarrow T$. Similarly for $\mathcal F'$ and $f'$. By the preceding example, this square determines a cocycle $\gamma\in\mathrm{\check Z}^1(T'/T,\,\mathcal F/\mathrm Z_{\mathcal F})$.

Replacing $\alpha$ by $\alpha\circ\alpha_0^{-1}$ for any $\alpha_0\in\mathcal{Aut}_{\mathcal F}(T')$ preserves the resulting class $[\gamma]\in\mathrm{H}^1(T,\mathcal{Aut}_{\mathcal F})$ (it replaces $\gamma$ by $\mathrm{pr}_1^*\alpha_0^{-1}\circ\gamma\circ\widetilde{\varphi}_{\mathcal{Aut}_{\mathcal F}}^{-1}(\mathrm{pr}_2^*\alpha_0)$), as does refining $T'$. Finally, changing $(T{\rightarrow} S, \mathcal F', \varphi')$ by a $T$-equivalent triple commutes with the formation of the descent data $\widetilde{\varphi}'_{\mathcal F}$, hence the desired function $\mathscr S(L,T)\rightarrow\mathscr I(\mathcal F,T)$ is well-defined.

To see that it is an injection, first observe that we only need to show that the fiber of the distinguished element in $\mathscr I(\mathcal F,T)\subseteq\mathrm{H}^1(T,\mathcal{Aut}_{\mathcal F})$ is trivial. Indeed, if we then apply this weaker statement to every inner form of $\mathcal F$, we will have completed the proof for $\mathcal F$. 
Thus it suffices to show that every triple $(T{\rightarrow} S,\mathcal F',\varphi')$ for which the constructed cocycle $\gamma$ is a boundary must~be $T$-equivalent to $(T{\rightarrow} S,\mathcal F,\varphi)$. Fix $\alpha : \mathcal F_{T'}\rightarrow\mathcal F'_{T'}$ and suppose that 
$$\gamma = \mathrm{pr}_1^*\alpha_0^{-1}\circ\widetilde{\varphi}_{\mathcal{Aut}_{\mathcal F}}^{-1}(\mathrm{pr}_2^*\alpha_0) = \mathrm{pr}_1^*\alpha_0^{-1}\circ\widetilde{\varphi}_{\mathcal F}^{-1}\circ\mathrm{pr}_2^*\alpha_0\circ\widetilde{\varphi}_{\mathcal F}$$
for some $\alpha_0\in\mathcal{Aut}_{\mathcal F}(T')$. This immediately implies that $\alpha\circ\alpha_0^{-1}$ commutes with the descent data and descends to a map $\mathcal F\rightarrow\mathcal F'$ over $T$. The definition of $T$-equivalence is then automatically satisfied by descent from $T'$.
\end{prf}

\begin{cor}
Two group sheaves $\mathcal F,\mathcal G$ on $S$ define the same band $L(\mathcal F) = L(\mathcal G)$ if and only if they are inner forms of each other.
\end{cor}
\begin{prf}
In the notation of the above proposition, $\mathscr S(L,S)\hookrightarrow\mathscr I(\mathcal F,S)$, and we only need to show that this map is a surjection. However, it is obvious that any inner form of $\mathcal F$ defines the same band as $\mathcal F$, since they are isomorphic over some covering and this isomorphism descends up to inner automorphisms.
\end{prf}

\begin{prop}\label{propnoncommles}
Let $\mathcal F$ be a group sheaf on $S$. There is a canonical map $\delta :\mathrm{H}^1(S,\mathcal F/\mathrm Z_{\mathcal F})\rightarrow\mathrm{H}^2(S,\mathrm Z_{\mathcal F})$ such that the following sequence
$$\cdots\rightarrow\mathrm{H}^1(S,\mathrm Z_{\mathcal F})\rightarrow\mathrm{H}^1(S,\mathcal F)\rightarrow\mathrm{H}^1(S,\mathcal F/\mathrm Z_{\mathcal F})\rightarrow\mathrm{H}^2(S,\mathrm Z_{\mathcal F})$$
of pointed sets is exact.
\end{prop}
\begin{prf}
Let $\mathcal P$ be a torsor of $\mathcal F/\mathrm Z_{\mathcal F}$. Then we let $\delta(\mathcal P)$ be the gerbe $\mathscr X\in\mathrm{H}^2(S,\mathrm Z_{\mathcal F})$ for which every fiber $\mathscr X(T)$ consists of pairs $(\mathcal P',a)$ where $\mathcal P'$ is an $\mathcal F$-torsor over $T$ and $a : \mathcal P'\times^{\mathcal F}(\mathcal F/\mathrm Z_{\mathcal F})\xrightarrow{\sim}\mathcal P_T$ an isomorphism of $(\mathcal F/\mathrm Z_{\mathcal F})$-torsors over $T$. Morphisms in $\mathscr X$ are defined in the obvious way. Here, $\mathcal P'\times^{\mathcal F}(\mathcal F/\mathrm Z_{\mathcal F})$ denotes the contracted product of $\mathcal P'$ and $(\mathcal F/\mathrm Z_{\mathcal F})$ seen as a $(\mathcal F/\mathrm Z_{\mathcal F})$-torsor over $T$. It's clear that $\mathcal P$ lifts to an $\mathcal F$-torsor over $S$ if and only if $\delta(\mathcal P)$ is trivial (that is, $\mathscr X(S)$ is nonempty).
The exactness of the rest of the sequence is well-known (a useful account is given in \cite[\textsection B.3]{Con12}).
\end{prf}

We now give a description of the subset of neutral elements for a group sheaf on $S$. This description underpins the abelianization theory in Section \ref{sectabel}, which has consequences even for bands which are not representable.

\begin{cor}\label{corneutelemdesc}
Let $\mathcal F$ be a group sheaf on $S$. Consider the bijection $\mathrm{H}^2(S,\mathrm Z_{\mathcal F})\rightarrow\mathrm{H}^2(S,\mathcal F)$ defined by taking $[\mathscr X]$ to $[\mathscr X].1_{\mathcal F}$. The image of the subset $\mathrm{im}(\delta)\subseteq\mathrm{H}^2(S,\mathrm Z_{\mathcal F})$ is exactly $\mathrm{N}^2(S,\mathcal F)$. 
\end{cor}
\begin{prf}
This becomes clear if we allow inner twists of the above sequence. See \cite[1.2.10]{DD99}~for details; the fact can also be deduced from the results in \cite[IV, \textsection3.2]{Gir71}.
\end{prf}

%% file: part1-1.tex
\subsection{\'Etale and Separable Bands}\label{ssectetandsep}
Let $\mathcal C\in\{k_\Et,k_\fppf\}$ in this subsection. The notion of a band on the \'etale site of a field has been reinterpreted using Galois cohomology in \cite{Spr66}, \cite{Brv93} and finally \cite{FSS98}. The equivalence of this definition with the general one is shown in \cite{DLA19}, and in this section we provide a similar discussion (with the goal of extending the Galois formulation to a certain class of \'etale-locally represented fppf bands). Encoded in the content of the above cited papers is the following proposition, which we now make explicit:

\begin{prop}\label{propetrefin}
Let $(k'/k,\mathcal F,\varphi)$ be a representative triple such that $k'$ is an \'etale $k$-algebra. Suppose that there exists an \'etale covering $R\rightarrow\Spec(k'\otimes_k k')$ such that $\varphi_R$ is contained inside the image of the map:
$$\mathcal{Isom}_{\mathrm{pr}_1^*\mathcal F,\,\mathrm{pr}_2^*\mathcal F}(R)\longrightarrow\mathcal{Out}_{\mathrm{pr}_1^*\mathcal F,\,\mathrm{pr}_2^*\mathcal F}(R)$$
Then $(k'/k,\mathcal F,\varphi)$ is refined by a nice triple $(k''/k,\mathcal F,\varphi)$, where $k''/k$ is a finite separable extension of fields.
\end{prop}
\begin{prf}
An \'etale $k$-algebra is a disjoint union of finite separable extensions of $k$, so, by possibly refining $(k'/k,\mathcal F,\varphi)$, we may assume that $k'/k$ is a finite Galois extension. Write the discrete set $\Spec(k'\otimes_k k')$ as $\bigsqcup_{t\in\Gal(k'/k)}\Spec(k')^{(t)}$. Then in particular $R = \bigsqcup_{t\in\Gal(k'/k)} R^{(t)}$ with \'etale maps $p^{(t)} : R^{(t)}\rightarrow\Spec(k')^{(t)}$.

Choose a large enough finite Galois extension $k''/k'$ such that each $R^{(t)}$ admits a $k''$-point over $\Spec(k')^{(t)}$. If we now write $\Spec(k''\otimes_k k'') = \bigsqcup_{s\in\Gal(k''/k)}\Spec(k'')^{(s)}$, then we can construct a factorization of the natural map $\Spec(k''\otimes_k k'')\rightarrow\Spec(k'\otimes_k k')$ through $R$ by fixing (any choice of) compositions of the form $\Spec(k'')^{(s)}\rightarrow R^{(\overline{s})}\rightarrow\Spec(k')^{(\overline{s})}$ for each $s\in\Gal(k''/k)$, where $\overline{s}$ is the natural image of $s$ in $\Gal(k'/k)$. This shows that the obvious refinement $(k''/k,\mathcal F,\varphi)$ of the starting triple is nice.
\end{prf}

\begin{defn}\label{defetlocrep}
We will say that a band on $\mathcal C$ is \textit{nicely representable} if it can be represented by a nice triple. Note that in this situation every representative triple can be refined by a nice triple (because any two coverings of $k$ have a common refinement).

In view of the above proposition, we will say that a band on $\mathcal C$ is \textit{\'etale-locally representable} if it can be represented by a nice triple over a finite separable extension $k'/k$. If $\mathcal C$ is the \'etale site of $k$, then of course every band has this property. Otherwise, note that an isomorphism of two \'etale-locally representable bands does not have to be defined over a separable extension~of~$k$.
\end{defn}

\begin{exmp}\label{exmpsmoothisnice}
A band on $k_\fppf$ represented by $(k'/k,\Gb,\varphi)$ for a locally algebraic $k'$-group $\Gb$ is nicely representable: For all $k''/k'$ there are exact sequences in which the last set is pointed
$$\mathcal{Isom}_{\mathrm{pr}_1^*\Gb,\,\mathrm{pr}_2^*\Gb}(k''\otimes_k k'')\longrightarrow\mathcal{Out}_{\mathrm{pr}_1^*\Gb,\,\mathrm{pr}_2^*\Gb}(k''\otimes_k k'')\longrightarrow\mathrm H^1(k''\otimes_k k'',\,\Gb/\mathrm Z_{\Gb})$$
so it suffices to show that $\mathrm H^1(\overline k\otimes_k\overline k,\Gb/\mathrm Z_{\Gb}) = 1$ (see \cite[Thm.\ 2.1]{Mar07} for this limit argument). Since $\Gb$ is locally algebraic, this holds by Lemma \ref{lemskewvanish}.
\end{exmp}

If a covering $\Spec(K)\rightarrow\Spec(k)$ in $\mathcal C$ comes from a finite Galois extension $K/k$, then fppf descent along this map can be formulated as Galois descent. Similarly, a (nice) representative triple $(K/k,\mathcal F,\varphi)$ can be given an equivalent Galois-theoretic definition. This is best expressed in the language of semiautomorphisms, which we now recall:

\begin{defn}
Let $K/k$ be a (possibly infinite) Galois extension with Galois group $\Gamma^K$. Let $\mathcal F$ be a group sheaf on $K$ (on the $K_\Et$ or $K_\fppf$ site, respectively). A $(K/k)$-\textit{semiautomorphism} of $\mathcal F$ is a pair $(\alpha,s)$ for $s\in\Gamma^K$ and $\alpha\in\mathrm{Isom}(s_*\mathcal F,\,\mathcal F)$, where $s_*\mathcal F$ denotes the pullback of $\mathcal F$ by $\Spec(s)$. The composition rule 
$$(\alpha,s)\circ(\beta,t)\coloneqq (\alpha\circ s_*\beta,s\circ t) \;:\; (s\circ t)_*\mathcal F\cong s_*(t_*\mathcal F)\xrightarrow{s_*\beta} s_*\mathcal F\xrightarrow{\alpha}\mathcal F$$
makes the set $\saut(\mathcal F/k)$ of $(K/k)$-semiautomorphisms of $\mathcal F$ into a group. Moreover, there~is~an obvious left exact sequence (which is in general not right exact):
\begin{equation}\label{eqdefsemiaut}
1\longrightarrow\aut(\mathcal F)\xrightarrow{\;\alpha\,\mapsto\,(\alpha,\mathrm{id})\;}\saut(\mathcal F/k)\xrightarrow{\;(\alpha,s)\,\mapsto\,s\;}\Gal(K/k)    
\end{equation}
Finally, the construction of pullbacks gives $(s_*\mathcal F)(K) = \mathcal F(K)$. Each $(K/k)$-semiautomorphism $\alpha$ therefore induces a map $\alpha_{K} : \mathcal F(K)=(s_*\mathcal F)(K)\rightarrow\mathcal F(K)$. This is a well-defined action of $\saut(\mathcal F/k)$ on $\mathcal F(K)$.
\end{defn}

\begin{rem}\label{remspringersplit}
In \cite{Spr66}, Springer works over the small \'etale site $k_\et$ on which a group sheaf is simply a Galois module. In that situation and for $K = k_s$, $\saut(\mathcal F/k)\rightarrow\Gal(k_s/k)$ admits a canonical homomorphic section (given by the trivial action of $\Gal(k_s/k)$, which corresponds to the constant sheaf on $k_\et$). In particular, then $\saut(\mathcal F/k) = \aut(\mathcal F)\times\Gal(k_s/k)$ as groups. Over more general sites, such a section need not exist.

Note also that, when $\mathcal F$ is represented by some group scheme $\Gb$ on $K$, the pullback $s_*\Gb$ is the same as the usual pullback by the morphism $\Spec(s)$. The following diagram illustrates that the data of a semiautomorphism $\alpha$ is, by the defining property of pullbacks (applied here to $\alpha^{-1}$), the same as that of an isomorphism $\Gb\rightarrow\Gb$ ``lying over'' $s^*\coloneqq\Spec(s)^{-1} : \Spec(K)\rightarrow\Spec(K)$.
\begin{center}\begin{tikzcd}[row sep = tiny]
    & s_*\Gb\arrow[ld, "\alpha", swap]\arrow[r]\arrow[ddd, "s_*p", pos=0.8]\arrow[rddd, phantom, "\!\!\ulcorner\;\;", very near start] & \Gb\arrow[ddd, "p"]\\
    \Gb\arrow[rdd, "p", swap]\arrow[rru, dashed, bend right=20] \\
    \; \\
    & \Spec(K)\arrow[r, "\Spec(s)"] & \Spec(K)\\
\end{tikzcd}\end{center}
This is an isomorphism of $k$-schemes. Looking at $\Gb$ as a functor of points on $k$-algebras,~the~map $\alpha_K : \Gb(K)\rightarrow\Gb(K)$ is given by taking the $k$-morphism $j : \Spec(K)\rightarrow\Gb$ to $j\circ s^*$.
\end{rem}

Consider the inclusions $(\mathcal F/\mathrm Z_{\mathcal F})(K)\subseteq\aut(\mathcal F)\subseteq\saut(\mathcal F/k)$. The sequence \eqref{eqdefsemiaut} induces~the following exact sequence for $(K/k)$-semiautomorphisms, important in the example below:
\begin{equation}\label{eqdefsemiout}
1\longrightarrow\frac{\aut(\mathcal F)}{(\mathcal F/\mathrm Z_{\mathcal F})(K)}\longrightarrow\frac{\saut(\mathcal F/k)}{(\mathcal F/\mathrm Z_{\mathcal F})(K)}\longrightarrow\Gal(K/k)    
\end{equation}
We will not introduce the usual notation $\out$ and $\sout$ for these groups as we want to always keep in mind the formation of the quotient $\mathcal F/\mathrm Z_{\mathcal F}$ (on the \'etale, resp. fppf site).

\begin{exmp}\label{exmpgalband}
Suppose given a band $L$ on $\mathcal C$ admitting a nice triple $(K/k,\mathcal F,\varphi)$ for a finite Galois extension $K/k$. There is a natural ring isomorphism $K\otimes_k K\cong \prod_{\sigma\in\Gal(K/k)}K$ taking $a\otimes b$ to $(a\,\sigma(b))_{\sigma}$. This gives that $\Spec(K)\times_{\Spec(k)}\Spec(K)\cong\Spec(K\otimes_k K)\cong\bigsqcup_{\sigma}\Spec(K)$ and there is a commutative diagram:
\begin{center}\begin{tikzcd}
\bigsqcup_{\sigma\in\Gal(K/k)}\Spec(K)\arrow[d, "{\rotatebox{90}{$\sim$}}", pos=0.4] & & \bigsqcup_{\sigma\in\Gal(K/k)}\Spec(K)\arrow[d, "{\rotatebox{90}{$\sim$}}", pos=0.4]\arrow[ll, swap,  "\bigsqcup_{\sigma}\Spec(\sigma)"]\\
\Spec(K)\times_{\Spec(k)}\Spec(K)\arrow[r, "\mathrm{pr}_1"] & \Spec(K) & \Spec(K)\times_{\Spec(k)}\Spec(K)\arrow[l, "\mathrm{pr}_2", swap]
\end{tikzcd}\end{center}
Thus $\mathrm{Isom}(\mathrm{pr}_1^*\mathcal F, \mathrm{pr}_2^*\mathcal F)$ is in canonical bijection with the sections of $\saut(\mathcal F/k)\rightarrow\Gal(K/k)$, not necessarily homomorphisms. Because the given triple is nice, we may choose such a section $f$ which represents a lift of $\varphi$. The difference of any two such sections takes values in $(\mathcal F/\mathrm Z_{\mathcal F})(K)$. Therefore $\varphi$ canonically defines the following section of the sequence \eqref{eqdefsemiout}
$$\kappa : \Gal(K/k)\longrightarrow\frac{\saut(\mathcal F/k)}{(\mathcal F/\mathrm Z_{\mathcal F})(K)}$$
which is moreover a homomorphism by the cocycle condition on $\varphi$. Conversely, a homomorphic section $\kappa$ for a given pair $(K,\mathcal F)$ defines a nice triple $(K,\mathcal F,\varphi)$, since $\kappa$ can always (noncanonically) be lifted to some section $f$.

Finally, if a section $f$ lifting $\kappa$ is a homomorphism (of $\Gal(K/k)$ into $\saut(\mathcal F/k)$), then it defines a $k$-form $\mathcal F_0$ of $\mathcal F$ by Galois descent and $L(\mathcal F_0)\simeq L$. If $L$ is an algebraic band, then $\mathcal F_0$ is representable by an algebraic group (recall that this is specific to Galois descent; in the case of fppf descent for algebraic, but non-affine $L$, the sheaf $\mathcal F_0$ could fail to be representable).
\end{exmp}

For the purpose of working with equivalence classes of triples defined over arbitrarily large finite Galois extensions, it is easiest to define them directly over the separable closure $k_s$ of $k$ with a ``continuity condition'' (that ensures that a representative of the class can always be found over a finite extension $K/k$). This will lead us to an alternative definition of bands with which we will be working for most of this paper.

From now on, all semiautomorphisms will be $(k_s/k)$-semiautomorphisms. Moreover, we restrict ourselves to working with algebraic bands; this condition seems essentially indispensable in the proof (direction (3)$\Rightarrow$(4)) of the crucial Proposition \ref{propcont} below -- even if we assume that a sheaf $\mathcal F$ over $k_s$ admits some $K$-form $\mathcal F_0$ for $K/k$ finite, a given (semi)automorphism of $\mathcal F$ can in general not be descended to any such $\mathcal F_0$.

\begin{defn}\label{defcont}
Let $\Gb$ be an algebraic group defined over $k_s$. Then it admits a $K$-form $G_0$ for some finite separable extension $K/k$. If we write $\Gamma_K\coloneqq\Gal(k_s/K)$, the group $G_0$ defines a section $\sigma : \Gamma_K\rightarrow\saut(\Gb/k)$ which is a homomorphism.

Let $f$ be a section (not necessarily a homomorphism) of $\saut(\Gb/k)\rightarrow\Gal(k_s/k)$. It is said to be \textit{continuous} (with respect to $K$ and $G_0$) if, for every $s\in\Gamma$, the map
$$\Gamma_K\rightarrow\aut(\Gb),\;\;\;\;
t\mapsto\sigma_t^{-1}f_s^{-1}f_{st}$$
is locally constant, with respect to the Krull topology on $\Gamma_K$. Note that this notation suppresses the base changes in the definition of composition of semiautomorphisms.
\end{defn}

This notion of continuity is taken from \cite[1.10]{FSS98}, where it is implicit that the definition is independent of the choice of $K$ and $G_0$. We explore this condition and deduce its equivalence to an (a priori stronger) statement which is the key to Galois descent along infinite extensions:

\begin{prop}\label{propcont}
With notations as in Definition \ref{defcont}, let $f : \Gamma\rightarrow\saut(\Gb/k)$ be a section of $\saut(\Gb/k)\rightarrow\Gal(k_s/k)$. The following conditions are equivalent:
\begin{enumerate}[\hspace{0.5 cm}(1)]
    \item $f$ is continuous with respect to some finite separable $K/k$ (and some $G_0$)
    \item $f$ is continuous with respect to all finite separable $K/k$ (and any $G_0$)
\end{enumerate}
In particular, we may just say that $f$ is continuous (suppressing mention of $K$ or $G_0$).
\begin{enumerate}[\hspace{0.5 cm}(1)]
\setcounter{enumi}{2}
    \item There exists a finite subextension $K/k$ of $k_s/k$ such that $f_{st} = f_s\sigma_t$ for all $s\in\Gamma$, $t\in\Gamma_K$
    \item The section $s\mapsto f_{s^{-1}}^{-1}$ is continuous
\end{enumerate}
In view of the last two statements, for any continuous $f$ there exists a finite subextension $K/k$ of $k_s/k$ (and a $K$-form $G_0$ of $\Gb$) such that $f_{st} = f_s\sigma_t$ and $f_{ts} = \sigma_tf_s$ for all $s\in\Gamma$, $t\in\Gamma_K$.
\end{prop}
\begin{proof}
Implication (2)$\Rightarrow$(1) is obvious. We now prove (1)$\Rightarrow$(3): For this, suppose that (1) holds with respect to some $K/k$. Then by definition of the Krull topology, we may choose for every $s\in\Gamma$ a finite extension $K(s)$ of $K$ such that the function $t\mapsto\sigma_t^{-1}f_s^{-1}f_{st}$ is constant on $\Gamma_{K(s)}\subseteq\Gamma_K$. In particular, it's equal to $1$, since $\sigma_1^{-1}f_s^{-1}f_{s1} = f_s^{-1}f_{s} = 1$.

Cover $\Gamma$ by the open sets $s\Gamma_{K(s)}$. Since $\Gamma$ is compact, we can find a finite subcover $s_j\Gamma_{K(s_j)}$. Finally, define a finite extension $K'/K$ as the compositum of all $K(s_j)$, so that $\Gamma_{K'}\subseteq\Gamma_{K(s_j)}$.
Now, given $s\in\Gamma$ and $t\in\Gamma_{K'}$, let $s' = s_j$ be such that $s\in s'\Gamma_{K(s')}$. Then $s = s't'$ and also $t, t'\in\Gamma_{K(s')}$, so that:
$$f_{st} = f_{s't't} = f_{s'}\sigma_{t't} = f_{s'}\sigma_{t'}\sigma_t = f_{s't'}\sigma_t = f_{s}\sigma_t$$

For the proof of (3)$\Rightarrow$(2), we use the algebraicity of $\Gb$. Fix $K$ and $G_0$ and assume that the property from (3) holds for some $K''/K$ and some $K''$-form $G_0''$ of $\Gb$. By enlarging $K''$, we may assume that $G_0''\simeq G_{0,\,k''}$. Now it suffices to prove that, given $s\in\Gamma$ and $t\in\Gamma_K$, the equality
$$\sigma_t^{-1}f_s^{-1}f_{st} = \sigma_{tt'}^{-1}f_s^{-1}f_{stt'}$$
holds for all $t'\in\Gamma_{K'}$ in some finite extension $K'/K''$. By $f_{stt'} = f_{st}\sigma_{t'}$ and $\sigma_{tt'} = \sigma_{t}\sigma_{t'}$, we get:
$$\sigma_{tt'}^{-1}f_s^{-1}f_{stt'} = \sigma_{t'}^{-1}\sigma_t^{-1}f_s^{-1}f_{st}(f_{st}^{-1}f_{stt'}) = \sigma_{t'}^{-1}(\sigma_t^{-1}f_s^{-1}f_{st})\sigma_{t'}$$
Finally, if $K'$ is such that the automorphism $\sigma_t^{-1}f_s^{-1}f_{st}\in\aut(\Gb)$ is defined for $G_{K'}$, then this automorphism commutes with $\sigma_{t'}$ for $t'\in\Gamma_{K'}$ by definition, which is what we wanted.

It remains to prove (3)$\Rightarrow$(4), as the converse will then hold by symmetry. Fix $K/k$ such that $f_{st} = f_s\sigma_t$ for all $s\in\Gamma$, $t\in\Gamma_K$. The continuity of $s\mapsto f_{s^{-1}}^{-1}$ amounts to saying that, for an arbitrary fixed $s$, the function 
$t\longmapsto\sigma_t^{-1}\big( f_{s^{-1}}^{-1}\big)^{-1}f_{(st)^{-1}}^{-1}$
is locally constant. Equivalently, we may consider the same expression with $s^{-1}$ substituted for $s$:
$$t\longmapsto\sigma_t^{-1}\left( f_s^{-1}\right)^{-1}f_{t^{-1}s}^{-1} = \sigma_{t^{-1}}f_sf_{s(s^{-1}ts)^{-1}}^{-1} = \sigma_{t^{-1}}f_s\left(f_s\sigma_{s^{-1}ts}^{-1}\right)^{-1} = \sigma_{t^{-1}}f_s\,\sigma_{s^{-1}ts}f_s^{-1}$$
We claim that, for a sufficiently large finite extension $K(s)/K$, this function sends all $t\in\Gamma_{K(s)}$ to $\mathrm{id}_\Gb$. This is moreover equivalent to $f_{ts} = \sigma_t f_s$ for all $t\in\Gamma_{K(s)}$. By a compactness argument as in (1)$\Rightarrow$(3) taken over all $s\in\Gamma$, this claim implies that there is some finite $K'/K$ such that condition (3) holds for $s\mapsto f_{s^{-1}}^{-1}$, which will end the proof.

This claim is in fact equivalent to saying that the map $f_s$ is defined over some Galois extension $K(s)/k$: Indeed, if we write $\overline{s}$ for the image of $s$ in $\Gal(K(s)/k)$, then $f_s : s_*\Gb\rightarrow\Gb$ descends to a (unique) map $\overline{f}_s : \overline{s}_*G_{0,\,K(s)}\rightarrow G_{0,\,K(s)}$ if and only if the following diagram commutes for all $t\in\Gamma_{K(s)}$ (which corresponds to the function in the claim being trivial):
\begin{center}\begin{tikzcd}
    & s_*\Gb\arrow[r, "f_s"] & \Gb\\
    s_*(s^{-1}ts)_*\Gb\arrow[ur, "s_*(\sigma_{s^{-1}ts})"]\arrow[r, equal] & t_*s_*\Gb\arrow[u]\arrow[r, "t_*f_s"] & t_*\Gb\arrow[u, "\sigma_t"]
\end{tikzcd}\end{center}
As $G_{0,\,K(s)}$ and $\overline{s}_*G_{0,\,K(s)}$ are algebraic groups, $\overline{f}_s$ must be defined over a sufficiently large finite extension $K(s)/k$. Then the diagram commutes, which proves the claim.
\end{proof}

\begin{exmp}\label{exmpgalband2}
Returning to Example \ref{exmpgalband} and the pair $(\mathcal F,\kappa)$, we will assume now that $\mathcal F$ is represented by an algebraic group $G$. We will keep the expression $\mathcal F/\mathrm Z_{\mathcal F}$ for the sheaf quotient taken in the topology of $\mathcal C$ (so if $\mathcal C = k_\Et$, then $(\mathcal F/\mathrm Z_{\mathcal F})(k_s) = \mathcal F(k_s)/\mathrm Z_{\mathcal F}(k_s)$). The pair $(G,\kappa)$ over $K$ defines by base change a pair $(\Gb,\widetilde{\kappa})$ over $k_s$, where $\Gb = G_{k_s}$ and the homomorphism 
$$\widetilde{\kappa} : \Gal(k_s/k)\longrightarrow\frac{\saut(\Gb/k)}{(\mathcal F/\mathrm Z_{\mathcal F})(k_s)}$$
is induced by $\widetilde{f} : \Gal(k_s/k)\rightarrow\saut(\Gb/k)$ such that $\widetilde{f}_s = f_{\overline{s}}\times_{\overline{s}^*} s^*$ for any $s\in\Gal(k_s/k)$ and its image $\overline{s}\in\Gal(K/k)$. Indeed, the element $\widetilde{f}_{st}^{-1}\widetilde{f}_s\widetilde{f}_t = (f_{\overline{st}}^{-1}f_{\overline{s}}f_{\overline{t}})\times_{\mathrm{id}_K}\mathrm{id}_{k_s}$ is~just~the~restriction of the element $f_{\overline{st}}^{-1}f_{\overline{s}}f_{\overline{t}}\in(\mathcal F/\mathrm Z_{\mathcal F})(K)\subseteq\aut(G)$, so it is in $(\mathcal F/\mathrm Z_{\mathcal F})(k_s)$ and the induced function $\widetilde{\kappa}$ is a homomorphism. The section $\widetilde{f}$ is clearly continuous in the sense of Definition \ref{defcont}.

Now, take another nice triple $(K'/k,G',\varphi')$ for which there is a $K''$-representable isomorphism with our starting triple $(K/k,G,\varphi)$, where $K''$ is a finite \textit{separable} extension of $k$ containing~both $K$ and $K'$. Then these two triples define isomorphic pairs over $k_s$ (see Definition \ref{defetsepband} below). 
Conversely, given any pair $(\Gb,\widetilde{\kappa})$ over $k_s$ with a $K$-form $G$ and continuous lift $\widetilde{f}$ of $\widetilde{\kappa}$, we get a nice triple $(K/k,G,\varphi)$ which represents a band $L$ depending only on the pair $(\Gb,\widetilde{\kappa})$. These two constructions are mutually inverse, up to obvious isomorphisms.

We sketch the construction of $\varphi$: 
Suppose that $\widetilde{f}_{st} = \widetilde{f}_s\sigma_t$ and $\widetilde{f}_{ts} = \sigma_t\widetilde{f}_s$ both hold for the natural homomorphism $\sigma : \Gal(k_s/K)\rightarrow\saut(\Gb/K)$ and $s\in\Gal(k_s/k),\, t\in\Gal(k_s/K)$.~As~in the proof of (3)$\Rightarrow$(4), this defines for each $\overline{s}\in\Gal(K/k)$ a unique semiautomorphism $f_{\overline{s}}$ of $G$ lying over $\overline{s}$ (the Galois descent of $f_s$ independent of the choice of $s$ lifting $\overline{s}$). It remains only to show that the section $f$ is a homomorphism up to inner automorphisms, i.e. that the element $f_{\overline{st}}^{-1}f_{\overline{s}}f_{\overline{t}}\in\aut(G)$ lies in $(\mathcal F/\mathrm Z_{\mathcal F})(K)$ for $s,t\in\Gal(k_s/k)$. Because $\smash{\widetilde{f}_{st}^{-1}\widetilde{f}_s\widetilde{f}_t}\in(\mathcal F/\mathrm Z_{\mathcal F})(k_s)$ and $\mathcal F/\mathrm Z_{\mathcal F}\hookrightarrow\mathcal{Aut}_G$ is an inclusion of sheaves, this again follows by descent. By the discussion in Example \ref{exmpgalband}, the section $f$ defines the desired map $\varphi$. 
\end{exmp}

\begin{rem}\label{remcontcond}
Examples \ref{exmpgalband} and \ref{exmpgalband2} show that nice triples correspond to our stronger interpretation of continuity (following Proposition \ref{propcont}), which is evidently connected to Galois descent. On the other hand, property (1) in Proposition \ref{propcont} is more directly related to the general definition of a band, where $t\mapsto\sigma_t^{-1}f_s^{-1}f_{st}$ being locally constant corresponds to the formation of the sheaf $\mathcal{Out}_{\mathrm{pr}_1^*\Gb,\,\mathrm{pr}_2^*\Gb}$.

See \cite[1.14]{FSS98} for a discussion of other variants of ``continuity'' for sections in literature:~In \cite{Brv93}, Borovoi gives a condition more general than ours and thus his results still hold in our case (when $\mathrm{char}(k) = 0$, which he assumes). On the other hand, Springer's condition~in~\cite{Spr66} is too strong: together with our Definition \ref{defcont}, it would imply that $\Gamma_K$ acts on $\Gb(k_s)$ by inner automorphisms (so if $\Gb$ is smooth commutative, $G_0$ must be a constant $K$-group scheme).
\end{rem}

\begin{defn}\label{defetsepband}
Let $\mathcal F$ be a group sheaf defined on the restriction of the site $\mathcal C$ to $\mathrm{Sch}/k_s$. Suppose it is represented by an algebraic group $\Gb$ over $k_s$ (we write $\mathcal F/\mathrm Z_{\mathcal F}$ for the sheaf quotient taken on~$\mathcal C$, e.g.\ $(\mathcal F/\mathrm Z_{\mathcal F})(k_s) = \Gb(k_s)/\mathrm Z_\Gb(k_s)\subseteq (\Gb/\mathrm Z_\Gb)(k_s)$ if $\mathcal C = k_\Et$). Fix a homomorphism:
$$\kappa : \Gal(k_s/k)\longrightarrow\frac{\saut(\Gb/k)}{(\mathcal F/\mathrm Z_{\mathcal F})(k_s)}$$
Suppose that $\kappa$ admits a continuous lift $f$ to $\saut(\Gb/k)$.
Such a pair $(\Gb,\kappa)$ will be called an \textit{\'etale band} if the site $\mathcal C$ with respect to which it is defined is $k_\Et$. If the site $\mathcal C$ is instead $k_\fppf$, then we suppose furthermore that $\mathrm{H}^1(k_s,\Gb/\mathrm Z_{\Gb}) = 1$ 
and call this pair a \textit{separable band}. These two definitions clearly agree if $k$ is a perfect field (which coincides with the case considered in \cite[\textsection2.2]{DLA19}) or if both $\Gb$ and $\mathrm Z_\Gb$ are smooth (as then $\Gb(k_s)/\mathrm Z_\Gb(k_s) = (\Gb/\mathrm Z_\Gb)(k_s)$).

Note that a particular choice of $K$, $G_0$ or $f$ is not part of the data of these pairs, and only their existence is required.
An isomorphism of two \'etale (resp. separable) bands $(\Gb,\kappa)$ and $(\Gb',\kappa')$ is an isomorphism $\alpha : \Gb\simeq\Gb'$ of algebraic groups such that
    the sections $\kappa,\kappa'$ commute with the natural group isomorphism (for $\mathcal F,\mathcal F'$ the associated \'etale, resp.\ fppf, sheaves on $k_s$)
$$\frac{\saut(\Gb/k)}{(\mathcal F/\mathrm Z_{\mathcal F})(k_s)}\xrightarrow{\;\;\;\sim\;\;\;}\frac{\saut(\Gb'/k)}{(\mathcal F'/\mathrm Z_{\mathcal F'})(k_s)}$$
    induced by the map $(\alpha',s)\mapsto (\alpha\circ\alpha'\circ s_*\alpha^{-1}, s)$ on semiautomorphisms. 
It is immediate that any continuous lift of $\kappa$ then corresponds to a continuous lift of $\kappa'$ and vice versa, because we know that $\alpha$ descends to some $\alpha_0 : G_0\simeq\Gb'_0$ over a finite separable extension of $k$.
\end{defn}

\begin{exmp}
Let $\Gb$ be an algebraic group on $k_s$. In \cite{FSS98}, an ``inner automorphism'' of $\Gb$ is defined to be conjugation by an element of $\Gb(k_s)$. This gives a definition of a band which agrees with our definition of \'etale band.
\end{exmp}

By Example \ref{exmpgalband2}, an ``\'etale band'' is (up to isomorphism) the same as an algebraic band on $k_\Et$. On the other hand, a given ``separable band'' only defines some algebraic band $L$ on $k_\fppf$ (representable \'etale-locally, in the sense of Definition \ref{defetlocrep}), but it is not a priori clear that this presentation of $L$ is unique. However, the following statement shows that this is indeed the case, thanks to our more restrictive definition of a separable band:

\begin{prop}\label{propcharsepband}
There is an equivalence of groupoid categories (that is, we consider only morphisms which are isomorphisms):
\begin{enumerate}[(i)]
    \item separable bands on $k$
    \item algebraic bands on $k_\fppf$ which admit representation by a nice triple $(K/k,G,\varphi)$ such that $K/k$ is separable and $\mathrm H^1(k_s, G/\mathrm Z_G) = 1$
\end{enumerate}
Note that the chosen representative triple is not part of the data of the objects in (ii), only the existence of such a triple.
\end{prop}
\begin{prf}
In view of Example \ref{exmpgalband2}, we only need to prove that two nice triples $(K/k,G,\varphi)$ and $(K'/k,G',\varphi')$ as in (ii) representing isomorphic bands on $k_\fppf$ must also define isomorphic separable bands. For this, it suffices to show that there is a $K''$-representable isomorphism between them for some finite separable extension $K''/K$.
Take any finite separable refinement $K''$ of $K,K'$. By Proposition \ref{propparambandforms}, these two triples correspond to two classes in the image of $\mathrm{H}^1(K'',G/\mathrm Z_G)$ in $\mathrm{H}^1(K'',\mathcal{Aut}_G)$, coinciding if and only if they are $K''$-representably isomorphic.
Because $\varinjlim\mathrm H^1(K'', G/\mathrm Z_G) = \mathrm H^1(k_s, G/\mathrm Z_G) = 1$, we may arrange for the two classes to coincide by enlarging $K''$. Here the limit is taken over all $K''/k$ finite separable (in a fixed separable closure $k_s/k$) and the equality holds because $G$ is algebraic.
\end{prf}

This proposition shows that a ``separable band'' is not a structure on a given band $L$ on $k_\fppf$, but rather a property that $L$ may or may not have. Moreover, if $L$ is locally represented~by~some group $G$ such that $\mathrm{H}^1(k_s,G_0/\mathrm Z_{G_0}) = 1$ for every $k_s$-form $G_0$ of $G_{\overline k}$, then we only need to know that $L$ is represented \'etale-locally over $k$. There are two obvious classes of algebraic groups $G$ for which this holds -- commutative and smooth ones. We are mostly interested in the second class. It follows that:
\begin{cor}
The notion of a smooth separable band is the same as that of a band on $k_\fppf$ represented \'etale-locally by a smooth algebraic group.
\end{cor}

%% file: part1-2.tex
\subsection{Diagrams with Smooth Separable Bands}

It is natural to study comparisons between separable and \'etale bands: Let $\pi$ denote the fppf-to-\'etale map of sites. An algebraic group $\Gb$ is identified with a presheaf on $\mathrm{Sch}/k_s$ (similarly on $\mathrm{Sch}/k$) which is an fppf sheaf. It is then also an \'etale sheaf, which we denote by $\pi_*\Gb$ (however, we may drop the notation $\pi_*$ when the context makes the chosen topology clear). While we can make identifications $\pi_*\mathrm Z_{\Gb} = \mathrm Z_{\pi_*\Gb}$ and $\pi_*\mathcal{Aut}_{\Gb} = \mathcal{Aut}_{\pi_*\Gb}$, there is in general only an inclusion $\pi_*\Gb/\pi_*\mathrm Z_{\Gb}\hookrightarrow\pi_*(\Gb/\mathrm Z_{\Gb})$. Note that $(\pi_*\Gb/\pi_*\mathrm Z_{\Gb})(k_s) = \Gb(k_s)/\mathrm Z_{\Gb}(k_s)$. We consider a commutative diagram with exact rows:
\begin{center}\begin{tikzcd}
0\arrow[r] & \Gb(k_s)/\mathrm Z_{\Gb}(k_s)\arrow[d, hook]\arrow[r] & \saut(\Gb/k)\arrow[d, equal]\arrow[r] & \dfrac{\saut(\Gb/k)}{\Gb(k_s)/\mathrm Z_{\Gb}(k_s)}\arrow[d, two heads]\arrow[r] & 0\\
0\arrow[r] & (\Gb/\mathrm Z_{\Gb})(k_s)\arrow[r] & \saut(\Gb/k)\arrow[r] & \dfrac{\saut(\Gb/k)}{(\Gb/\mathrm Z_{\Gb})(k_s)}\arrow[r] & 0\\
\end{tikzcd}\end{center}
\vspace{-22pt}

\begin{defn}\label{deflyingover}
If $\Gb$ is an algebraic group on $k_s$ such that $\mathrm{H}^1(k_s,\Gb/\mathrm Z_{\Gb}) = 1$, then any \'etale band $(\pi_*\Gb, \kappa)$ induces a separable band $(\Gb,\overline{\kappa})$ by the composition
\begin{center}\begin{tikzcd}
\overline{\kappa} \;:\; \Gamma\arrow[r, "\kappa"] & \dfrac{\saut(\Gb/k)}{\Gb(k_s)/\mathrm Z_{\Gb}(k_s)}\arrow[r, two heads] & \dfrac{\saut(\Gb/k)}{(\Gb/\mathrm Z_{\Gb})(k_s)}
\end{tikzcd}\end{center}
where $\Gamma\coloneqq\Gal(k_s/k)$. In this situation we say that $(\pi_*\Gb, \kappa)$ \textit{lies over} $(\Gb,\overline{\kappa})$. 
Note that a given separable band $(\Gb,\overline{\kappa})$ has an \'etale band lying over it if and only if it admits a continuous lift $f$ which is a ``homomorphism up to $\Gb(k_s)/\mathrm Z_{\Gb}(k_s)$, not just up to $(\Gb/\mathrm Z_{\Gb})(k_s)$''.
\end{defn}

When talking about the $\mathrm H^i_\et$ cohomology of the higher direct image \'etale sheaves $\mathrm{R}^j\pi_*\Gb$ on $k_\Et$, we will identify it with the Galois cohomology groups $\mathrm{H}^i(\Gamma, \mathrm{H}^j(k_s, \Gb))$. 
Recall that the center (see discussion preceding Proposition \ref{propcentacth2}) of a band is defined over $k$, hence so are all the \'etale sheaves $\mathrm{R}^j\pi_*\mathrm Z_L$.

\begin{rem}\label{remconjact}
Let $\alpha$ be a (semi)automorphism of $\Gb$. Then $\alpha$ acts on $\mathcal{Aut}_{\Gb}$ by: 
$$\mathrm{int}(\alpha) : \alpha_0\longmapsto\alpha\circ\alpha_0\circ\alpha^{-1}$$ 
For the natural inclusion of $\Gb/\mathrm Z_{\Gb}$ (or $\pi_*\Gb/\pi_*\mathrm Z_{\Gb}$) into $\mathcal{Aut}_{\Gb}$, these two actions agree. Indeed, for $x\in(\Gb/\mathrm Z_{\Gb})(k_s)$ and $y\in\Gb(R)$, where $R$ is a $\overline{k}$-algebra, we have:
$$\mathrm{int}(\alpha(x))(y) = \alpha(x)\cdot y\cdot\alpha(x)^{-1} = \alpha(x\cdot \alpha^{-1}(y)\cdot x^{-1}) = (\alpha\circ\mathrm{int}(x)\circ\alpha^{-1})(y) = \mathrm{int}(\alpha)(\mathrm{int}(x))(y)$$
This fact will often be useful in calculations, where we will interchangeably write $\alpha\circ\mathrm{int}(x)\circ\alpha^{-1}$ and $\alpha(x)$ (by identifying $\Gb/\mathrm Z_{\Gb}$ with its image). A good example is the following proposition:
\end{rem}

\begin{prop}\label{propparamlet}
Let $(\Gb,\overline{\kappa})$ be a separable band on $k$ and moreover suppose $\mathrm{H}^1(k_s,\Gb) = 1$. Let $\mathscr{L}_\et(\Gb,\overline{\kappa})$ be the set of all \'etale bands $(\pi_*\Gb,\kappa)$ lying above $(\Gb,\overline{\kappa})$ up to conjugation of $\kappa$ by (the conjugation action of) an element of $(\Gb/\mathrm Z_{\Gb})(k_s)$. The following properties hold:
\begin{enumerate}[\hspace{0.5cm}(a)]
    \item There is a class in $\mathrm{H}^2(\Gamma, \mathrm{H}^1(k_s, \mathrm Z_{\Gb}))$ which vanishes if and only if $\mathscr{L}_\et(\Gb,\overline{\kappa})\neq\varnothing$.
    \item Suppose that $\mathscr{L}_\et(\Gb,\overline{\kappa})\neq\varnothing$.\ Then the group $\mathrm{H}^1(\Gamma, \mathrm{H}^1(k_s, \mathrm Z_{\Gb}))$ acts freely and transitively on $\mathscr{L}_\et(\Gb,\overline{\kappa})$.
\end{enumerate}
\end{prop}
\begin{prf}
By the assumption that $\mathrm{H}^1(k_s,\Gb) = 1$, there is a short exact sequence
\begin{equation}\label{eqsepet}
    1\rightarrow\Gb(k_s)/\mathrm Z_{\Gb}(k_s)\rightarrow (\Gb/\mathrm Z_{\Gb})(k_s)\rightarrow \mathrm{H}^1(k_s, \mathrm Z_{\Gb})\rightarrow 1
\end{equation}
of groups. Each semiautomorphism $\alpha$ of $\Gb$ acts on this sequence by conjugation, as discussed in the above remark: In particular, any continuous lift $f : \Gamma\rightarrow\saut(\Gb/k)$ of $\overline{\kappa}$ makes $\mathrm{H}^1(k_s, \mathrm Z_{\Gb})$ into a $\Gamma$-module (independent of the chosen lift); we write ${^s}h$ for $f_s(h)$. 
\smallskip

For (a), we fix some such lift $f$. Then the automorphisms $g_{s,t}\coloneqq f_sf_tf_{st}^{-1}\subseteq\aut(\Gb/k)$ lie in $(\Gb/\mathrm Z_{\Gb})(k_s)$. Any continuous lift $f'$ of $\overline{\kappa}$ is of the form $f' = jf$ for a locally constant function $j : \Gamma\rightarrow(\Gb/\mathrm Z_{\Gb})(k_s)$. We are interested in finding $j$ such that $(jf)_s(jf)_t(jf)_{st}^{-1}\in\Gb(k_s)/\mathrm Z_{\Gb}(k_s)$ for all $s,t\in\Gamma$, as $jf$ would then define an \'etale band in $\mathscr{L}_\et(\Gb,\overline{\kappa})$.

Consider the images $\overline{g}_{s,t}\in\mathrm{H}^1(k_s, \mathrm Z_{\Gb})$ of $g_{s,t}$. Expressing $f_{stu}$ in two different ways gives
$$f_s(g_{t,u})\cdot g_{s,tu} = g_{s,t}\cdot g_{st,u}$$
(c.f. Remark \ref{remconjact}). Therefore $\overline{g}$ is a cocycle in $\mathrm{Z}^2(\Gamma,\mathrm{H}^1(k_s, \mathrm Z_{\Gb}))$. Replacing the lift $f$ by $jf$ (for a locally constant function $j$ as above) has the effect of replacing $\overline{g}$ by the cohomologous cocycle $\overline{j}_s {^s}(\overline{j}_t) \overline{g}_{s,t} \overline{j}_{st}^{-1}$ valued in $\mathrm{H}^1(k_s, \mathrm Z_{\Gb})$.
Thus the class $[\overline{g}]$ is independent of choice.

It is neutral if and only if there exists a locally constant function $h : \Gamma\rightarrow\mathrm{H}^1(k_s, \mathrm Z_{\Gb})$ such that $1 = h_s {^s}(h_t) \overline{g}_{s,t} h_{st}^{-1}$. As the map $(\Gb/\mathrm Z_{\Gb})(k_s)\rightarrow\mathrm{H}^1(k_s, \mathrm Z_{\Gb})$ is surjective, any such $h$ admits a locally constant lift $j : \Gamma\rightarrow(\Gb/\mathrm Z_{\Gb})(k_s)$. Thus $[\overline{g}]$ is neutral if and only if $j$ can be chosen so that $(jf)_s(jf)_t(jf)_{st}^{-1}\in\Gb(k_s)/\mathrm Z_{\Gb}(k_s)$.
\smallskip

For (b), take $(\pi_*\Gb,\kappa)\in\mathscr{L}_\et(\Gb,\overline{\kappa})$  and a lift $f : \Gamma\rightarrow\saut(\Gb/k)$ of $\kappa$. A continuous lift $jf$ of $\overline{\kappa}$, for a locally constant function $j : \Gamma\rightarrow(\Gb/\mathrm Z_{\Gb})(k_s)$, defines some \'etale band $(\pi_*\Gb,j\cdot\kappa)$ if and only if $(jf)_s(jf)_t(jf)_{st}^{-1}\in G(k_s)/\mathrm Z_G(k_s)$ (and every \'etale band in $\mathscr{L}_\et(\Gb,\overline{\kappa})$ is obtained in this way). Equivalently,
$$1 = (\overline{jf})_s(\overline{jf})_t(\overline{jf})_{st}^{-1} = \overline{j}_s\cdot \overline{(f_s j_t f_s^{-1})}\cdot \overline{j}_{st}^{-1} = \overline{j}_s\cdot {^s}(\overline{j}_t)\cdot \overline{j}_{st}^{-1}$$
in $\mathrm{H}^1(k_s, \mathrm Z_{\Gb})$, which says exactly that $\overline{j}$ is a cocycle in $\mathrm{Z}^1(\Gamma, \mathrm{H}^1(k_s, \mathrm Z_{\Gb}))$. Conversely, it is clear that every cocycle $\overline{j}$ lifts to a locally constant function $j : \Gamma\rightarrow(\Gb/\mathrm Z_{\Gb})(k_s)$ which defines a continuous section of some \'etale band.

Given two functions $j,j' : \Gamma\rightarrow(\Gb/\mathrm Z_{\Gb})(k_s)$ which define cohomologous cocycles $\overline{j}\simeq\overline{j'}$, we see that there is an element $\overline{a}$, for $a\in(\Gb/\mathrm Z_{\Gb})(k_s)$, such that $\overline{j'} = \overline{a}^{-1}\overline{j} \,{^s}\overline{a}$. Equivalently,
$$j'_s \equiv a^{-1}j_sf_s\, af_s^{-1}   \;\;\mathrm{mod}\;\Gb(k_s)/\mathrm Z_{\Gb}(k_s)$$
which shows that $jf$ and $j'f$ define bands that are mutually conjugate by the image of $a$ (more precisely, by the image of $\mathrm{int}(a)\in\aut(\Gb)$) in $\saut(\Gb/k)/(\Gb(k_s)/\mathrm Z_{\Gb}(k_s))$.
\end{prf}

\begin{cor}\label{corparamlet}
Every smooth algebraic separable band admits an \'etale band lying over it.
\end{cor}
\begin{prf}
It suffices that $\mathrm{H}^2(\Gamma, \mathrm{H}^1(k_s, Z)) = 0$ for $Z = \mathrm Z_{\Gb}$. This holds for any algebraic group $Z$ by Proposition \ref{proph2h1coh}.
\end{prf}

We now study smooth algebraic separable bands in the case when they are globally represented. Let $G$ be a smooth algebraic group over $k$ and let $\Gb\coloneqq G_{k_s}$. There are an associated \'etale band $(\Gb,\kappa)$ and separable band $(\Gb,\overline{\kappa})$ over $k$. The observations we make now are essential for some technical steps in the later sections.

\begin{prop}\label{propdiaglet}
There is a commutative diagram as follows, with all rows and columns being exact sequences of pointed sets:
\vspace{-17pt}

\begin{center}\begin{equation}\label{diagbignoncomm}
\!\!\!\!\!\,
\begin{tikzcd}[cramped, column sep = small]
&\mathrm{H}^1(\Gamma, G(k_s))\arrow[d]\arrow[r, "\sim"] &
\mathrm{H}^1(k, G)\arrow[d]\\
\mathrm{H}^0(\Gamma, \mathrm{H}^1(k_s, \mathrm Z_G))\arrow[d, equal]\arrow[r] &
\mathrm{H}^1\!\!\left(\!\Gamma, \dfrac{G(k_s)}{\mathrm Z_G(k_s)}\right)\arrow[d, "\delta"]\arrow[r] &
\mathrm{H}^1\!\!\left(\! k, \dfrac{G}{\mathrm Z_G}\right)\arrow[d, "\widetilde{\delta}"]\arrow[r] &
\mathrm{H}^1(\Gamma, \mathrm{H}^1(k_s, \mathrm Z_G))\arrow[d, equal]\\
\mathrm{H}^0(\Gamma, \mathrm{H}^1(k_s, \mathrm Z_G))\arrow[r] &
\mathrm{H}^2(\Gamma, \mathrm Z_G(k_s))\arrow[r] &
\mathrm{H}^2(k, \mathrm Z_G)\arrow[r] &
\mathrm{H}^1(\Gamma, \mathrm{H}^1(k_s, \mathrm Z_G))\arrow[r, "\mathrm{obs}"] & \mathrm{H}^3(\Gamma, \mathrm Z_G(k_s))
\end{tikzcd}\!\!\!\!\!\!\!\!\!\!\!\!\!\!\!\!
\end{equation}\end{center}
Moreover, given a representative $P$ of $[P]\in\mathrm{H}^1(\Gamma, G(k_s)/\mathrm Z_G(k_s))$ (resp. $[P]\in\mathrm{H}^1(k, G/\mathrm Z_G)$), consider the diagram \eqref{diagbignoncomm} associated to $G$ and the analogous diagram {$_P$\eqref{diagbignoncomm}} associated to the inner twist ${_P}G$ of $G$. Then the subdiagram of \eqref{diagbignoncomm} with initial object $\mathrm{H}^1(\Gamma, G(k_s)/\mathrm Z_G(k_s))$ (resp. $\mathrm{H}^1(k, G/\mathrm Z_G)$) maps canonically to the respective subdiagram of {$_P$\eqref{diagbignoncomm}} via compatible bijections $\tau_P$ induced on all objects (which do not preserve pointedness of sets), such that $\tau_P([P]) = 1$.
\end{prop}

\begin{prf}
We first show the construction of the maps, which immediately explains exactness everywhere: The columns are given by Proposition \ref{propnoncommles}. The middle row is the long exact sequence associated to \eqref{eqsepet}, up to applying the isomorphism $\mathrm{H}^1(k, G/\mathrm Z_G)\cong\mathrm{H}^1(\Gamma, (G/\mathrm Z_G)(k_s))$. The bottom sequence is the one of low-degree terms associated to the flat-to-\'etale spectral~sequence of $\mathrm Z_G$, in which we have also used that $\mathrm{H}^2(k, \mathrm Z_G) = \ker\big(\mathrm{H}^2(k, \mathrm Z_G)\rightarrow\mathrm{H}^0(\Gamma, \mathrm{H}^2(k_s, \mathrm Z_G))\big)$ (which holds because $\mathrm{H}^2(k_s, \mathrm Z_G) = 0$ by Proposition \ref{proph2h1coh}).

The commutativity of the diagram is not obvious (due to our lack of tools from homological algebra in the noncommutative setting) and \textit{we postpone its proof} to the end of this subsection. Finally, we fix a Galois (resp. fppf) torsor $P$ representing a class $[P]\in\mathrm{H}^1(\Gamma, G(k_s)/\mathrm Z_G(k_s))$ (resp. $[P]\in\mathrm{H}^1(k, G/\mathrm Z_G)$) and explain how it twists part of the diagram \eqref{diagbignoncomm}. Note that this is an inner twist of $G$, but not a pure inner twist (see Definition \ref{definntwist}) so there is in general no bijection between the sets $\mathrm H^1(k, G)$ and $\mathrm H^1(k, {_P}G)$. However, there is a canonical identification $\mathrm Z_G\xrightarrow{\sim}\mathrm Z_{{_P}G}$ (so in particular, the bottom rows of \eqref{diagbignoncomm} and {$_P$\eqref{diagbignoncomm}} are the same).

The translations $\tau_P$ on the sets $\mathrm{H}^1(\Gamma, G(k_s)/\mathrm Z_G(k_s))$ (if applicable), on $\mathrm{H}^1(k, G/\mathrm Z_G)$ and on $\mathrm{H}^1(\Gamma, \mathrm{H}^1(k_s, \mathrm Z_G))$ are given by the usual Galois/fppf theory of twisting by $P$ (and its images), cf.\ \cite[I, \textsection5]{Ser97} and \cite[\textsection B.2]{Con12}. In particular, if we started with $[P]\in\mathrm{H}^1(\Gamma, G(k_s)/\mathrm Z_G(k_s))$, then the image of $[P]$ in $\mathrm{H}^1(\Gamma, \mathrm{H}^1(k_s, \mathrm Z_G))$ is trivial, so its translation is just the identity map to $\mathrm{H}^1(\Gamma, \mathrm{H}^1(k_s, \mathrm Z_{{_P}G}))$.

The bottom row of \eqref{diagbignoncomm} consists of groups, so we define $\tau_P$ there as subtraction of the image of $[P]$. It remains to show that these two constructions are compatible with the vertical maps in the diagram: For $\mathrm{H}^1(\Gamma, \mathrm{H}^1(k_s, \mathrm Z_G))$, this is well-known to correspond to twisting in the Abelian case. For $\delta$ and ${_P}\delta$, this is \cite[I, Prop.\ 44]{Ser97} (note that the map $\tau_c$ there goes in the opposite direction). For $\widetilde{\delta}$ and ${_P}\widetilde{\delta}$, the statement follows similarly and is easiest to see by looking at $\tau_P^{-1}$, which is defined by contracted products: Take a right torsor $P'$ of ${_P}(G/\mathrm Z_G) = {_P}G/\mathrm Z_{_PG}$, then $\tau_P^{-1}([P']) = [P'\times^{_P(G/\mathrm Z_G)}P]\in\mathrm{H}^1(k, G/\mathrm Z_G)$ (as $P$ is a left torsor of $_PG$). We want to show that $\widetilde{\delta}(\tau_P^{-1}([P'])) = ({_P}\widetilde{\delta})([P']) + \widetilde{\delta}([P])$. By Propositions \ref{propcentacth2} and \ref{propnoncommles} it suffices to prove that an obvious natural map $\mathscr X\wedge^{\mathrm Z_G}\mathscr X'\rightarrow\mathscr X''$ of $\mathrm Z_G$-gerbes is an equivalence (where $\wedge^{\mathrm Z_G}$ denotes the contracted product of $\mathrm Z_G$-gerbes), which can be checked locally from the definition.
\end{prf}

This diagram captures the connections between the different cohomology sets appearing in the study of separable bands. As stated above, it still remains to prove its commutativity. The proof is a technicality which occupies the remainder of this subsection, so we first discuss some immediate properties which the proposition suggests:

\begin{cor}\label{correprlet}
For a fixed $G$ as above, there is an exact sequence of pointed sets:
$$\mathrm{H}^0(\Gamma, \mathrm{H}^1(k_s, \mathrm Z_G))\longrightarrow\mathrm{N}^2(k,\Gb,\kappa)\longrightarrow\mathrm{N}^2(k,\Gb,\overline{\kappa})\longrightarrow\mathscr L_\et(\Gb,\overline{\kappa})
$$
The image of the rightmost map agrees with the image of the map (cf. Proposition~\ref{propparamlet}(b))
$$\mathrm{H}^1(k, G/\mathrm Z_G)\longrightarrow\mathrm{H}^1(\Gamma, \mathrm{H}^1(k_s, \mathrm Z_G))
\xrightarrow[\scalebox{0.7}{\raisebox{10pt}{$\sim$}}]
{\;[j]\,\mapsto\,[(\Gb,\,j\cdot\kappa)]\;}\mathscr L_\et(\Gb,\overline{\kappa})
$$
\vspace{-15pt}

\noindent
and is exactly the set of globally representable \'etale bands lying over $(\Gb,\overline{\kappa})$ (up to conjugation). Hence this map is surjective if and only if every \'etale band lying over $(\Gb,\overline{\kappa})$~is~representable.
\end{cor}
\begin{prf}
Recall that, by Corollary \ref{corneutelemdesc}, the two sets $\im\delta$ and $\im\widetilde{\delta}$ (as in \eqref{diagbignoncomm}) are canonically identified with the sets $\mathrm{N}^2(k,\Gb,\kappa)$ and $\mathrm{N}^2(k,\Gb,\overline{\kappa})$, respectively. Exactness of the sequence is now immediate from the commutativity of the diagram \eqref{diagbignoncomm}. Finally, to see that the desired image is exactly the subset of globally representable \'etale bands lying over $(\Gb,\overline{\kappa})$, we just note that every such global representative is an inner form of $G$ and that the map in the statement sends an inner form to its corresponding \'etale band.
\end{prf}

\begin{cor}
For a fixed $G$ as above, there is an exact sequence of pointed sets:
$$\mathrm{H}^0(\Gamma, \mathrm{H}^1(k_s, \mathrm Z_G))\longrightarrow\mathrm{H}^2(k,\Gb,\kappa)\longrightarrow\mathrm{H}^2(k,\Gb,\overline{\kappa})\longrightarrow\mathscr L_\et(\Gb,\overline{\kappa})
\xrightarrow{\;\mathrm{obs}\;}\mathrm{H}^3(\Gamma,\mathrm Z_G(k_s))
$$
\end{cor}
\begin{prf}
Similar to the above. The ``obstruction class'' (to the existence of gerbes bound by a given $(\Gb,\kappa)$ in $\mathscr L_\et(\Gb,\overline{\kappa})$) on the right will be discussed more in the next section.
\end{prf}

Now we turn our attention to the commutativity of the diagram \eqref{diagbignoncomm}. The top square is clearly commutative, by the forgetful functor which sends an \'etale torsor of $G$ (on the big site) to an fppf torsor, applied to the rectangle:
\begin{center}\begin{tikzcd}
\mathrm{H}^1(\Gamma, G(k_s))\arrow[r]\arrow[d, "\rotatebox{90}{$\sim$}", pos=0.4] & \mathrm{H}^1\!\left(\!\Gamma, \dfrac{G(k_s)}{\mathrm Z_G(k_s)}\right)\arrow[r] & \mathrm{H}^1\!\left(\!\Gamma, \dfrac{G}{\mathrm Z_G}(k_s)\!\right)\arrow[d, "\rotatebox{90}{$\sim$}", pos=0.4]\\
\mathrm{H}^1(k, G)\arrow[rr] & & \mathrm{H}^1\!\left(\!k, \dfrac{G}{\mathrm Z_G}\right)
\end{tikzcd}\end{center}
The bottom squares of the diagram are more difficult. The flat-to-\'etale spectral sequence is given by the composition of derived functors
$\mathrm R_\Et\Gamma(k, -)\circ\mathrm R\pi_* = \mathrm R\Gamma(k, -)$
and the cohomology group $\mathrm{H}^2(k,\mathrm Z_G)$ appears by a natural identification with $\mathbf{H}^2_\et(k, \mathrm R\pi_*\mathrm Z_G)$. We are thus lead to consider long exact sequences in nonabelian hypercohomology of complexes of group sheaves:
\begin{defn}
Let $\mathcal F\rightarrow\mathcal G\rightarrow\mathcal H$ be a complex of group sheaves on a site $\mathcal C$ with final element $S$. We now review some constructions (a good reference for which is \cite{Ald08}):

An $(\mathcal F{\rightarrow}\mathcal G)$\textit{-torsor} is a pair $(\mathcal P,a)$, where $\mathcal P$ is an $\mathcal F$-torsor and $a : \mathcal P\times^{\mathcal F}\mathcal G\xrightarrow{\,\sim\,}\mathcal G$~is~a~fixed trivialization of the induced $\mathcal G$-torsor. There is an obvious gerbe $\mathrm{TORS}(\mathcal F{\rightarrow}\mathcal G)$ with~a~forgetful functor $\mathrm{TORS}(\mathcal F{\rightarrow}\mathcal G)\rightarrow\mathrm{TORS}(\mathcal F)$. We write $\mathbf H^1(S,\, \mathcal F{\rightarrow}\mathcal G)$ for the set of isomorphism classes of objects in the fiber $\mathrm{TORS}(\mathcal F{\rightarrow}\mathcal G)(k)$. This set was already~considered by Deligne in \cite[2.4.3]{Del79} (see \cite[\textsection1]{Brv92} for an equivalent definition using cocycles in the case of Galois modules) and it agrees with the first cohomology group in the Abelian case.

Suppose now that $\mathcal F, \mathcal G, \mathcal H$ are commutative. An $(\mathcal F{\rightarrow}\mathcal G{\rightarrow}\mathcal H)$\textit{-gerbe} is a pair $(\mathscr X,a)$, where $\mathscr X$ is a gerbe bounded by $\mathcal F$ and $a$ is a map $\mathscr X\rightarrow\mathrm{TORS}(\mathcal G{\rightarrow}\mathcal H)$ of gerbes such that the composition $\mathscr X\rightarrow\mathrm{TORS}(\mathcal G)$ is a map of gerbes bound by the map $\mathcal F\rightarrow\mathcal G$ (in~an~obvious~sense). We write $\mathbf H^2(S,\, \mathcal F{\rightarrow}\mathcal G{\rightarrow}\mathcal H)$ for the set of isomorphism classes of such pairs. By \cite[Prop.\ 3.2.3]{Ald08}, this definition agrees with the usual second cohomology group of $[\mathcal F\rightarrow\mathcal G\rightarrow\mathcal H]$.

Finally, let $(\mathcal F^\bullet,\delta^\bullet)$ be a complex of group sheaves on $\mathcal C$ concentrated in nonnegative degrees (here $\delta^j : \mathcal F^j\rightarrow\mathcal F^{j+1}$ with $\im\delta^j$ normal in $\mathcal F^{j+1}$).
We define $\mathbf H^1(S, \mathcal F^\bullet)\coloneqq\mathbf H^1(S,\, \mathcal F^0{\rightarrow}\ker\delta^1)$. If the complex is Abelian, then $\mathbf H^2(S, \mathcal F^\bullet)\coloneqq\mathbf H^2(S,\, \mathcal F^0{\rightarrow}\mathcal F^1{\rightarrow}\ker\delta^2)$. Again, we recover the usual hypercohomology groups in the Abelian case because this truncation is functorial and induces isomorphisms in first (resp. second) degree. We observe that these groups are functorial and preserved by quasi-isomorphisms; 
\end{defn}

\begin{prop}
Let $[\mathcal F^0\rightarrow\mathcal F^1]\rightarrow[\mathcal G^0\rightarrow\mathcal G^1]$ be a quasi-isomorphism of complexes of group sheaves concentrated in degrees $0$ and $1$. Then the map $\mathbf H^1(S,\, \mathcal F^0{\rightarrow}\mathcal F^1)\rightarrow\mathbf H^1(S,\, \mathcal G^0{\rightarrow}\mathcal G^1)$ 
\end{prop}
\begin{prf}
We only need to construct an inverse: Given a pair $(\mathcal P',a')\in\mathbf H^1(S,\, \mathcal G^0{\rightarrow}\mathcal G^1)$, consider the $\mathcal G^0$-equivariant morphism of sheaves (of sets):
$$\mathcal P'\rightarrow\mathcal P'\times^{\mathcal G^0}\mathcal G^1\xrightarrow{\;a'\;}\mathcal G^1$$
We claim that the fiber product $\mathcal P\coloneqq\mathcal F^1\times_{\mathcal G^1}\mathcal P'$ is a $\mathcal G^0$-torsor. Indeed, it is enough to check this over a covering of $S$, so in particular it suffices that the natural map $\mathcal F^0\rightarrow\mathcal F^1\times_{\mathcal G^1}\mathcal G^0$ is an isomorphism. This is true for quasi-isomorphisms by the $5$-lemma (for nonabelian group sheaves). Finally, there is a map $a$ of $\mathcal F^1$-torsors defined as:
$$\mathcal P\times^{\mathcal F^0}\mathcal F^1\cong\left(\mathcal F^1\times_{\mathcal G^1}\mathcal P'\right)\times^{\left(\mathcal F^1\times_{\mathcal G^1}\mathcal G^0\right)}\left(\mathcal F^1\times_{\mathcal G^1}\mathcal G^1\right)\cong\mathcal F^1\times_{\mathcal G^1}\left(\mathcal P'\times^{\mathcal G^0}\mathcal G^1\right)\xrightarrow{{\;1\,\times\, a'\;}}\mathcal F^1\times_{\mathcal G^1}\mathcal G^1\cong\mathcal F^1$$
It is now simple to check that $(\mathcal P',a')\mapsto(\mathcal P,a)$ defines the desired inverse.
\end{prf}

\begin{prop}\label{proptechnoncomm}
Let $1\rightarrow\mathcal F^\bullet\rightarrow\mathcal G^\bullet\xrightarrow{p}\mathcal H^\bullet\rightarrow 1$ be a short exact sequence of complexes of group sheaves on $\mathcal C$. Suppose that $\mathcal F^0$ lies in the center of $\mathcal G^0$. Then there is a functorial (in such short exact sequences) map $\mathbf H^1(S, \mathcal H^\bullet)\rightarrow\mathbf H^2(S, \mathcal F^\bullet)$ which generalizes the one in Proposition \ref{propnoncommles} when these complexes are concentrated in degree $0$.
\end{prop}
\begin{prf}
We construct functorially the following truncation, which is also a short exact sequence of complexes:
$$1\longrightarrow\left[\begin{tikzcd}
    \mathcal F^0\arrow[d] \\ \mathcal F^1\arrow[d] \\ \ker\delta^2_{\mathcal F} 
\end{tikzcd}\right]
\longrightarrow
\left[\begin{tikzcd}
    \mathcal G^0\arrow[d] \\ p^{-1}(\ker\delta^1_{\mathcal H})\arrow[d] \\ \ker\delta^2_{\mathcal F} 
\end{tikzcd}\right]
\longrightarrow
\left[\begin{tikzcd}
    \mathcal H^0\arrow[d] \\ \ker\delta^1_{\mathcal H}\arrow[d] \\ 1
\end{tikzcd}\right]\longrightarrow 1$$
Next, we pick an $(\mathcal H^0{\rightarrow} \ker\delta^1_{\mathcal H})$-torsor $(\mathcal P,a)$ to which we will assign a $(\mathcal F^0{\rightarrow}\mathcal F^1{\rightarrow} \ker\delta^2_{\mathcal F})$-gerbe as follows:

First, we choose the gerbe $\mathscr X$ bound by $\mathcal F^0$ in the same way as in Proposition \ref{propnoncommles}. Thus each fiber $\mathscr X(T)$ consists of pairs $(\mathcal P',a')$, where $\mathcal P'$ is an $\mathcal G^0$-torsor over $T$ and ${a' : \mathcal P'\times^{\mathcal G^0}\mathcal H^0\xrightarrow{\sim}\mathcal P_T}$ an isomorphism of $(\mathcal F/\mathrm Z_{\mathcal F})$-torsors over $T$. This already shows that our construction generalizes the case concentrated in degree $0$. Note that the centrality of $\mathcal F^0$ in $\mathcal G^0$ is necessary for $\mathscr X$ to be bound by $\mathcal F^0$.

Second, we construct the morphism of gerbes $b : \mathscr X\rightarrow\mathrm{TORS}(\mathcal F^1{\rightarrow} \ker\delta^2_{\mathcal F})$. Take any $(\mathcal P',a')$ in $\mathscr X(T)$. We will now assign to it a $\mathcal F^1$-torsor over $T$. Note that there is a $p^{-1}(\ker\delta^1_{\mathcal H})$-torsor $\mathcal P'\times^{\mathcal G^0}p^{-1}(\ker\delta^1_{\mathcal H})$ over $T$. It admits a $p^{-1}(\ker\delta^1_{\mathcal H})$-equivariant map to $\ker\delta^1_{\mathcal H}$, which is given by the following chain of morphisms:
\begin{align*}
\mathcal P'\times^{\mathcal G^0}p^{-1}(\ker\delta^1_{\mathcal H})\xhookrightarrow{\;\;\;\;\;\;\;\;\;\;}
&\left(\mathcal P'\times^{\mathcal G^0}p^{-1}(\ker\delta^1_{\mathcal H})\right)\times^{p^{-1}(\ker\delta^1_{\mathcal H})}\ker\delta^1_{\mathcal H}\\
\cong\;\;
&\left(\mathcal P'\times^{\mathcal G^0}\mathcal H^0\right)\times^{\mathcal H^0}\ker\delta^1_{\mathcal H}
\;\;\xrightarrow{\;\;a'\,\times\, 1\;\;}\;\;
\mathcal P_T\times^{\mathcal H^0}\ker\delta^1_{\mathcal H}
\;\;\xrightarrow{\;\;a_T\;\;}\;\;\ker\delta^1_{\mathcal H}
\end{align*}
The kernel of this map is an $\mathcal F^1$-torsor $\mathcal K$ over $T$. Finally, there is a map $a''$ of $\ker\delta^2_{\mathcal F}\,$-torsors
$$\mathcal K\times^{\mathcal F^1}\ker\delta^2_{\mathcal F}\longrightarrow\left(\mathcal P'\times^{\mathcal G^0}p^{-1}(\ker\delta^1_{\mathcal H})\right)\times^{p^{-1}(\ker\delta^1_{\mathcal H})}\ker\delta^2_{\mathcal F}\xrightarrow{\;\;\sim\;\;}\mathcal P'\times^{\mathcal G^0}\ker\delta^2_{\mathcal F}\;\cong\;\ker\delta^2_{\mathcal F}$$
where the last isomorphism comes from the fact that the map $\mathcal G^0\rightarrow\delta^2_{\mathcal F}$ is trivial. Thus the pair $(\mathcal K, a'')$ determines an object of $\mathrm{TORS}(\mathcal F^1{\rightarrow} \ker\delta^2_{\mathcal F})(T)$ which we assign to $(\mathcal P',a')$. As~this assignment is clearly functorial, we have successfully defined $b$.

Finally, the pair $(\mathscr X,b)$ is the desired $(\mathcal F^0{\rightarrow}\mathcal F^1{\rightarrow} \ker\delta^2_{\mathcal F})$-gerbe. All steps of this construction are clearly functorial in the starting exact sequence.
\end{prf}

\begin{rem}\label{remhypernonab}
Using ideas from the last two propositions, one can generalize the $7$-term exact sequence of Proposition \ref{propnoncommles} to complexes. To prove that the map constructed in Proposition \ref{proptechnoncomm} agrees with the analogous map in Abelian hypercohomology, one can adapt the methods of \cite[Ch. IV, \textsection3.4]{Gir71}. We will not need this.
\end{rem}

We are now ready to finish the proof of Proposition \ref{propdiaglet}. Let $\mathcal I^\bullet$ be a $\pi_*$-acyclic resolution of $\mathrm Z_G$ on the fppf site. Applying $\pi_*$ gives the following short exact sequence of (vertical) central short exact sequences:
\begin{center}\begin{tikzcd}[column sep = large]
& 0\arrow[d] & 0\arrow[d] & 0\arrow[d]\\
0\arrow[r] & \pi_*\mathrm Z_G\arrow[d]\arrow[r] & \pi_*\mathcal I^0\arrow[d]\arrow[r] & \dfrac{\pi_*\mathcal I^0}{\pi_*\mathrm Z_G}\arrow[d]\arrow[r] & 0\\
0\arrow[r] & \pi_*G\arrow[d]\arrow[r] & \pi_*\!\left(\dfrac{\mathcal I^0\times G}{\mathrm Z_G}\right)\arrow[d]\arrow[r] & \pi_*\!\left(\dfrac{\mathcal I^0}{\mathrm Z_G}\right)\arrow[d]\arrow[r] & 0\\
0\arrow[r] & \dfrac{\pi_*G}{\pi_*\mathrm Z_G}\arrow[r]\arrow[d] & \pi_*\!\left(\dfrac{G}{\mathrm Z_G}\right)\arrow[r]\arrow[d] & \mathrm R^1\pi_*\mathrm Z_G\arrow[r]\arrow[d] & 0\\
& 0 & 0 & 0
\end{tikzcd}\end{center}
%
We may immediately deduce the following commutative diagram of complexes of group sheaves:
\begin{center}\begin{tikzcd}[cramped, column sep = 6.8pt]
\lcomp\hspace{-12pt} & 0\arrow[rr] & &  0\arrow[d, line width = 1, shift left = 6, shorten <=3pt, shorten >=3pt]\arrow[rr] & & \mathrm R^1\pi_*\mathrm Z_G & \hspace{-12pt}\rcomp & & & \arrow[lll, "\sim", swap, line width = 1]
\lcomp\hspace{-12pt} & \pi_*\mathrm Z_G\arrow[rr] & & \pi_*\mathcal I^0\arrow[d, line width = 1, shift left = 4, shorten <=3pt, shorten >=3pt]\arrow[rr] & & \pi_*\!\left(\dfrac{\mathcal I^0}{\mathrm Z_G}\right)\! & \hspace{-12pt}\rcomp\arrow[rrr, line width = 1] & & & 
\lcomp\hspace{-12pt} & \pi_*\mathrm Z_G\arrow[rr] & & 0\arrow[d, line width = 1, shift left = -3, shorten <=3pt, shorten >=3pt]\arrow[rr] & & 0 & \hspace{-12pt}\rcomp\\
\lcomp\hspace{-12pt} & 0\arrow[rr] & & \mathrm R^1\pi_*\mathrm Z_G\arrow[d, line width = 1, shift left = 6, shorten <=3pt, shorten >=3pt]\arrow[rr] & & \mathrm R^1\pi_*\mathrm Z_G & \hspace{-12pt}\rcomp & & & \arrow[lll, "\sim", swap, line width = 1]
\lcomp\hspace{-12pt} & \pi_*G\arrow[rr] & & \pi_*\!\left(\dfrac{\mathcal I^0{\times} G}{\mathrm Z_G}\right)\arrow[d, line width = 1, shift left = 4, shorten <=3pt, shorten >=3pt]\arrow[rr] & & \pi_*\!\left(\dfrac{\mathcal I^0}{\mathrm Z_G}\right)\! & \hspace{-12pt}\rcomp\arrow[rrr, line width = 1] & & & 
\lcomp\hspace{-12pt} & \pi_*G\arrow[rr] & & 1\arrow[d, line width = 1, shift left = -3, shorten <=3pt, shorten >=3pt]\arrow[rr] & & 1 & \hspace{-12pt}\rcomp\\
\lcomp\hspace{-12pt} & 0\arrow[rr] & & \mathrm R^1\pi_*\mathrm Z_G\arrow[rr] & & 0 & \hspace{-12pt}\rcomp & & & \arrow[lll, "\sim", swap, line width = 1]
\lcomp\hspace{-12pt} & \dfrac{\pi_*G}{\pi_*\mathrm Z_G}\arrow[rr] & & \pi_*\!\left(\dfrac{G}{\mathrm Z_G}\right)\arrow[rr] & & 1 & \hspace{-12pt}\rcomp\arrow[rrr, line width = 1] & & & 
\lcomp\hspace{-12pt} & \dfrac{\pi_*G}{\pi_*\mathrm Z_G}\arrow[rr] & & 1\arrow[rr] & & 1 & \hspace{-12pt}\rcomp
\end{tikzcd}\end{center}
Here all three columns are central extensions of complexes, while the first two columns are quasi-isomorphic. Using Proposition \ref{proptechnoncomm}, we now find the following commutative diagram in nonabelian hypercohomology:
\begin{center}\begin{tikzcd}
\mathrm{H}^0(\Gamma, \mathrm{H}^1(k_s, \mathrm Z_G))\arrow[d, equal] &
\mathbf{H}^1_\et\!\left(\!k,\, \left[\dfrac{\pi_*G}{\pi_*\mathrm Z_G}\rightarrow\pi_*\!\left(\dfrac{G}{\mathrm Z_G}\right)\right]\right)\arrow[l, "\sim", swap]\arrow[d]\arrow[r] &
\mathrm{H}^1\!\left(\!\Gamma, \dfrac{G(k_s)}{\mathrm Z_G(k_s)}\right)\arrow[d, "\delta"]\\
\mathrm{H}^0(\Gamma, \mathrm{H}^1(k_s, \mathrm Z_G)) &
\mathbf{H}^2_\et\!\left(\!k,\, \left[\pi_*\mathrm Z_G\rightarrow \pi_*\mathcal I^0\rightarrow \pi_*\!\left(\dfrac{\mathcal I^0}{\mathrm Z_G}\right)\right]\right)\arrow[l, "\sim", swap]\arrow[r] &
\mathrm{H}^2(\Gamma, \mathrm Z_G(k_s))
\end{tikzcd}\end{center}
The left-most vertical map is indeed an equality, as can be deduced from its definition (or as in Remark \ref{remhypernonab}). This recovers the left-most square in the diagram \eqref{diagbignoncomm}.

For the two remaining squares in the middle of the diagram, we will consider the following commutative diagram, whose columns are again central extensions of complexes:
\begin{center}\begin{tikzcd}[cramped, column sep = 6.8pt]
\lcomp\hspace{-12pt} & \pi_*\mathrm Z_G\arrow[rr] & & 0\arrow[d, line width = 1, shift left = -8, shorten <=3pt, shorten >=3pt] & \hspace{-12pt}\rcomp\arrow[rrrrr, line width = 1] & & & & & 
\lcomp\hspace{-12pt} & \pi_*\mathcal I^0\arrow[d, line width = 1, shift left = 12, shorten <=3pt, shorten >=3pt]\arrow[rr] & & \pi_*\!\left(\dfrac{\mathcal I^0}{\mathrm Z_G}\right)\! & \hspace{-12pt}\rcomp\arrow[rrrrr, line width = 1] & & & & & 
\lcomp\hspace{-12pt} & 0\arrow[d, line width = 1, shift left = 11, shorten <=3pt, shorten >=3pt]\arrow[rr] & & \mathrm R^1\pi_*\mathrm Z_G & \hspace{-12pt}\rcomp\\
\lcomp\hspace{-12pt} & \pi_*G\arrow[rr] & & 1\arrow[d, line width = 1, shift left = -8, shorten <=3pt, shorten >=3pt] & \hspace{-12pt}\rcomp\arrow[rrrrr, line width = 1] & & & & & 
\lcomp\hspace{-12pt} & \pi_*\!\left(\dfrac{\mathcal I^0{\times} G}{\mathrm Z_G}\right)\arrow[d, line width = 1, shift left = 12, shorten <=3pt, shorten >=3pt]\arrow[rr] & & \pi_*\!\left(\dfrac{\mathcal I^0}{\mathrm Z_G}\right)\! & \hspace{-12pt}\rcomp\arrow[rrrrr, line width = 1] & & & & & 
\lcomp\hspace{-12pt} & \mathrm R^1\pi_*\mathrm Z_G\arrow[d, line width = 1, shift left = 11, shorten <=3pt, shorten >=3pt]\arrow[rr] & & \mathrm R^1\pi_*\mathrm Z_G & \hspace{-12pt}\rcomp\\
\lcomp\hspace{-12pt} & \dfrac{\pi_*G}{\pi_*\mathrm Z_G}\arrow[rr] & & 1 & \hspace{-12pt}\rcomp\arrow[rrrrr, line width = 1] & & & & & 
\lcomp\hspace{-12pt} & \pi_*\!\left(\dfrac{G}{\mathrm Z_G}\right)\arrow[rr] & & 1 & \hspace{-12pt}\rcomp\arrow[rrrrr, line width = 1] & & & & & 
\lcomp\hspace{-12pt} & \mathrm R^1\pi_*\mathrm Z_G\arrow[rr] & & 0 & \hspace{-12pt}\rcomp
\end{tikzcd}\end{center}
Again applying Proposition \ref{proptechnoncomm} we deduce the following:
\begin{center}\begin{tikzcd}
\mathrm{H}^1\!\left(\!\Gamma, \dfrac{G(k_s)}{\mathrm Z_G(k_s)}\right)\arrow[d, "\delta"]\arrow[r] &
\mathrm{H}^1\!\left(\!\Gamma, \dfrac{G}{\mathrm Z_G}(k_s)\!\right)\arrow[d]\arrow[r] &
\mathrm{H}^1(\Gamma, \mathrm{H}^1(k_s, \mathrm Z_G))\arrow[d, equal]\\
\mathrm{H}^2(\Gamma, \mathrm Z_G(k_s))\arrow[r] &
\mathbf{H}^2_\et\!\left(\!k,\, \left[\pi_*\mathcal I^0\rightarrow \pi_*\!\left(\dfrac{\mathcal I^0}{\mathrm Z_G}\right)\right]\right)\arrow[r] &
\mathrm{H}^1(\Gamma, \mathrm{H}^1(k_s, \mathrm Z_G))
\end{tikzcd}\end{center}
To finish, we must replace the middle vertical map by the one in diagram \eqref{diagbignoncomm}. For this, it is enough to show that the following diagram commutes:
\begin{center}\begin{tikzcd}
\mathrm{H}^1\!\left(\!\Gamma, \dfrac{G}{\mathrm Z_G}(k_s)\!\right)\arrow[d]\arrow[r, equal] &
\mathrm{H}^1\!\left(\!\Gamma, \dfrac{G}{\mathrm Z_G}(k_s)\!\right)\arrow[d]\arrow[r, "\sim"] &
\mathrm{H}^1\!\!\left(\! k, \dfrac{G}{\mathrm Z_G}\!\right)\arrow[d]\arrow[r, equal] &
\mathrm{H}^1\!\!\left(\! k, \dfrac{G}{\mathrm Z_G}\!\right)\arrow[d, "\widetilde{\delta}"]\\
\mathbf{H}^2_\et\!\left(\!k,\, \left[\pi_*\mathcal I^0\rightarrow \pi_*\!\left(\dfrac{\mathcal I^0}{\mathrm Z_G}\right)\right]\right)\arrow[r] &
\mathbf{H}^2_\et(k,\pi_*\mathcal I^\bullet)\arrow[r, "\sim"] &
\mathbf{H}^2(k,\mathcal I^\bullet) &
\mathrm{H}^2(k,\mathrm Z_G)\arrow[l, "\sim", swap]
\end{tikzcd}\end{center}
As discussed earlier, all horizontal maps are isomorphisms, including the lower left one (which follows from Proposition \ref{proph2h1coh}). The commutativity of this diagram follows similarly to the previous ones, however for the middle square we must use that the construction in Proposition \ref{proptechnoncomm} is functorial with respect to the canonical maps between \'etale and fppf cohomology (nonabelian in degree $1$), but this is obvious since it was defined uniformly over an arbitrary~site. Finally, we know that the vertical map on the right given by Proposition \ref{proptechnoncomm} (which explicitly generalizes Proposition \ref{propnoncommles}) is indeed the map $\smash{\tilde{\delta}}$ from \eqref{diagbignoncomm}, because the map $\smash{\tilde{\delta}}$ in \eqref{diagbignoncomm} was constructed following Proposition \ref{propnoncommles}.

This argument completes the proof of Proposition \ref{propdiaglet}.

%% file: part1-3.tex
\subsection{Čech Cohomology of Locally Algebraic Bands}\label{ssectcech}
Let $\mathcal C$ again denote one of the sites $k_\Et$ and $k_\fppf$. When $\mathcal C$ is the \'etale site of $k$, the set $\mathrm{H}^2(k, L)$ for a band $L$ on $\mathcal C$ can be given a definition in terms of Čech cocycles (see \cite{FSS98} or \cite{DLA19}). Implicitly, this rests on the fact that every \'etale-locally representable band on $\mathcal C$ is nicely representable (see Definition \ref{defetlocrep}), as suggested by the following remark:

\begin{rem}\label{rembreen}
Given a representative triple $(T{\rightarrow} S,\mathcal F,\varphi)$, it is not in general nice: that is, $\varphi$ exists as an isomorphism only over some covering $R\rightarrow T\times_S T$. This prevents us from defining Čech cocycles valued in this band over $T\rightarrow S$, as the descent datum (defined over $T\times_S T$ and not just locally over a covering) must appear in their definition; see below.
One solution to this problem is to consider cocycles more general than Čech cocycles (see \cite[\textsection2.2, in particular p.\ 38]{Bre94}) which incorporate ``iterated'' coverings of $T\times_S T$. They can also be formulated using the notion of hypercoverings. However, this description is then simply a reformulation which ultimately gives, completely formally, the same $\mathrm{H}^2$ set as the one defined using gerbes.
\end{rem}

We will instead show that, when a band $L$ on $\mathcal C$ is locally algebraic (thus in particular nicely representable; see Example \ref{exmpsmoothisnice}), its Čech cohomology is definable and in canonical bijection with the $\mathrm{H}^2$ set of gerbes. In fact, our proof does not strictly require that $L$ is locally algebraic, only that it is nicely representable and that the center $\mathrm Z_L$ is an locally algebraic group over $k$.
Note that this bijectivity is a nontrivial property even when the band is globally represented by a commutative algebraic group (it is classical, however: see the following lemma) and it fails for general commutative group sheaves. The essential reason for this was noted in Definition~\ref{defh2}: gerbes, much like representative triples, are not ``nice'' in general.


\begin{lem}\label{lemrosentensfield}
Let $Z$ be a locally algebraic commutative group over a field $k$. For~each~$n\geq 0$:
\begin{enumerate}[\hspace{0.8 cm}(a)]
    \item $\mathrm{H}^1(\overline k{^{\,\otimes_k (n+1)}}, Z) = 0$, where $\overline k{^{\,\otimes_k n}}$ denotes the $n$-fold tensor product of $\overline k$ over $k$
    \item The natural map $\mathrm{\check H}^n(k, Z)\rightarrow\mathrm{H}^n(k, Z)$ is an isomorphism
\end{enumerate}
\end{lem}
\begin{prf}
These are \cite[Prop.\ 2.9.5, 2.9.6]{RosTD}. See also \cite[III, Prop.\ 6.1]{Mil06}.
\end{prf}

This lemma will be crucial in interpreting the following definition.

\begin{defn}\label{defh2cech}
Let $L$ be a nicely representable band on $\mathcal C$. Then any representative triple of $L$ can be refined into a nice triple $(k'/k,\mathcal F,\varphi)$, where $k'/k$ is a finite extension of fields. We will write $\mathrm{\check Z}^2(k'/k,\mathcal F,\varphi)$ for the set of \textit{cocycles} $(f,g)$, where $f\in\mathcal{Isom}_{\mathrm{pr}_1^*\mathcal F,\,\mathrm{pr}_2^*\mathcal F}(k'\otimes_k k')$ and $g\in\mathcal F(k'\otimes_k k'\otimes_k k')$ are such that:
\begin{itemize}
    \item $\varphi$ is the image of $f$ in $\mathcal{Out}_{\mathrm{pr}_1^*\mathcal F,\,\mathrm{pr}_2^*\mathcal F}(k'\otimes_k k')$
    \item $(\mathrm{pr}_{13}^*f)^{-1}\circ(\mathrm{pr}_{23}^*f)\circ(\mathrm{pr}_{12}^*f)$ is the image of $g^{-1}$ in $(\mathcal F/\mathrm Z_{\mathcal F})(k'\otimes_k k'\otimes_k k')$
    \item $(\mathrm{pr}^*_{12}f)^{-1}(\mathrm{pr}^*_{234}g)\cdot\mathrm{pr}^*_{124}g = \mathrm{pr}^*_{123}g\cdot\mathrm{pr}^*_{134}g$ holds in $\mathcal F(k'\otimes_k k'\otimes_k k'\otimes_k k')$
\end{itemize}
Two cocycles $(f,g)$ and $(f',g')$ are called equivalent if there exists an element $h\in\mathcal F(k'\otimes_k k')$ such that:
\begin{itemize}
    \item $f' = f\circ\mathrm{int}(h)^{-1}$ holds in $\mathcal{Isom}_{\mathrm{pr}_1^*\mathcal F,\,\mathrm{pr}_2^*\mathcal F}(k'\otimes_k k')$
    \item $g'\cdot\mathrm{pr}^*_{13}h\cdot g^{-1} = \mathrm{pr}^*_{12}h\cdot (\mathrm{pr}^*_{12}f)^{-1}(\mathrm{pr}^*_{23}h)$ holds in $\mathcal F(k'\otimes_k k'\otimes_k k')$
\end{itemize}
The set of equivalence classes in $\mathrm{\check Z}^2(k'/k,\mathcal F,\varphi)$ will be denoted by $\mathrm{\check H}^2(k'/k,\mathcal F,\varphi)$. As finite normal extensions $k'/k$ form a directed set (with unique maps), we may form a direct limit in the usual way, denoted by $\mathrm{\check H}^2(k,L)$ and called the \textit{second Čech cohomology set} of $L$. A class in $\mathrm{\check H}^2(k,L)$ is called \textit{neutral} if it can be represented by a cocycle of the form $(f,1)$.

As the name suggests, this set is canonically preserved by isomorphisms of bands. After this section, the notation $\mathrm{\check H}^2$ will be dropped for the more common $\mathrm{H}^2$, as the two sets will be identified in all situations.
\end{defn}

\begin{rem}\label{remetbandh2def}
When $\mathcal C$ is the \'etale site and $\mathcal F = \Gb$ is an algebraic group, the above definition agrees with the one given in \cite{FSS98} and \cite{DLA19}. The key point is that, given a $K$-form $G_0$ of $\Gb$ for $K/k$ finite separable, the maps $\varinjlim G_0(K')\rightarrow G_0(k_s)$ and ${\varinjlim\mathcal{Aut}_{G_0}(K')\rightarrow\mathcal{Aut}_{G_0}(k_s)}$ are isomorphisms, where the limits are taken over all finite separable extensions $K'/K$ in $k_s$. This allows us to replace $f,g$ with their continuous variants over $k_s$:

Let $L = (\Gb,\kappa)$ be an \'etale algebraic band. We write $\mathrm{\check Z}^2(k,L)$ for the set of \textit{cocycles} $(f,g)$, where $f : \Gamma\rightarrow\saut(\Gb/k)$ is a continuous lift of $\kappa$ and $g : \Gamma\times\Gamma\rightarrow\Gb(k_s)$ is a locally constant function such that, for each $s,t,u\in\Gamma$, the following identities hold:
\begin{itemize}
    \item $f_s\circ f_t = \mathrm{int}(g_{s,t})\circ f_{st}$
    \item $f_s(g_{t,u})\cdot g_{s,tu} = g_{s,t}\cdot g_{st,u}$
\end{itemize}
Two cocycles $(f,g)$ and $(f',g')$ are called equivalent if there exists a locally constant function $h : \Gamma\rightarrow\Gb(k_s)$ such that:
\begin{itemize}
    \item $f'_s = \mathrm{int}(h_s)\circ f_s$
    \item $g'_{s,t}\cdot h_{st}\cdot g_{s,t}^{-1} = h_s\cdot f_s(h_t)$
\end{itemize}
The set of equivalence classes in $\mathrm{\check Z}^2(k,L)$ is canonically bijective to the set $\mathrm{\check H}^2(k,L)$ constructed above. Note that by our conventions of writing descent maps throughout this paper, all of the semiautomorphisms $f_s$ point in the opposite direction from the isomorphism $f$ in the previous definition.
\end{rem}

\begin{exmp}
Let $L$ be a band on $\mathcal C$ which is nicely represented over some $k'/k$. If $\mathcal C$ is the fppf site, assume furthermore that $\mathrm Z_L$ is a locally algebraic group over $k$. Let $f$ be a fixed lift of $\varphi$ to $\mathcal{Isom}_{\mathrm{pr}_1^*\mathcal F,\,\mathrm{pr}_2^*\mathcal F}(k'\otimes_k k')$. We claim that any class in $\mathrm{\check H}^2(k,L)$ is represented by a cocycle of the form $(f,g)$ for some $g$ (up to possibly enlarging $k'/k$):

Indeed, let $(f',g')$ be an arbitrary cocycle, which we can assume to be in $\mathrm{\check Z}^2(k'/k,\mathcal F,\varphi)$. Then $f^{-1}\circ f'\in(\mathcal F/\mathrm Z_{\mathcal F})(k'\otimes_k k')$. Up to enlarging $k'$, we may pick $h\in\mathcal F(k'\otimes_k k')$ such that $f^{-1}\circ f' = \mathrm{int}(h)^{-1}$. Indeed, this is trivial over the \'etale site, as $k'\otimes_k k' = \prod_{\Gal(k'/k)} k'$ for finite Galois $k'/k$. Over the fppf site, this is a consequence of $\mathrm{H}^1(\overline{k}\otimes_k\overline{k}, \mathrm Z_L) = 0$, by Lemma \ref{lemrosentensfield}(a).
Then the definition of equivalence of cocycles gives us $g$ in terms of $g'$ and $h$. We may check that the pair $(f,g)$ is now a cocycle, necessarily equivalent to $(f',g')$. 
The same is clearly true for the Galois-theoretic description of cocycles.

We also note that a band $L$ is globally representable if and only if it admits a neutral class. Then the cocycle $(f,1)$ represents a choice of descent data $f$, which is not unique (the neutral classes parametrize global representatives of $L$). By the above discussion, every class $[f,g]$~of~$L$ can then be represented by this $f$ and some $g$ with $\mathrm{int}(g) = 1$, that is: $g\in\mathrm{Z}_L(k'\otimes_k k')$. 

When $L = (\Gb,\kappa)$ is an \'etale band, take a continuous lift $f : \Gamma\rightarrow\saut(\mathcal F/k)$ of $\kappa$ and a function $h : \Gamma\rightarrow (\mathcal F/\mathrm Z_{\mathcal F})(k_s)$. Then the lift $f' = hf$ is continuous if and only if $h$ is locally constant, and every continuous lift $f'$ of $\kappa$ arises in this way. In particular, $(\mathcal F,\kappa)$ is globally representable if and only if there exists a locally constant $h$ such that $hf$ is a homomorphism. 
\end{exmp}

We do not attempt to introduce Galois-theoretic cocycles for separable bands, as even a smooth separable band $L$ can have a nonsmooth center $\mathrm Z_L$ (for example, $\mathrm{Z}_{\mathrm{GL}_p} = \mu_p$) whose fppf cohomology cannot be calculated on the \'etale site. In particular, we now assume that $\mathcal C$ is the fppf site of $k$ and we work only with (fppf) cocycles in the sense of Definition \ref{defh2cech} with values in a nicely representable band $L$ such that $\mathrm Z_L$ is a locally algebraic group on $k$.

\begin{defn}\label{defcechtogerbe}
We define the map $\mathrm{\check H}^2(k, L)\rightarrow\mathrm{H}^2(k, L)$: Take a pair $(f,g)\in\mathrm{\check Z}^2(k'/k,\mathcal F,\varphi)$ and consider the trivial gerbe $\mathrm{TORS}(\mathcal F)$ on $k'_\fppf$. We will now show that $(f,g)$ defines $2$-descent data on $\mathrm{TORS}(\mathcal F)$ in the $2$-stack $\mathrm{STACK}_k$ of stacks on $k$. There is an equivalence of gerbes
$$\widetilde{f}\; :\; \mathrm{pr}_1^*\,\mathrm{TORS}(\mathcal F)\cong \mathrm{TORS}(\mathrm{pr}_1^*\mathcal F)\xrightarrow{\;\;\mathrm{TORS}(f)\;\;}\mathrm{TORS}(\mathrm{pr}_2^*\mathcal F)\cong\mathrm{pr}_2^*\,\mathrm{TORS}(\mathcal F)$$
over $k'\otimes_k k'$ induced by $f$. Now, any element $g'\in\mathcal F(k'\otimes_k k'\otimes_k k')$ such that
\begin{equation}\label{eqconjcocyc}
(\mathrm{pr}_{12}^*f)^{-1}\circ(\mathrm{pr}_{23}^*f)^{-1} = \mathrm{int}(g')\circ(\mathrm{pr}_{13}^*f)^{-1}
\end{equation}
holds defines a natural transformation $\widetilde{g}'$ in the following diagram:
\begin{center}\begin{tikzcd}[column sep = huge]
\mathrm{pr}_{13}^*\mathrm{pr}_1^*\,\mathrm{TORS}(\mathcal F) & &
\mathrm{pr}_{13}^*\mathrm{pr}_2^*\,\mathrm{TORS}(\mathcal F)
\arrow[ll, "\;\;\;\;\;\;\;\;\;\;\;\;\;\;\;\;\;\;\;\;(\mathrm{pr}_{12}^*\widetilde{f})^{-1}\circ\,(\mathrm{pr}_{23}^*\widetilde{f})^{-1}"{name = Dw}, bend left = 10]
\arrow[ll, "(\mathrm{pr}_{13}^*\widetilde{f})^{-1}\;\;\;\;\;\;\;\;\;\;\;\;\;\;\;\;\;\;\;\;"{name = Up}, bend right = 10, swap]
\arrow[shorten <=5pt, shorten >=5pt, Rightarrow, from = Up, to = Dw, "\widetilde{g}'"]
\end{tikzcd}\end{center}
Indeed, given an object $X$ in a fiber of $\mathrm{pr}_{13}^*\mathrm{pr}_2^*\,\mathrm{TORS}(\mathcal F)$, there is a well-defined, natural map
$$\widetilde{g}'_X\;:\;X\times^{\mathrm{pr}_{13}^*\mathrm{pr}_2^*\mathcal F}_{(\mathrm{pr}_{13}^*\widetilde{f})^{-1}}\mathrm{pr}_{13}^*\mathrm{pr}_1^*\mathcal F\xrightarrow{\;\;\;\;\;\;\;\;\;\;} X\times^{\mathrm{pr}_{13}^*\mathrm{pr}_2^*\mathcal F}_{(\mathrm{pr}_{12}^*\widetilde{f})^{-1}\circ\,(\mathrm{pr}_{23}^*\widetilde{f})^{-1}}\mathrm{pr}_{13}^*\mathrm{pr}_1^*\mathcal F$$
of $\mathrm{pr}_{13}^*\mathrm{pr}_1^*\mathcal F$-torsors (where $X\times^{\mathcal G}_\alpha\mathcal H$ denotes the ``change of structure'' of $X$ along $\alpha : \mathcal G\rightarrow\mathcal H$) given by taking $[x,a]$ to $[x,g'a]$. The identity \eqref{eqconjcocyc} ensures that this assignment is well-defined on the classes $[x,a] = [xb,(\mathrm{pr}_{13}^*\widetilde{f}^{-1}(b))^{-1}a]$. This statement is essentially an application of the much more succinct \cite[Lemma 1.5(i)]{Bre94}; note that our arrows $f$ point in a different direction from the arrows $\lambda,\mu$ in Breen's paper.

In particular, when $g' = g$, the cocycle condition of $g$ ensures that $(\widetilde{f},\widetilde{g})$ is a $2$-descent datum in $\mathrm{STACK}_k$ and that, by gluing, we have constructed a gerbe $\mathscr X$ on $k$ (see \cite[\textsection 2.6]{Bre94}). This gluing commutes with the descent datum $f$ of $\mathcal F$, in the sense that $\mathscr X$ is canonically bounded by $L$. The class $[\mathscr X]\in\mathrm{H}^2(k, L)$ is independent of the chosen cocycle (and of refining $k'/k$), since if $(f,g)\sim (f',g')$ via $h$, then $\mathrm{TORS}(f)$ and $\mathrm{TORS}(f')$ are related by a natural transformation $\widetilde{h}$ (analogous to the above) which satisfies a cocycle condition with respect to $\widetilde{g},\widetilde{g}'$ following from the definition of cocycle equivalence. Therefore the identity functor $\mathrm{id}_{\mathrm{TORS}(\mathcal F)}$ descends to an isomorphism of resulting gerbes $\mathscr X\simeq\mathscr X'$ (which is non-unique, it depends on $h$). A similar argument is made more explicit in the proof of Proposition \ref{proph2actcomm} below.
\end{defn}

\begin{rem}
The above constructed map restricts to a bijection between neutral classes in $\mathrm{\check H}^2(k, L)$ and in $\mathrm{H}^2(k, L)$. Indeed, both subsets are parametrized by global representatives $\mathcal F_0$ of $L$ over $k$ (up to pure inner forms) and it's straightforward to see that the class of $\mathcal F_0$ in $\mathrm{\check H}^2(k, L)$ maps to $[\mathrm{TORS}(\mathcal F_0)]\in\mathrm{H}^2(k, L)$.

Furthermore, suppose that $L$ is the band of a commutative group sheaf $\mathcal F$ on $k$. Then $f$ is uniquely determined, so we only write $g$ for the cocycle $(f,g)$. In \cite[IV, \textsection3.5]{Gir71} it is shown that, if we start from a cocycle $g\in\mathrm{\check Z}^2(k'/k,\mathcal F)$ corresponding to a descent datum as above, which defines $\mathscr X$ as a descent of some trivial gerbe, then the image of $[g]$ along the canonical map $\mathrm{\check H}^2(k, \mathcal F)\rightarrow\mathrm{H}^2(k, \mathcal F)$ is again the class of $\mathscr X$. In particular, our construction agrees with the usual Čech-to-derived homomorphism when $L = L(\mathcal F)$.
\end{rem}

Next, we claim that there is a natural action of $\mathrm{\check H}^2(k,\mathrm Z_L)$ on $\mathrm{\check H}^2(k,L)$ via $[g_0].[f,g] = [f,g_0g]$, which is free and transitive when $\mathrm{\check H}^2(k, L)\neq\varnothing$. Recall that an analogous action was defined for the gerbe-theoretic $\mathrm{H}^2$ set in Proposition \ref{propcentacth2}. 

Indeed, this is a clearly well-defined action on cocycles. Moreover, observe that $g_0 = \mathrm dh$~if~and only if there is an equivalence $(f,g)\sim(f,g_0g)$ through $h\in\mathcal F(k'\otimes_k k')$ (up to enlarging $k'$; in fact, then automatically $h\in\mathrm Z_{\mathcal F}(k'\otimes_k k')$ since $\mathrm{int}(h) = f^{-1}f$). Therefore, when $\mathrm{\check H}^2(k, L)\neq\varnothing$, this action descends to a free action of $\mathrm{\check H}^2(k,\mathrm Z_L)$ on $\mathrm{\check H}^2(k,L)$. Finally, as it was already shown that any two given classes can be written as $[f,g]$ and $[f,g']$ for the same $f$, it suffices to note that then $\mathrm{int}(g) = \mathrm{int}(g')$ and that $g_0 = g'g^{-1}\in\mathrm Z_{\mathcal F}(k'\otimes_k k'\otimes_k k')$ is a cocycle.

\begin{prop}\label{proph2actcomm}
The following diagram of actions commutes:
\begin{center}\begin{tikzcd}[column sep = tiny]
    \mathrm{\check H}^2(k,\mathrm Z_L)\arrow[d]
\hspace{-5pt}&\hspace{-10pt}{\times}\hspace{-10pt}&\hspace{-5pt}
    \mathrm{\check H}^2(k, L)\arrow[d]
    \arrow[rrrrrrrr]
    &&&&&&&& \mathrm{\check H}^2(k, L)\arrow[d]\\
    \mathrm{H}^2(k,\mathrm Z_L)
\hspace{-5pt}&\hspace{-10pt}{\times}\hspace{-10pt}&\hspace{-5pt}
    \mathrm{H}^2(k, L)
    \arrow[rrrrrrrr]
    &&&&&&&& \mathrm{H}^2(k, L)
\end{tikzcd}\end{center}
\end{prop}
\begin{prf}
Fix some $k'/k$ and respective cocycles $g_0$ and $(f,g)$ defined over $k'$. By the above-given $2$-descent procedure, $g_0$ defines a $\mathrm Z_L$-gerbe $\mathscr X_0$ and $(f,g)$ defines an $L$-gerbe $\mathscr X$. The bottom action is defined by the contracted product $(\mathscr X_0,\mathscr X)\mapsto\mathscr X_0\wedge^{\mathrm Z_L}\mathscr X$. Over $k'$, this contracted product admits a trivialization (as $\mathscr X_0$, $\mathscr X$ are trivial over $k'$ by construction)
$$T \;:\; \mathrm{TORS}(\mathrm Z_{\mathcal F})\wedge^{\mathrm Z_{\mathcal F}}\mathrm{TORS}(\mathcal F)\xrightarrow{\;\;\;\sim\;\;\;}\mathrm{TORS}(\mathcal F)$$
by mapping the class of a pair $[P_0,P]$ of torsors to the torsor $P_0\times^{\mathrm Z_{\mathcal F}}P$ of $\mathcal F$. This product makes sense since $\mathrm Z_{\mathcal F}\hookrightarrow\mathcal F$ is central and the right action of $\mathrm Z_{\mathcal F}$ on $P$ is equivalently a left action, $z.p\coloneqq p.z$.
We need to descend this trivialization (by $(f,g_0),(f,g)$ on the left, and $(f,g_0g)$ on the right) to an equivalence of gerbes on $k$. Consider the following diagram of functors
\begin{center}\begin{tikzcd}
\mathrm{pr}_1^*\mathrm{TORS}(\mathrm Z_{\mathcal F})\wedge^{\mathrm{pr}_1^*\mathrm Z_{\mathcal F}}\mathrm{pr}_1^*\mathrm{TORS}(\mathcal F)\arrow[r, "\mathrm{pr}_1^*T"] & \mathrm{pr}_1^*\mathrm{TORS}(\mathcal F)\\
\mathrm{pr}_2^*\mathrm{TORS}(\mathrm Z_{\mathcal F})\wedge^{\mathrm{pr}_2^*\mathrm Z_{\mathcal F}}\mathrm{pr}_2^*\mathrm{TORS}(\mathcal F)\arrow[u, "\widetilde{f}^{-1}\wedge\widetilde{f}^{-1}"]\arrow[r, "\mathrm{pr}_2^*T"] & \mathrm{pr}_2^*\mathrm{TORS}(\mathcal F)\arrow[u, "\widetilde{f}^{-1}"]
\end{tikzcd}\end{center}
for which there is a natural transformation
$$\omega \;:\; \mathrm{pr}_1^*T\circ\big(\widetilde{f}^{-1}\wedge\widetilde{f}^{-1}\big)\Longrightarrow\widetilde{f}^{-1}\circ\mathrm{pr}_2^*T$$
defined for each pair $(P_0,P)$ by the canonical isomorphism (which is well-defined):
\begin{align*}
\big(P_0\times^{\mathrm{pr}_2^*\mathrm Z_{\mathcal F}}\mathrm{pr}_1^*\mathrm Z_{\mathcal F}\big)\times^{\mathrm{pr}_1^*\mathrm Z_{\mathcal F}}\big(P\times^{\mathrm{pr}_2^*\mathcal F}\mathrm{pr}_1^*\mathcal F\big)
\xrightarrow{\;\;\left[[p_0,z],\,[p,x]\right]\,\mapsto\,\left[[p_0,p],\,zx\right]\;\;}
\big(P_0\times^{\mathrm{pr}_2^*\mathrm Z_{\mathcal F}}P\big)\times^{\mathrm{pr}_2^*\mathcal F}\mathrm{pr}_1^*\mathcal F
\end{align*}
Descent of $1$-morphisms between objects in $2$-stacks is itself a $1$-descent condition, hence~we~need to show that $\omega^{-1}$ satisfies the appropriate equivalence property (analogous to the one satisfied by $h$ in the last point of Definition \ref{defh2cech}) over $k'\otimes_k k'\otimes_k k'$, which amounts to the identity
$$\mathrm{pr}^*_{13}\omega\circ(\mathrm{pr}^*_{13}\mathrm{pr}^*_1 T)(\widetilde{g_0}\wedge\widetilde{g})^{-1}\;=\;\widetilde{(g_0g)}^{-1}\!\circ\big(\mathrm{pr}^*_{12}\widetilde{f}\big)^{-1}(\mathrm{pr}^*_{23}\omega)\circ\mathrm{pr}^*_{12}\omega$$
with $\widetilde{g_0},\widetilde{g}$ as in the above construction of $\mathscr X_0,\mathscr X$. In slightly more elaborate terms, this identity corresponds to the commutativity of the following diagram:
\begin{center}\begin{tikzcd}
& \mathrm{pr}^*_{13}\mathrm{pr}^*_1 T\circ\mathrm{pr}^*_{13}(\widetilde{f}^{-1}\wedge\widetilde{f}^{-1})\arrow[ddrr, "\mathrm{pr}^*_{13}\omega", swap]\\[10pt]
\mathrm{pr}^*_{13}\mathrm{pr}^*_1 T\circ\mathrm{pr}^*_{12}(\widetilde{f}^{-1}\wedge\widetilde{f}^{-1})\circ\mathrm{pr}^*_{23}(\widetilde{f}^{-1}\wedge\widetilde{f}^{-1})\arrow[ur, "(\mathrm{pr}^*_{13}\mathrm{pr}^*_1 T)(\widetilde{g_0}\wedge\widetilde{g})^{-1}", swap]\arrow[dd, "\mathrm{pr}^*_{12}\omega"]\\[-10pt]
& &[-100pt] & \mathrm{pr}^*_{13}\widetilde{f}^{-1}\circ\mathrm{pr}^*_{13}\mathrm{pr}^*_2 T\\[-10pt]
\mathrm{pr}^*_{12}\widetilde{f}^{-1}\circ\mathrm{pr}^*_{13}\mathrm{pr}^*_2 T\circ\mathrm{pr}^*_{23}(\widetilde{f}^{-1}\wedge\widetilde{f}^{-1})\arrow[dr, "\big(\mathrm{pr}^*_{12}\widetilde{f}\big)^{-1}(\mathrm{pr}^*_{23}\omega)"]\\[10pt]
& \mathrm{pr}^*_{12}\widetilde{f}^{-1}\circ\mathrm{pr}^*_{23}\widetilde{f}^{-1}\circ\mathrm{pr}^*_{13}\mathrm{pr}^*_2 T\arrow[uurr, "\widetilde{(g_0g)}^{-1}"] & & 
\end{tikzcd}\end{center}

Given any pair of torsors $[P_0,P]$ in $\mathrm{pr}_{13}^*\mathrm{pr}_2^*\mathrm{TORS}(\mathrm Z_{\mathcal F})\wedge^{\mathrm{pr}_{13}^*\mathrm{pr}_2^*\mathrm Z_{\mathcal F}}\mathrm{pr}_{13}^*\mathrm{pr}_2^*\mathrm{TORS}(\mathcal F)$, this commutativity is now straightforward to check. Both paths send an element $\big[[p_0,z,z'],[p,x,x']\big]$~in
$$\left(P_0\!\times^{\mathrm{pr}_{23}^*\mathrm{pr}_2^*\mathrm Z_{\mathcal F}\!}\mathrm{pr}_{23}^*\mathrm{pr}_1^*\mathrm Z_{\mathcal F}\!\times^{\mathrm{pr}_{12}^*\mathrm{pr}_2^*\mathrm Z_{\mathcal F}\!}\mathrm{pr}_{12}^*\mathrm{pr}_1^*\mathrm Z_{\mathcal F}\right)\times^{\mathrm{pr}_{13}^*\mathrm{pr}_1^*\mathrm Z_{\mathcal F}}\left(P\!\times^{\mathrm{pr}_{23}^*\mathrm{pr}_2^*\mathcal F\!}\mathrm{pr}_{23}^*\mathrm{pr}_1^*\mathcal F\!\times^{\mathrm{pr}_{12}^*\mathrm{pr}_2^*\mathcal F\!}\mathrm{pr}_{12}^*\mathrm{pr}_1^*\mathcal F\right)$$
to the element $$\left[[p_0,p],\;g_0^{-1}g^{-1}\cdot(\mathrm{pr}_{12}^*\widetilde{f})^{-1}(xz)\cdot x'z'\right]\;\in\;\left(P_0\!\times^{\mathrm{pr}_{13}^*\mathrm{pr}_2^*\mathrm Z_{\mathcal F}\!}P\right)
\!\times^{\mathrm{pr}_{13}^*\mathrm{pr}_2^*\mathcal F\!}\mathrm{pr}_{13}^*\mathrm{pr}_1^*\mathcal F$$
and this finishes the proof.
\end{prf}

\begin{cor}\label{cornicelyrepcech1}
When $\mathrm{\check H}^2(k, L)\neq\varnothing$, the function $\mathrm{\check H}^2(k, L)\rightarrow\mathrm{H}^2(k, L)$ is a bijection.
\end{cor}
\begin{prf}
By the existence of such a function, the condition $\mathrm{\check H}^2(k, L)\neq\varnothing$ implies $\mathrm{H}^2(k, L)\neq\varnothing$. Now the respective actions of $\mathrm{\check H}^2(k,\mathrm Z_L)$ and $\mathrm{H}^2(k,\mathrm Z_L)$ are both free and transitive, so the result follows from Proposition \ref{proph2actcomm} and Lemma \ref{lemrosentensfield}(b).
\end{prf}

It remains to consider the case when $\mathrm{\check H}^2(k, L) = \varnothing$ and show that then also $\mathrm{H}^2(k, L) = \varnothing$. For this, we introduce the obstruction class $\check e(L)\in\mathrm{\check H}^3(k,\mathrm Z_L)$ to the nonemptiness of $\mathrm{\check H}^2(k, L)$, in analogy with the obstruction class $e(L)\in\mathrm{H}^3(k,\mathrm Z_L)$ constructed in \cite[VI, \textsection 2]{Gir71}. This requires some technical care:

\begin{defn}
Let $j : \mathrm Z_L\rightarrow C$ be a map of commutative locally algebraic groups over $k$. Given a nice triple $(k'/k,\mathcal F,\varphi)$ representing $L$, consider the nice triple $(k'/k,\mathcal F\times_{\mathrm Z_{\mathcal F}} C,\varphi\times\varphi_C)$. Here, $\mathcal F\times_{\mathrm Z_{\mathcal F}} C\coloneqq(\mathcal F\times C)/\mathrm Z_{\mathcal F}$ is the pushforward group, $\varphi_C$ is the descent datum of $C$, and $\varphi\times\varphi_C$ is defined in an obvious way using that $\mathrm Z_{\mathcal F}$ is central. The resulting band is defined up to unique isomorphism depending only on $L$ and we denote it by $L\wedge^{\mathrm Z_L}\!C$. Its center~is~canonically identified with $C$ and thus, in particular, $\varphi\times\varphi_C$ restricts to $\varphi_C$ on $C$.

There is, moreover, a natural map $\check v^{(2)} : \mathrm{\check H}^2(k, L\wedge^{\mathrm Z_L}\!C)\longrightarrow\mathrm{\check H}^2(k, Q)$ for $Q\coloneqq C/\mathrm{im}(\mathrm Z_L)$, which we now define: Note that the group sheaf morphism $\mathcal F\times_{\mathrm Z_{\mathcal F}} C\rightarrow Q$ defined by $[g,c]\mapsto\overline{c}$ commutes with inner automorphisms. Because of this and because $\varphi_C$ clearly commutes~with the descent datum $\varphi_Q$, the following map
$$\mathrm{\check Z}^2(k'/k,\mathcal F\times_{\mathrm Z_{\mathcal F}} C,\varphi\times\varphi_C)
\xrightarrow{\;\;(F,\,[g,c])\,\mapsto\,\overline{c}\;\;}\mathrm{\check Z}^2(k'/k,Q)$$
is well-defined, that is: $\overline{c}$ does indeed satisfy the cocycle property. It is obvious that cohomologous cocycles get mapped to cohomologous cocycles. 
\end{defn}

\begin{prop}\label{propeinh3}
There is a unique class $\check e(L)\in\mathrm{\check H}^3(k,\mathrm Z_L)$ such that:
\begin{itemize}
    \item For any short exact sequence $0\rightarrow\mathrm Z_L\rightarrow C\rightarrow Q\rightarrow 0$ of commutative fppf group sheaves on $k$ such that $C$ is a locally algebraic group, the image of the composition
    $$\mathrm{\check H}^2(k, L\wedge^{\mathrm Z_L}\!C)\xrightarrow{\;\;\check v^{(2)}\;\;}\mathrm{\check H}^2(k, Q)\rightarrow\mathrm{H}^2(k, Q)\xrightarrow{\;\;\;\delta\;\;\;}\mathrm{H}^3(k,\mathrm Z_L)\cong\mathrm{\check H}^3(k,\mathrm Z_L)$$
    is contained in the one-element set $\{\check e(L)\}$ 
\end{itemize}
Moreover, it satisfies the following two properties:
\begin{itemize}
    \item $\check e(L) = 0$ if and only if $\mathrm{\check H}^2(k, L)\neq\varnothing$
    \item For any homomorphism $\mathrm Z_L\rightarrow C$ of commutative locally algebraic groups on $k$, the natural map $\mathrm{\check H}^3(k,\mathrm Z_L)\rightarrow\mathrm{\check H}^3(k, C)$ takes $\check e(L)$ to $\check e(L\wedge^{\mathrm Z_L}\!C)$
\end{itemize}
\end{prop}

Note that the last point makes sense, since $L\wedge^{\mathrm Z_L}\!C$ satisfies the same global assumptions we have on $L$: It is a nicely represented fppf band whose center is a locally algebraic group on $k$.

\begin{prf}
The first property actually defines $\check e(L)$ only if we know that there exists such a sequence with $\mathrm{\check H}^2(k, L\wedge^{\mathrm Z_L}\!C)\neq\varnothing$. For this reason, we initially define $\check e(L)$ by an explicit construction such that it satisfies the second stated property.

Fix a nice triple $(k'/k,\mathcal F,\varphi)$ representing $L$. Let $f$ be a lift of $\varphi$ and choose $g\in\mathcal F(k'\otimes_k k'\otimes_k k')$ such that $(\mathrm{pr}_{13}^*f)^{-1}\circ(\mathrm{pr}_{23}^*f)\circ(\mathrm{pr}_{12}^*f) = \mathrm{int}(g)^{-1}$ holds, which we can do by Lemma~\ref{lemrosentensfield}(a) applied to $\mathrm Z_L$. Then $(f,g)$ is a cocycle if and only if $g$ satisfies its cocycle property with respect to $f$, equivalently if the following difference vanishes:
$$\mathrm{d}g\coloneqq(\mathrm{pr}^*_{12}f)^{-1}(\mathrm{pr}^*_{234}g)\cdot\mathrm{pr}^*_{124}g\cdot\mathrm{pr}^*_{134}g^{-1}\cdot\mathrm{pr}^*_{123}g^{-1}\in\mathcal F(k'\otimes_k k'\otimes_k k'\otimes_k k')$$
Note that the choice of $f$ is always implicit in $g\mapsto\mathrm{d}g$. By construction, one has
\begin{align*}
\mathrm{int}(\mathrm{d}g) = \mathrm{pr}^*_{12}f^{-1}\,&{\circ}\,\mathrm{int}(\mathrm{pr}^*_{234}g)\circ\mathrm{pr}^*_{12}f\circ\mathrm{int}(\mathrm{pr}^*_{124}g)\circ\mathrm{int}(\mathrm{pr}^*_{134})g^{-1}\circ\mathrm{int}(\mathrm{pr}^*_{123}g)^{-1}\\
    = \mathrm{pr}^*_{12}f^{-1}\,&{\circ}\,
(\mathrm{pr}^*_{23}f^{-1}\circ\mathrm{pr}^*_{34}f^{-1}\circ\mathrm{pr}^*_{24}f)\circ\mathrm{pr}^*_{12}f\\
&\begin{aligned}
\circ\,(\mathrm{pr}^*_{12}f^{-1}\circ\mathrm{pr}^*_{24}f^{-1}\circ\mathrm{pr}^*_{14}f)
&\circ(\mathrm{pr}^*_{14}f^{-1}\circ\mathrm{pr}^*_{34}f\circ\mathrm{pr}^*_{13}f)\\
&\circ(\mathrm{pr}^*_{13}f^{-1}\circ\mathrm{pr}^*_{23}f\circ\mathrm{pr}^*_{12}f) \,=\, \mathrm{id}
\end{aligned}
\end{align*}
and $\mathrm{d}g$ hence lies in $\mathrm Z_{\mathcal F}(k'\otimes_k k'\otimes_k k'\otimes_k k')$. Similarly, one sees that it is in fact a cocycle in $\mathrm{\check Z}^3(k'/k,\mathrm Z_{\mathcal F})$. We now show that the class $\check e(L)\coloneqq-[\mathrm{d}g]$ is independent of the choice of~$f$~and~$g$.\;

First, by replacing $g$ with some $g'$, but keeping $f$ the same, we get $\mathrm{int}(g') = \mathrm{int}(g)$ and hence $g^{-1}g'\in\mathrm Z_{\mathcal F}(k'\otimes_k k'\otimes_k k')$ so that $\mathrm{d}g' = \mathrm{d}g+\mathrm{d}(g^{-1}g')$ and $[\mathrm{d}g'] = [\mathrm{d}g]$. Thus the class depends at most on the choice of $f$. Replacing $f$ by $f'$, we may write $f' = f\circ\mathrm{int}(h)^{-1}$ for $h\in\mathcal F(k'\otimes_k k')$ (again by Lemma \ref{lemrosentensfield}(a)). Thus, we are free to take $g' = \mathrm{pr}^*_{12}h\cdot (\mathrm{pr}^*_{12}f)^{-1}(\mathrm{pr}^*_{23}h)\cdot g\cdot\mathrm{pr}^*_{13}h^{-1}$ for which $(\mathrm{pr}_{13}^*f')^{-1}\circ(\mathrm{pr}_{23}^*f')\circ(\mathrm{pr}_{12}^*f') = \mathrm{int}(g')^{-1}$. We calculate
\begin{align*}
\mathrm{d}g' &= (\mathrm{pr}^*_{12}f')^{-1}(\mathrm{pr}^*_{234}g')\cdot\mathrm{pr}^*_{124}g'\cdot(\mathrm{pr}^*_{134}g')^{-1}\cdot(\mathrm{pr}^*_{123}g')^{-1}\\
&= \left(\mathrm{int}(\mathrm{pr}^*_{12}h)\circ\mathrm{pr}^*_{12}f^{-1}\right)\left(\mathrm{pr}^*_{23}h\cdot(\mathrm{pr}^*_{23}f)^{-1}(\mathrm{pr}^*_{34}h)\cdot \mathrm{pr}^*_{234}g\cdot\mathrm{pr}^*_{24}h^{-1}\right)\\
&
\hspace{-30pt}\begin{aligned}
\cdot\left(\mathrm{pr}^*_{12}h\cdot (\mathrm{pr}^*_{12}f)^{-1}(\mathrm{pr}^*_{24}h)\cdot \mathrm{pr}^*_{124}g\cdot\mathrm{pr}^*_{14}h^{-1}\right)
&\cdot\left(\mathrm{pr}^*_{14}h\cdot\mathrm{pr}^*_{134}g^{-1}\cdot(\mathrm{pr}^*_{13}f)^{-1}(\mathrm{pr}^*_{34}h)^{-1}\cdot\mathrm{pr}^*_{13}h^{-1}\right)\\
&\cdot\left(\mathrm{pr}^*_{13}h\cdot\mathrm{pr}^*_{123}g^{-1}\cdot(\mathrm{pr}^*_{12}f)^{-1}(\mathrm{pr}^*_{23}h)^{-1}\cdot\mathrm{pr}^*_{12}h^{-1}\right)
\end{aligned}\\
&\;\begin{aligned}
= S\cdot \left((\mathrm{pr}^*_{12}f)^{-1}(\mathrm{pr}^*_{234}g)\right.\cdot&\left.(\mathrm{pr}^*_{12}f)^{-1}(\mathrm{pr}^*_{24}h)^{-1}\cdot\mathrm{pr}^*_{12}h^{-1}\right)\\
\cdot\,\mathrm{pr}^*_{12}h\;\cdot&\;(\mathrm{pr}^*_{12}f)^{-1}(\mathrm{pr}^*_{24}h)\cdot \left(\mathrm{pr}^*_{124}g\cdot\mathrm{pr}^*_{134}g^{-1}\cdot\mathrm{pr}^*_{123}g^{-1}\right)\cdot S^{-1}
\end{aligned}\\
&= S\cdot\mathrm{d}g\cdot S^{-1} = \mathrm{d}g
\end{align*}
where we have used that $\mathrm{d}g$ is central and where $S$ is defined as:
\begin{align*}
S \coloneqq&\; \mathrm{pr}^*_{12}h\circ(\mathrm{pr}^*_{12}f)^{-1}(\mathrm{pr}^*_{23}h)\cdot(\mathrm{pr}^*_{12}f\circ\mathrm{pr}^*_{23}f)^{-1}(\mathrm{pr}^*_{34}h)\\
=&\; \mathrm{pr}^*_{12}h\circ(\mathrm{pr}^*_{12}f)^{-1}(\mathrm{pr}^*_{23}h)\cdot\mathrm{pr}^*_{123}g\cdot(\mathrm{pr}^*_{13}f)^{-1}(\mathrm{pr}^*_{34}h)\cdot\mathrm{pr}^*_{123}g^{-1}
\end{align*}
The obstruction class $\check e(L)$ is thus defined independently of choices. It is the neutral class if and only if, for an arbitrary $(f,g)$ as above, we may write $\mathrm dg = \mathrm dg_0$ with $g_0\in\mathrm Z_{\mathcal F}(k'\otimes_k k'\otimes_k k')$. Equivalently, the pair $(f, g_0^{-1}g)$ is a cocycle. This shows that $\check e(L)$ satisfies the second point of the proposition statement.

Next, we show that $\check e(L)$ satisfies the third point. For any choice of $(f,g)$ as above, there is an analogous pair $(f\times\varphi_C,[g,1])$ for $L\wedge^{\mathrm Z_L}\!C$. A representative of $-\check e(L\wedge^{\mathrm Z_L}\!C)$ is thus given by $\mathrm{d}[g,1] = [\mathrm{d}g,1] = [1,\mathrm{d}g]$. However, this corresponds to the image of $\mathrm{d}g$ in $C(k'\otimes_k k'\otimes_k k'\otimes_k k')$ and therefore $\check e(L)$ maps to $\check e(L\wedge^{\mathrm Z_L}\!C)$ in $\mathrm{\check H}^3(k,C)$.

Finally, we return to the first point in the statement. Note immediately that the composition (in which the first map is actually an isomorphism by the $5$-lemma, which we do not use)
$$\check\delta \;:\; \mathrm{\check H}^2(k, Q)\rightarrow\mathrm{H}^2(k, Q)\xrightarrow{\;\;\;\delta\;\;\;}\mathrm{H}^3(k,\mathrm Z_L)\cong\mathrm{\check H}^3(k,\mathrm Z_L)$$
can, by \cite[Prop. F.2.1]{RosTD}, be seen as the connecting homomorphism from the snake lemma: Take a cocycle $q\in\mathrm{\check Z}^2(k'/k, Q)$ and suppose that it admits a preimage $c\in C(k'\otimes_k k'\otimes_k k')$ (this in fact always happens by virtue of Lemma \ref{lemrosentensfield}(a)). Since $\mathrm dq = 0$, the element $\mathrm dc$ lies in $\mathrm Z_L(k'\otimes_k k'\otimes_k k'\otimes_k k')$ and it is a cocycle by $\mathrm d^2c = 0$. Then $\check\delta([q]) = [\mathrm dc]$.

In our situation, we have an element $(F,\,[g,c])\in\mathrm{\check Z}^2(k'/k,\mathcal F\times_{\mathrm Z_{\mathcal F}} C,\varphi\times\varphi_C)$ mapping to some $q = \overline{c}\in\mathrm{\check Z}^2(k'/k,Q)$ by the definition of $\check v^{(2)}$. Moreover, we are free to assume that $F = f\times\varphi_C$ (up to enlarging $k'/k$) for some lift $f$ of $\varphi$, without affecting the given class in $\mathrm{\check H}^2(k, L\wedge^{\mathrm Z_L}\!C)$. 
Then $(\mathrm{pr}_{13}^*F)^{-1}\circ(\mathrm{pr}_{23}^*F)\circ(\mathrm{pr}_{12}^*F) = \mathrm{int}([g,c])^{-1}$ implies $(\mathrm{pr}_{13}^*f)^{-1}\circ(\mathrm{pr}_{23}^*f)\circ(\mathrm{pr}_{12}^*f) = \mathrm{int}(g)^{-1}$.
Next, we have $0 = \mathrm d[g,c] = [\mathrm dg,\mathrm dc] = [1,\mathrm dg+\mathrm dc]$, since $\mathrm dg$ is central. This shows that $\mathrm dc = -\mathrm dg$ and, finally, that $\check e(L) = -[\mathrm dg] = [\mathrm dc] = \check\delta([q])$ as required. 

To finish the proof we must still show that there does exist a monomorphism of commutative locally algebraic groups $\mathrm Z_L\hookrightarrow C$ such that $\mathrm{\check H}^2(k, L\wedge^{\mathrm Z_L}\!C)\neq\varnothing$; equivalently that $\check e(L\wedge^{\mathrm Z_L}\!C) = 0$. By the third point in the statement, it suffices that $\check e(L)$ maps to $0$ in $\mathrm{\check H}^3(k,C)\cong\mathrm{H}^3(k,C)$.

By Proposition \ref{proph2h1coh}, $\mathrm Z_L$ admits a monomorphism to a smooth commutative locally algebraic group $Z$. We~choose $C$ to be the Weil restriction $\mathrm{R}_{k'/k}(Z_{k'})$ (well-defined by Proposition \ref{proplocalgquasiproj}) for some finite field extension $k'/k$, with $k'$ to be specified later.
By Proposition \ref{propweiladjmorph}, there is now a monomorphism
$$j \;:\; \mathrm Z_L\xhookrightarrow{\;\;\;\;\;\;} Z\xhookrightarrow{\;\;\;\;\;\;}\mathrm{R}_{k'/k}(Z_{k'}) = C$$
of locally algebraic groups. The composition
$$\mathrm{H}^3(k, \mathrm Z_L)\longrightarrow\mathrm{H}^3(k, Z)\longrightarrow\mathrm{H}^3(k, C)=\mathrm{H}^3(k, \mathrm{R}_{k'/k}(Z_{k'}))\xrightarrow{\;\;\sim\;\;}\mathrm{H}^3(k', Z)$$
incorporates both $j_* : \mathrm{H}^3(k,\mathrm Z_L)\rightarrow\mathrm{H}^3(k, C)$ and the natural map $\mathrm{H}^3(k, Z)\rightarrow\mathrm{H}^3(k', Z)$. The isomorphism at the end appears by Shapiro's lemma (Proposition \ref{propshapiro}), since $Z$ is smooth. Therefore, by sufficiently enlarging $k'$, we may assume that the (fixed) image of $\check e(L)$ in $\mathrm{H}^3(k, Z)$ maps to $0$ in $\mathrm{H}^3(k', Z)$. This gives the desired group $C$.
\end{prf}


In \cite[VI, Thm.\ 2.3]{Gir71}, the existence of a class $e(L)\in\mathrm{H}^3(k, \mathrm Z_L)$ is shown, with an analogous universal property and the property that $e(L) = 0$ if and only if $\mathrm{H}^3(k, L)\neq\varnothing$. We now state the final result of this section:

\begin{cor}\label{cornicelyrepcech2}
Let $L$ be a nicely representable fppf band such that $\mathrm Z_L$ is a locally algebraic group over $k$. The canonical function $\mathrm{\check H}^2(k, L)\rightarrow\mathrm{H}^2(k, L)$ from Definition \ref{defcechtogerbe} is a bijection.
\end{cor}
\begin{prf}
In view of Corollary \ref{cornicelyrepcech1}, in this proof we will only need to show that the isomorphism $\mathrm{\check H}^3(k, \mathrm Z_L)\rightarrow\mathrm{H}^3(k, \mathrm Z_L)$ maps $\check e(L)$ to Giraud's $e(L)$ (see paragraph above). Then the condition $e(L) = 0$ will imply $\check e(L) = 0$.

Take a short exact sequence $0\rightarrow\mathrm Z_L\rightarrow C\rightarrow Q\rightarrow 0$ of commutative fppf group sheaves over $k$ such that $C$ is a locally algebraic group and $\mathrm{\check H}^2(k, L\wedge^{\mathrm Z_L}\!C)\neq\varnothing$ (such a sequence was constructed in the preceding proof). We write the following diagram

\begin{center}\begin{tikzcd}
    \mathrm{\check H}^2(k, L\wedge^{\mathrm Z_L}\!C)\arrow[r, "\check v^{(2)}"]\arrow[d] & \mathrm{\check H}^2(k, Q)\arrow[r, "\check\delta"]\arrow[d] & \mathrm{\check H}^3(k,\mathrm Z_L)\arrow[d]\\
    \mathrm{H}^2(k, L\wedge^{\mathrm Z_L}\!C)\arrow[r, "v^{(2)}"] & \mathrm{H}^2(k, Q)\arrow[r, "\delta"] & \mathrm{H}^3(k,\mathrm Z_L)
\end{tikzcd}\end{center}
in which the map $v^{(2)}$ is as in \cite[VI, \textsection2]{Gir71} and the rest are clear. It suffices to show that this diagram commutes, since then the image of $\check e(L)$ in $\mathrm{H}^3(k,\mathrm Z_L)$ must lie in the one-element set $\{e(L)\}$ by \cite[VI, Thm.\ 2.3]{Gir71}. Moreover, we only need to prove that the left square of the diagram commutes, as for the right square this was already noted above.

We recall the construction of the map $v^{(2)}$: The band $L\wedge^{\mathrm Z_L}\!C$ is defined in \cite[IV, \textsection1.6]{Gir71} as the ``cokernel of $\mathrm Z_L\rightrightarrows L\times C$'' (and it is clear that our construction satisfies this universal property). The map $L\times C\twoheadrightarrow C\rightarrow Q$ defines a ``morphism of bands'' $v : L\wedge^{\mathrm Z_L}\!C\rightarrow Q$ which induces $v^{(2)}$. In more explicit terms, this is simply a descent condition on $\mathcal F\times_{\mathrm Z_{\mathcal F}}C\rightarrow Q$ (again, this map is even fixed by inner automorphisms since its kernel is $\mathcal F$).

Suppose given a cocycle $(F,[g,c])\in\mathrm{\check Z}^2(k'/k,\mathcal F\times_{\mathrm Z_{\mathcal F}} C,\varphi\times\varphi_C)$ and, by our construction, the corresponding $(L\wedge^{\mathrm Z_L}\!C)$-gerbe $\mathscr X$. A $Q$-gerbe $\mathscr Y$ is a representative of $v^{(2)}([\mathscr X])$ if and only if the morphism of gerbes over $k'$
$$T \;:\; \mathscr X_{k'}\cong\mathrm{TORS}(\mathcal F\times_{\mathrm Z_{\mathcal F}}C)\longrightarrow\mathrm{TORS}(Q)\cong\mathscr Y_{k'}$$
given by the ``change of structure'' functor $P\rightsquigarrow P\times^{\mathcal F\times_{\mathrm Z_{\mathcal F}}C} Q$ descends to $k$. The $2$-descent datum on $\mathrm{TORS}(\mathcal F\times_{\mathrm Z_{\mathcal F}}C)$ is the one defined by $(F,[g,c])$. It is now easily checked that, if $\mathrm{TORS}(Q)$ is equipped with the $2$-descent datum coming from the cocycle $\overline{c}\in\mathrm{\check Z}^2(k'/k,Q)$ representing $\check v^{(2)}([F,[g,c]])$, then the obvious natural transformation relating the two sides satisfies a similar cocycle condition to the one in the proof of Proposition \ref{proph2actcomm}. This now implies that $v^{(2)}([\mathscr X])$ is indeed the image of $\check v^{(2)}([F,[g,c]])$ in $\mathrm{H}^2(k,Q)$, which ends the proof.
\end{prf}

%% file: part2-1.tex
\subsection{Bands Represented by Weil Restrictions}
Recall that a ``reductive'' group in this paper is reductive connected. We take as our starting point the following well-known result:

\begin{thm}\label{thmreddouai}
Every \'etale band $(\Gb, \kappa)$ on $k$ locally represented by a reductive group $\Gb$ is globally representable (by a quasi-split reductive group).
\end{thm}

\begin{rem}
This was proven by Douai (\cite[V, Prop.\ 3.2]{Dou76}) in much greater generality, over any base scheme and in \'etale, fppf or fpqc topology. A slightly more concrete proof of this fact is given by Borovoi (\cite[\textsection 3]{Brv93}) for \'etale bands over a field $k$ of characteristic $0$, using the canonical retraction $\saut(\Gb/k)\rightarrow(\Gb/\mathrm Z_{\Gb})(k_s)$ given by split $k$-forms. 
When $\mathrm{char}(k) > 0$, one sees that the same proof goes through for separable bands, but we do not use this.
\end{rem}

The main results of this section crucially depend on the following abstract property of fields:

\begin{prop}
Suppose that $[k : k^p] = p$. Then the following statements hold:
\begin{enumerate}[(a)]
    \item If $k'/k$ is purely inseparable (and algebraic), then $k' = k^{p^{-n}}$ for some $n\in\mathbf Z_{\geq 0}\cup\{\infty\}$.
    \item If $K'/k$ is algebraic, then there exist $k',K\subseteq K'$ such that $K/k$ is separable, $k'/k$ is purely inseparable and $K' = Kk' = K\otimes_k k'$ (as $k$-algebras).
\end{enumerate}
\end{prop}
\begin{prf}
For (a), recall that the minimal polynomial over $k$ of an element $\alpha\in k'$ is of the form $X^{p^n}-\alpha^{p^n}\in k[X]$ for some $n = n(\alpha)$. Then $k(\alpha)\subseteq k^{p^{-n}}$ and $[k(\alpha):k] = p^n$. By assumption, $[k^{p^{-n}} : k] = p^n$ and thus $k(\alpha) = k^{p^{-n}}$. Finally, $k' = \bigcup_{\alpha\in k'}k(\alpha)$.

For (b), first let $K\coloneqq k_s\cap K'$, so that $K/k$ is separable. We have $[K : K^p] = [k : k^p] = p$ by \cite[Thms.\ 25.3, 26.7]{Mat90}, so applying (a) to $K'/K$ gives $K' = K^{p^{-n}}$ for some $n\in\mathbf Z_{\geq 0}\cup\{\infty\}$. Choosing $k'\coloneqq k^{p^{-n}}$ then gives $Kk' = K^{p^{-n}} = K'$ by comparing degrees over $K$.
\end{prf}

In particular, for an arbitrary algebraic field extension $k'_s/k_s$, the extension $k'_s/k$ is normal. Moreover, if $k'/k$ is such that $k'_s = k_s\otimes_k k'$, then $\Gal(k'_s/k') = \Gal(k_s/k)\eqqcolon\Gamma$.

\begin{prop}\label{propsemiautpsred}
Suppose that $[k : k^p] = p$ and let $k'_s/k_s$ be a finite field extension. Let $\Ga$ be a pseudo-reductive group over $k'_s$ and denote by $\Gb = \mathrm R(\Ga)$ its Weil restriction to $k_s$. Then there is a natural commutative diagram
\begin{center}\begin{tikzcd}
    \Ga(k'_s)\arrow[r]\arrow[d, equal] & \aut(\Ga)\arrow[r, hook]\arrow[d, "\rotatebox{90}{$\sim$}", pos=0.4] & \saut(\Ga/k')\arrow[d, "\rotatebox{90}{$\sim$}", pos=0.4]\\
    \Gb(k_s)\arrow[r] & \aut(\Gb)\arrow[r, hook] & \saut(\Gb/k)
\end{tikzcd}\end{center}
where the left horizontal maps are given by conjugation and the vertical maps are isomorphisms.
\end{prop}
\begin{prf}
The middle vertical arrow $Q : \aut(\Ga)\rightarrow\aut(\Gb)$ is defined by the functoriality of Weil restriction (which preserves the group structures on $\Ga$, $\Gb$). This also shows commutativity of the left square, because $\Ga(k'_s)$ is the set of sections of the structure morphism $\Ga\rightarrow\Spec(k'_s)$ and $\mathrm R(\Spec(k'_s)) = \Spec(k_s)$. The remaining vertical arrow $\saut(\Ga)\rightarrow\saut(\Gb)$ is analogous and an extension of $\mathrm{id}_\Gamma : \Gamma\rightarrow\Gamma$ by $Q$, so it suffices to show that $Q$ is an isomorphism.

We define an inverse $P$ to $Q$: Pick $\alpha\in\aut(\Gb)$ and consider $\alpha_{k'_s}\in\aut(\smash{\Gb_{k'_s}})$. The adjoint counit $q : \smash{\Gb_{k'_s}}\rightarrow\Ga$ is a surjection and $\ker(q)$ is exactly the unipotent radical by Proposition \ref{propweilrestrunip}, thus invariant under automorphisms of $\Gb$. Therefore $\alpha_{k'_s}$ induces via the surjective map $q$ an automorphism $P(\alpha)\in\aut(\Ga)$. Moreover, as the following commutative diagram shows
\begin{center}\begin{tikzcd}[row sep = tiny]
& &&& \Gb_{k'_s}\arrow[rrr, "q"] & \arrow[ddddd, "\mathrm{R}", shorten < = 3mm, shorten > = 4mm, squiggly={pre length=10pt, post length=15pt}] && \Ga\\
&&& \Gb_{k'_s}\arrow[rrr, "q"]\arrow[ur, "\alpha_{k'_s}"] &&& \Ga\arrow[ur, "P(\alpha)", swap, pos=0.3]\\
\; \\
\; \\
& \Gb\arrow[rrr, "\iota"] &&& \mathrm{R}(\Gb_{k'_s})\arrow[rrr, "\mathrm{R}(q)"] &&& \Gb\\
\Gb\arrow[rrr, "\iota"]\arrow[ur, "\alpha"] &&& \mathrm{R}(\Gb_{k'_s})\arrow[rrr, "\mathrm{R}(q)"]\arrow[ur, "\mathrm{R}(\alpha_{k'_s})", pos=0.2] && \; & \Gb\arrow[ur, "\mathrm{R}(P(\alpha)) \;=\; Q(P(\alpha))", swap]
\end{tikzcd}\end{center}
the automorphisms $\alpha$ and $Q(P(\alpha))$ are related by the identity map $\mathrm{id}_{\Gb} = \mathrm{R}(q)\circ\iota$ coming from the definition of an adjoint pair, and thus they agree.
Finally, if $\alpha'\in\aut(\Ga)$, then it is related to $Q(\alpha')_{k'_s}$ by $q$ (this can be checked on the functor of points, using Proposition \ref{propweiladjmorph}
) and we thus have an equality $P(Q(\alpha')) = \alpha'$.
\end{prf}

\begin{cor}\label{correprred}
In the notation of the above proposition, let $\Ga$ be a reductive group over $k'_s$. Then every \'etale band on $k$ of the form $(\Gb,\kappa)$ is globally representable.
\end{cor}
\begin{prf}
There is the following commutative diagram with exact rows
\begin{center}\begin{tikzcd}
    1\arrow[r] & \dfrac{\Ga(k'_s)}{\mathrm Z_\Ga(k'_s)}\arrow[r]\arrow[d, "\rotatebox{90}{$\sim$}", pos=0.4] & \saut(\Ga/k')\arrow[r, hook]\arrow[d, "\rotatebox{90}{$\sim$}", pos=0.4] & \dfrac{\saut(\Ga/k')}{\Ga(k'_s)/\mathrm Z_\Ga(k'_s)}\arrow[d, "\rotatebox{90}{$\sim$}", pos=0.4]\arrow[r] & 1\\
    1\arrow[r] & \dfrac{\Gb(k_s)}{\mathrm Z_\Gb(k_s)}\arrow[r] & \saut(\Gb/k)\arrow[r, hook] & \dfrac{\saut(\Gb/k)}{\Gb(k_s)/\mathrm Z_\Gb(k_s)}\arrow[r] & 1
\end{tikzcd}\end{center}
and vertical isomorphisms; the formation of Weil restrictions commutes with taking centers by Proposition \ref{propweilrestrunip}, so $\mathrm Z_\Gb = \mathrm R(\mathrm Z_\Ga)$. The data of the band $(\Gb,\kappa)$ clearly agrees with the data of a band of the form $(\Ga,\kappa')$. Now, by Theorem \ref{thmreddouai}, the second band admits a continuous lift which is a homomorphism, and hence so does the first one.
\end{prf}

The above proof is made short by the identification of $\Gal(k'_s/k')$ = $\Gal(k_s/k)$, which we have seen is possible when $[k : k^p] = p$. We now show that this representability fails in general without such an assumption (and even for separable bands).

For this, suppose that $\Gb$ is a perfect pseudo-reductive group (that is, $\mathcal D(\Gb) = \Gb$; what is also called a \textit{pseudo-semisimple group} in \cite{CP15}). It is shown in \cite[Prop.\ 6.2.2]{CP15} that the group sheaf $\mathcal{Aut}_\Gb$ on $k_s$ is representable by an affine algebraic group, which then admits a maximal smooth algebraic subgroup $\mathcal{Aut}^\mathrm{sm}_\Gb$ and for which $\mathcal{Aut}^\mathrm{sm}_\Gb(k_s) = \mathcal{Aut}_\Gb(k_s) = \aut(\Gb)$ (cf.\ \cite[Lem.\ C.4.1]{CGP15}). The connected-\'etale sequence of $\mathcal{Aut}^\mathrm{sm}_\Gb$ then gives a short exact sequence on $k_s$-points, where $E_\Gb$ is a finite \'etale group,
$$1\longrightarrow \big(\mathcal{Aut}^\mathrm{sm}_\Gb\big)^0(k_s)\longrightarrow \aut(\Gb)\longrightarrow E_\Gb(k_s)\longrightarrow 1$$
and there is also a short exact sequence of algebraic groups shown in \cite[Prop.\ 6.2.4]{CP15}
$$1\longrightarrow \Gb/\mathrm Z_\Gb\longrightarrow \big(\mathcal{Aut}^\mathrm{sm}_\Gb\big)^0\longrightarrow \frac{Z_{\Gb,\,\Cb}}{\Cb/\mathrm Z_\Gb}\longrightarrow 1$$
where $\Cb$ is a fixed Cartan subgroup of $\Gb$, and $Z_{\Gb,\,\Cb} = \mathcal{Aut}^\mathrm{sm}_{\Gb,\,\Cb}$ is the maximal smooth algebraic subgroup of the commutative and representable functor of automorphisms of $\Gb$ fixing $\Cb$. Since all of these groups are smooth, this in particular shows that the following sequence is exact:
\begin{equation}\label{eqautcp}
    1\longrightarrow \frac{Z_{\Gb,\,\Cb}}{\Cb/\mathrm Z_\Gb}(k_s)\longrightarrow \frac{\aut(\Gb)}{(\Gb/\mathrm Z_\Gb)(k_s)}\longrightarrow E_\Gb(k_s)\longrightarrow 1
\end{equation}
In \cite[\textsection6.3]{CP15}, the map $E_\Gb(k_s)\hookrightarrow\aut(\mathrm{Dyn}(\Gb))$ to the automorphism group of the ``canonical Dynkin diagram'' is introduced, shown to be injective, and even a bijection when the group $\Gb$ is absolutely pseudo-simple, as will be the case in the following example:

\begin{exmp}\label{exmprepringen}
We choose a field $k$ and an extension $k'_s/k_s$ such that there does not exist $k'/k$ with $k'_s = k_sk'$ (take any finite extension $k''/k$ which admits no purely inseparable~subextension $k'/k$ such that $k''/k'$ is separable; it then suffices to take $k'_s\coloneqq k_sk''$). In this situation, the extension $k'_s/k$ is not normal (cf. \cite[V, Prop.\ 6.11]{Lan02}), and the automorphisms $t\in\Gamma\coloneqq\Gal(k_s/k)$ which extend to automorphisms of $k'_s/k$ form a proper open subgroup $\Gamma_K\subsetneq\Gamma$ of finite~index (corresponding to a finite separable extension $K/k$ over which there exists $K'/K$ such that $k'_s = k_sK'$). We may also assume that $k'_s\subseteq k_s^{1/p^n}$ for simplicity. Over any field~$k$~which~admits this setup, we will now exhibit a separable band $(\Gb,\kappa)$, locally represented by a pseudo-reductive group $\Gb$, but not globally representable. The proof follows several steps:
\smallskip

\textit{Step 1:} We take $\Ga\coloneqq\mathrm{SL}_{p^n\!,\,k_s}$ and $\Gb\coloneqq\mathrm R(\Gaa)$ over $k_s$. As $\Ga$ is absolutely simple,~semisimple and simply connected, the group $\Gb$ is perfect by \cite[Cor.\ A.7.11]{CGP15}. Moreover, we claim that there exists no algebraic group $G$ over $k$ such that $G_{k_s}\simeq\Gb$: Indeed, by Theorem \ref{thmgenstandfunct}(b), this would imply the existence of an extension $k'/k$ such that $k'_s = k_s\otimes_k k'$, which is a contradiction. It now remains to construct any separable band on $k$ which is of the form $(\Gb,\kappa)$.

\textit{Step 2:} For this, we first note that the naive attempt at an assignment 
\begin{equation}\label{eqmapinexmpdynk}
    f_t\longmapsto\mathrm R_{k'_s/k_s}(f_{t,\,k'_s})
\end{equation}
between $\saut(\Ga/k)$ and $\saut(\Gb/k)$ is well-defined only for $t\in\Gamma_K$, in the above notation. However, there is a canonical isomorphism $\mathrm{Dyn}(\Ga)\cong\mathrm{Dyn}(\Gb)\eqqcolon D$ (analogous to the based root datum, see \cite[Rem.\ 6.3.6,\, Exmp.\ 6.3.8]{CP15}) and we may define $\saut(D/k)$ as the group of ``semiautomorphisms'' $(s,\; \alpha: s_*D\rightarrow D)$ with $s\in\Gamma$. Here $s_*D$ is defined naturally by pulling back the data defining $D$ (on $\Ga$, equivalently $\Gb$; see \cite[pp.\ 624, 625]{Mil17} for the definitions). In particular, there are two (a priori unrelated) canonical isomorphisms $s_*D\cong\mathrm{Dyn}(s_*\Ga)$ and $s_*D\cong\mathrm{Dyn}(s_*\Gb)$. In view of the discussion following \eqref{eqautcp}, we have just defined isomorphisms (which commute with projections to $\Gamma$):
$$\frac{\saut(\Ga/k)}{(\Ga/\mathrm Z_\Ga)(k_s)}\xrightarrow{\;\;\sim\;\;}\saut(D/k)\xleftarrow{\;\;\sim\;\;}\frac{\saut(\Gb/k)}{\big(\mathcal{Aut}^\mathrm{sm}_\Gb\big)^0(k_s)}$$
There is an obvious band $(\Ga,\kappa')$ represented globally by a split $k$-form $G'$ of $\Ga$ over $k$ (given by the Existence Theorem for split reductive groups) with an associated continuous lift $f$ which is a homomorphism. It defines a homomorphic section $\overline{\kappa} : \Gamma\longrightarrow\saut(\Gb/k)/\big(\mathcal{Aut}^\mathrm{sm}_\Gb\big)^0(k_s)$. We pick a lift $\overline{f} : \Gamma\rightarrow\saut(\Gb/k)$ of $\overline{\kappa}$, which we can assume to be ``continuous'' (in the sense that $\overline{f}_{st} = \overline{f}_{s}\cdot\mathrm R(f_{t,\,k'_s})$ for $t\in\Gamma_K$, by \eqref{eqmapinexmpdynk}), but not a homomorphism in general.

In fact, we remark that the full power of the Existence Theorem is not actually necessary for this step: The section $\overline{\kappa}$ is already given by the fact that $s_*D$ and $D$ are canonically isomorphic (by taking each vertex $s_*v$ to $v$). A continuous lift $\overline{f}$ also exists since $\Gb$ is of finite type, so we may take an extension $k_0/k$ and some $k_0$-form $G_{k_0}$ which is \textit{pseudo-split} (i.e.\ admits a split maximal torus) and thus locally agrees with $\overline{\kappa}$.

\textit{Step 3:} We will now modify $\overline{f}$ so as to lift the section $\overline{\kappa}$ to a band $(\Gb,\kappa)$. Arguing as in the proof of Proposition \ref{propparamlet}, the obstruction to this lifting problem lies in $\mathrm{H}^2(\Gamma, U(k_s))$, where
$$U(k_s)\coloneqq\ker\left(\frac{\saut(\Gb/k)}{(\Gb/\mathrm Z_\Gb)(k_s)}\longrightarrow\frac{\saut(\Gb/k)}{\big(\mathcal{Aut}^\mathrm{sm}_\Gb\big)^0(k_s)}\right)\cong\frac{Z_{\Gb,\,\Cb}}{\Cb/\mathrm Z_\Gb}(k_s)$$
and the isomorphism on the right comes from \eqref{eqautcp}. Moreover, the set of all such bands $(\Gb,\kappa)$, if nonempty, is parametrized by the group $\mathrm{H}^1(\Gamma,U(k_s))$. We may assume that $\Cb = \mathrm R(\Caa)$ for some Cartan subgroup $\Ca$ of $\Ga$, and recall that then $Z_{\Ga,\,\Ca} = \Ca/\mathrm Z_{\Ga}$ because $\Ga$ is reductive. This allows us to calculate, for the projection $\pi : \Spec(k'_s)\rightarrow\Spec(k_s)$, that
$$U\coloneqq\frac{Z_{\Gb,\,\Cb}}{\Cb/\mathrm Z_\Gb} = \frac{Z_{\mathrm{R}(\Gaa),\,\mathrm{R}(\Caa)}}{\mathrm{R}(\Caa)/\mathrm Z_{\mathrm{R}(\Gaa)}}
\cong\frac{\mathrm{R}\!\left(\big(Z_{\Ga,\,\Ca}\big)_{k'_s}\right)}{\mathrm{R}(\Caa)/\mathrm Z_{\mathrm{R}(\Gaa)}} = \frac{\mathrm{R}\!\left(\Caa/\mathrm Z_{\Ga,\,k'_s}\right)}{\mathrm{R}(\Caa)/\mathrm{R}(\mathrm Z_{\Ga,\,k'_s})}\cong \mathrm R^1\pi_*\big(\mathrm Z_{\Ga,\,k'_s}\big)$$
where the first isomorphism comes from \cite[Prop.\ 6.1.7]{CP15}. The second isomorphism is shown as in Example \ref{exmpweilrestr}, in which it is also proven that $U\cong\mathrm R^1\pi_*(\mu_{p^n,\,k'_s})$ is a unipotent group. In particular, $\mathrm{H}^2(\Gamma, U(k_s)) = 0$ by Proposition \ref{propunipvanish}.

We have thus shown the existence of a separable band $(\Gb,\kappa)$ on $k$ with the desired properties. As an additional remark, note that $\mathrm{H}^1(\Gamma,U(k_s)) = \mathrm{H}^1(k, U)$ is nonzero for general $k$ in view of Example \ref{exmpweilrestr}, in which case the lifting problem in Step 3 even has multiple solutions. 
\end{exmp}

%% file: part2-2.tex
\subsection{Reduction Steps in Low Characteristic}
We work over a field $k$ such that $[k : k^p] = p$. To extend the statement of Corollary \ref{correprred} to an arbitrary pseudo-reductive group $\Gb$ over $k_s$, we employ the structure theory recalled in Subsections \ref{ssectpsred2} and \ref{ssectpsred3}: There is a factorization $\Gb = \Gb_1\times\Gb_2$ uniquely functorial with respect to isomorphisms in $\Gb$, where $\Gb_1$ is generalized standard and $\Gb_2$ is totally non-reduced (and occurs only when $p = 2$). It suggests the treatment of two important and, in practice, similar cases in our proof before the rest; those are the totally non-reduced case and the ``primitive'' case of generalized standard groups (which may involve basic exotic groups when $p\in\{2,3\}$), respectively:

\begin{lem}\label{lemreprtotnonred}
If $\Gb$ is a totally non-reduced pseudo-reductive group over $k_s$, then every \'etale band over $k$ of the form $(\Gb,\kappa)$ is globally representable.
\end{lem}
\begin{prf}
By Proposition \ref{propredbasnonred}(a), there is a nonzero finite reduced $k_s$-algebra $k'_s = \prod_i k'_i$, a group $\Ga$ over $k'_s$ with basic non-reduced pseudo-simple fibers and an isomorphism $j : \mathrm{R}(\Ga)\rightarrow\Gb$ such that the triple $(k'_s/k_s,\Ga,j)$ is uniquely functorial in isomorphisms of $\Gb$. We immediately reduce to the case when the fields $k'_i$ are all isomorphic as $k_s$-algebras (otherwise the automorphisms of $k'_s/k_s$ do not permute them transitively, so any band can be written as a product of bands), and then to the analogous situation where $k'_s/k_s$ is a field extension and $\Ga$ is a finite product of basic non-reduced pseudo-simple groups over $k'_s$. Then, in particular, any automorphism of $\Ga$ induces a permutation of this finite set of factors by pseudo-simplicity.

Proposition \ref{propredbasnonred}(b) shows that there is a factor-wise map $\Ga\rightarrow\mathrm R_{K/k'_s}\big((\Ga)^{ss}_K\big)\eqqcolon \overline{\mathcal G'}$, where $K\coloneqq(k'_s)^{1/2}$ and $(\Ga)^{ss}_K\simeq\mathrm{Sp}_{2n,\,K}$. This map induces an isomorphism on $k'_s$-points, and it is uniquely functorial with respect to isomorphisms in $\Ga$ (since any such isomorphism is just a product of maps followed by a permutation). Weil restrictions therefore define a homomorphism $\Gb = \mathrm R(\Ga)\longrightarrow\mathrm R(\overline{\mathcal G'})\eqqcolon\overline{\mathcal G}$ such that $\Gb(k_s)\cong\overline{\mathcal G}(k_s)$, which is also uniquely functorial with respect to isomorphisms in $\Gb$ by Proposition \ref{propsemiautpsred}. As the formation of $\overline{\mathcal G}$ commutes with pullbacks $s_*$ for $s\in\Gamma$, any \'etale band $(\Gb,\kappa)$ defines a band $(\overline{\mathcal G},\widetilde{\kappa})$.

Because the center of $\mathrm{Sp}_{2n}$ is $\mu_2$ (and the groups $\Gb,\overline{\mathcal G}$ are smooth), we have equalities:
\vspace{-10pt}

$$1 = \mathrm Z_{\overline{\mathcal G}}(k_s) = \mathrm Z_{\overline{\mathcal G}(k_s)} = \mathrm Z_{\Gb(k_s)} = \mathrm Z_\Gb(k_s)$$
\vspace{-10pt}

\noindent
It follows that both $\mathrm H^2(k,\Gb,\kappa)$ and $\mathrm H^2(k,\overline{\mathcal G},\widetilde{\kappa})$ are single-element sets, where we also know that $\mathrm N^2(k,\overline{\mathcal G},\widetilde{\kappa})\neq\varnothing$ by Corollary \ref{correprred}. 
Any cocycle $(f,g)\in\mathrm Z^2(k,\Gb,\kappa)$ induces a cocycle $(\widetilde{f},g)$ representing the neutral class of $(\overline{\mathcal G},\widetilde{\kappa})$. The locally constant function $h : \Gamma\rightarrow\Gb(k_s)\cong\overline{\mathcal G}(k_s)$ which relates $(\widetilde{f},g)$ to $(\mathrm{int}(h)\circ\widetilde{f},1)$ then also relates $(f,g)$ to $(\mathrm{int}(h)\circ f,1)$.
\end{prf}

\begin{lem}\label{lemreprprim}
Suppose given a nonzero finite reduced $k_s$-algebra $k'_s = \prod_{i=1}^s k'_i$ and a $k'_s$-group $\Ga$ such that each fiber $\Ga_{k'_i}$ is absolutely pseudo-simple and either semisimple simply connected or basic exotic pseudo-reductive. If $\Gb\coloneqq\mathrm R(\Ga)$ is the associated Weil restriction over $k_s$, then every \'etale band over $k$ of the form $(\Gb,\kappa)$ is globally representable.
\end{lem}
\begin{prf}
Arguing as in the first half of the previous proof, we may reduce by Theorem \ref{thmgenstandfunct}(b)~and Proposition \ref{propsemiautpsred} to the situation where $k'_s/k_s$ is a field extension, $\overline{\mathcal G'}$ is a reductive group over $k'_s$ (we take $\overline{\mathcal G'} = \Ga$ if $\Ga$ is semisimple simply connected; otherwise we use Proposition \ref{propredbasexpsred}), there is a uniquely functorial map $\Gb\rightarrow\mathrm R(\overline{\mathcal G'})\eqqcolon\overline{\mathcal G}$ such that $\Gb(k_s)\cong\overline{\mathcal G}(k_s)$ and along which any \'etale band $(\Gb,\kappa)$ defines a band $(\overline{\mathcal G},\widetilde{\kappa})$. Note that, as in the previous proof, we drop the (absolute) pseudo-simplicity condition on $\Ga$: this group $\Ga$ is only a \textit{product} of pseudo-simple groups, which are permuted by isomorphisms (same for $\Gb$, $\overline{\mathcal G'}$, $\overline{\mathcal G}$).

Proposition \ref{propredbasexpsred}(b) now says that the map $\mathrm{Isom}(s_*\Ga,\Ga)\rightarrow\mathrm{Isom}(s_*\overline{\mathcal G'},\overline{\mathcal G'})$ is a bijection for all $s\in\Gamma$; by Proposition \ref{propsemiautpsred} we get the same for $\mathrm{Isom}(s_*\Gb,\Gb)\rightarrow\mathrm{Isom}(s_*\overline{\mathcal G},\overline{\mathcal G})$. Applying Corollary \ref{correprred}, we find a continuous lift $\widetilde{f}$ of $\widetilde{\kappa}$ which is a homomorphism. The corresponding lift $f$ of $\kappa$, for which
$\!\widetilde{\!\phantom{f}\smash{f_sf_tf_{st}^{-1}}} = \widetilde{f_s}\widetilde{f_t}(\widetilde{f_{st}})^{-1} = 1$,
is thus also a homomorphism.
\end{prf}

%% file: part2-3.tex
\subsection{Representability of \'Etale Pseudo-Reductive Bands}
In view of the previous subsection, it remains to use the structure theory of generalized standard pseudo-reductive groups to deduce the statement for an arbitrary pseudo-reductive group over $k$. We will do so by reducing the problem to the primitive case handled in Lemma \ref{lemreprprim} (itself a special case of Corollary \ref{correprred} unless $\mathrm{char}(k)\in\{2,3\}$) and, for this, it will be useful to first discuss in the abstract the conditions for constructing auxiliary bands by lifting them along certain morphisms.

\begin{defn}
Suppose given a homomorphism $\alpha : \Gb\rightarrow\Hb$ of algebraic groups on $k_s$.~For~an \'etale band $(\Hb,\kappa)$ with a continuous lift $f : \Gamma\rightarrow\saut(\Hb/k)$, where $\Gamma\coloneqq\Gal(k_s/k)$, we~will~say that a set-theoretical section $f' : \Gamma\rightarrow\saut(\Gb/k)$ is \textit{compatible} with $f$ (with respect to $\alpha$) if $\alpha\circ f'_s = f_s\circ s_*\alpha$ holds for all $s\in\Gamma$.
\end{defn}

\begin{prop}
In the above situation, the section $f'$ is automatically continuous (in the sense of Definition \ref{defcont}). 
\end{prop}
\begin{prf}
Since $\Gb,\Hb$ are of finite type over $k$, the homomorphism $\alpha$ descends to a homomorphism $\alpha_0 : G_0\rightarrow H_0$ over some finite extension $K/k$. Denote the Galois data associated to $G_0,H_0$ by $\sigma_t',\sigma_t$ for $t\in\Gamma_K\coloneqq \Gal(k_s/K)$, respectively, and note that then $\alpha\circ\sigma'_t = \sigma_t\circ t_*\alpha$.

By continuity of $f$, we have (up to enlarging $K$) that $f_{st} = f_s\circ s_*\sigma_t$ for any $s\in\Gamma$. Then also
\begin{align*}
\alpha\circ f'_{st} = f_{st}\circ (st)_*\alpha = (f_s\circ s_*\sigma_t)\circ s_*t_*\alpha &= f_s\circ s_*(\sigma_t\circ t_*\alpha)\\ &= f_s\circ s_*(\alpha\circ\sigma'_t) = (f_s\circ s_*\alpha)\circ s_*\sigma'_t = \alpha\circ f'_s\circ s_*\sigma'_t
\end{align*}
by definition of compatibility, so $f'_{st} = f'_s\circ s_*\sigma'_t$ holds for all $s\in\Gamma$ and $t\in\Gamma_K$.
\end{prf}

Note that, given such $\alpha,f$ as above and any compatible section $f'$, if $\Gb$ is smooth~(so~its~automorphisms are determined by their effect on $k_s$-points) and the map $\alpha$ induces an isomorphism $\Gb(k_s)\cong\Hb(k_s)$, then $f'$ defines an \'etale band $(\Gb,\kappa')$ ``lifting'' $(\Hb,\kappa)$. Indeed, if $f_sf_tf_{st}^{-1} = \mathrm{int}(g)$ for some $g\in\Hb(k_s)$, then $\mathrm{int}(g)^{-1}f'_sf'_t(f'_{st})^{-1}$ is the identity map on $\Gb(k_s)$ and hence also on $\Gb$. This situation has been seen in the proofs of Lemmas \ref{lemreprtotnonred} and \ref{lemreprprim}.

We are interested in conditions under which $f'$ can be chosen to define an \'etale band $(\Gb,\kappa')$ in a more general case, when $\alpha$ is not known to induce an isomorphism on $k_s$ points. This is axiomatized in the following statement: 

\begin{lem}\label{lemmainlift}
Let $\alpha : \Gb\rightarrow\Hb$ be a homomorphism of algebraic groups on $k_s$ and consider an \'etale band $(\Hb,\kappa)$ defined by a continuous lift $f : \Gamma\rightarrow\saut(\Hb/k)$. Suppose that the inclusion $\alpha(\mathrm Z_\Gb(k_s))\subseteq\mathrm Z_\Hb(k_s)$ holds and that there exists the following factorization
\begin{center}\begin{tikzcd}
    & \Gb(k_s)/\mathrm Z_\Gb(k_s)\arrow[r, hook]\arrow[dl, bend right = 30, dashed]\arrow[d, "\alpha"] & \aut(\Gb/k)\subseteq\saut(\Gb/k)\\
    M\arrow[r, hook] & \Hb(k_s)/\mathrm Z_\Hb(k_s)\arrow[r, hook] & \aut(\Hb/k)\subseteq\saut(\Hb/k)
\end{tikzcd}\end{center}
for some subgroup $M$ of $\Hb(k_s)/\mathrm Z_\Hb(k_s)$ such that the following conditions hold:
\begin{enumerate}[(1)]
    \item There exists some (continuous) section $f' : \Gamma\rightarrow\saut(\Gb/k)$ compatible with $f$
    \item For each $m\in M$, the conjugation action $\mathrm{int}(m)\in\aut(\Hb)$ lifts to a unique automorphism $\widetilde{m}\in\aut(\Gb)$ compatibly with $\alpha$ 
    \item For all $s,t\in\Gamma$, the inner automorphism $g_{s,t}\coloneqq f_sf_tf_{st}^{-1}\in\Hb(k_s)/\mathrm Z_\Hb(k_s)$ lies in $M$
    \item The action of the semiautomorphisms $f_s$ on $\Hb(k_s)/\mathrm Z_\Hb(k_s)\subseteq\aut(\Hb)$ (by conjugation, see Remark \ref{remconjact}) restricts to an action on $M$
\end{enumerate}
Then $f',f$ and $g$ define a band $\overline{L}$ on the site $k_\et$ with $k_s$-points $Q\coloneqq\mathrm{coker}\big(\Gb(k_s)/\mathrm Z_\Gb(k_s)\longrightarrow M\big)$ and a class $[\mathrm{int}(f),g]\in\mathrm H^2(\Gamma, \overline{L})$. Moreover, if this class is neutral, then there is a continuous lift $f_0$ of $\kappa$ and a section $f'_0 : \Gamma\rightarrow\saut(\Gb/k)$ compatible with $f_0$ defining a band $(\Gb,\kappa')$.
\end{lem}

This statement requires some clarification: First, the 4 conditions are stated roughly in order from the most difficult to the easiest to obtain in practice. Condition (2) implies that the dashed map in the diagram is an inclusion (of a normal subgroup) and the group $Q$ is well-defined. We will be applying the lemma only in the case $M = \Hb(k_s)/\mathrm Z_\Hb(k_s)$ (which makes properties (3) and (4) trivial), however it is instructive to state it in this generality to highlight the key properties and to enable the later discussion in Remark \ref{remnilpproof}.

Second, any band $L$ on the site $k_\et$ is locally represented by a sheaf, hence we may indeed associate to it a group $Q$ of ``$k_s$-points'' (defined uniquely up to inner automorphisms). Then $Q$ is a noncommutative module over some open subgroup of $\Gamma$, and is a $\Gamma$-module ``up to inner automorphisms'' with a continuity property; in other words, there is a homomorphism
$$\kappa \;:\; \Gamma\longrightarrow\dfrac{\aut(Q)}{Q/\mathrm Z(Q)}$$
admitting a continuous lift in a sense analogous to Definition \ref{defcont} (this is Springer's definition of a ``kernel'' in \cite{Spr66}, except that our continuity condition is more relaxed). 
We will call $L$ a \textit{Galois band} and define its cohomology set $\mathrm H^2(\Gamma, L)$ as in Remark \ref{remetbandh2def}. This is the only instance in which we use this notion (apart from Proposition \ref{propliftnilp}, which is included for completion); it is convenient here since we do not wish to assume that the cokernel $Q$ appearing in the statement of Lemma \ref{lemmainlift} corresponds to a sheaf on the big \'etale site of $k_s$.

Finally, we will apply the lemma only in situations when $Q$ is a commutative $p$-group (using the statements proven in Subsection \ref{ssectvanthm}), as then $\mathrm H^2(\Gamma, \overline{L}) = \mathrm H^2(\Gamma, Q) = 0$. More generally,~this holds when $Q$ is a nilpotent $p$-group (cf.\ Proposition \ref{propliftnilp}), which we will not need.

\begin{prf}[of Lemma \ref{lemmainlift}]
By conditions (1) and (4), the compatible automorphisms $f'_s$ and $f_s$ define an automorphism $\overline f_s$ of the quotient group $Q = M/\big(\Gb(k_s)/\mathrm Z_\Gb(k_s)\big)$, which exists by (2) in view of the discussion above. Moreover, the section $\overline f : \Gamma\rightarrow\aut(Q)$ is continuous since $f',f$ are. Now (3) says that $g_{s,t}\in M$ for all $s,t\in\Gamma$ and we claim that the composition $\overline g : \Gamma\times\Gamma\rightarrow Q$ is locally constant:

In fact, $g$ itself is locally constant because $f$ is continuous: For fixed $s,t\in\Gamma$, a large enough finite separable extension $K/k$ and any $u\in\Gamma_K$, we have $g_{s,tu} = f_sf_{tu}f_{stu}^{-1} = f_sf_t\sigma_u\sigma_u^{-1}f_{st}^{-1} = g_{s,t}$ and (by Proposition \ref{propcont}(4)) $g_{us,t} = f_{us}f_{t}f_{ust}^{-1} = \sigma_ug_{s,t}\sigma_u^{-1}$. Up to refining $K$, we may assume that the automorphism $g_{s,t}$ is defined over $K$; then it is fixed under conjugation by $\sigma_u$. Hence, $g_{s,t}$ is locally constant in both $s$ and $t$.

Let $\overline L = (Q,\overline\kappa)$ be the Galois band defined by $\overline f$. We conclude that the pair $(\overline f,\overline g)$ (which can also be written as $(\mathrm{int}(f), g)$ in the sense of Remark \ref{remconjact}) is a cocycle in $\mathrm Z^2(\Gamma,\overline L)$: The~cocycle condition $\mathrm{int}(f_s)(g_{t,u})\cdot g_{s,tu} = g_{s,t}\cdot g_{st,u}$ is immediate from the definition of $g_{s,t}$ in terms of $f$. The resulting class $[\overline f,\overline g]\in\mathrm H^2(\Gamma,\overline L)$ is neutral if and only if there exists some locally constant function $\overline{h} : \Gamma\rightarrow Q$ (with a locally constant lift $h : \Gamma\rightarrow M$) such that: 
$$1 = \overline{h}_s \big(\overline{f}_s(\overline{h}_t)\big) \overline{g}_{s,t} (\overline{h}_{st})^{-1} = \overline{h_s(f_sh_tf_s^{-1})(f_sf_tf_{st}^{-1})h_{st}^{-1}} = \overline{(hf)_s(hf)_t(hf)_{st}^{-1}}\quad\textrm{ in }Q$$
If that is the case, observe that each automorphism $\mathrm{int}(h_s)\in\aut(\Hb)$ lifts uniquely to~an~automorphism $\widetilde h_s\in\aut(\Gb)$ by condition (2). We write $f_0 = hf$ and $f'_0 = \widetilde hf'$ (clearly a compatible pair of continuous sections) and we claim that $f'_0 : \Gamma\rightarrow\saut(\Gb/k)$ defines an \'etale band~of~the form $(\Gb,\kappa')$. This is equivalent to saying that the automorphism
$$(g'_0)_{s,t}\coloneqq \big(f'_0\big)_s\big(f'_0\big)_t\big(f'_0\big)^{-1}_{st} = \big(\widetilde hf'\big)_s\big(\widetilde hf'\big)_t\big(\widetilde hf'\big)^{-1}_{st}\in\aut(\Gb)$$
lies in $\Gb(k_s)/\mathrm Z_\Gb(k_s)$. We know that $(g'_0)_{s,t}$ commutes with $(g_0)_{s,t}\coloneqq (hf)_s(hf)_t(hf)_{st}^{-1}\in M$ by the compatibility of sections $f'_0$ and $f_0$. However, by the defining condition of $h$, we get that $(g_0)_{s,t}$ maps to $1\in Q$; thus it is in the image of $\Gb(k_s)/\mathrm Z_\Gb(k_s)\hookrightarrow M$. Again using condition~(2), the lift $(g'_0)_{s,t}\in\aut(\Gb)$ of $(g_0)_{s,t}$ is unique, and thus $(g'_0)_{s,t}\in\Gb(k_s)/\mathrm Z_\Gb(k_s)$.
\end{prf}

While the preceding example allows us to ``pull back'' bands along certain morphisms, the following lemma uses that to deduce representability of the original band from the new one.

\begin{lem}\label{lemliftrepr}
Let $\alpha : \Gb\rightarrow\Hb$ be a map of algebraic groups on $k_s$ such that $\alpha(\mathrm Z_\Gb(k_s))\subseteq\mathrm Z_\Hb(k_s)$. Suppose given two \'etale bands $(\Gb,\kappa')$ and $(\Hb,\kappa)$, which are induced by compatible continuous lifts $f' : \Gamma\rightarrow\saut(\Gb/k)$ and $f : \Gamma\rightarrow\saut(\Hb/k)$, respectively.

Consider the induced map $\overline{\alpha} : \Gb(k_s)/\mathrm Z_\Gb(k_s)\rightarrow\Hb(k_s)/\mathrm Z_\Hb(k_s)$ and suppose furthermore that $f_sf_tf_{st}^{-1} = \widetilde{\alpha}(f'_sf'_t(f'_{st})^{-1})$ for all $s,t\in\Gamma$. If $(\Gb,\kappa')$ is representable, then so is $(\Hb,\kappa)$.
\end{lem}
\begin{prf}
By representability of $(\Gb,\kappa')$, there is a locally constant function $i : \Gamma\rightarrow\Gb(k_s)/\mathrm Z_{\Gb}(k_s)$ such that $if'$ is a homomorphism. Recall that we always identify $f'_s(i_t)$ and $f'_si_t(f'_s)^{-1}$ through the inclusion of inner automorphisms (Remark \ref{remconjact}), which by the compatibility of $f'$ and $f$ with $\alpha$ gives the following equality:
$$f_s(\widetilde{\alpha}\circ i)_tf_s^{-1} = f_s\big((\widetilde{\alpha}\circ i)_t\big) = \widetilde{\alpha}\big(f'_s(i_t)\big) = \widetilde{\alpha}\big(f'_si_t(f'_s)^{-1}\big)$$
We may calculate inside $\saut(\Hb/k)$, using the property $f_sf_tf_{st}^{-1} = \widetilde{\alpha}(f'_sf'_t(f'_{st})^{-1})$, that:
\begin{align*}
\big((\widetilde{\alpha}\circ i)f\big)_s\big((\widetilde{\alpha}\circ i)f\big)_t\big((\widetilde{\alpha}\circ i)f\big)_{st}^{-1} 
&= (\widetilde{\alpha}\circ i)_s\cdot f_s(\widetilde{\alpha}\circ i)_tf_s^{-1}\cdot f_sf_tf_{st}^{-1}\cdot (\widetilde{\alpha}\circ i)_{st}^{-1}\\
&= \widetilde{\alpha}\left(i_s\cdot f'_si_t(f'_s)^{-1}\cdot f'_sf'_t(f'_{st})^{-1}\cdot i_{st}^{-1}\right)\\
&= \widetilde{\alpha}\left((if')_s(if')_t(if')_{st}^{-1}\right) = 1
\end{align*}
This shows that $(\widetilde{\alpha}\circ i)f$ is also a homomorphism.
\end{prf} 

\begin{thm}\label{thmmainpsredrepr}
Suppose given a field $k$ such that $[k:k^p] = p$ and a pseudo-reductive group $\Gb$ over $k_s$. Then every \'etale band on $k$ of the form $(\Gb,\kappa)$ is globally representable. 
\end{thm}
\begin{prf}
By Theorem \ref{thmpsredstrmain}, we immediately reduce to the two cases where $\Gb$ is either totally non-reduced or generalized standard (because a general $\Gb$ decomposes functorially-in-isomorphisms into a product of groups of the two types). As the first case is settled in Lemma \ref{lemreprtotnonred}, we may assume that $\Gb$ is a generalized standard pseudo-reductive group. We proceed in two steps, first resolving the case when $\Gb$ is perfect (i.e.\ $\mathcal D(\Gb) = \Gb$; see Subsection \ref{ssectpsred1}):

\textit{Step 1: $\Gb$ is perfect.} Let $j : \mathrm R(\Ga)\rightarrow\Gb$ be the map coming from any generalized standard presentation of $\Gb$, uniquely functorial with respect to isomorphisms in $\Gb$ (see Theorem \ref{thmgenstandfunct}(b), which also shows that $j$ is surjective). As its formation commutes with pullbacks $s_*$, $s\in\Gamma$ (this again uses the fact that $[k : k^p] = p$ as in Proposition \ref{propsemiautpsred}), an element $f$ of $\saut(\Gb/k)$ lifts uniquely to an element $f'$ of $\saut(\mathrm R(\Ga)/k)$. There is moreover an exact sequence
$$1\longrightarrow\dfrac{\mathrm R(\Ga)(k_s)}{\mathrm Z_{\mathrm R(\Ga)}(k_s)}\longrightarrow\dfrac{\Gb(k_s)}{\mathrm Z_\Gb(k_s)}\longrightarrow\dfrac{\mathrm H^1(k_s,\ker\phi)}{\im\mathrm Z_\Gb(k_s)}\longrightarrow1$$
where $\ker\phi = \ker j$ is central in $\mathrm R(\Ga)$ (see Example \ref{exmppsredpres}). By Proposition \ref{proppsredcentralsurj} and by this centrality, we have $\mathrm Z_{\mathrm R(\Ga)} = j^{-1}(\mathrm Z_\Gb)$, explaining the left-exactness in the above sequence.

Next, we apply Lemma \ref{lemmainlift} to $j : \mathrm R(\Ga)\rightarrow\Gb$ and $M = \Gb(k_s)/\mathrm Z_\Gb(k_s)$. Condition (3) and (4) hold trivially, while (1) and (2) are direct consequences of the previous paragraph.~We~conclude that the obstruction to lifting a band $(\Gb,\kappa)$ to $\mathrm R(\Ga)$ lies in $\mathrm H^2(\Gamma,Q)$, where $Q$ is a quotient of $\mathrm H^1(k_s,\ker\phi)$. By Proposition \ref{proph2h1coh}, this group is a (commutative) $p$-group and so $\mathrm H^2(\Gamma,Q) = 0$. 
Lemma \ref{lemliftrepr} now tells us that it is enough to know the theorem for $\mathrm R(\Ga)$ instead of $\Gb$, but this is exactly the content of Lemma \ref{lemreprprim}.

\textit{Step 2: $\Gb$ is an arbitrary generalized standard pseudo-reductive group.} Semiautomorphisms of $\Gb$ restrict to semiautomorphisms of the derived subgroup $\mathcal D(\Gb)$, which is characteristic in $\Gb$ (and perfect, by Proposition \ref{proppsredmulunip}). As above, we would like to apply Lemma \ref{lemmainlift} to the map $\mathcal D(\Gb)\hookrightarrow\Gb$ and to $M = \Gb(k_s)/\mathrm Z_{\Gb}(k_s)$; the four conditions of the lemma are clearly satisfied. Again similarly to Step 1, we consider the sequence
$$1\longrightarrow \frac{\mathcal D(\Gb)(k_s)}{\mathrm Z_{\mathcal D(\Gb)}(k_s)}\longrightarrow \frac{\Gb(k_s)}{\mathrm Z_{\Gb}(k_s)}\longrightarrow \frac{\Gb(k_s)}{\mathcal D(\Gb)(k_s)\cdot\mathrm Z_{\Gb}(k_s)}\longrightarrow 1$$
which we note is (left-)exact by Proposition \ref{proppsredsubgrpcenter} and we have to show that the commutative quotient on the right is a $p$-group. The required property follows from the exact sequence
$$1\longrightarrow \frac{(\mathcal D(\Gb)\cdot\mathrm Z_{\Gb})(k_s)}{\mathcal D(\Gb)(k_s)\cdot\mathrm Z_{\Gb}(k_s)}\longrightarrow \frac{\Gb(k_s)}{\mathcal D(\Gb)(k_s)\cdot\mathrm Z_{\Gb}(k_s)}\longrightarrow \frac{\Gb(k_s)}{(\mathcal D(\Gb)\cdot\mathrm Z_{\Gb})(k_s)}\longrightarrow 1$$
by Proposition \ref{propquotsumpts} on the left and by Corollary \ref{corpsredmulunip} on the right (which says that the group $\Gb/(\mathcal D(\Gb)\cdot\mathrm Z_{\Gb})$ is unipotent and thus the quotient on the right is killed by some $p^n$). We finish using Lemma \ref{lemliftrepr} and the previous step, as $\mathcal D(\Gb)$ is perfect.
\end{prf}

\begin{rem}\label{remnilpproof}
It is possible to write the above proof without breaking it up into two parts, by applying Lemma \ref{lemmainlift} straight to $\mathrm R(\Ga)\rightarrow\Gb$ and $M = \mathcal D(\Gb)(k_s)/\mathrm Z_{\mathcal D(\Gb)}(k_s)$ (for $\Gb$ possibly not perfect). However, showing that $M$ satisfies all of the necessary assumptions would still require going through the second step of the proof.
Alternatively, one may choose~${M = \Gb(k_s)/\mathrm Z_{\Gb}(k_s)}$ and~obtain for $Q$ the (possibly noncommutative) extension
$$1\longrightarrow \dfrac{\mathrm H^1(k_s,\ker\phi)}{\im\mathrm Z_\Gb(k_s)}\longrightarrow Q\longrightarrow \frac{\Gb(k_s)}{\mathcal D(\Gb)(k_s)\cdot\mathrm Z_{\Gb}(k_s)}\longrightarrow 1$$
which is a $p$-group as shown above. However, it is not clear whether $Q$ is a nilpotent group, which makes it difficult to apply Proposition \ref{propliftnilp}.
\end{rem}

\begin{cor}\label{cormainpsredrepr}
Suppose given a field $k$ such that $[k:k^p] = p$ and a pseudo-reductive group $\Gb$ over $k_s$. Then every separable band on $k$ of~the form $(\Gb,\kappa)$ is globally representable. 
\end{cor}
\begin{prf}
The separable band $(\Gb,\kappa)$ admits an \'etale band lying over it (in the sense of Definition \ref{deflyingover}) by Corollary \ref{corparamlet}. We now use the previous theorem to find a global representative for both bands.
\end{prf}

%% file: part3-1.tex
\subsection{Preliminaries on Abelianization}
Let $k$ be a global or local field of characteristic $p > 0$, so in particular $[k : k^p] = p$. Given an \'etale or separable band $(\Gb,\kappa)$ over $k$ locally represented by a pseudo-reductive group $\Gb$, the main results of the previous section show that the subset $\mathrm N^2(k,\Gb,\kappa)\subseteq\mathrm H^2(k,\Gb,\kappa)$ is nonempty. It is a natural question to ask for a characterization of this subset; in particular, to ask when the given inclusion is an equality. We start with the latter question:

It is a classical fact reviewed in Subsection \ref{ssectdouai} that equality holds when $\Gb$ is a semisimple group. Equivalently, given a semisimple group $G$ over $k$, both of the maps of pointed sets
$$\delta :\mathrm{H}^1(\Gamma,G(k_s)/\mathrm Z_G(k_s))\rightarrow\mathrm{H}^2(\Gamma,\mathrm Z_G(k_s))
\;\;\textrm{ and }\;\;
\widetilde{\delta}:\mathrm{H}^1(k,G/\mathrm Z_G)\rightarrow\mathrm{H}^2(k,\mathrm Z_G)$$
are surjective (Theorem \ref{thmsurjh1h2ss}). We now show that the same is true when $G$ is merely perfect and pseudo-reductive. Such groups are called \textit{pseudo-semisimple} in \cite{CGP15} and \cite{CP15}, as they are necessarily semisimple whenever they are reductive.

\begin{thm}\label{thmsurjh1h2perfpsred}
Let $k$ be a global or local field of positive characteristic and suppose given a perfect pseudo-reductive group $G$ over $k$. Then the above two maps $\delta$ and $\widetilde{\delta}$ are surjective.
\end{thm}
\begin{prf}
By Proposition \ref{propdiaglet}, both of the maps $\delta$ and $\widetilde{\delta}$ appear in the following commutative diagram with exact rows. Moreover, the two horizontal maps on the right are both surjections by Corollary \ref{correprlet} (since each \'etale band represented by a pseudo-reductive group is representable by Theorem \ref{thmmainpsredrepr}), which will become important in the last of the 3 steps of this proof.
\begin{center}
\begin{tikzcd}
\mathrm{H}^0(\Gamma, \mathrm{H}^1(k_s, \mathrm Z_G))\arrow[d, equal]\arrow[r] &
\mathrm{H}^1\!\!\left(\!\Gamma, \dfrac{G(k_s)}{\mathrm Z_G(k_s)}\right)\arrow[d, "\delta"]\arrow[r] &
\mathrm{H}^1\!\!\left(\! k, \dfrac{G}{\mathrm Z_G}\right)\arrow[d, "\widetilde{\delta}"]\arrow[r, two heads] &
\mathrm{H}^1(\Gamma, \mathrm{H}^1(k_s, \mathrm Z_G))\arrow[d, equal]\\
\mathrm{H}^0(\Gamma, \mathrm{H}^1(k_s, \mathrm Z_G))\arrow[r] &
\mathrm{H}^2(\Gamma, \mathrm Z_G(k_s))\arrow[r] &
\mathrm{H}^2(k, \mathrm Z_G)\arrow[r, two heads] &
\mathrm{H}^1(\Gamma, \mathrm{H}^1(k_s, \mathrm Z_G))
\end{tikzcd}
\end{center}
\vspace{-5pt}

\textit{Step 1: The surjectivity of $\delta$ follows from the surjectivity of $\widetilde{\delta}$ in general.} To see this, take an arbitrary $x\in\mathrm{H}^2(\Gamma, \mathrm Z_G(k_s))$ and pick $z\in\mathrm{H}^1(k, G/\mathrm Z_G)$ such that ${\widetilde{\delta}(z) = \im(x)\in\mathrm{H}^2(k, \mathrm Z_G)}$.~The image of $z$ in $\mathrm{H}^1(\Gamma, \mathrm{H}^1(k_s, \mathrm Z_G))$ is $0$, so $z = \im(y)$ for some $y\in\mathrm{H}^1(\Gamma, G(k_s)/\mathrm Z_G(k_s))$. Twisting by a representative $P$ of $y$ (see Proposition \ref{propdiaglet}), we get $\im(\tau_P(x)) = (_P\widetilde{\delta})(\im(\tau_P(y))) = 0$, so $\tau_P(x)$ is in the image of $\mathrm{H}^0(\Gamma, \mathrm{H}^1(k_s, \mathrm Z_G))$. There is hence an element $y'\in\mathrm{H}^1(\Gamma, {_P}G(k_s)/\mathrm Z_G(k_s))$ with $(_P\delta)(y') = \tau_P(x)$. We conclude that $\delta(\tau_P^{-1}(y')) = x$, as we wanted.

\textit{Step 2: $\delta$ is surjective when $G$ is primitive.} By a primitive pseudo-reductive group over the field $k$ (and more generally: a nonzero finite reduced $k$-algebra, see Definition \ref{defgenstand}), we mean an absolutely pseudo-simple group $G$ which is either basic exotic or semisimple simply connected. If $G$ is semisimple, this is Theorem \ref{thmsurjh1h2ss}. Otherwise, Proposition \ref{propredbasexpsred} gives a map $G\rightarrow\mathcal G$ to a semisimple simply connected group $\mathcal G$ over $k$, inducing a bijection $G(k_s)\cong\mathcal G(k_s)$ on points. Since $G,\mathcal G$ are smooth, we also have $\mathrm Z_G(k_s) = \mathrm Z_{G(k_s)}\cong\mathrm Z_{\mathcal G(k_s)} = \mathrm Z_{\mathcal G}(k_s)$. From this, we conclude that $\delta$ is surjective when $G$ is basic exotic pseudo-reductive.

\textit{Step 3: $\widetilde{\delta}$ is surjective for all perfect pseudo-reductive $G$.} Using Theorem \ref{thmpsredstrmain}, we have $G = G_1\times G_2$ with $G_1$ (perfect) generalized standard and $G_2$ totally non-reduced. Since $\mathrm Z_{G_2} = 0$, we may assume that $G = G_1$. Then there is a nonzero finite reduced $k$-algebra $k'$ and a primitive $k'$-group $G'$, with a map $j : \mathrm R(G')\coloneqq\mathrm R_{k'/k}(G')\rightarrow G$ which is a surjection (because $G = \mathcal D(G)$). By this surjectivity and Proposition \ref{proppsredcentralsurj}, we have $\mathrm Z_{\mathrm R(G')} = j^{-1}(\mathrm Z_G)$. This implies that $\mathrm R(G')/\mathrm Z_{\mathrm R(G')}\rightarrow G/\mathrm Z_G$ is an isomorphism and that $\widetilde{\delta}$ factors as:
$$\mathrm{H}^1(k,G/\mathrm Z_G)\cong\mathrm{H}^1(k,\mathrm R(G')/\mathrm Z_{\mathrm R(G')})\longrightarrow\mathrm{H}^2(k,\mathrm Z_{\mathrm R(G')})\twoheadrightarrow\mathrm{H}^2(k,\mathrm Z_G)$$
The surjection on the right comes from the fact that $\mathrm{H}^3(k,\ker(\mathrm Z_{\mathrm R(G')}\rightarrow\mathrm Z_G)) = 0$ since we are working over a global or (nonarchimedean) local field, by \cite[Prop.\ 3.1.2]{RosTD}. We conclude that it suffices to prove surjectivity of $\widetilde{\delta}$ when $G = \mathrm R(G')$, which we now assume.

If $k' = \prod k'_i$, then $\mathrm R(G') = \prod \mathrm R_{k'_i/k}(G'_i)$, so we may assume that $k'/k$ is a field extension. Let $k''/k$ is the maximal separable subextension and write $\mathrm R = \mathrm R_{k''/k}\circ\mathrm R_{k'/k''}$. As Weil restriction along separable field extensions is exact on all algebraic groups (Proposition \ref{proppreshapiro}), preserves cohomology (Shapiro's lemma, Lemma \ref{propshapiro}) and commutes with the formation of centers (that is, $\mathrm Z_{\mathrm R(G')} = \mathrm R(\mathrm Z_{G'})$ by Proposition \ref{propweilrestrunip}), we may replace $k$ by $k''$ and reduce to the case when $k'/k$ is purely inseparable. Then $\Gal(k'_s/k') = \Gal(k_s/k) = \Gamma$ and the above diagram gives

\begin{equation}\label{eqdiagtempgg}
\begin{tikzcd}
\mathrm{H}^1\!\!\left(\!\Gamma, \dfrac{G'(k'_s)}{\mathrm Z_{G'}(k'_s)}\right)\arrow[d, "\delta"]\arrow[r] &
\mathrm{H}^1\!\!\left(\! k, \dfrac{G}{\mathrm Z_G}\right)\arrow[d, "\widetilde{\delta}"]\arrow[r, two heads] &
\mathrm{H}^1(\Gamma, \mathrm{H}^1(k_s, \mathrm Z_G))\arrow[d, equal]\\
\mathrm{H}^2(\Gamma, \mathrm Z_{G'}(k'_s))\arrow[r] &
\mathrm{H}^2(k, \mathrm Z_G)\arrow[r, two heads] &
\mathrm{H}^1(\Gamma, \mathrm{H}^1(k_s, \mathrm Z_G))
\end{tikzcd}\tag{$\bigstar_{G',G}$}\end{equation}
where we have again used that $\mathrm Z_{\mathrm R(G')} = \mathrm R(\mathrm Z_{G'})$ (Proposition \ref{propweilrestrunip}) and that $\mathrm R(G')(k_s) = G'(k'_s)$. Note also that, if $\mathcal G$ is any \'etale $k$-form of $G$, then $\mathcal G\simeq\mathrm R_{K/k}(\mathcal G')$ for some finite extension $K/k$ and a primitive $K$-group $\mathcal G'$ by Corollary \ref{corconradetform}. In fact, $K\otimes_k k_s\simeq k'\otimes_k k_s$ implies that the (purely inseparable) field extensions $K/k$ and $k'/k$ are isomorphic, so we may take $K = k'$. We therefore get an analogous diagram $(\bigstar_{\mathcal G',\mathcal G})$ for $\mathcal G',\mathcal G$ as we did for $G',G$. Furthermore, we note that the leftmost vertical map in $(\bigstar_{\mathcal G',\mathcal G})$ is surjective, $\mathrm{H}^1(\Gamma,\mathcal G'(k'_s)/\mathrm Z_{\mathcal G'}(k'_s))\twoheadrightarrow\mathrm{H}^2(\Gamma,\mathrm Z_{\mathcal G'}(k'_s))$, by Step~2 applied to $\mathcal G'$ over $k'$ (which is itself a local or global field).

Now take any $x\in\mathrm{H}^2(k, \mathrm Z_G)$. By surjectivity of the horizontal maps, we may take some point $y\in\mathrm{H}^1(k, G/\mathrm Z_G)$ such that $x$ and $y$ have the same image in $\mathrm{H}^1(\Gamma, \mathrm{H}^1(k_s, \mathrm Z_G))$. Twisting by a representative $P$ of $y$ (again, see Proposition \ref{propdiaglet}), we therefore get $\tau_P(x)$ which maps to $0$ in $\mathrm{H}^1(\Gamma, \mathrm{H}^1(k_s, \mathrm Z_G))$. Now, $_P\delta$ is a surjective map by the above discussion for $(\bigstar_{\mathcal G',\mathcal G})$ and $\mathcal G = {_P}G$. We conclude that there is $y'\in\mathrm{H}^1(k, {_P}G/\mathrm Z_G)$ such that ${(_P\widetilde{\delta})(y') = \tau_P(x)}$,~hence~${\widetilde{\delta}(\tau_P^{-1}(y')) = x}$. 
This proves Step 3 and, together with Step 1, the proposition statement.
\end{prf}

\begin{cor}
Any \'etale or separable band $(\Gb,\kappa)$ over $k$ locally represented by a perfect pseudo-reductive group $\Gb$ satisfies the equality $\mathrm N^2(k,\Gb,\kappa) = \mathrm H^2(k,\Gb,\kappa)$.
\end{cor}
\begin{prf}
This is immediate from Corollary \ref{corneutelemdesc}.
\end{prf}

\begin{rem}
In the proof of the above theorem, we use only the \'etale form of Theorem \ref{thmsurjh1h2ss}. However, to deduce the surjectivity of $\delta$ in \'etale cohomology, we still pass through surjectivity of $\widetilde{\delta}$ in fppf cohomology (unlike $\mathrm R(G')/\mathrm Z_{\mathrm R(G')}\rightarrow G/\mathrm Z_G$, the map $G'(k'_s)/\mathrm Z_{G'}(k'_s)\rightarrow G(k_s)/\mathrm Z_G(k_s)$ is usually not an isomorphism). Here separable bands play a conceptual role, formally gluing~the two statement in \'etale cohomology through twisting of the diagram in Proposition \ref{propdiaglet}.

A direct proof of this theorem (without twisting or reducing to a special case) could also be done by analogy with the proof of Theorem \ref{thmsurjh1h2ss}, using the generalization of anisotropicity to smooth connected commutative algebraic groups to be developed in the next subsection. 
\end{rem}

Now let $G$ be any pseudo-reductive group over a global or local field $k$. The abelianization results of \cite{Brv93} in the case when $G$ is reductive (which is the essential case when $\mathrm{char}(k) = 0$) were generalized by Gonz\'alez-Avil\'es in \cite{GA12} to fppf cohomology of reductive groups over any base scheme $S$ which is ``of Douai type'': this condition means that $\mathrm H^1(S, G/\mathrm Z_G)\rightarrow\mathrm H^2(S, \mathrm Z_G)$ is surjective for any semisimple group $G$. 

The main object appearing in this theory is the ``quasi-abelian crossed module'' $[G^{sc}\rightarrow G]$ (see \cite[\textsection 3]{GA12}), given by the simply connected cover $G^{sc}\twoheadrightarrow G^{ss}\subseteq G$, for reductive $G$. Many results are shown for this crossed module, in particular (\cite[Thm.\ 5.5(ii)]{GA12}) that the map
\begin{equation}\label{eqabmapga}
    \mathrm{ab}^2 :\mathrm H^2(S, G)\longrightarrow\mathbf H^2_{ab}(S, [G^{sc}\rightarrow G])\coloneqq\mathbf H^2(S, [\mathrm Z_{G^{sc}}\rightarrow\mathrm Z_G])
\end{equation}
has kernel equal to $\mathrm N^2(S, G)$. The property of being quasi-abelian crucially includes the fact that $G = \im(G^{sc})\cdot\mathrm Z_G$. The obvious analogue of this crossed module for a (generalized)~standard pseudo-reductive group is the complex $[\mathrm R(G')\rightarrow G]$, as in Theorem \ref{thmgenstandfunct}. However, the equality $G = \mathcal D(G)\cdot\mathrm Z_G$ does not hold in general for pseudo-reductive groups $G$ (unless they are perfect, or generated by tori; see Subsection \ref{ssectpsred1}), so we cannot expect to apply the quasi-abelian~theory of Gonz\'alez-Avil\'es directly here. 
Instead, let us discuss in the abstract the expected properties of an ``abelianization map'' $\mathrm{ab}^2$ to take as a starting point:

\begin{exmp}\label{exmppsredabelprelim}
Given a pseudo-reductive group $G$ over a local or global field $k$, suppose that $\ell : \mathrm H^2(k, \mathrm Z_G)\rightarrow A$ is some homomorphism of Abelian groups. If
$$\im\left(\widetilde{\delta}:\mathrm{H}^1(k,G/\mathrm Z_G)\rightarrow\mathrm{H}^2(k,\mathrm Z_G)\right)\subseteq \ker(\ell)$$
then there exists a function $\mathrm{ab}^2 : \mathrm H^2(k,G)\rightarrow A$ defined by fixing any element $n\in\mathrm N^2(k,G)$ and taking $a.n$ to $\ell(a)$, where $a\in\mathrm{H}^2(k,\mathrm Z_G)$. This is a well-defined function by Proposition \ref{propcentacth2} and independent of the choice of $n$ by Corollary \ref{corneutelemdesc}. If moreover $\im(\widetilde{\delta}) = \ker(\ell)$, then~clearly $\ker(\mathrm{ab}^2) = \mathrm N^2(k,G)$, and we call $\mathrm{ab}^2$ an \textit{abelianization map} of the second cohomology of $G$.

Because $k$ is local or global, we have that the group $\mathrm H^i(k, C)$ vanishes for any $i > 2$ and any commutative algebraic group $C$ over $k$. This implies that the homomorphism
$$\mathbf H^2(k, [\mathrm Z_{\mathrm R(G')}\rightarrow\mathrm Z_G])\longrightarrow\mathrm H^2(k, \coker(\mathrm Z_{\mathrm R(G')}\rightarrow\mathrm Z_G)) = \mathrm H^2(k, \mathrm Z_G/\mathrm Z_{\mathcal D(G)})$$
is an isomorphism. Under this identification, our construction recovers the abelianization~map \eqref{eqabmapga} by taking $\ell = \ell_0$ for the canonical map $\ell_0 : \mathrm H^2(k, \mathrm Z_G)\rightarrow \mathrm H^2(k, \mathrm Z_G/\mathrm Z_{\mathcal D(G)})\eqqcolon A_0$ whenever $\im(\widetilde{\delta})\subseteq\ker(\ell_0)$. We will see below that this inclusion indeed holds when $G = \mathcal D(G)\cdot\mathrm Z_G$.

On the other hand, we claim that the converse inclusion $\im(\widetilde{\delta})\supseteq\ker(\ell_0)$ always holds: Indeed, it follows directly from the commutative diagram with exact bottom row (cf.\ Proposition~\ref{proppsredsubgrpcenter})
\begin{center}\begin{tikzcd}
\mathrm{H}^1\!\left(\! k, \dfrac{\mathcal D(G)}{\mathrm{Z}_{\mathcal D(G)}}\right)\arrow[r]\arrow[d, two heads, "\widetilde\delta_{\mathcal D(G)}"] & \mathrm{H}^1\!\left(\! k, \dfrac{G}{\mathrm{Z}_G}\right)\arrow[d, "\widetilde\delta_G"]\\
\mathrm{H}^2(k, \mathrm{Z}_{\mathcal D(G)})\arrow[r] & \mathrm{H}^2(k, \mathrm{Z}_G)\arrow[r, "\ell_0"] & \mathrm{H}^2(k, \mathrm{Z}_G/\mathrm{Z}_{\mathcal D(G)})
\end{tikzcd}\end{center}
in which the left vertical map is a surjection by Theorem \ref{thmsurjh1h2perfpsred}. This ends the example.
\end{exmp}

Returning to our main discussion, we claim that such an abelianization map does exist~for~any pseudo-reductive group $G$ (and in fact for any smooth connected group $G$ over $k$) and the proof of this claim will occupy the remainder of the entire section; see Corollary \ref{corabelmainpsred}, Theorem~\ref{thmmainabelh2}. In view of the preceding example, we must first make a correct choice of the codomain group~$A$. To directly measure the failure of the equality $G = \mathcal D(G)\cdot\mathrm Z_G$ which informs our choice of $A$, consider the following commutative diagram (in which the bottom row is part of a long exact sequence in cohomology of commutative algebraic groups):
\vspace{-10pt}

\begin{center}\begin{equation}\label{eqabelprelim}
\begin{tikzcd}
\mathrm{H}^1\!\left(\! k, \dfrac{G}{\mathrm Z_{G}}\right)\arrow[d]\arrow[r, "\widetilde{\delta}"] &
\mathrm{H}^2(k, \mathrm Z_G)\arrow[d, two heads, "\ell_0"]\arrow[r, dashed, "\ell"] &
A\\
\mathrm{H}^1\!\left(\! k, \dfrac{G}{\mathrm Z_{G}\cdot\mathcal D(G)}\right)\arrow[r] &
\mathrm{H}^2\!\left(\! k, \dfrac{\mathrm Z_G}{\mathrm Z_{\mathcal D(G)}}\right)\arrow[r, two heads, "\eta"] &
\mathrm{H}^2\!\left(\! k, \dfrac{G}{\mathcal D(G)}\right)
\end{tikzcd}
\end{equation}\end{center}
Here, the map $\eta$ is a surjection because $G/(\mathrm Z_{G}\cdot\mathcal D(G))$ is unipotent (Corollary \ref{corpsredmulunip}) and thus its $\mathrm H^2$ group vanishes. The map $\ell_0$ is a surjection since $\mathrm H^3(k, \mathrm Z_{\mathcal D(G)}) = 0$ (over $k$ local or global). Moreover, the composition $\eta\circ\ell_0\circ\widetilde{\delta}$ is trivial by commutativity of this diagram (which itself comes from the functoriality of connecting homomorphisms associated to central subgroups). This proves the chain of inclusions:
$$\ker(\ell_0)\subseteq\mathrm{im}(\widetilde{\delta})\subseteq\ker(\eta\circ\ell_0)$$
When $G = \mathrm Z_{G}\cdot\mathcal D(G)$, then $\eta$ is an isomorphism, so the two inclusions are equalities and both $\ell_0$ and $\eta\circ\ell_0$ induce abelianization maps through the construction in the above example.

In the following subsection (more precisely, Corollary \ref{corabelmainpsred}), we will show that the inclusion on the right is in fact always an equality, $\mathrm{im}(\widetilde{\delta}) = \ker(\eta\circ\ell_0)$. Thus the correct choice~of~the~starting map $\ell$ is:
$$\ell : \mathrm H^2(k, \mathrm Z_G)\longrightarrow \mathrm H^2(k, G/\mathcal D(G))\eqqcolon A$$
On the other hand, the example below shows that the inclusion on the left can in general be a strict inclusion:

\begin{exmp}\label{exmpabelmain}
Consider the situation from Example \ref{exmppsredh2notinj}, where we had $k'\subseteq k^{1/p^n}$ and a standard pseudo-reductive group $G\coloneqq\mathrm{R}(\mathbf{GL}_{p^n}) = \mathrm{R}_{k'/k}(\mathbf{GL}_{p^n})$. After applying Shapiro's lemma to $\mathrm{R}(\mathbf{GL}_{p^n})/\mathrm{R}(\mathbf G_{\mathrm m})\cong \mathrm{R}(\mathbf{GL}_{p^n}/\mathbf G_{\mathrm m})$ (this identification holds only since, in this specific case, $\mathrm Z_{\mathbf{GL}_{p^n}}\simeq\mathbf G_{\mathrm m}$ is smooth), we can write down part of the diagram \eqref{eqabelprelim} explicitly:

\begin{center}\begin{tikzcd}
\mathrm{H}^1\!\left(\! k', \dfrac{\mathbf{GL}_{p^n}}{\mathbf G_{\mathrm m}}\right)\arrow[r, "\widetilde\delta"] &
\mathrm{Br}(k')\arrow[d, two heads, "\ell_0"]\\
& \mathrm{Br}(k)\arrow[r, two heads, "\eta"] &
\mathrm{Br}(k')
\end{tikzcd}\end{center}

\noindent
If $k$ is a local field, we know that $\ker(\eta) =  {_{p^n}}\mathrm{Br}(k)\simeq \mathbf Z/p^n$ and that $\ell_0$ is an isomorphism (\cite[Thm.\ 8.9]{Har20}). Moreover, the inclusion $\mathbf{SL}_{p^n}/\mu_{p^n}\hookrightarrow\mathbf{GL}_{p^n}/\mathbf G_{\mathrm m}$ is an isomorphism, so we get by Theorem \ref{thmsurjh1h2ss} and Hilbert's theorem 90 the precise image of the following composition:
\begin{center}\begin{tikzcd}
\widetilde\delta\; :\;\mathrm{H}^1\!\left(\! k', \dfrac{\mathbf{GL}_{p^n}}{\mathbf G_{\mathrm m}}\right)\cong\mathrm{H}^1\!\left(\! k', \dfrac{\mathbf{SL}_{p^n}}{\mu_{p^n}}\right)\arrow[r, two heads] & \mathrm{H}^2(k, \mu_{p^n})\cong {_{p^n}}\mathrm{Br}(k')\arrow[r, hook] & \mathrm{Br}(k')
\end{tikzcd}\end{center}
This means that $\ell_0(\mathrm{im}(\widetilde\delta)) = \ker(\eta)$, hence we exactly have $\ker(\ell_0)\subsetneq\mathrm{im}(\widetilde\delta) = \ker(\eta\circ \ell_0)$.
\end{exmp}

%% file: part3-2.tex
\subsection{Abelianization for Pseudo-reductive Groups}\label{ssectgenanis}
To prove the properties explained in the previous subsection, we will need to generalize the notion of an ``anisotropic torus''. Recall how that is a torus $T$ for which holds $\Hom(\mathbf G_{\mathrm m}, T) = 0$ (or equivalently $\Hom(T, \mathbf G_{\mathrm m}) = 0$; see Proposition \ref{propequivcondanistor}). The statements we now show hold over any field $k$ and it will be useful to keep in mind the multiplicative-unipotent decomposition $0\rightarrow G^m\rightarrow G\rightarrow G^u\rightarrow 0$ of an arbitrary commutative affine algebraic group $G$ over $k$: 

By \cite[IV, \textsection 3, 1.1 and 1.4]{DG70}, there exists a short exact sequence $0\rightarrow G^m\rightarrow G\rightarrow G^u\rightarrow 0$ of commutative affine algebraic groups, which splits when the field $k$ is perfect. Here, $G^u$ is unipotent and $G^m$ is of multiplicative type, and both groups are smooth (resp.\ connected) when $G$ is. In particular, this sequence splits over the perfect closure $\bigcup_n k^{p^{-n}}$ of an arbitrary field $k$, so also over a finite extension $k'/k$ (since all the groups are algebraic). 

\begin{prop}\label{propanisdef}
Suppose given a commutative affine algebraic group $G$ over $k$. The three following conditions are equivalent:
\begin{enumerate}[\hspace{0.5 cm}(1)]
    \item $\Hom(\mathbf G_{\mathrm m}, G) = 0$
    \item There is no surjection $G\twoheadrightarrow\mathbf G_{\mathrm m}$
    \item The unique maximal torus $T\subseteq G$ is anisotropic
\end{enumerate}
These conditions (1)-(3) are implied by the following condition (4):
\begin{enumerate}[\hspace{0.5 cm}(1)]
\setcounter{enumi}{3}
    \item $\Hom(G, \mathbf G_{\mathrm m}) = 0$
\end{enumerate}
The converse implication also holds when $G$ is smooth and connected.
\end{prop}
\begin{prf}
If $T$ is the maximal torus in $G$, then $T\subseteq G^m\subseteq G$ and every map $\mathbf G_{\mathrm m}\rightarrow G$ lands into $T$. This immediately shows (1)$\Leftrightarrow$(3). Next, we show (3)$\Rightarrow$(2): Any map $f : G\rightarrow\mathbf G_{\mathrm m}$ induces a map $\overline{f} : G^u\rightarrow\mathbf G_{\mathrm m}/f(G^m)$, which is $0$ (since its domain is unipotent, and the codomain is of multiplicative type). Thus $f(G) = f(G^m)$. If $f(T) = 0$, then  $f|_{G^m}$ factors through the finite group $G^m/T$, so $f(G^m)\neq\mathbf G_{\mathrm m}$ and $f$ is not a surjection.

The implication (4)$\Rightarrow$(2) is obvious, and its converse holds when $G$ is smooth and connected because then any nonzero map $G\rightarrow\mathbf G_{\mathrm m}$ is surjective. It remains to prove implication (2)$\Rightarrow$(3), as follows:
Suppose given a nonzero map $f: T\rightarrow\mathbf G_{\mathrm m}$, which is necessarily a surjection. By \cite[Prop.\ 2.1.1]{RosTD}, there exists a finite subgroup $F\subseteq G^m$ such that $T'\coloneqq G^m/F$ is smooth and connected, hence a torus. Since $G^m/T$ is also finite, the induced map $T\twoheadrightarrow T'$ is an isogeny which induces a surjection
$$f' \;:\; G^m\twoheadrightarrow T'\xrightarrow{\;\;f\;\;} \mathbf G_{\mathrm m}/f(T\cap F)\simeq\mathbf G_{\mathrm m}$$
Now, as remarked at the beginning of this proof, there is a (purely inseparable) field extension $k'/k$ over which exists a retraction $G_{k'}\twoheadrightarrow G^m_{k'}$. By Proposition \ref{propweiladjmorph} and Corollary \ref{corweiladjmorph}, this gives a map $G\rightarrow\mathrm R_{k'/k}(G^m_{k'})\eqqcolon\mathrm R(G^m_{k'})$ which fits into the following commutative diagram, in which the unit maps $\iota$ are closed immersions:
\begin{center}\begin{tikzcd}
    & G^m\arrow[ld, hook]\arrow[d, "\iota", hook]\arrow[r, "f'", two heads] & \mathbf G_{\mathrm m}\arrow[d, "\iota", hook]\arrow[dr, "{[k' : k]}", two heads, pos=0.3]\\
    G\arrow[r] & \mathrm R(G^m_{k'})\arrow[r] & \mathrm R(\mathbf G_{\mathrm m,\,k'})\arrow[r, "\mathrm N", pos=0.3] & \mathbf G_{\mathrm m}
\end{tikzcd}\end{center}
Here, the map $\mathrm N = \mathrm N_{k'/k}$ is defined by the norm map $(k'\otimes_k A)^\times\rightarrow A^\times$ of $A$-vector spaces, for $k$-algebras $A$. An element $1\otimes a$ then maps to $[k':k]\cdot a$, as indicated. Finally, the bottom row of this diagram gives a map $G\rightarrow\mathbf G_{\mathrm m}$ which restricts to a surjection $G^m\twoheadrightarrow\mathbf G_{\mathrm m}$ and is thus itself a surjection.
\end{prf}

\begin{defn}
Suppose given a commutative affine algebraic group $G$ over a field $k$. If the conditions (1)-(3) of the above proposition hold, we will say that $G$ is \textit{anisotropic}. Note that this definition includes finite, as well as unipotent, commutative groups $G$.

When $k'/k$ is a purely inseparable field extension, the group $G$ is anisotropic (over $k$) if and only the same holds for $G_{k'}$ over $k'$. This is a consequence of the analogous statement for the maximal torus $T\subseteq G$, which is well-known from duality with lattices.

Suppose given a nonzero finite reduced $k$-algebra $k' = \prod k'_i$ and a commutative $k'$-group $G'$ with affine algebraic factors $\prod G'_i$ over $k'_i$. By the defining property of Weil restrictions of group schemes (Corollary \ref{corweiladjmorph}), the group $\mathrm R_{k'/k}(G') = \prod\mathrm R_{k'_i/k}(G')$ is anisotropic over $k$ if and only if, for each $i$, the group $G'_i$ is anisotropic over $k'_i$.
\end{defn}

\begin{prop}\label{propanisses}
Let $0\rightarrow G'\rightarrow G\rightarrow G''\rightarrow 0$ be a short exact sequence of commutative affine algebraic groups over $k$. Then $G$ is anisotropic if and only if both $G'$ and $G''$ are.
\end{prop}
\begin{prf}
Obvious from the conditions (1) and (2).
\end{prf}

We will use this generalized notion of anisotropicity mainly when the commutative group $G$ is smooth and connected, in which case it includes condition (4) above: $\smash{\widehat G(k)} = \Hom(G, \mathbf G_{\mathrm m}) = 0$. If $k$ is a local field, this condition implies $\mathrm H^2(k, G) = 0$ (by local duality, \cite[Thm.~1.2.2]{RosTD}). It is also crucial in the proof of the proposition to be proven next, generalizing a statement about anisotropicity of tori (Lemma \ref{lemanistor}) very useful for working over global fields.

\begin{lem}
Let $G$ be an algebraic group over a field $k$. A $G$-torsor over $k$ which trivializes over some separable (not necessarily algebraic) extension $K/k$ must also trivialize over a finite separable extension~of~$k$.
\end{lem}
\begin{prf}
See \cite[the beginning of Exmp.\ C.4.3 and Prop.\ 1.1.9(i)]{CGP15}.
\end{prf}

\begin{lem}\label{corunipinjplace}
Let $k$ be a global field and fix some place $v$ of $k$. If $U$ is a commutative unipotent algebraic group over $k$, then the map $\mathrm{H}^1(k, \widehat{U})\rightarrow\mathrm{H}^1(k_v, \widehat{U})$ is injective.
\end{lem}
\begin{prf}
Because $\mathbf G_{\mathrm m}$ is of multiplicative type, $\widehat{U}(K) = 0$ for all field extensions $K/k$. We know, by \cite[Lem.\ 2.1.1 and Lem.\ 2.1.5]{RosTD}, that there is a short exact sequence $0\rightarrow F\rightarrow U\rightarrow V\rightarrow 0$ where $F$ is finite and $V$ is split unipotent. The dual sequence $0\rightarrow\widehat V\rightarrow\widehat U\rightarrow\widehat F\rightarrow 0$ is exact~by \cite[Prop.\ 2.3.1]{RosTD}. The sheaf $\widehat{F}$ is representable by an infinitesimal group, so
$$\ker\left(\mathrm{H}^1(k, \widehat{F})\rightarrow\mathrm{H}^1(k_v, \widehat{F})\right) \xhookrightarrow{\;\;\;\;\;\;}\mathrm{H}^1_\et(k, \widehat{F}) = \mathrm{H}^1(\Gamma, \widehat{F}(k_s)) = 0$$
using the previous lemma. As $\mathrm{H}^1(k, \widehat{V}) = 0$, the composition ${\mathrm{H}^1(k, \widehat{U})\rightarrow\mathrm{H}^1(k, \widehat{F})\rightarrow\mathrm{H}^1(k_v, \widehat{F})}$ is an injection factoring through $\mathrm{H}^1(k, \widehat{U})\rightarrow\mathrm{H}^1(k_v, \widehat{U})$.
\end{prf}

\begin{prop}\label{propanisgen}
Let $k$ be a global field and $G$ a smooth connected commutative algebraic group~over $k$. If $v$ is a place of $k$ such that $G_{k_v}$ is anisotropic over $k_v$, then $\mathrm{H}^1(k, \widehat{G})\rightarrow\mathrm{H}^1(k_v, \widehat{G})$ is injective. In particular, if such a place exists, then $\Sh^2(G) = \Sh^1(\widehat{G})^* = 0$.
\end{prop}
\begin{prf}
Let $0\rightarrow M\rightarrow G\rightarrow U\rightarrow 0$ be the multiplicative-unipotent exact sequence of $G$. The group $M$ is smooth connected, hence a torus, and $M_{k_v}$ is anisotropic when $\mathcal{Hom}(\mathbf G_{\mathrm m},G)(k_v) = 0$. Consider the following commutative diagram with exact rows
\begin{center}\begin{tikzcd}
\widehat{M}(k)\arrow[r]\arrow[d] & \mathrm{H}^1(k,\widehat{U})\arrow[r]\arrow[d, hook] & \mathrm{H}^1(k,\widehat{G})\arrow[r]\arrow[d] & \mathrm{H}^1(k,\widehat{M})\arrow[d, hook]\\
\widehat{M}(k_v)\arrow[r] & \mathrm{H}^1(k_v,\widehat{U})\arrow[r] & \mathrm{H}^1(k_v,\widehat{G})\arrow[r] & \mathrm{H}^1(k_v,\widehat{M})
\end{tikzcd}\end{center}
in which the two vertical injections come from Lemma \ref{corunipinjplace} (for $U$) and Lemma \ref{lemanistor} (for $M$). We are now done by the fact that $\widehat{M}(k_v) = 0$.
\end{prf}

The core statement in this section is the theorem below, following a quick lemma. We will apply it in two cases, $A = \mathrm Z_G$ and (in Section 5) $A = 1$. Before reading the proof, we suggest that the reader understand the proof of Theorem \ref{thmsurjh1h2ss}, as the main ideas are very similar.

\begin{lem}\label{lemcartapprox}
Let $G$ be a primitive group over a field $k$ (cf.\ Definition \ref{defgenstand}) with $\mathrm{char}(k) > 0$. If $k$ is local, then $G$ admits an anisotropic Cartan subgroup. If $k$ is global and $S$ is a finite set of places of $k$, then $G$ admits a Cartan subgroup $C$ such that $C_{k_v}$ is anisotropic for each $v\in S$.
\end{lem}
\begin{prf}
A Cartan subgroup $C\subseteq G$ is anisotropic if and only if the unique maximal torus $T$ of $G$ with $T\subseteq C$ is anisotropic. The group $G$ is absolutely pseudo-simple and either basic exotic or semisimple simply connected. If it is semisimple, then this is Lemma \ref{lemlocexanis} (in the local case) with addition of Lemma \ref{lemtorapprox} (in the global case).

If $G$ is instead basic exotic pseudo-reductive, then we may use Proposition \ref{propredbasexpsred}(c) to find a semisimple group $\mathcal G$ over (the local or global field) $k$ and a map $G\rightarrow\mathcal G$ such that: given any maximal torus $\mathcal T\subseteq\mathcal G$, there exists a maximal torus $T\subseteq G$ mapped isogeneously onto $\mathcal T$. The kernel of an isogeny is finite, hence Proposition \ref{propanisses} implies that, for any field extension $K/k$, the torus $T_K$ is anisotropic (over $K$) if and only if $\mathcal T_K$ is. We now apply the previous paragraph to $\mathcal G$ and deduce the result. 
\end{prf}

\begin{thm}\label{thmabelmainpsred}
Let $k$ be a global or local field of positive characteristic. Suppose given some pseudo-reductive group $G$ over $k$ and a central subgroup $A\subseteq\mathrm Z_G$. Then the map
$$\mathrm{H}^1\!\left(\! k, \dfrac{G}{A}\right)\longrightarrow\mathrm{H}^1\!\left(\! k, \dfrac{G}{A\cdot\mathcal D(G)}\right)$$
of pointed sets is a surjection.
\end{thm}

\begin{prf}
First, note that this statement is trivial when $G$ is commutative or perfect. By Theorem \ref{thmpsredstrmain}, we may assume that $G$ is noncommutative generalized standard pseudo-reductive (since cohomology commutes with products and all totally non-reduced pseudo-reductive groups are perfect). Given a fixed generalized standard representation $(k'/k, G', T'_0, C_0)$ of $(G,T_0)$, for any choice of maximal torus $T'$ in $G'$, there exists by Theorem \ref{thmgenstandfunct}(c) a unique maximal torus $T$ in $G$ and a unique generalized standard presentation of $(G,T)$ of the form $(G', k'/k, T', C)$ with $C = \mathrm Z_G(T)$. After making this choice of $T\subseteq G$, we will have an isomorphism
$$\frac{G}{\mathcal D(G)}\cong \frac{C}{C\cap\mathcal D(G)} = \frac{C}{\im(\phi)}$$
for $C' = \mathrm Z_{G'}(T')$ and the map $\phi : \mathrm R(C')\rightarrow C$ in the generalized standard presentation (see the construction in Example \ref{exmppsredpres}; the inclusion $\im(\phi)\subseteq C\cap\mathcal D(G)$ is an equality since $\im(\phi)$ must be a Cartan subgroup of $\mathcal D(G)$). With this, we proceed similarly to the proof of Theorem~\ref{thmsurjh1h2ss}, by first proving the theorem for a local field $k$ and then deducing the global case, in both cases carefully choosing $T'$ in $G'$ with the necessary properties:
\medskip

\textit{The local case:} If $k' = \prod k'_i$ and $G' = \prod G'_i$, we apply Lemma \ref{lemcartapprox} to each $G'_i$ to get a Cartan subgroup
$C' = \prod C'_i$ in $G'$ with all $C'_i$ anisotropic. It corresponds to a maximal torus $T\subseteq G$ with $C = \mathrm Z_G(T)$, 
and we consider the associated generalized standard presentation, as explained above. Note the containment of groups $A\subseteq\mathrm Z_G\subseteq C$ and also that the map
$$\mathrm R_{k'/k}(C')\xrightarrow{\;\;\;\phi\;\;\;}\frac{\im(\phi)}{A\cap\im(\phi)} = \frac{C\cap\mathcal D(G)}{A\cap(C\cap\mathcal D(G))} = \frac{C\cap\mathcal D(G)}{A\cap\mathcal D(G)}$$
is a surjection. The Weil restriction $\mathrm{R}(C')$ is anisotropic, hence so is the (smooth and connected) group $\im(\phi)/(A\cap\im(\phi))$. By local duality (\cite[Thm.\ 1.2.2]{RosTD}), we conclude vanishing on~the right of the following commutative diagram
\begin{center}\begin{tikzcd}
\mathrm{H}^1\!\left(\! k, \dfrac{C}{A}\right)\arrow[r]\arrow[d] & \mathrm{H}^1\!\left(\! k, \dfrac{C}{A\cdot\im(\phi)}\right)\arrow[r]\arrow[d, "\rotatebox{90}{$\sim$}", pos=0.4] & \mathrm{H}^2\!\left(\! k,\dfrac{\im(\phi)}{A\cap\im(\phi)}\right) = 0\\
\mathrm{H}^1\!\left(\! k, \dfrac{G}{A}\right)\arrow[r] & \mathrm{H}^1\!\left(\! k, \dfrac{G}{A\cdot\mathcal D(G)}\right)
\end{tikzcd}\end{center}
with exact top row, which shows that the first map is a surjection. It follows that the bottom map is also surjective, which we wanted to prove.
\medskip

\textit{The global case:} Fix an element $x\in\mathrm{H}^1(k,G/(A\cdot\mathcal D(G)))$. By Lemma \ref{lemcohdirsum}, its local images $x_v\in\mathrm{H}^1(k_v,G/(A\cdot\mathcal D(G)))$ are $0$ for almost all $v$. Take a nonempty finite set $S$ of places~which includes all $v$ such that $x_v\neq 0$. Now, consider $G' = \prod G'_i$ over $k' = \prod k'_i$ and write, for each $i$,
$$k'_i\otimes_k k_v = \prod\nolimits_{w\mid v} k'_{i,w}$$
where the places $w$ depend on $i$. We apply Lemma \ref{lemcartapprox} to each $G'_i$ and the set $S_i$ of all places $w$ of $k'_i$ which lie over places in $S$. This gives a Cartan subgroup $C' = \prod C'_i$ of $G'$ such that, for each $i$ and each place $w\in S_i$ of $k'_i$, the group $(C'_i)_{k'_{i,w}}$ is anisotropic. In particular, for the Weil restriction $\mathrm R(C') = \mathrm R_{k'/k}(C')$ and any place $v\in S$, the group $\mathrm R(C')_{k_v}$ is anisotropic.

As in the local case, consider the associated generalized standard presentation, with $C\subseteq G$ and the group $A\subseteq\mathrm Z_G\subseteq C$. 
The map $\mathrm{R}(C')\twoheadrightarrow\im(\phi)/(A\cap\im(\phi))$ is a surjection, hence our choice of $C$ implies that, for all $v\in S$,
$$\left(\frac{\im(\phi)}{A\cap\im(\phi)}\right)_{\!k_v}\;\textrm{ is anisotropic over $k_v$, and thus }\;\;\mathrm{H}^2\!\left(\! k_v,\dfrac{\im(\phi)}{A\cap\im(\phi)}\right) = 0$$ 
by local duality, as before. Consider the following commutative diagram with exact rows and products taken over all places $v$ of $k$:
\begin{center}\begin{tikzcd}
\mathrm{H}^1\!\left(\! k, \dfrac{C}{A}\right)\arrow[r]\arrow[d] & \mathrm{H}^1\!\left(\! k, \dfrac{C}{A\cdot\im(\phi)}\right)\arrow[r]\arrow[d] & \mathrm{H}^2\!\left(\! k,\dfrac{\im(\phi)}{A\cap\im(\phi)}\right)\arrow[d]\\
\prod_v\mathrm{H}^1\!\left(\! k_v, \dfrac{C}{A}\right)\arrow[r] & \prod_v\mathrm{H}^1\!\left(\! k_v, \dfrac{C}{A\cdot\im(\phi)}\right)\arrow[r] & \prod_v\mathrm{H}^2\!\left(\! k_v,\dfrac{\im(\phi)}{A\cap\im(\phi)}\right)
\end{tikzcd}\end{center}
By the identification $C/(A\cdot\im(\phi))\cong G/(A\cdot\mathcal D(G))$, our fixed element $x$ is in ${\mathrm{H}^1(k,C/(A\cdot\im(\phi)))}$ and we want to show that it lies in the image of $\mathrm{H}^1(k,C/A)$ (therefore also in the image of $\mathrm{H}^1(k,G/A)$). Since we started with $x$ arbitrary, this will suffice to finish the proof.

Equivalently, we need to show that $\im(x)\in\mathrm{H}^2(k, \im(\phi)/(A\cap\im(\phi)))$ is $0$. We already know that $\im(x)_v = \im(x_v) = 0$ for all $v$ (for $v\notin S$ this holds by choice of $S$, and for $v\in S$ by choice of $C$ as shown above). It remains only to observe that $\Sh^2(\im(\phi)/(A\cap\im(\phi))) = 0$, which follows from Proposition \ref{propanisgen} since the base change of this group to $k_v$ is anisotropic for some $v$ in the nonempty set $S$.
\end{prf}

\begin{cor}\label{corabelmainpsred}
Let $k$ be a global or local field with $\mathrm{char}(k) > 0$, and $G$ a pseudo-reductive group over $k$. The sequence
$$\mathrm{H}^1\!\left(\! k, \dfrac{G}{\mathrm Z_{G}}\right)\xrightarrow{\;\;\;\widetilde{\delta}\;\;\;}\mathrm{H}^2(k, \mathrm Z_G)\longrightarrow\mathrm{H}^2\!\left(\! k, \dfrac{G}{\mathcal D(G)}\right)\longrightarrow 0$$
of pointed sets is exact.
\end{cor}
\begin{prf}
Write down the commutative diagram of pointed sets (cf.\ \eqref{eqabelprelim} after Example \ref{exmppsredabelprelim})
\begin{center}\begin{tikzcd}
\mathrm{H}^1\!\left(\! k, \dfrac{G}{\mathrm Z_{G}}\right)\arrow[d]\arrow[r, "\widetilde{\delta}"] &
\mathrm{H}^2(k, \mathrm Z_G)\arrow[d, two heads, "\ell_0"]\arrow[r, two heads] &
\mathrm{H}^2\!\left(\! k, \dfrac{G}{\mathcal D(G)}\right)\arrow[d, equal]\\
\mathrm{H}^1\!\left(\! k, \dfrac{G}{\mathrm Z_{G}\cdot\mathcal D(G)}\right)\arrow[r] &
\mathrm{H}^2\!\left(\! k, \dfrac{\mathrm Z_G}{\mathrm Z_{\mathcal D(G)}}\right)\arrow[r, two heads] &
\mathrm{H}^2\!\left(\! k, \dfrac{G}{\mathcal D(G)}\right)
\end{tikzcd}\end{center}
whose bottom row is an exact sequence of commutative groups. We only need to check that, given $x\in\ker\big(\mathrm{H}^2(k, \mathrm Z_G)\rightarrow\mathrm{H}^2(k, G/\mathcal D(G))\big)$, it lies in $\im(\widetilde\delta)$.

By exactness, the element $\ell_0(x)\in\mathrm{H}^2(k, \mathrm Z_G/\mathrm Z_{\mathcal D(G)})$ lies in the image of $\mathrm{H}^1(k, G/(\mathrm Z_G\cdot\mathcal D(G)))$. Using the preceding theorem (for $A = \mathrm Z_G$), we find an element $y\in\mathrm{H}^1(k, G/\mathrm Z_G)$ which maps to $\ell_0(x)$. Twisting this diagram by $\tau_P$ for $[P]=y$, we get $\tau_P(\ell_0(x)) = 0$ and thus $\tau_P(x)\in\ker({_P}\ell_0)$. Now, the \'etale $k$-form $_PG$ of $G$ is also pseudo-reductive, and thus we may apply the results from Example \ref{exmppsredabelprelim} to $_PG$ to find that $\ker({_P}\ell_0)\subseteq\im({_P}\widetilde\delta)$ (as a consequence of Theorem \ref{thmsurjh1h2perfpsred}). Twisting back, we get that $x\in\im(\widetilde\delta)$.
\end{prf}

In the notation of Example \ref{exmppsredabelprelim}, this corollary shows that $\ell : \mathrm{H}^2(k, \mathrm Z_G)\longrightarrow\mathrm{H}^2(k, G/\mathcal D(G))$ can be used to construct an abelianization map for $\mathrm H^2(k,G)$. 

%% file: part3-3.tex
\subsection{Abelianization for Smooth Connected Separable Bands}
The statements presented above form the main technical case of pseudo-reductive groups. It remains only to sum up these results, including straightforward generalizations to the case of general (smooth connected affine algebraic) groups, in view of some basic reduction statements about unipotent groups:

\begin{lem}\label{lemuniph1surj}
Let $G$ be an algebraic group over a field $k$ and suppose given a normal unipotent subgroup $U\subseteq G$. Then the map $\mathrm H^1(k,G)\rightarrow\mathrm H^1(k,G/U)$ is a surjection.
\end{lem}
\begin{prf}
This is \cite[Lem.\ 2.4]{NNR24}, an almost direct formal consequence of the fact that (every $k$-form of) $U$ is filtered by commutative unipotent groups with trivial $\mathrm H^2$ group.
\end{prf}

\begin{thm}\label{thmmainabelh1}
Let $k$ be a global or local field with $\mathrm{char}(k) > 0$. Suppose given a smooth and connected affine algebraic group $G$ over $k$, with a central subgroup $A\subseteq\mathrm Z_G$. Then the map
$$\mathrm{H}^1\!\left(\! k, \dfrac{G}{A}\right)\longrightarrow\mathrm{H}^1\!\left(\! k, \dfrac{G}{A\cdot\mathcal D(G)}\right)$$
of pointed sets is a surjection.
\end{thm}
\begin{prf}
Let $U\coloneqq\mathscr R_{u,k}(G)$ be the unipotent radical and $p : G\rightarrow G/U\eqqcolon Q$ denote the maximal pseudo-reductive quotient of $G$. Then $p(\mathcal D(G)) = \mathcal D(Q)$ and $p(A)\subseteq p(\mathrm Z_G)\subseteq\mathrm Z_Q$. There is a short exact sequence
$$0\longrightarrow \dfrac{U}{U\cap(A\cdot\mathcal D(G))}\longrightarrow \dfrac{G}{A\cdot\mathcal D(G)}\longrightarrow \dfrac{Q}{p(A)\cdot \mathcal D(Q)}\longrightarrow 0$$
of (commutative) affine algebraic groups and we consider the commutative diagram 
\begin{center}\begin{tikzcd}
\mathrm{H}^1\!\left(\! k, \dfrac{U}{U\cap A}\right)\arrow[d, two heads]\arrow[r] &
\mathrm{H}^1\!\left(\! k, \dfrac{G}{A}\right)\arrow[d]\arrow[r] &
\mathrm{H}^1\!\left(\! k, \dfrac{Q}{p(A)}\right)\arrow[d, two heads]\arrow[r] & 1\\
\mathrm{H}^1\!\left(\! k, \dfrac{U}{U\cap(A\cdot\mathcal D(G))}\right)\arrow[r] &
\mathrm{H}^1\!\left(\! k, \dfrac{G}{A\cdot\mathcal D(G)}\right)\arrow[r] &
\mathrm{H}^1\!\left(\! k, \dfrac{Q}{p(A)\cdot \mathcal D(Q)}\right)\arrow[r] & 1
\end{tikzcd}\end{center}
of pointed sets with exact rows. Here the rightmost vertical map is surjective by Theorem \ref{thmabelmainpsred}. The remaining surjective maps are as in the preceding lemma.

Given an arbitrary element $x\in\mathrm{H}^1(k, G/(A\cdot\mathcal D(G)))$, we first take $y\in\mathrm{H}^1(k, G/A)$ such that the images of $x$ and $y$ in $\mathrm{H}^1(k, Q/(p(A)\cdot\mathcal D(Q)))$ agree. Using the map $G/A\rightarrow G/\mathrm Z_G$, it makes sense to twist $G$ by a representative $P$ of $[P] = y$. The above diagram is mapped by $\tau_P$ to the corresponding diagram associated to ${_P}G\rightarrow{_P}Q$ (without the first column; although the formation of unipotent radicals $U$ and ${_P}U$ commutes with passing to \'etale $k$-forms, there is in general no well-defined action on their torsors by $\tau_P$).

Now, $\tau_P(x)$ goes to $0$ on the right, hence there exists $z\in\mathrm{H}^1(k, {_P}U/({_P}U\cap A))$ mapping to $\tau_P(x)$. If $z'\in\mathrm H^1(k, {_P}G/A)$ is the image of $z$, then $x$ is the image of $\tau_P^{-1}(z')$.
\end{prf}

We now formulate the main result of this section:

\begin{defn}\label{defmainabel}
Let $L = (\Gb,\kappa)$ be a separable band on a local or global field $k$, represented by a smooth (and connected) affine algebraic group $\Gb$ over $k_s$. Any lift $f$ of $\kappa$ defines the same descent datum on the quotient $\Gb/\mathcal D(\Gb)$, which is hence represented by a unique (up to unique isomorphism) commutative affine algebraic group $L_\ab$ on $k$, called the \textit{maximal Abelian quotient} (or \textit{abelianization}) of $L$. When there is no confusion, we will denote it by $G_\ab$ as well.

The assignment $[f,g]\mapsto[\overline g]$ defines the \textit{abelianization map}
$$\mathrm{ab}^2 : \mathrm H^2(k,L)\rightarrow\mathrm H^2(k,L_\ab)$$ for the second cohomology set of $L$. See Subsection \ref{ssectcech} for the definition of Čech cohomology of $L$, and in particular Proposition \ref{proph2actcomm} for the action of $\mathrm H^2(k,\mathrm Z_L)$ on $\mathrm H^2(k,L)$. The naturality of this action along the obvious map $\mathrm H^2(k,\mathrm Z_L)\rightarrow\mathrm H^2(k,L_\ab)$ shows that the abelianization map described here agrees with the abelianization map constructed in Example \ref{exmppsredabelprelim} (in the case when $L$ is globally represented by $G$ and when $A = G/\mathcal D(G)$). 
\end{defn}

\begin{thm}\label{thmmainabelh2}
Let $k$ be a local or global field of positive characteristic and let $L = (\Gb,\kappa)$ be a smooth connected affine separable band on $k$. Then the sequence
$$\mathrm{N}^2(k,L)\xhookrightarrow{\;\;\;\;\;\;}\mathrm{H}^2(k, L)\xrightarrow{\;\;\mathrm{ab}^2\;\;}\mathrm{H}^2(k,L_\ab)$$
of sets (the last of which is pointed) is exact.
\end{thm}
\begin{prf}
We write $\Qb\coloneqq\Gb/\Ub$ for the unipotent radical $\Ub$ of $\Gb$ over $k_s$. We prove the theorem by induction on the length of the descending central series
$$1 = \Ub_0\subsetneq\Ub_1\subsetneq\ldots\subsetneq\Ub_n = \Ub$$
with $\Ub_i = [\Ub_{i+1},\Ub]$ for $0\leq i < n$. If $n = 0$, then $\Gb = \Qb$ is pseudo-reductive. In particular, it admits a global representative $G$ by Corollary \ref{cormainpsredrepr} and the statement to be proven reduces to Corollary \ref{corabelmainpsred} applied to $G$.

Otherwise, let $\Hb\coloneqq\Gb/\Ub_1$ and suppose the theorem holds for $\Hb$. By construction, $\Ub_1\subseteq\mathrm Z_\Ub$ and it is in particular commutative. Since the formation of the unipotent radical $\Ub$ is preserved by (semi)automorphisms of $\Gb$ and the descending central series is characteristic in $\Ub$, we get that any lift $f$ of $\kappa$ has a well-defined restriction to $\Ub_1$. There is hence an induced~band~$\overline L = (\Hb,\overline{\kappa})$ (independent of the choice of $f$, as any two lifts differ by an inner automorphism of $\Gb$, resp. of the quotient $\Hb$) and a non-unique $k$-form $U_f$ of $\Ub_1$. Let $G_\ab$ and $H_\ab$ denote the respective maximal Abelian quotients. There is a short exact sequence
$$0\longrightarrow\im(U_f\rightarrow G_\ab)\longrightarrow G_\ab\longrightarrow H_\ab\longrightarrow 0$$
of algebraic groups over $k$. As $U_f$ is unipotent, we get an isomorphism ${\mathrm H^2(k,G_\ab)\cong\mathrm H^2(k,H_\ab)}$. The following diagram is commutative with exact bottom row:
\begin{center}\begin{tikzcd}[row sep = 15pt]
\mathrm N^2(k,\Gb,\kappa)\arrow[d]\arrow[r, hook] & \mathrm H^2(k,\Gb,\kappa)\arrow[d, "p"]\arrow[r] & \mathrm H^2(k,G_\ab)\arrow[d, "\rotatebox{90}{$\sim$}", pos=0.4]\\
\mathrm N^2(k,\Hb,\overline{\kappa})\arrow[r, hook] & \mathrm H^2(k,\Hb,\overline{\kappa})\arrow[r] & \mathrm H^2(k,H_\ab)
\end{tikzcd}\end{center}
To finish the proof, we need to show that there is an equality:
$$p^{-1}\big(\mathrm N^2(k,\Hb,\overline{\kappa})\big) = \mathrm N^2(k,\Gb,\kappa)$$
Suppose given $[f,g]\in\mathrm H^2(k,\Gb,\kappa)$ with neutral image $[\overline f,\overline g]\in\mathrm H^2(k,\Hb,\overline\kappa)$. Then there is some $\overline h\in\Hb(k'\otimes_k k')$ (here we implicitly work with a $k'$-form of $\Hb$ for $k'/k$ finite) such that 
$$1\cdot\mathrm{pr}^*_{13}\overline h\cdot\overline g^{-1} = \mathrm{pr}^*_{12}\overline h\cdot (\mathrm{pr}^*_{12}\overline f)^{-1}(\mathrm{pr}^*_{23}\overline h)$$
holds in $\Hb(k'\otimes_k k'\otimes_k k')$. By applying Lemma \ref{lemrosentensfield}(a) to $U_f$, we may assume (up to enlarging the extension $k'/k$) that $\overline h$ is the image of some $h\in\Gb(k'\otimes_k k')$ and then:
$$g'\coloneqq\mathrm{pr}^*_{12}h\cdot (\mathrm{pr}^*_{12}f)^{-1}(\mathrm{pr}^*_{23}h)\cdot g\cdot\mathrm{pr}^*_{13}h^{-1}\;\in U_f(k'\otimes_k k'\otimes_k k')$$
Then $(f',g')$ is a cocycle for $f' = f\circ\mathrm{int}(h)^{-1}$. It defines an element of $\mathrm H^2(k,U_{f'})$, where $U_{f'}$ is a ($k'/k$)-form of $U_f$ (corresponding to taking the descent datum on $\Ub_1$ given by $f'$ instead~of~$f$). As $U_{f'}$ is unipotent, the class $[f',g']=[f,g]\in\mathrm H^2(k,\Gb,\kappa)$ is neutral.
\end{prf}

\begin{rem}\label{remmainabel}
This proof shows that, in the statement of the above theorem, $\mathrm H^2(k,L_\ab)$ can also be replaced by the isomorphic commutative group $\mathrm H^2(k,L^{psred}_\ab)$, coming from the abelianized pseudo-reductive quotient of the band $L$. This is consistent with the original formulation by Borovoi (\cite[5.4]{Brv93}) who considers the reductive quotient in characteristic $0$.
\end{rem}

%% file: part4-0.tex
\subsection{The Springer Band}\label{ssectspringer}

The topic of algebraic Springer bands on $k_\Et$ (with respect to the different definitions of bands and of the $\mathrm H^2$ set) has already been covered in \cite[\textsection2.3]{DLA19}. For simplicity, we restrict our discussion to $\mathcal C = k_\fppf$ and the following situation: Suppose given an affine algebraic group $G$ and a scheme $X$ of finite type over $k$ which is a \textit{homogeneous space} of $G$. By this we mean that $G$ acts on $X$ (from the right) and that, for every object $S$ of $\mathcal C$ and every element $x\in X(S)$, the map $r^x : G_S\rightarrow X_S$ defined by $r^x(g) = x.g$ is surjective.

We define a gerbe $\mathscr X$ on $\mathcal C$ as follows: Let the fiber $\mathscr X(S)$ have as its objects pairs $(Y,p)$, where $Y_S$ is a $G_S$-torsor and $p : Y_S\rightarrow X_S$ is a $G_S$-equivariant map (necessarily a surjection, by transitivity of the $G$-action on $X$). Morphisms between these pairs are defined in an obvious way; they are all isomorphisms. There is on $\mathscr X$ a structure of a fibered category over $\mathcal C$ given by pullbacks; $\mathscr X$ is clearly a stack and, as $X(k')\neq\varnothing$ for some finite $k'/k$, indeed a gerbe.

\begin{defn}
This gerbe is considered in \cite[IV, 5.1]{Gir71}. Its automorphism band $L_{X}\coloneqq L(\mathscr X)$ is called the \textit{band of stabilizers}, or the \textit{Springer band}, of the homogeneous space $X$. We will see below that it is locally represented by geometric stabilizers of the action of $G$ on $X$.

The class $\xi_{X}\coloneqq[\mathscr X]\in\mathrm H^2(k, L_{X})$ is the \textit{Springer class}. The question of its neutrality will be fundamental to our applications.
\end{defn}

We now give a construction of this band in terms of representative triples. Let $k'/k$ be a finite field extension such that $X(k')\neq\varnothing$. Any choice of $x\in X(k')$ defines a stabilizer $\Hb\coloneqq\ker(r^x)$ (itself an affine algebraic group over $k'$) and a descent datum on $\Hb$ up to inner automorphisms:

For this, let $\Gb\coloneqq G_{k'}$ and have $\varphi_G : \mathrm{pr}_1^*\Gb\rightarrow\mathrm{pr}_2^*\Gb$ be the descent datum corresponding to $G$; and similarly for $X$. There exists an fppf covering $R\rightarrow\Spec(k'\otimes_k k')$ and a point $g_x\in G(R)$ such that 
$\mathrm{pr}_2^*(x) = \varphi_X(\mathrm{pr}_1^*(x).g_x) = \varphi_X(\mathrm{pr}_1^*(x)).\varphi_G(g_x)$. This defines a map:
$$f\coloneqq\varphi_G\circ\mathrm{int}(g_x^{-1}) \,:\, (\mathrm{pr}_1^*\Hb)_R\xrightarrow{\sim}(\mathrm{pr}_2^*\Hb)_R
\;\;\textrm{ because }\mathrm{pr}_j^*\Hb = \mathrm{Stab}_{\mathrm{pr}_1^*\Gb}(\mathrm{pr}_1^*(x))\textrm{ for }j = 1,2$$
It is straightforward to see that $f$ descends to an element $\varphi_H\in\mathcal{Out}_{\mathrm{pr}_1^*\Hb,\,\mathrm{pr}_2^*\Hb}(k'\otimes_k k')$ (following the definition in \eqref{eqdefshquot} for the covering $R\rightarrow\Spec(k'\otimes_k k')$), which is moreover independent of the chosen $g_x$. The point is that $g_x$ may be replaced with another element $g\cdot g_x$, for $g\in G(R)$, if and only if $g$ is in $(\mathrm{pr}_1^*\Hb)(R)$, which affects $f$ only up to inner automorphisms of $\Hb$. 

\begin{rem}\label{remwecanchoosef}
It is in fact true, by the algebraicity of $\Hb$ and Example \ref{exmpsmoothisnice}, that we may choose $R = \Spec(k'\otimes_k k')$ up to enlarging $k'/k$. This fact will be used below to study Čech~cohomology.
\end{rem}

It is nontrivial to show directly that $(k'/k,\Hb,\varphi_H)$ is a representative triple (in the sense of Definition \ref{defband}). However, it follows automatically once we give a natural identification of $\Hb$ with the automorphism sheaf $\mathcal{Aut}_{(\Gb,r^x)}$ of the pair $(\Gb,r^x)\in\mathscr X(k')$. This will also show that the triple represents exactly the band $L_{X}$.

\begin{prop}\label{propspringertriple}
The triple $(k'/k,\Hb,\varphi_H)$ represents a band canonically isomorphic to $L_X$. In particular, this band is unique up to unique isomorphism, independent of the choice of $k'/k$ or the point $x\in X(k')$ in its construction.
\end{prop}
\begin{prf}
Let $\Gb\textrm{-}\mathcal{Aut}_\Gb$ denote the sheaf of automorphisms of the trivial $\Gb$-torsor $\Gb$. The canonical morphism $\ell : \Gb\longrightarrow\Gb\textrm{-}\mathcal{Aut}_\Gb$ defined by left-multiplication is an isomorphism. The image of $\Hb\subseteq\Gb$ by $\ell$ is identified with the subsheaf $\mathcal{Aut}_{(\Gb,r^x)}$ of $\Gb\textrm{-}\mathcal{Aut}_\Gb$. Indeed, $r^x(hg) = x.hg = r^x(g)$ for any local elements $h,g$ of $\smash{\Hb,\Gb}$. We show that this identification descends to an isomorphism with the band $L_X$, by considering the following commutative diagram:
\begin{center}\begin{tikzcd}[column sep = small]
(\mathrm{pr}_1^*\Hb)_R\arrow[d, "\mathrm{pr}_1^*f"]\arrow[rrr, "\mathrm{pr}_1^*\ell"] & & & (\mathrm{pr}_1^*\mathcal{Aut}_{(\Gb,r^x)})_R\arrow[r, equal] & \mathcal{Aut}_{(\mathrm{pr}_1^*(\Gb,r^x))_R}\arrow[d, dashed]\\
(\mathrm{pr}_2^*\Hb)_R\arrow[rrr, "\mathrm{pr}_2^*\ell"] & & & (\mathrm{pr}_2^*\mathcal{Aut}_{(\Gb,r^x)})_R\arrow[r, equal] & \mathcal{Aut}_{(\mathrm{pr}_2^*(\Gb,r^x))_R}
\end{tikzcd}\end{center}
Here, $f$ and $R$ are as in the construction of $\varphi_H$. We need to show that there is an isomorphism $\varphi : \mathrm{pr}_1^*(\Gb,r^x)\xrightarrow{\sim}\mathrm{pr}_2^*(\Gb,r^x)$ such that the dashed map on the right (uniquely determined by commutativity of the diagram) is of the form $\mathrm{int}(\varphi)$. We claim that $\varphi\coloneqq \varphi_G\circ\ell_{g_x}^{-1} = \ell_{\varphi_G(g_x)}^{-1}\circ\varphi_G$ satisfies the required property.

Indeed, given a local section $h$ of $(\mathrm{pr}_1^*\Hb)_R$ and a local section $g$ of $(\mathrm{pr}_2^*\Gb)_R$, we have:
\begin{align*}
\ell_{f(h)}(g) = f(h).g = \varphi_G(g_x^{-1}hg_x).g &= \varphi_G\left(g_x^{-1}\cdot h\cdot (g_x\cdot\varphi_G^{-1}(g))\right)\\
&= \left((\varphi_G\circ\ell_{g_x}^{-1})\circ\ell_h\circ(\varphi_G\circ\ell_{g_x}^{-1})^{-1}\right)(g) = \big(\mathrm{int}(\varphi)(\ell_h)\big)(g)
\end{align*}
By the definition of a band associated to a gerbe (Definition \ref{defh2}), this shows that $(k'/k,\Hb,\varphi_H)$ represents a band canonically isomorphic to $L_X$.
\end{prf}

In fact, this proof shows more. The left action of $\Hb$ on the trivial right $\Gb$-torsor $\Gb$ preserves the map $r^x : \Gb\rightarrow\Xb$. This makes $\Gb$ into a left $\Hb$-torsor over $\Xb$, and this structure commutes with the gluing data of $\Hb$ in $L_X$. Moreover, any pair $(Y,p)\in\mathscr X(k')$ can be assumed to be of the form $(\Gb,r^x)$ by enlarging $k'/k$. This shows:

\begin{cor}\label{corspringertors}
Suppose that $\xi_X$ is neutral and let $H$ be a global representative of $L_X$ corresponding to this class. Then there is a $G$-torsor $Y$ lying over $X$ such that, furthermore, $Y$ is a left $H$-torsor over $X$.
\end{cor}

Note that here $H$ is not necessarily a subgroup of $G$, and in general only admits a map to $G$ over some extension of $k$. However, $H$ is always a subgroup of a pure inner form ${_Y}G$ of $G$, which will be discussed further below in Remark \ref{remsprinnerform}.


\begin{cor}\label{corsprbandissep}
Suppose both $\Hb$ and $X$ are smooth. Then $L_X$ is \'etale-locally representable, and thus a separable band (in the sense of Definition \ref{defetsepband}).
\end{cor}
\begin{prf}
Because $X$ is smooth, we may assume in the above construction of a representative $(k'/k,\Hb,\varphi_H)$ of $L_X$ that the fixed element $x\in X(k')$ comes from a separable extension $k'/k$. When constructing a lift $f$ of $\varphi_X$, we must choose a covering $R\rightarrow\Spec(k'\otimes_k k')$ for which there exists $g_x\in G(R)$ such that $\mathrm{pr}_1^*(x).g_x = \varphi_X^{-1}(\mathrm{pr}_2^*(x))$. Now, $L_X$ is \'etale-locally representable if, up to enlarging the finite separable extension $k'/k$, we may suppose that $R = \Spec(k'\otimes_k k')$. We claim this can indeed be done: The sequence
$$1\longrightarrow\Hb(k'\otimes_k k')\longrightarrow G(k'\otimes_k k')\xrightarrow{\;\mathrm{pr}_1^*(r^x)\;} X(k'\otimes_k k')
\longrightarrow\mathrm{H}^1(k'\otimes_k k',\,\Hb)$$
is an exact sequence of pointed sets (for the point $\mathrm{pr}_1^*(x)\in X(k'\otimes_k k')$). For a large enough Galois extension $k'/k$, the image of $\varphi_X^{-1}(\mathrm{pr}_2^*(x))$ in the set $\mathrm{H}^1(k'\otimes_k k',\Hb)\cong\prod_{\Gal(k'/k)}\mathrm{H}^1(k',\Hb)$ vanishes by smoothness of $\smash{\Hb}$. Finally, a band representable \'etale-locally by a smooth algebraic group is separable by Proposition \ref{propcharsepband}.
\end{prf}

As mentioned above in Remark \ref{remwecanchoosef}, the band $L_X$ is algebraic and thus nicely representable (in the sense of Definition \ref{defetlocrep}). We may hence use the Čech theory of Subsection \ref{ssectcech} to study the class $\xi_X$: A nice representative triple of $L_X$ corresponds to
$x\in X(k')$ and $g_x\in G(k'\otimes_k k')$ such that $\mathrm{pr}_1^*(x).g_x = \varphi_X^{-1}(\mathrm{pr}_2^*(x))$. For $f = \varphi_G\circ\mathrm{int}(g_x)^{-1}$, define an element:
$$h_x = \mathrm{d}g_x\coloneqq\mathrm{pr}_{12}^*g_x\cdot(\mathrm{pr}_{12}^*f^{-1})(\mathrm{pr}_{23}^*g_x)\cdot\mathrm{pr}_{13}^*g_x^{-1}\in\Gb(k'\otimes_k k'\otimes_k k')$$
It is easy to see that in fact $h_x\in\Hb(k'\otimes_k k'\otimes_k k')$, by checking that $\mathrm{pr}_{13}^*\mathrm{pr}_1^*(x).h_x = \mathrm{pr}_{13}^*\mathrm{pr}_1^*(x)$. Moreover, $(f,h_x)$ is a cocycle in the sense of Definition \ref{defh2cech}: This is because, first,
\begin{align*}
(\mathrm{pr}_{13}^*f)^{-1}\circ(\mathrm{pr}_{23}^*f)\circ(\mathrm{pr}_{12}^*f)
&= \mathrm{pr}_{13}^*(\varphi_G\circ\mathrm{int}(g_x^{-1}))^{-1}\circ\mathrm{pr}_{23}^*(\varphi_G\circ\mathrm{int}(g_x^{-1}))\circ\mathrm{pr}_{12}^*(\varphi_G\circ\mathrm{int}(g_x^{-1}))\\
&= \mathrm{int}(\mathrm{pr}_{13}^*g_x\cdot(\mathrm{pr}_{12}^*f^{-1})(\mathrm{pr}_{23}^*g_x^{-1})\cdot\mathrm{pr}_{12}^*g_x^{-1})\circ\mathrm{id} 
\;=\; \mathrm{int}(h_x)^{-1}  
\end{align*}
since $(\mathrm{pr}_{13}^*\varphi_G)^{-1}\circ(\mathrm{pr}_{23}^*\varphi_G)\circ(\mathrm{pr}_{12}^*\varphi_G) = \mathrm{id}$. Second, $h_x$ satisfies the cocycle property, which is straightforward to check and not surprising since formally $h_x = \mathrm{d}g_x$.

\begin{prop}\label{propcechsprclassdef}
The cocycle $(f,h_x)\in\mathrm{\check Z}^2(k'/k,\Hb,\varphi_H)$ represents the class $\xi_X\in\mathrm{H}^2(k, L_X)$.
\end{prop}
\begin{prf}
Recall the definition of the bijection $\mathrm{\check H}^2(k, L_X)\rightarrow\mathrm{H}^2(k, L_X)$ from Subsection \ref{ssectcech}. Over $k'$ we have the inclusion $\Hb\hookrightarrow\Gb$ and an equivalence of gerbes over $k'$:
$$T \;:\; \mathrm{TORS}(\Hb)\xrightarrow{\;\;P\,\rightsquigarrow\,(P\times^\Hb\Gb,\, r^x)\;\;}\mathscr X_{k'}$$
Indeed, a quasi-inverse is given by taking $(Y,p)$ to the fiber $p^{-1}(x)$ which is an $\Hb$-torsor. Now we proceed as in the proof of Proposition \ref{proph2actcomm} to show that $T$ descends to an equivalence of gerbs over $k$. The natural transformation we require is of the form:
$$\omega_P \;:\; \left(P\times^{\mathrm{pr}_2^*\Hb}\mathrm{pr}_1^*\Hb\times^{\mathrm{pr}_1^*\Hb}\mathrm{pr}_1^*\Gb,\;\;r^{\mathrm{pr}_1^*(x)}\right)\longrightarrow\left(P\times^{\mathrm{pr}_2^*\Hb}\mathrm{pr}_2^*\Gb\times^{\mathrm{pr}_2^*\Gb}\mathrm{pr}_1^*\Gb,\;\;r^{\varphi_x^{-1}(\mathrm{pr}_2^*(x))}\right)$$
for objects $P$ in $\mathrm{TORS}(\mathrm{pr}_2^*\Hb)$. Taking $\omega$ to be defined by right-multiplication by $g_x$, we see~that it then needs to satisfy exactly the following equivalence condition:
$$h_x\cdot\mathrm{pr}_{13}^*g_x\cdot 1 = \mathrm{pr}_{12}^*g_x\cdot(\mathrm{pr}_{12}^*f^{-1})(\mathrm{pr}_{23}^*g_x)$$
However, this holds by definition of $h_x$.
\end{prf}

One may also directly (without comparing it to $\xi_X$) deduce that the class $[(f,h_x)]$ is neutral if and only if there exists a $G$-torsor $Y$ over $k$ as in Corollary \ref{corspringertors}: To see this, suppose that $h_x = \mathrm{pr}_{12}^*h\cdot(\mathrm{pr}_{12}^*f^{-1})(\mathrm{pr}_{23}^*h)\cdot\mathrm{pr}_{13}^*h^{-1}$ for some $h\in\Hb(k'\otimes_k k')$. Then up to replacing $g_x$ by $h^{-1}g_x$, we may assume that $h_x = 1$ and $g_x\in\mathrm{Z}^1(k'/k,G)$ (in the sense of the definition of $h_x$) and we equip $\Gb$ with descent data of the form $\varphi_G\circ\ell_{g_x}^{-1}$. By effectivity of fppf descent for affine schemes, this defines an affine $k$-scheme $Y$ of finite type, and the right action of $G$ on this scheme descends from $k'$ to $k$, as does the $G$-equivariant map $Y\rightarrow X$ (which carries the additional structure of an $\Hb$-torsor). Clearly, the converse also holds; any such $Y$ gives that the class $[(f,h_x)]$ is neutral.

\begin{rem}\label{remsprinnerform}
In the situation just described, a $k$-form $H$ of $\Hb$ is defined by the descent datum $\varphi_H = \varphi_G\circ\mathrm{int}(g_x^{-1})$ with $g_x\in\mathrm{Z}^1(k'/k,G)$. Just as in Corollary \ref{corspringertors}, this makes $Y$ into a left $H$-torsor over $X$.

We observe that the descent datum $\varphi_G\circ\mathrm{int}(g_x^{-1})$ also defines a pure inner $k$-form ${_Y}G$ of $G$ (see Example \ref{exmptwisting}). This form acts on $Y$ from the left, making $Y$ into a left ${_Y}G$-torsor over $k$. The actions of $H$ and ${_Y}G$ on $Y$ agree with the natural inclusion $H\hookrightarrow{_Y}G$, but the trivial action of $H$ on $X$ does not in general extend to an action of ${_Y}G$.

Finally, any torsor $Z$ of $H$ (equivalently, a cocycle $h\in\mathrm{Z}^1(k'/k,H)$ up to enlarging $k'/k$) allows us to replace $H$, $Y$ and ${_Y}G$ by the twists ${_Z}H$, ${_Z}Y$ and ${_Z}({_Y}G) = {_{_ZY}}G$ such that all the above properties are preserved: ${_Z}Y$ is a (right) $G$-torsor over $k$, a left ${_Z}({_Y}G)$-torsor over $k$ and, compatibly, a left ${_Z}H$-torsor over $X$. In this way, we parametrize (by the corresponding classes in $\mathrm H^1(k,H)$) all the homogeneous spaces of $G$ lying over $X$. Note that ${_Z}H$ is a pure inner twist of $H$ and thus represents the same neutral class $\xi_X\in\mathrm N^2(k,L_X)$.
\end{rem}

%% file: part4-1.tex
\subsection{The Brauer-Manin Obstruction}
Let $k$ be a global field and $X$ be a scheme of finite type over $k$. We write $\Br(X)\coloneqq\mathrm{H}^2(X,\mathbf G_{\mathrm m})$ for the cohomological Brauer group of $X$ (some authors consider only the torsion subgroup in this definition, but this does not make a difference for geometrically integral $X$). We now recall the definition of a variant of the Brauer-Manin obstruction on $X$ given by functors of the following form
$$\Be_S(X)\coloneqq\mathrm{ker}\left(\frac{\mathrm{Br}(X)}{\mathrm{im}\,\mathrm{Br}(k)}\longrightarrow\prod_{v\in\Omega\setminus S}\frac{\mathrm{Br}(X_{k_v})}{\mathrm{im}\,\mathrm{Br}(k_v)}\right)$$
where $\Omega$ is the set of all places of $k$, and $S\subseteq\Omega$ a subset. We also write $\Be(X)\coloneqq\Be_\varnothing(X)$.

\begin{defn}\label{defbm}
Suppose that $X(\mathbf A)\neq\varnothing$. Then in particular $X(k_v)\neq\varnothing$ for all $v\in\Omega$; the converse holds when $X$ is geometrically integral (by the proof of \cite[Thm.\ 7.7.2]{Poo17}), for example a homogeneous space of a smooth algebraic group.

The canonical maps $\Br(k_v)\rightarrow\Br(X_{k_v})$ induced by the structure morphism are injections (since a left inverse is given by any point $P_v\in X(k_v)$). There is thus a commutative diagram with exact rows (the exactness on the left of the top row is a consequence of the injectivity of the map $\Br(k)\hookrightarrow\bigoplus_v\Br(k_v)$, as part of the Brauer-Hasse-Noether theorem):
\begin{equation}\label{eqdiagbm}
\begin{tikzcd}[row sep=15pt]
0\arrow[r] & \Br(k)\arrow[r]\arrow[d] & \Br(X)\arrow[r]\arrow[d] & \dfrac{\Br(X)}{\Br(k)}\arrow[d]\arrow[r] & 0\\
0\arrow[r] & \prod_v\Br(k_v)\arrow[r] & \prod_v\Br(X_{k_v})\arrow[r] & \prod_v\dfrac{\Br(X_{k_v})}{\Br(k_v)}\arrow[r] & 0
\end{tikzcd}\end{equation}
The snake lemma defines a morphism $\Be(X)\longrightarrow\big(\prod_v\Br(k_v)\big)/\Br(k)$ which actually lands into $\big(\;\!\!\;\!\!\bigoplus_v\Br(k_v)\big)/\Br(k)$. The Brauer-Hasse-Noether theorem then identifies this quotient with $\mathbf Q/\mathbf Z$ via the invariant maps $\mathrm{inv}_v :\Br(k_v)\xrightarrow{\sim}\mathbf Q/\mathbf Z$ of class field theory; see \cite[Def.\ 2.1]{Don24}.

The resulting functorial homomorphism $BM_X : \Be(X)\rightarrow\mathbf Q/\mathbf Z$ is called the \textit{Brauer-Manin obstruction to the Hasse principle on $X$ given by the functor $\Be$}. To see why, observe that, if $X(k)\neq\varnothing$, then the map $\Br(k)\rightarrow\Br(X)$ admits a left inverse, which forces the connecting homomorphism in the above diagram to be trivial. Therefore $BM_X\neq 0$ implies $X(k) = \varnothing$.
One of the question of the theory of Brauer-Manin obstruction is whether the converse implication holds for the scheme $X$. If so, then we say that ``the Brauer-Manin obstruction given by $\Be(X)$ is the only obstruction to the Hasse principle on $X$''.
\vfill

The above construction is the case $S = \varnothing$ of a more general theory for a (finite) set $S\subseteq\Omega$: 
For any choice of points $(P_v)\in X(k_S) = \prod_{v\in S}X(k_v)$, we replace the middle column of \eqref{eqdiagbm}~by
$$\Br(X)\longrightarrow\prod\nolimits_{v\in S} \Br(k_v)\times\prod\nolimits_{v\notin S} \Br(X_{k_v})$$
defined on the first factor by pulling back via $(P_v)$; leaving the first column intact, we obtain a right column with kernel $\Be_S(X)$. The above procedure then gives a map $\Be_S(X)\rightarrow\mathbf Q/\mathbf Z$. This defines, functorially in $X$, a locally constant (as follows from \cite[Cor.\ 8.2.11]{Poo17}) function
$$BM_{X,\,S} \;:\; X(k_S)\longrightarrow\mathrm{Hom}(\Be_S(X),\mathbf Q/\mathbf Z)\eqqcolon \Be_S(X)^*$$
with respect to the local topology on $X(k_S)$ in the sense of \cite{CesTC}. When $S = \varnothing$, this recovers the above construction, as $k_S$ is the zero ring and the unique element of $BM_{X,\,\varnothing}(X(0))\subseteq\Be(X)^*$ is $BM_X$. The function $BM_{X,\,S}$ is called the \textit{Brauer-Manin obstruction to weak approximation on $X$ with respect to $S$ given by the functor $\Be_S$}. See \cite[Def.\ 2.3]{Don24} for a more general~formulation and for similar obstructions to strong approximation, which we will not need.

Again, we explain this name by the fact that the closure of $\overline{X(k)}$ of $X(k)$ in $X(k_S)$ lies in the kernel $X(k_S)^{\Be_S}\coloneqq BM_{X,\,S}^{-1}(0)$. If $BM_{X,\,S}\neq 0$, then $\overline{X(k)}\neq X(k_S)$. If the converse~implication holds then we say that ``the Brauer-Manin obstruction given by $\Be_S(X)$ is the only obstruction~to weak approximation on $X$ with respect to $S$'', or that ``property $\mathrm{BMWA}_{X,\,S}$ holds''. We~will~also understand property $\mathrm{BMWA}_{X,\,S}$ to trivially hold when $X(\mathbf A) = \varnothing$. 

Note that some authors define $\mathrm{BMWA}_{X,\,S}$ as the (a priori stronger) property $\overline{X(k)} = X(k_S)^{\Be_S}$. By our definition, if $S\subseteq T$, then $\mathrm{BMWA}_{X,\,T}$ implies $\mathrm{BMWA}_{X,\,S}$. This in particular holds when $S = \varnothing$, which is the Hasse principle case (and which is usually handled separately in practice).
\end{defn}

\begin{exmp}
Let $X$ be a homogeneous space of a commutative affine algebraic group $G$ over $k$ and suppose that $X(k_v)\neq\varnothing$ for all $v\in\Omega$. Then $X(\mathbf A)\neq\varnothing$ and $\mathrm{BMWA}_{X,\,S}$ holds~for~all finite $S$ by \cite[\textsection3, \textsection4]{Don24}. In fact, the statements there include the case of infinite $S\neq\Omega$. 
\end{exmp}

We now record some statements for later use. First, let $\Br_1(X)\coloneqq\ker(\Br(X)\rightarrow\Br(X_{k_s}))$, where $k_s$ denotes the separable closure of the field $k$, and similarly for $\Br_1(X_{k_v})$ with $v\in\Omega$. As the term $\Br_1(X)$ naturally appears in standard arguments using the Hochschild-Serre spectral sequence, many authors define $\Be_S(X)$ with $\Br_1$ instead of $\Br$. It is well-known that, whenever $S\neq\Omega$, there is in fact an equality of these two definitions: that is,
\vspace{-10pt}

\begin{equation*}
    \Be_S(X) = \mathrm{ker}\left(\Br_\mathrm{a}(X)\longrightarrow\prod\nolimits_{v\in\Omega\setminus S}\Br_\mathrm{a}(X_{k_v})\right)
\end{equation*}
\vspace{-7pt}

\noindent
where $\Br_\mathrm{a}(X)\coloneqq\Br_1(X)/\Br(k)$, $\Br_\mathrm{a}(X_{k_v})\coloneqq\Br_1(X_{k_v})/\Br(k_v)$ are the \textit{algebraic Brauer~groups}. Indeed, to prove this equality, it suffices only to show that $\Be_S(X)\subseteq\Br_\mathrm{a}(X)$. Given an element $A\in\Br(X)$ representing a class in $\Be_S(X)$, then $A_{k_v}\in\Br(k_v)$ for some $v\notin S$ and thus $A_K = 0$ for a finite separable extension $K/k_v$. Because the field extension $K/k$ is separable, a limiting argument gives a smooth finite-type $k$-algebra $R$ with $A_R = 0$. As $\Spec(R)$ admits a $k_s$-point, we conclude that $A_{k_s} = 0$, as was to be proven.

\begin{prop}\label{proptechbm}
Let $X,Y$ be smooth, geometrically integral schemes of finite type over $k$. 
\begin{enumerate}[a)]
    \item If $Y_{k_s}$ is rational over the separable closure $k_s$ and $Y(k)\neq\varnothing$ holds, then the homomorphism $\Br_\mathrm{a}(X)\oplus\Br_\mathrm{a}(Y)\rightarrow\Br_\mathrm{a}(X\times Y)$ is an isomorphism. The same holds for $\Be_S$ when $S\neq\Omega$.\quad\quad\quad\quad
    \item Given an open immersion $X\rightarrow Y$, the homomorphism $\Be_S(Y)\rightarrow\Be_S(X)$ is an isomorphism for every finite subset $S\subseteq\Omega$.
\end{enumerate}
\end{prop}
\begin{prf}
Statement (a) is \cite[Lem.\ 6.6(ii)]{San81}. It also holds for $\Be_S$ in view of the above discussion (using $S\neq\Omega$) since we may fix $k_s\subseteq (k_v)_s$ and work over all completions $k_v$ for $v\notin S$.

For statement (b), we use the terminology of \cite[\textsection2]{BvH09} (in which all proofs were given in characteristic $0$, but the ones we need remain valid in positive characteristic): We first note the functorial identification $\Br_\mathrm{a}(X)\xrightarrow{\sim}\mathbf{H}^2(k, \mathrm{UPic}(X_{k_s}))$ shown in \cite[Cor.\ 2.20]{BvH09}, using that $\mathrm{H}^3(k, \mathbf G_{\mathrm m}) = 0$ over a global or local function field. The proof of \cite[Cor.\ 2.15]{BvH09} shows that $\Sh^2_S(k, \mathrm{UPic}(Y_{k_s}))\rightarrow\Sh^2_S(k, \mathrm{UPic}(X_{k_s}))$ is an isomorphism for finite $S$, because the vanishing of the group $\Sh^2_\omega(k,\mathcal Z)$ in that proof implies the vanishing of all of its subgroups $\Sh^2_S(k,\mathcal Z)$. Combining these two facts, we get that the map $\Be_S(Y)\rightarrow\Be_S(X)$ is an isomorphism.
\end{prf}

\begin{lem}\label{lemweiltorbm}
Let $k'/k$ be a finite field extension and let $A\coloneqq\mathrm R_{k'/k}(\mathbf G_{\mathrm m}^r)$ be the Weil restriction of a split torus. Suppose given an $A$-torsor $Y$ over a smooth, geometrically integral $k$-scheme $X$ of finite type. 
Then $Y(k_v)\rightarrow X(k_v)$ is surjective for all $v\in\Omega$.

Suppose that $X(\mathbf A)\neq\varnothing$. For any finite $S\subseteq\Omega$, the map $\Be_S(X)\rightarrow\Be_S(Y)$ is an isomorphism and, if the equality $\overline{Y(k)} = Y(k_S)^{\Be_S}$ holds in $Y(k_S)$, then $\overline{X(k)} = X(k_S)^{\Be_S}$ holds in $X(k_S)$.
\end{lem}
\begin{prf}
First, we recall that $\mathrm{H}^1(K, \mathbf G_{\mathrm m}) = 0$ for any field $K$ by Hilbert's theorem 90. If $K/k$ is a separable field extension, then $K\otimes_k k'$ is a finite product of fields, so
$$\mathrm{H}^1(K, A) = \mathrm{H}^1(K\otimes_k k', \mathbf G_{\mathrm m})^r = 0$$
by Shapiro's lemma (Proposition \ref{propshapiro}).
Given $x\in X(k_v)$, the fiber $Y_x$ is an $A$-torsor over $k_v$. As $k_v/k$ is separable, this implies that $Y_x$ is a trivial torsor and so $Y_x(k_v)\neq\varnothing$.

Because $X$ is smooth over $k$, the field extension $k(X)/k$ is separable and thus $Y\times_X\Spec(k(X))$ is a trivial $A$-torsor over $k(X)$. This implies that $Y\rightarrow X$ has a rational section. In other words, $X$ admits an open subscheme $U$ (of finite type over $k$) such that $Y_U\simeq A\times_k U$ as $U$-schemes.\quad\quad

Next, observe that $A$ is an open subscheme of the affine space $\mathrm{R}_{k'/k}(\mathbf A^r)\simeq\mathbf A^{r\cdot[k':k]}$ and~thus $\Be_S(A) = 0$ by Proposition \ref{proptechbm}(b) and by the Hochschild-Serre spectral sequence for $\smash{\mathbf A^{r\cdot[k':k]}}$. 
Moreover, $A$ is therefore $k$-rational. The map $\Be_S(X)\rightarrow\Be_S(Y)$ is now a chain of isomorphisms
$$\Be_S(X)\xrightarrow{\,\;\sim\;\,}\Be_S(U)\xrightarrow{\,\;\sim\;\,}\Be_S(Y_U)\xleftarrow{\,\;\sim\;\,}\Be_S(Y)$$
where the first and last arrow are isomorphisms by Proposition \ref{proptechbm}(b) and the middle one by Proposition \ref{proptechbm}(a), since $\Be_S(A) = 0$.

Suppose now that $X(\mathbf A)\neq\varnothing$; then $Y(\mathbf A)\neq\varnothing$ by surjectivity of the maps $Y(k_v)\rightarrow X(k_v)$ (and geometric integrality of $Y$). The preimage of $X(k_S)^{\Be_S}$ under the continuous surjective map $f : Y(k_S)\rightarrow X(k_S)$ is exactly $Y(k_S)^{\Be_S}$ because $\Be_S(X)\cong\Be_S(Y)$. To prove the~final~statement, fix an arbitrary open set $U\subseteq X(k_S)^{\Be_S}$. If the nonempty open set $f^{-1}(U)\subseteq Y(k_S)^{\Be_S}$ contains a point of $Y(k)$, then that point maps to a point of $X(k)$ in $U$.
\end{prf}


The following important lemma is the main input of all our work with bands:
\begin{lem}\label{lemphslift}
Let $X$ be a homogeneous space of a smooth algebraic group $G$, with $X(k_v)\neq\varnothing$ for all $v\in\Omega$. Suppose that the geometric stabilizer $\Hb$ is smooth and connected. If~$BM_X = 0$, then there exists a homogeneous space $Y$ of $G$ with a $G$-equivariant map $Y\rightarrow X$.
\end{lem}
\begin{prf}
In view of the previous subsection, we want to show that the Springer class $\xi_X$ is neutral whenever $BM_X = 0$. Because $\Hb, G, X$ are all smooth, the smooth connected band $L_X$ is a separable band by Corollary \ref{corsprbandissep}. Now Theorem \ref{thmmainabelh2} tells us that the class $\xi_X$ is neutral if and only if its image $\xi_X^\ab$ in $\mathrm{H}^2(k,H_\ab)$ is $0$. Moreover, for any $v\in\Omega$ and a choice of $x\in X(k_v)$, there exist $G_{k_v}$-equivariant maps $G_{k_v}\rightarrow X_{k_v}$ defined by $g\mapsto x.g$, which implies that $(\xi_X^\ab)_v = 0$ in $\mathrm{H}^2(k_v,H_\ab)$. Therefore, $\xi_X^\ab\in\Sh^2(H_\ab)$.

Next, recall the notion of the Cartier dual $\widehat{H}_\ab\coloneqq\mathcal{Hom}(H_\ab,\mathbf G_{\mathrm m})$. The Poitou-Tate pairing
$$\langle-, -\rangle_{PT} \;:\; \Sh^2(H_\ab)\times\Sh^1(\widehat{H}_\ab)\longrightarrow\mathbf Q/\mathbf Z$$
is defined in \cite[\textsection5.13]{RosTD} and shown to be nondegenerate in \cite[Thm.\ 1.2.10]{RosTD}. In other words, 
$\xi_X^\ab = 0$ if and only if $\langle\xi_X^\ab, A\rangle_{PT} = 0$ for every class $A\in\Sh^1(\widehat{H}_\ab)$. Therefore, it suffices to show that there exists a map $\phi$ which makes the following diagram commute (up to sign):
\begin{center}\begin{tikzcd}[row sep = small]
    \Sh^1(\widehat{H}_\ab)\arrow[dd, dashed, "\phi"]\arrow[drr, "\langle\xi_X^\ab{,}\, -\rangle_{PT}", pos=0.1]\\
    & & \mathbf Q/\mathbf Z\\
    \Be(X)\arrow[urr, "BM_X", swap, pos=0.3]
\end{tikzcd}\end{center}
Such a map indeed exists; this very technical statement is shown in Appendix \ref{sectlqf}. Finally, we may conclude that, if $BM_X = 0$, then $\langle\xi_X^\ab, -\rangle_{PT} = 0$ and $\xi_X^\ab = 0$.
\end{prf}

%% file: part4-2.tex
\subsection{The Main Result and Proof}
Let $k$ be a global function field. The following theorem is the~main result of this paper, whose proof crucially depends on the above lemma.

\begin{thm}\label{thmmainbmobs}
Let $G$ be an affine algebraic group which is an extension of a pseudo-reductive group by a split unipotent group. Suppose that $X$ is a homogeneous space of $G$ whose geometric stabilizer $\Hb$ (defined over a finite separable extension $K'/k$) is smooth and connected. Then:
\begin{enumerate}[a)]
    \item The Brauer-Manin obstruction given by $\Be(X)$ is the only obstruction to the Hasse principle on $X$. 
    \item For any finite set $S\in\Omega$ of places of $k$, the Brauer-Manin obstruction given by $\Be_S(X)$ is the only obstruction to weak approximation on $X$ with respect to $S$. More precisely,~if~$X(\mathbf A)\neq\varnothing$, then the equality $\overline{X(k)} = X(k_S)^{\Be_S}$ holds.
\end{enumerate}
\end{thm}

The remainder of this section will be devoted to the proof of Theorem \ref{thmmainbmobs}. We first present large reduction steps in the following two propositions, then prove statement (a). After that, we will be able to suppose in the proof of (b) that there exists an equivariant map $G\rightarrow X$ and then finish through a series of lemmas on approximation.
Note that, as we have discussed in Definition \ref{defbm}, statement (a) is merely the case $S = \varnothing$ of statement (b). Nevertheless,~it~is~still instructive to separate them because the proof of (a) is the place in which it is essential to use the theory of bands, via Lemma \ref{lemphslift}. In comparison, the proof of (b) for $S\neq\varnothing$ is longer,~but slightly more self-contained since it does not include any more references to bands.

\begin{prop}\label{propmainreduction1}
In the proof of Theorem \ref{thmmainbmobs} we may assume that $G$ is an extension of a commutative affine algebraic group $C$ by a smooth affine algebraic group $D$ such that:
\begin{itemize}
    \item Any \'etale $k$-form $D'$ of $D$ satisfies $\mathrm{H}^1(k,D') = 1$
    \item Let $A$ be a commutative pseudo-reductive group over $k$. For any \'etale $k$-form $G'$ of $G\times A$, there exists a unique subgroup $D'$ of $G'$ such that any isomorphism $G'_{k_s}\simeq (G\times A)_{k_s}$ of algebraic groups over $k_s$ restricts to an isomorphism $D'_{k_s}\simeq D_{k_s}$
\end{itemize}
\end{prop}
\begin{prf}
By assumption, $G$ is an extension of a pseudo-reductive group $Q$ by a split unipotent group $U$. By Corollary \ref{corstrpsredetforms}, there is a surjection of pseudo-reductive groups $\widetilde{Q}\twoheadrightarrow Q$ with smooth connected kernel, where $\widetilde{Q} = D_0\rtimes C$ for a commutative pseudo-reductive group $C$ and a smooth group $D_0$, such that each $k$-form $D'_0$ of $D_0$ satisfies $\mathrm{H}^1(k,D'_0) = 1$. Furthermore, if $A$ is a commutative pseudo-reductive group over $k$, then for any \'etale $k$-form $\widetilde{Q}'$ of $\widetilde{Q}\times A$, then there exists a unique subgroup $D'_0$ of $\widetilde{Q}'$ such that any isomorphism $\widetilde{Q}'_{k_s}\simeq (\widetilde{Q}\times A)_{k_s}$~of~algebraic groups over $k_s$ restricts to an isomorphism $D'_{0,\,k_s}\simeq D_{0,\,k_s}$.

We define $\widetilde{G}\coloneqq G\times_Q\widetilde{Q}$, which is an algebraic group and an extension of $\widetilde{Q}$ by $U$. Similarly, if a commutative pseudo-reductive group $A$ over $k$ is given, then $\widetilde{G}\times A$ is an extension of $\widetilde{Q}\times A$ by $U$. Note that then $U$ is necessarily the unipotent radical $\mathscr R_{u,k}(\widetilde{G}\times A)$ of $\widetilde{G}\times A$, hence its formation commutes with taking \'etale $k$-forms of $\widetilde{G}$.
We let $D\coloneqq\ker\big(\widetilde{G}\twoheadrightarrow\widetilde{Q}\twoheadrightarrow C\big)$
so it is an extension of $D_0$ by $U = \mathscr R_{u,k}(D)$. Clearly $\smash{C = \widetilde{G}/D}$.

It is now clear that an \'etale $k$-form $\widetilde{G}'$ of $\widetilde{G}\times A$ contains a $k$-form $D'$ of $D$ which is an extension of a $k$-form $D'_0$ of $D_0$ by a $k$-form $U'$ of $U$. More generally, every \'etale $k$-form $D'$ of $D$ satisfies the same description, for the same reason that $U = \mathscr R_{u,k}(D)$. Since $U'_{k_s}\simeq U_{k_s}$,~the unipotent group $U'$ is also split (this property descends along separable extensions). We conclude that $\mathrm{H}^1(k, U') = 1$ and $\mathrm{H}^1(k,D'_0) = 1$, so that $\mathrm{H}^1(k,D') = 1$.

Now $X$ is a homogeneous space of the group $\widetilde{G}$ and its geometric stabilizer is an extension of $\Hb$ by the smooth connected group $\ker(\widetilde{Q}\rightarrow Q)$, hence itself smooth and connected. We conclude that we may replace $G$ by $\widetilde{G}$ in the statement of Theorem \ref{thmmainbmobs}, without changing $X$, and that the groups $D$ and $C$ have the desired properties.
\end{prf}

The following step is inspired by a similar reduction in \cite[\textsection4]{Brv96} 
and the desired property is independent of the particular choice of geometric stabilizer $\Hb$ (and separable extension $K'/k$).\quad\quad

\begin{prop}\label{propmainreduction2}
In the proof of Theorem \ref{thmmainbmobs} we may, in addition to the previous proposition, assume that the group $(\Hb_{k_s}\cap D_{k_s})/\mathcal D(\Hb)_{k_s}$ is unipotent.
\end{prop}
\begin{prf}
Suppose given $G,X,\Hb$ as in Theorem \ref{thmmainbmobs}, and assume moreover the existence of $C,D$ as in Proposition \ref{propmainreduction1}. Furthermore, we may assume $X(k_v)\neq\varnothing$ for all $v$ (and thus $X(\mathbf A)\neq\varnothing$, since $X$ is geometrically integral) and $BM_X = 0$, as otherwise there is nothing to prove.

The commutative quotient $H_\ab$ of $L_X$ is a smooth connected affine algebraic group, therefore it contains a torus $T$ such that $H_\ab/T$ is unipotent. This torus splits over some finite separable extension $k'/k$, and there is moreover a retraction $(H_\ab)_{k''}\rightarrow T_{k''}$ over some purely inseparable finite extension $k''/k'$ (cf.\ the proof of Proposition \ref{propanisdef}). The Weil restriction $A\coloneqq \mathrm R_{k''/k}(T_{k''})$ and the retraction define a homomorphism $\mu : H_\ab\rightarrow A$. The composition $T\rightarrow H_\ab\rightarrow A$ is the adjoint unit, which is a closed immersion as seen in Proposition \ref{propweiladjmorph}. This shows that $\ker(\mu)$ is unipotent. Moreover, $A$ is a commutative pseudo-reductive group.

By Lemma \ref{lemphslift}, there is a $G$-torsor $Y$ with a $G$-equivariant map $Y\rightarrow X$. Moreover, there is a global representative $H$ of $L_X$ such that $Y$ is a left $H$-torsor over $X$. Define a contracted product $X'\coloneqq A\times^H Y$, with respect to the composition $H\rightarrow H_\ab\rightarrow A$. This is clearly an $A$-torsor over $X$. Moreover, we claim that there is a natural right action of $A\times G$ on $X'$ defined (in terms of local sections) by $[a_0,y].(a,g) = [a_0+a,y.g]$. If so, then this action makes $X'$ into a homogeneous space of $G\times A$.

To check that this action is well-defined, we only need to check it does not depend on the particular representatives of a pair $[a_0,y]$. This can be done locally. Choose a separable extension $K/k$ over which we may fix a $G$-equivariant isomorphism $Y_K\simeq G_K$. Then there is a subgroup inclusion $\iota : H_K\hookrightarrow G_K$ whose image is the stabilizer in $G_K$ of the image of $1_G$ in $X$, and we have $[a_0,g_0] = [a_0-\mu(h^\ab),hg_0]$ for any local section $h$ of $H_K$. We simply check that
$$[a_0-\mu(h^\ab),hg_0].(a,g) = [a_0+a-\mu(h^\ab),hg_0g] = [a_0+a,g_0g] = [a_0,g_0].(a,g)$$
which is what we needed to show. Furthermore, the stabilizer of the point $[a_0,g_0]\in X'(K)$ is the preimage of $[a_0,g_0]$ with respect to the map
$$A_K\times G_K\xrightarrow{\;\;(a,g)\,\mapsto\,[a_0,g_0].(a,g)\;\;} A_K\times^{H_K} Y_K = X'_K$$
and that is exactly the image of $(-\mu,\iota) : H_K\longrightarrow A_K\times G_K$.

If Theorem \ref{thmmainbmobs}(b) (of which Theorem \ref{thmmainbmobs}(a) is only a special case, as we discussed above) holds for the $A$-torsor $X'$ over $X$, then it holds for $X$ by Lemma \ref{lemweiltorbm}. Therefore, it suffices to prove Theorem \ref{thmmainbmobs} for $\widetilde{G}\coloneqq G\times A$ and its homogeneous space $X'$.
Clearly, $\widetilde{G}$ is an extension of $\widetilde{C}\coloneqq C\times A$ by $D$. The triple $(\widetilde{G},\widetilde{C},D)$ still has the properties stated in Proposition \ref{propmainreduction1} because, for any commutative pseudo-reductive group $A'$, we have that $\widetilde{G}\times A' = G\times(A\times A')$ and that $A\times A'$ is a commutative pseudo-reductive group, which reduces the statement to the original triple $(G,C,D)$.

Finally, consider the stabilizer $\Hb\coloneqq (-\mu,\iota)(H_K)$ of $X'_K$ in $\widetilde{G}_K$. Because $D\subseteq G$, the quotient $(\Hb_{k_s}\cap D_{k_s})/\mathcal D(\Hb)_{k_s}$ is a subgroup of $\ker(\mu)_{k_s}$, but this group is unipotent and so is each one of its subgroups.
\end{prf}

\begin{rem}
Borovoi's analogue of this proposition, which he presented in terms of Galois descent in \cite[\textsection4]{Brv96}, can be interpreted as showing that $A$ bounds the gerbe of all possible lifts $X'\rightarrow X$ (satisfying the properties in the above proof). The associated class in $\mathrm{H}^2(k,A)$ is shown to be $0$ over all $k_v$ (because $X(k_v)\neq\varnothing$), hence it lies in $\Sh^2(A)$ which is known to be trivial. The lift $X'$ thus exists even when $BM_X\neq 0$; however our construction shows that when the property $BM_X = 0$ does hold, it (together with Lemma \ref{lemphslift}) is strong enough that we do not even need to use the fact that $\Sh^2(A) = 0$.

In more conceptual terms, Borovoi's class in $\mathrm{H}^2(k,A)$ can be seen to be the image of the Springer class $\xi_X\in\mathrm{H}^2(k,L_X)$ along the ``morphism of bands'' $L_X\rightarrow L(H_\ab)\rightarrow L(A)$. Since the image of $\xi_X$ is $0$ in $\mathrm{H}^2(k,H_\ab)$, it is also $0$ in $\mathrm{H}^2(k,A)$.
\end{rem}

\begin{prf}[of Theorem \ref{thmmainbmobs}(a)]
Given $G,C,D$ as in Proposition \ref{propmainreduction1} without loss of generality, let $X$ be a homogeneous space of $G$ with $X(\mathbf A)\neq\varnothing$ and $BM_X = 0$. By Proposition \ref{propmainreduction2} we~may also assume that the geometric stabilizer $\Hb$ is smooth connected and~that~${(\Hb_{k_s}\cap D_{k_s})/\mathcal D(\Hb)_{k_s}}$ is unipotent. Then Lemma \ref{lemphslift} produces a homogeneous space $Y$ of $G$ with a $G$-equivariant map $Y\twoheadrightarrow X$, which makes it into a left $H$-torsor over $k$ (for a $k$-form $H$ of $\Hb$).

Because $D\subseteq G$ is a normal subgroup, we may quotient $X$ to form a homogeneous space $D{\setminus}X$ of the commutative group $C = D{\setminus}G$ (this is a clear application of the effective fppf descent for affine schemes). There exists a $G$-equivariant map $X\rightarrow D{\setminus}X$, so in particular $(D{\setminus}X)(k_v)\neq\varnothing$ for all $v$. Since $BM_{D{\setminus}X}$ factors through $BM_X = 0$, and the Brauer-Manin obstruction given by $\Be$ is the only obstruction to the Hasse principle for homogeneous spaces of commutative affine algebraic groups (by \cite[Thm.\ 3.2]{Don24}), there exists a point $p\in (D{\setminus}X)(k)$, which we fix.
\smallskip

The fiber $X_p$ is a homogeneous space of $D$. The fiber $Y_p$ is a left $H$-torsor over $X_p$, which also admits a free (but not necessarily transitive) action of $D$. We now recall Remark \ref{remsprinnerform}, by which the action of $H$ on $Y$ extends to an action of a pure inner $k$-form ${_Y}G$ of $G$ containing $H$. Now, by our assumption on $D$, there is a uniquely determined $k$-form ${_Y}D\subseteq{_Y}G$ of $D$ (which is then necessarily obtained through twisting by an element $g_x$, same as ${_Y}G$). We claim that the left action of ${_Y}D$ on $Y$ restricts to an action on $Y_p$:

This can be checked locally: Passing to a finite extension $k'/k$ and making a $G_{k'}$-equivariant identification $Y_{k'}\simeq G_{k'}$, we get an isomorphism $({_Y}G)_{k'}\simeq G_{k'}$ and, consequently, $({_Y}D)_{k'}\simeq D_{k'}$. Then the left action of ${_Y}D$ on $Y$ can be seen as left-multiplication, and in terms of local sections we have $d.y = y.(y^{-1}dy)$, where the action on the right-hand side is by $D$ from the right (since $D\subseteq G$ is a normal subgroup). In particular, this shows that if $(Y_p)_{k'}$ is stable under the action of $D_{k'}$ from the right, then it is stable under the action of $({_Y}D)_{k'}$ from the left.

Moreover, it is immediate that $E\coloneqq (H\cdot {_Y}D)\subseteq {_Y}G$ acts both freely and transitively on $Y_p$, making it into an $E$-torsor over $k$.
Again by Remark \ref{remsprinnerform}, we may twist by a cocycle~in~$\mathrm{H}^1(k, H)$ (equivalently, by some $H$-torsor $Z$) to get an action of ${_Z}E = {_Z}X\cdot {_Z}({_Y}D)$ on ${_Z}(Y_p)$, with there still being a map ${_Z}(Y_p)\rightarrow X_p$ defined over $k$. To finish the proof, we only need to show that $Z$ can be chosen such that ${_Z}(Y_p)$ admits a $k$-point.
\smallskip

Fix any $H$-torsor $Z$. Because ${_Z}({_Y}D)$ is a normal subgroup of ${_Z}E$, we may form a quotient ${_Z}(Y_p)/{_Z}({_Y}D)\cong {_Z}(Y_p/{_Y}D)$, which is a torsor of the group ${_Z}(H/(H\cap{_Y}D))$ over $k$. Suppose that ${_Z}(Y_p/{_Y}D)$ has a $k$-point $q$. Then the fiber $({_Z}(Y_p))_q$ is a ${_Z}({_Y}D)$-torsor, but $\mathrm{H}^1(k, {_Z}({_Y}D)) = 1$ by our assumption on $D$. Hence, $({_Z}(Y_p))_q$ has a $k$-point, and so do ${_Z}(Y_p)$ and $X$.

Therefore, it suffices to show that $Z$ can be chosen so that ${_Z}(Y_p/{_Y}D)$ is the trivial torsor~of ${_Z}(H/(H\cap{_Y}D))$. By general results on twisting, it is equivalent that the class of $Z$ in $\mathrm{H}^1(k, H)$ is mapped to the class of $Y_p/{_Y}D$ in $\mathrm{H}^1(k, H/(H\cap{_Y}D))$ by the natural map of cohomology sets. This map can be written as a composition of two factors
$$\mathrm{H}^1(k, H)\longrightarrow\mathrm{H}^1\left(\!k, \frac{H}{\mathcal D(H)}\right)\longrightarrow\mathrm{H}^1\left(\!k, \frac{H}{H\cap{_Y}D}\right)$$
which are in fact both surjective. Indeed, the first map is surjective by Theorem \ref{thmmainabelh1} since $H$ is smooth and connected (for $A$ = 1). The second map is surjective by Proposition \ref{propunipfppfvanish} since the group $(H\cap{_Y}D)/\mathcal D(H)$ is unipotent by assumption (as it is isomorphic to $(\Hb_{k_s}\cap D_{k_s})/\mathcal D(\Hb)_{k_s}$ when base-changed to $k_s$).
\end{prf}

This finishes the proof of Theorem \ref{thmmainbmobs}(a). Before proving statement (b), we require several results of independent interest on (weak and strong) approximation at the level of $\mathrm H^1$ sets in the natural topology; see \cite{CesTC} and \cite{CesPT} for the relevant definitions.

\begin{prop}\label{proph1approx}
Let $G$ be a commutative affine algebraic group over a global (function) field $k$. If $S$ is any (possibly infinite) subset of the set of places of $k$, then the natural~homomorphism
$$\mathrm H^1(k,G)\longrightarrow\ker\left(\mathrm H^1(\mathbf A_S,G)\rightarrow\Sh^1_S(\widehat G)^*\right)$$
of topological groups has dense image. In particular, if $G$ is smooth and connected, then the map is a surjection and $\mathrm H^1(k,G)\rightarrow\mathrm H^1(\mathbf A_S,G)\rightarrow\Sh^1_S(\widehat G)^*$ is an exact sequence of Abelian~groups.
\end{prop}
\begin{prf}
The last statement follows since $\mathrm H^1(\mathbf A_S,G)$ is discrete when $G$ is smooth and connected (by \cite[Prop.\ 3.7]{CesPT}). For the rest, we construct a commutative diagram as follows: 
\begin{center}\begin{tikzcd}[column sep=large]
    & \mathrm H^1(\mathbf A^S, G)\arrow[r, "f^S"]\arrow[d, "i^S", hook] & \ker\left(\mathrm{H}^1(k,\widehat{G})^*\twoheadrightarrow \Sh^1_S(\widehat{G})^*\right)\arrow[d, hook]\arrow[r, equal]\! &[-35pt] \!\ker(q)\\
    \mathrm H^1(k, G)\arrow[r]\arrow[d, equal] & \mathrm H^1(\mathbf A, G)\arrow[r, "f"]\arrow[d, "p_S", two heads] & \ker\left(\mathrm{H}^1(k,\widehat{G})^*\twoheadrightarrow \Sh^1(\widehat{G})^*\right)\arrow[d, "q", two heads]\arrow[rr] &[-35pt] &[-42pt] 0\\
    \mathrm H^1(k, G)\arrow[r] & \mathrm H^1(\mathbf A_S, G)\arrow[r, "f_S"] & \ker\left(\Sh_S^1(\widehat{G})^*\twoheadrightarrow \Sh^1(\widehat{G})^*\right)
\end{tikzcd}\end{center}
Here, the middle row is part of the exact Poitou-Tate sequence from \cite[Thm.\ 1.2.9]{RosTD}, in which the map $f$ is defined by summing local pairings. It descends to a map $f_S$ on quotients, since the local duality pairings for $v\notin S$ with elements of $\Sh_S^2(\widehat{G})$ are all trivial. Finally, the left column is (split) exact and the right column is obviously exact (after taking algebraic duals of inclusions of the discrete torsion Abelian groups $\Sh^1(\widehat{G})\subseteq\Sh^1_S(\widehat{G})\subseteq\mathrm{H}^1(k,\widehat{G})$, which admit canonical topologies making them coincide with the profinite Pontryagin duals of said groups), inducing the map $f^S$ on top and closing the diagram.

Consider $x_S\in\ker(f_S)$ with open neighborhood $U\subseteq\ker(f_S)$, and fix a preimage $x\in\mathrm H^1(\mathbf A, G)$ of $x_S$. We need to show that $U$ has nonempty preimage in $\mathrm H¹(k,G)$. By definition of subspace topology, there is an open set $V\subseteq\mathrm H^1(\mathbf A_S, G)$ such that $U = V\cap\ker(f_S)$. Because $f$ is open (by \cite[Props.\ 5.14.5, 5.14.6]{RosTD}), the set $f(p_S^{-1}(V))$ is open in $\im(f)$. Since  ${q(f(x)) = f_S(x_S) = 0}$, the intersection $f(p_S^{-1}(V))\cap\ker(q)$ is an open nonempty subset of $\ker(q)$. 

We claim that $f^S$ has dense image. To see why this ends the proof, observe that then there exists $y^S\in\mathrm H^1(\mathbf A^S, G)$ with $f^S(y^S)\in f(p_S^{-1}(V))\cap\ker(q)$, from which we conclude that:
$$i^S(y^S)\in f^{-1}(f(p_S^{-1}(V))) = p_S^{-1}(V)+\im\left(\mathrm H^1(k, G)\rightarrow\mathrm H^1(\mathbf A, G)\right)$$
After projecting along $p_S$, we get that $0\in V+\im\big(\mathrm H^1(k, G)\rightarrow\mathrm H^1(\mathbf A_S, G)\big)$. Hence, there exists an element $z_S\in V$ in the image of $\mathrm H^1(k, G)$, but then $z_S\in V\cap\ker(f_S) = U$, as desired.

It remains only to prove that $f^S$ indeed has dense image. For contradiction, suppose that the (compact Hausdorff) quotient $\ker(q)/\overline{\im(f^S)}$ is nontrivial. Then its Pontryagin dual is also nontrivial, so there is a nontrivial \textit{character} (i.e.\ continuous homomorphism) $\ker(q)\rightarrow\mathbf R/\mathbf Z$ restricting to $0$ on $\im(f^S)$. It extends to a character $\mathrm H^1(k,\widehat G)^*\rightarrow\mathbf R/\mathbf Z$ (since $\ker(q)$ is closed; see \cite[Thm.\ 2.1.2]{Rud62}) which is not induced by a character $\Sh^1_S(\widehat{G})^*\rightarrow\mathbf R/\mathbf Z$. 
This implies that the sequence of Pontryagin duals
$$\left(\Sh^1_S(\widehat{G})^*\right)^D\longrightarrow\left(\mathrm H^1(k,\widehat G)^*\right)^D\longrightarrow\left(\mathrm H^1(\mathbf A^S, G)\strut\right)^D$$
is not exact as a sequence of Abelian groups; a contradiction, because, by the perfect pairing~of $\mathrm H^1(\mathbf A^S, G)$ and $\mathrm H^1(\mathbf A^S,\widehat G)$ (see proof of \cite[Prop.\ 5.3.5]{RosTD}), this sequence is simply the exact sequence $\Sh^1_S(\widehat{G})\hookrightarrow\mathrm H^1(k,\widehat G)\rightarrow\mathrm H^1(\mathbf A^S,\widehat G)$ from the definition of $\Sh^1_S(\widehat{G})$.
\end{prf}

We now axiomatize the following step, to be used repeatedly in what follows:

\begin{lem}\label{lemapproxabstr}
Let $0\rightarrow H\rightarrow G\rightarrow Q\rightarrow 0$ be a short exact sequence of algebraic groups over a global (function) field $k$, such that $\mathrm H^1(k,G)\rightarrow\mathrm H^1(k,Q)$ is a surjection. Suppose given some (possibly infinite) subset $S$ of places of $k$. If, for each $k$-form $_YH$ of $H$ induced by a pure inner twist $_YG$ of $G$, the map $\mathrm H^1(k,{_Y}H)\rightarrow\mathrm H^1(\mathbf A_S,{_Y}H)$ has dense image, then the natural map
$$\mathrm H^1(k,G)\longrightarrow\mathrm H^1(k,Q)\times_{\mathrm H^1(\mathbf A_S,\,Q)}\mathrm H^1(\mathbf A_S,G)$$
has dense image (in the fiber product topology). If moreover $\mathrm H^1(k,Q)\rightarrow\mathrm H^1(\mathbf A_S,Q)$ has dense image, then so does $\mathrm H^1(k,G)\rightarrow\mathrm H^1(\mathbf A_S,G)$.
\end{lem}
\begin{prf}
We will keep track of this data in the following commutative diagram with exact rows:
\begin{center}\begin{tikzcd}[column sep=large]
    \mathrm H^1(k, H)\arrow[r]\arrow[d] & \mathrm H^1(k, G)\arrow[r]\arrow[d] & \mathrm H^1(k, Q)\arrow[r]\arrow[d] & 1\\
    \mathrm H^1(\mathbf A_S, H)\arrow[r] & \mathrm H^1(\mathbf A_S, G)\arrow[r] & \mathrm H^1(\mathbf A_S, Q)
\end{tikzcd}\end{center}
The set $\mathrm H^1(k,Q)$ is discrete, so the first assertion of the lemma is equivalent to saying that, for a fixed element $z\in\mathrm H^1(k,Q)$ and nonempty open set $U\subseteq\mathrm H^1(\mathbf A_S,G)$ such that the image of $z$ in $\mathrm H^1(\mathbf A_S,Q)$ lands into the image of $U$, there exists an element $x\in\mathrm H^1(k,G)$ mapping to $z$ and into $U$ simultaneously.

Since $\mathrm H^1(k,G)\rightarrow\mathrm H^1(k,Q)$ is surjective, we may find a lift $y\in\mathrm H^1(k,G)$ of $z$. Twisting by a representative $Y$ of $y$, we send $z$ to $\tau_Y(z) = 1$. Because $\tau_Y : \mathrm H^1(\mathbf A_S,G)\rightarrow\mathrm H^1(\mathbf A_S,{_Y}G)$ is a homeomorphism (\cite[Prop.\ 3.6]{CesPT}), the set $\tau_Y(U)$ is open and has nonempty open preimage in $\mathrm H^1(\mathbf A_S,H)$. By assumption, the map $\mathrm H^1(k,{_Y}H)\rightarrow\mathrm H^1(\mathbf A_S,{_Y}H)$ has dense image, and there exists thus some $x_0\in\mathrm H^1(k,{_Y}H)$ mapping into $\tau_Y(U)$. If $\im(x_0)$ denotes its image in $\mathrm H^1(k,{_Y}G)$, then $x = \tau_Y^{-1}(x_0)$ is the desired element, proving the first part of the lemma.

For the second assertion, suppose that $\mathrm H^1(k,Q)\rightarrow\mathrm H^1(\mathbf A_S,Q)$ has dense image. In view of the first part of the lemma, we only have to show that, given an arbitrary open $U\subseteq\mathrm H^1(\mathbf A_S,G)$, there exists some $z\in\mathrm H^1(k,Q)$ mapping into the image of $U$ in $\mathrm H^1(\mathbf A_S,Q)$. For this, it suffices to note that the map $\mathrm H^1(\mathbf A_S,G)\rightarrow\mathrm H^1(\mathbf A_S,Q)$ is open by \cite[Prop.\ 3.11(d)]{CesPT} (or at least~by its proof; see remark below).
\end{prf}

\begin{rem}
The cited \cite[Prop.\ 3.11(d)]{CesPT} is actually stated under the condition that $H$ is a \textit{commutative} group-algebraic space (which is flat and of finite presentation). We claim that this commutativity condition (which is used only in one step of the proof) can be dropped when $H$ is a scheme, as is the case in our situation. In fact, the proof of \cite[Prop.~3.11(d)]{CesPT} seems to mistakenly invoke \cite[Cor.\ B.8]{CesTC} instead of \cite[Prop.\ B.13]{CesTC}: These two statements are analogous for our needs, however \cite[Prop.\ B.13]{CesTC} applies in the (intended) generality of commutative group-algebraic spaces and \cite[Cor.\ B.8]{CesTC} applies if $H$ admits a spread-out model over some $\mathcal O_{k,\,T}$ which is (at least over each closed fiber) a flat scheme of finite type. The second situation clearly holds when $H$ is an algebraic group.
\end{rem}

\begin{lem}\label{lemapproxliftunip}
Let $U$ be a unipotent algebraic group over a global function field $k$. If $S\neq\Omega$~is~a proper subset of places, then the natural morphism $\mathrm H^1(k,U)\rightarrow\mathrm H^1(\mathbf A_S,U)$ has dense image.
\end{lem}
\begin{prf}
We prove this statement by induction on the length of the derived series of $U$. As the derived series is preserved under base change, this length is greater for $U$ than for any $k$-form of $\mathcal D(U)$. By the preceding lemma, we reduce to the case when $U$ is commutative. We now~finish by Proposition \ref{proph1approx} (since $\Sh^1_S(k,\widehat U) = 0$ by Lemma \ref{corunipinjplace} applied to some $v\notin S$).

Alternatively, with a simpler argument, the commutative unipotent group $U$ admits a central series whose subquotients are subgroups of $\mathbf G_{\mathrm a}$, so by induction on the minimal length of such a series we may suppose $U\subseteq\mathbf G_{\mathrm a}$. Finally, there is an exact sequence $0\rightarrow U\rightarrow\mathbf G_{\mathrm a}\rightarrow\mathbf G_{\mathrm a}\rightarrow 0$, and we finish by strong approximation for $k\hookrightarrow\mathbf A_S$.
\end{prf}

\begin{lem}\label{lemapproxliftabquot}
Let $G$ be a smooth connected affine algebraic group over a global function field $k$. If $S$ is a finite set of places of $k$, then the following statements hold:
\begin{enumerate}[a)]
    \item The following natural map of pointed sets is a surjection:
    $$\mathrm H^1(k,G)\longrightarrow\mathrm H^1(k,G/\mathcal D(G))\times_{\mathrm H^1(k_S,\,G/\mathcal D(G))}\mathrm H^1(k_S,G)$$
    \item The natural map $\mathrm H^1(k,\mathcal D(G))\rightarrow\mathrm H^1(k_S,\mathcal D(G))$ is a surjection.
\end{enumerate}
\end{lem}
\begin{prf}
Since $G$ and $\mathcal D(G)$ are smooth, all of the sets in the statement are discrete. For a pure inner form $_YG$ of $G$, there is an identification $_Y\mathcal D(G)\cong\mathcal D({_Y}G)$. Therefore, statement (a) follows from (b) by Lemma \ref{lemapproxabstr} as soon as we notice that $\mathrm H^1(k,G)\rightarrow\mathrm H^1(k,G/\mathcal D(G))$ is surjective by Theorem \ref{thmmainabelh1}. For statement (b), let $U\coloneqq\mathscr R_{u,k}(G)$ be the unipotent radical and $p : G\twoheadrightarrow G/U$ the projection to the maximal pseudo-reductive quotient of $G$. Then $p(\mathcal D(G)) = \mathcal D(G/U)$ and the kernel of $\mathcal D(G)\twoheadrightarrow\mathcal D(G/U)$ is unipotent, so by Lemmas \ref{lemapproxabstr} and \ref{lemapproxliftunip} it suffices to prove only that $\mathrm H^1(k,Q)\rightarrow\mathrm H^1(k_S,Q)$ is a surjection for any perfect pseudo-reductive group $Q$. Note that, in our case, $Q = \mathcal D(G/U)$ is perfect even though $\mathcal D(G)$ does not have to be.

By Theorem \ref{thmpsredstrmain}, we may assume that $Q$ is noncommutative generalized standard pseudo-reductive (since cohomology commutes with products and since all totally non-reduced pseudo-reductive groups have trivial $\mathrm H^1$ set by Proposition \ref{propredbasnonred} and Shapiro's lemma).  Then there is a central surjection $j : \mathrm R(Q')\rightarrow Q$, for the Weil restriction $\mathrm R = \mathrm R_{k'/k}$ from some nonzero finite reduced $k$-algebra $k' = \prod k'_i$ and a $k'$-group $Q'$ with primitive fibers $Q'_i$. Let $T$ be a finite set of places of $k$ strictly containing $S$. Just as in the global case of Theorem \ref{thmabelmainpsred} we may choose $C' = \prod C'_i$, for Cartan subgroups $C'_i\subseteq Q'_i$, such that $\mathrm R(C')_v$ is anisotropic for all $v\in T$. Since $j$ is central, we have $\ker(j)\subseteq\mathrm R(C')$ and we may write a commutative diagram with exact rows:
\begin{center}\begin{tikzcd}
    \mathrm H^1(k_S, \mathrm R(C'))\arrow[r]\arrow[d] & \mathrm H^1(k_S, j(\mathrm R(C')))\arrow[r, two heads]\arrow[d] & \mathrm H^2(k_S, \ker(j))\arrow[r]\arrow[d, equal] & \mathrm H^2(k_S, \mathrm R(C'))\arrow[r, equal] &[-20pt] 0\\
    \mathrm H^1(k_S, \mathrm R(Q'))\arrow[r] & \mathrm H^1(k_S, Q)\arrow[r] & \mathrm H^2(k_S, \ker(j))
\end{tikzcd}\end{center}
By anisotropicity of the smooth connected groups $\mathrm R(C')_v$ over all $v\in S$ and local duality (recall the discussion surrounding Proposition \ref{propanisses}), the top right group vanishes. We now claim that the column $\mathrm H^1(k_S, j(\mathrm R(C')))\rightarrow\mathrm H^1(k_S, Q)$ is a surjection. Using the diagram, it suffices to show that the bottom right map is an injection, which we now do: If $x,y\in\mathrm H^1(k_S, Q)$ are two classes with the same image in $\mathrm H^2(k_S, \ker(j))$, let $Y$ be a torsor of $Q$ representing $y$ and, importantly, splitting over an \'etale cover of $k_S$ (which exists since $Q$ is smooth).~The~composition
$$Q = \mathrm R(Q')/\ker(j)\twoheadrightarrow\mathrm R(Q')/\mathrm Z_{\mathrm R(Q')}\hookrightarrow\mathcal{Aut}_{\mathrm R(Q')}$$
allows us to twist by $Y$ (fixing $\ker(j)$) and get the sequence ${0\rightarrow \ker(j)\rightarrow {_Y}(\mathrm R(Q'))\rightarrow {_Y}Q\rightarrow 0}$. Now, by Corollary \ref{corconradetform}, Shapiro's lemma, Kneser-Bruhat-Tits and Proposition~\ref{propredbasexpsred}, we get $\mathrm H^1(k_S, {_Y}(\mathrm R(Q'))) = 1$. As the twist $\tau_Y(x)$ maps to the trivial element in $\mathrm H^1(k_S, Q)$, it~is itself trivial. Therefore $x = y$, which proves the claim.

After the claim just proven, it remains only to show that $\mathrm H^1(k, C)\rightarrow\mathrm H^1(k_S, C)$ is a surjection, for the smooth connected affine algebraic group $C\coloneqq j(\mathrm R(C'))$ with $C_v$ anisotropic for all $v\in T$. Applying Proposition \ref{propanisgen} for some $v\in T\setminus S$, we get that $\Sh^1_{\Omega\setminus\{v\}}(\widehat C) = 0$ and, as $S\subseteq\Omega\setminus\{v\}$, conclude by Proposition \ref{proph1approx} that $\mathrm H^1(k, C)\rightarrow\mathrm H^1(k_S, C)$ is surjective, as desired.
\end{prf}

\begin{rem}
Lemma \ref{lemapproxliftabquot}(b) holds for any $k$-form $H$ of $\mathcal D(G)$ (even if $H$ is a priori~not~known to be the derived subgroup of any group): If $Q = \mathcal D(G/U)$ as in the proof, the surjectivity of $\mathcal D(G)\twoheadrightarrow Q$ implies that $\mathcal D(\mathcal D(G))\twoheadrightarrow \mathcal D(Q) = Q$ is surjective, so $\mathcal D(G)/\mathcal D(\mathcal D(G))$ is unipotent. Therefore every $k$-form $H$ is a smooth connected affine algebraic group with $H/\mathcal D(H)$ unipotent, so by Lemma \ref{lemapproxliftabquot}(a) and Lemma \ref{lemapproxliftunip} the map $\mathrm H^1(k, H)\rightarrow\mathrm H^1(k_S, H)$ is surjective. 
\end{rem}

We are finally ready to finish the proof of the main result.

\begin{prf}[of Theorem \ref{thmmainbmobs}(b)]
We again let $G,C,D,X,\Hb$ be as in Propositions \ref{propmainreduction1} and \ref{propmainreduction2}, so the formation of the normal subgroup $D\subseteq G$ commutes with taking \'etale $k$-forms of $G$ (and $\mathrm H^1(k,D') = 1$ for each \'etale $k$-form $D'$ of $D$), the quotient $C = G/D$ is commutative and $(H\cap D)/\mathcal D(H)$ is unipotent. Suppose given a point $x_S\in X(k_S)^{\Be_S} = \ker(X(k_S)\rightarrow\Be_S(X)^*)$, which we fix. Then the map $BM_X : \Be(X)\rightarrow\mathbf Q/\mathbf Z$ (which can be constructed by evaluation at any point in $X(\mathbf A)$) is trivial, so $X(k)\neq\varnothing$ by Theorem \ref{thmmainbmobs}(a) and we may assume $X = H\backslash G$ for some (smooth and connected) subgroup $H$ of $G$.

We need to prove that $x_S$ lies in the closure of $X(k)\subseteq X(k_S)$. Consider the quotient scheme $X/D\coloneqq \big(H/(H\cap D)\big)\backslash\big(G/D\big)$. By functoriality of the Brauer-Manin pairing, there is a map $X(k_S)^{\Be_S}\rightarrow (X/D)(k_S)^{\Be_S}$ whose codomain is exactly the closure of $(X/D)(k)\subseteq (X/D)(k_S)$ by \cite[Thm.\ 4.5]{Don24}. In the following diagram, whose rows are exact sequences of pointed sets, the first column is surjective and the last is injective (by the stated properties of $D$):
\begin{center}\begin{tikzcd}[column sep=large]
    G(k)\arrow[r]\arrow[d, two heads] & X(k)\arrow[r]\arrow[d] & \mathrm H^1(k, H)\arrow[r]\arrow[d] & \mathrm H^1(k, G)\arrow[d, hook]\\
    \dfrac{G}{D}(k)\arrow[r] & \dfrac{X}{D}(k)\arrow[r] & \mathrm H^1\!\left(\!k,\dfrac{H}{H\cap D}\right)\arrow[r] & \mathrm H^1\!\left(\!k,\dfrac{G}{D}\right)
\end{tikzcd}\end{center}
Fix an open neighborhood $U\subseteq X(k_S)$ of $x_S$, and consider their images $\overline{x_S}\in\overline U$ in $(X/D)(k_S)$. We claim that $\overline{x_S}$ lies in the interior of $\overline U$: To see this, define a map $r^{x_S} : G_{k_S}\rightarrow X_{k_S}$ (which is in general not compatible with the above diagram) by ${g_S}\rightarrow {x_S}.{g_S}$. The induced composition
$$G(k_S)\longrightarrow (G/D)(k_S)\xrightarrow{\,r^{\overline{x_S}}\,} (X/D)(k_S)$$
is open by \cite[Prop.\ 4.3(a)]{CesTC}, since both $D$ and (any $k_S$-form of) $H/(H\cap D)$ are $k_S$-smooth. However, it maps the open set $(r^{x_S})^{-1}(U)$ into $\overline U$, which proves the claim.

By the aforementioned denseness of $(X/D)(k)$ in $(X/D)(k_S)^{\Be_S}$, we may choose an element $\overline x\in(X/D)(k)\cap\overline U$. Now, the third column $\mathrm H^1(k,H)\rightarrow\mathrm H^1(k,H/H\cap D)$ in the above diagram is surjective, just as in the proof of Theorem \ref{thmmainbmobs}(a).
A diagram chase then gives that $\overline x$ admits a preimage $x$ in $X(k)$. In other words, the fiber $X'\coloneqq X_{\overline x}$ (of the map $X\rightarrow X/D$) has a $k$-point and also $X'(k_S)\cap U\neq\varnothing$. To complete the proof, it only remains to prove that $X'$ satisfies weak approximation (with respect to the finite set $S$).

Returning to the beginning of the proof, we may assume (for simplicity of notation) that $x$ is the image of $1_G\in G(k)$. This has the effect of replacing $H$ by a $k$-form defined via conjugation inside $G_{k_s}$, which preserves the supposed property that the quotient $(H\cap D)/\mathcal D(H)$ is unipotent. Then $X' = (H\cap D)\backslash D$ and we consider the diagram:
\begin{center}\begin{tikzcd}[column sep=large]
    (H\cap D)(k)\arrow[r]\arrow[d] & D(k)\arrow[r]\arrow[d] & X'(k)\arrow[r]\arrow[d] & \mathrm H^1(k, H\cap D)\arrow[r]\arrow[d] & 1\\
    (H\cap D)(k_S)\arrow[r] & D(k_S)\arrow[r] & X'(k_S)\arrow[r] & \mathrm H^1(k_S, H\cap D)
\end{tikzcd}\end{center}
By construction, $D$ is an extension of a product $D_1\times D_2$ by a split unipotent group $\mathscr R_u(G)$, where $D_1 = \mathrm R_{k'/k}(\prod_i D'_i)$ is the Weil restriction of a product of primitive groups and $D_2 = \mathrm R_{K/k}(Q)$ for $Q$ basic non-reduced (if $D_2\neq 1$, then $\mathrm{char}(k) = 2$ and there is a map $Q\rightarrow\mathrm R_{K^{1/2}/K}(\mathrm{Sp}_{2n,\,K^{1/2}})$ inducing bijections on points over all separable extensions of $K$). 
It thus follows that $D$ satisfies weak approximation, i.e.\ that $D(k)\rightarrow D(k_S)$ has dense image, by reduction to the well-known cases of $\mathbf G_{\mathrm a}$ and of semisimple simply-connected groups (using Proposition \ref{propredbasexpsred}).

Since $D$ is smooth, the map $X'(k_S)\rightarrow\mathrm H^1(k_S, H\cap D)$ is open by \cite[Prop.\ 4.3(b)]{CesTC}. We claim furthermore that the map $\mathrm H^1(k, H\cap D)\rightarrow\mathrm H^1(k_S, H\cap D)$ has dense image, which clearly suffices to complete the proof (given a nonempty open $U'\subseteq X'(k_S)$, we construct by a diagram chase some $x'\in X'(k)$ such that $x'.d_S\in U'$ with $d_S\in D(k_S)$; then weak approximation for $D$ gives $d\in D(k)$ such that $x'.d\in U'$). This claim will follow from an application of Lemma \ref{lemapproxabstr} to the exact sequence $0\rightarrow \mathcal D(H)\rightarrow H\cap D\rightarrow (H\cap D)/\mathcal D(H)\rightarrow 0$. For this, we need to check three conditions: First, the map
$$\mathrm H^1\!\left(\!k,\dfrac{H\cap D\,}{\mathcal D(H)}\right)\longrightarrow\mathrm H^1\!\left(\!k_S,\dfrac{H\cap D\,}{\mathcal D(H)}\right)$$
has dense image by Lemma \ref{lemapproxliftunip}. Second, we consider the commutative diagram
\begin{center}\begin{tikzcd}[column sep=large]
    \dfrac{H}{H\cap D}(k)\arrow[r]\arrow[d, equal] & \mathrm H^1\big(k, H\cap D\big)\arrow[r]\arrow[d] & \mathrm H^1\big(k, H\big)\arrow[r]\arrow[d] & \mathrm H^1\!\left(\!k,\dfrac{H}{H\cap D}\right)\arrow[d, equal]\\
    \dfrac{H}{H\cap D}(k)\arrow[r] & \mathrm H^1\!\left(\!k,\dfrac{H\cap D\,}{\mathcal D(H)}\right)\arrow[r] & \mathrm H^1\!\left(\!k,\dfrac{H}{\mathcal D(H)}\right)\arrow[r] & \mathrm H^1\!\left(\!k,\dfrac{H}{H\cap D}\right)
\end{tikzcd}\end{center}
with exact rows and surjective third column (by Theorem \ref{thmmainabelh1}), which implies that the second column is surjective as well (because the group $(H/(H\cap D))(k)$ acts transitively on fibers in both rows). Third, we consider $k$-forms $_Y\mathcal D(H)$ of $\mathcal D(H)$ induced by $(H\cap D)$-torsors $Y$. Because $H\cap D\subseteq H$, there is an inclusion $_Y\mathcal D(H)\subseteq {_Y}H$ which gives an identification $_Y\mathcal D(H)\cong \mathcal D({_Y}H)$; then Lemma \ref{lemapproxliftabquot}(b) implies that the map $\mathrm H^1(k, {_Y}\mathcal D(H))\rightarrow\mathrm H^1(k_S, {_Y}\mathcal D(H))$ has dense image. This establishes all three conditions to apply Lemma \ref{lemapproxabstr}.
\end{prf}

\newpage

%% file: partA-1.tex
\subsection{Basic Properties of Pseudo-reductive Groups}\label{ssectpsred1}
In this subsection, $k$ is an arbitrary field. Given an affine algebraic group $G$ over $k$, recall that the \textit{derived subgroup} $\mathcal D(G)$ of $G$ is its unique largest subgroup such that $G/\mathcal D(G)$ is commutative (this is a well-behaved notion if $G$ is affine or smooth; see \cite[\textsection 6.d]{Mil17}). 

\begin{prop}\label{proppsredmulunip}
Let $G$ be a pseudo-reductive group over a field $k$. Let $\mathcal D(G)$ be its derived subgroup and let $G^t$ be the subgroup of $G$ generated by its (maximal) tori.

Then $\mathcal D(G)$ is perfect (that is, $\mathcal D(G) = \mathcal D(\mathcal D(G))$) and $\mathcal D(G)\subseteq G^t$. The short exact sequence
$$0\rightarrow\frac{G^t}{\mathcal D(G)}\rightarrow\frac{G}{\mathcal D(G)}\rightarrow\frac{G}{G^t}\rightarrow 0$$
is the multiplicative-unipotent decomposition of the commutative algebraic group $G/\mathcal D(G)$ (see the beginning of Subsection \ref{ssectgenanis}).
\end{prop}
\begin{prf}
We only use pseudo-reductivity to say that $\mathcal D(G)$ is perfect (\cite[Prop.\ 1.2.6]{CGP15}). It is a general fact, shown for smooth connected $G$ in [ibid., Prop.\ A.2.11], that $G/G^t$ is unipotent (and trivial, if $G$ is perfect). By [ibid., Cor.\ A.2.7], $\mathcal D(G) = \mathcal D(G)^t\subseteq G^t$. Finally, $G^t/\mathcal D(G)$ is smooth, connected, commutative and generated by tori, hence a torus.
\end{prf}

\begin{cor}\label{corpsredmulunip}
If $\mathrm Z_G$ denotes the center of $G$, then $G/(\mathrm Z_G\cdot \mathcal D(G))$ is unipotent.
\end{cor}
\begin{prf}
By \cite[Lem.\ 1.2.5(iii)]{CGP15}, each torus $T\subseteq G$ belongs to $\mathrm Z_G\cdot \mathcal D(G)$.
\end{prf}
\medskip

Recall that, if a torus $T\subseteq G$ is an (inclusion-wise) maximal torus, then so is $T_{\overline k}\subseteq G_{\overline k}$, and any two maximal tori are geometrically conjugate (see \cite[Cor.\ A.2.6 and Prop.\ A.2.10]{CGP15}). The centralizer $C = \mathrm Z_G(T)$ of any maximal torus $T$ in $G$ is called a \textit{Cartan subgroup} of $G$.

\begin{prop}\label{propcartderv}
Let $G$ be a pseudo-reductive group over a field $k$, and take any Cartan subgroup $C$ of $G$. Then $C$ is commutative pseudo-reductive and $G = \mathcal D(G)\cdot C$.

Suppose that $N$ is a smooth connected normal subgroup of $G$. Then $N$ is pseudo-reductive and also $C\cap N$ is a Cartan subgroup of $N$. (This holds in particular for $N = \mathcal D(G)$.)
\end{prop}
\begin{prf}
These are \cite[Prop.\ 1.2.4, Prop.\ 1.2.6 and Lem.\ 1.2.5(ii)]{CGP15}.
\end{prf}

\begin{prop}\label{proppsredsubgrpcenter}
Let $G$ be a pseudo-reductive group over a field $k$. Then $\mathrm Z_G$ is exactly the intersection of all Cartan subgroups of $G$. Consequently, if $N$ is a smooth connected normal subgroup of $G$, then $\mathrm Z_N = N\cap\mathrm Z_G$.
\end{prop}
\begin{prf}
As $\mathrm Z_G$ is contained in every Cartan subgroup of $G$, it is clearly contained in their intersection $I$. Moreover, $I$ commutes with every Cartan subgroup $C$ (since $C$ is commutative). Therefore, for the inverse inclusion $I\subseteq\mathrm Z_G$, it suffices to show that $G$ is generated by its Cartan subgroups (cf.\ \cite[Prop.\ 1.92]{Mil17}). But this is clear, since the property $G = G^t\cdot C$ follows from the two propositions above, as does then the final claim in the statement.
\end{prf}

The following proposition appears to generalize \cite[Lem.\ 4.1.1]{CP15}, where it is assumed in addition that $H$ is pseudo-reductive (and whose proof also depends on deep statements about the structure of $G$ and $H$, while our proof is significantly simpler).

\begin{prop}\label{proppsredcentralsurj}
Let $G$ be pseudo-reductive and $f : G\twoheadrightarrow H$ be a surjective homomorphism of smooth algebraic groups over $k$ which is central (that is, $\ker(f)\subseteq\mathrm Z_G$). Then the induced~map on centers $\mathrm Z_G\rightarrow\mathrm Z_H$ is also a flat surjection (i.e.\ an epimorphism of algebraic groups).
\end{prop}
\begin{prf}
Any surjection induces a map between centers, so $f(\mathrm Z_G)\subseteq\mathrm Z_H$. The surjection $f$ maps Cartan subgroups onto Cartan subgroups (by \cite[Prop.\ A.2.8]{CGP15}), so $\mathrm Z_H\subseteq f(C)$ holds for any Cartan subgroup $C$ of $G$. But $\ker(f)\subseteq C$ too, so $f^{-1}(\mathrm Z_H)\subseteq C$. As this holds for all $C$, we have $f^{-1}(\mathrm Z_H)\subseteq\mathrm Z_G$, so $\mathrm Z_H = f(f^{-1}(\mathrm Z_H))\subseteq f(\mathrm Z_G)$ as $f$ is an epimorphism.
\end{prf}

Propositions \ref{propcartderv} and \ref{proppsredsubgrpcenter} above give, for any Cartan subgroup $C\subseteq G$, an isomorphism
$$\frac{C}{C\cap\mathcal D(G)}\xrightarrow{\;\sim\;}\frac{G}{\mathcal D(G)}\;\;\textrm{ and show that }\;\;
0\rightarrow\frac{\mathrm Z_G}{\mathrm Z_{\mathcal D(G)}}\rightarrow\frac{G}{\mathcal D(G)}\rightarrow\frac{G}{\mathrm Z_G\cdot\mathcal D(G)}\rightarrow 0$$
is a short exact sequence of commutative algebraic groups.

Reductive groups satisfy the property that $G = \mathcal D(G)\cdot\mathrm Z_G$. In fact, there is an almost-direct product, that is, an isogeny $\mathcal D(G)\times Z\twoheadrightarrow G$, where $Z$ is the maximal torus of $\mathrm Z_G$, and the derived subgroup $\mathcal D(G)$ is moreover semisimple. This property will be suitably generalized to pseudo-reductive groups in the next two subsections, but we now show that the naive generalization fails (even for ``standard'' pseudo-reductive groups, to be defined below). Moreover, we compute explicitly in the following example that the map $$\mathrm{H}^2\!\left(k,\frac{\mathrm Z_G}{\mathrm Z_{\mathcal D(G)}}\right)\longrightarrow\mathrm{H}^2\!\left(k, \frac{C}{C\cap\mathcal D(G)}\right) = \mathrm{H}^2\!\left(k, \frac{G}{\mathcal D(G)}\right)$$ is not always an injection. On the other hand, it is always a surjection, since $G/(\mathrm Z_G\cdot\mathcal D(G))$~is unipotent by Corollary \ref{corpsredmulunip} and thus has vanishing $\mathrm H^2$ (see Proposition \ref{propunipfppfvanish}).

\begin{exmp}\label{exmppsredh2notinj}
Suppose $k'/k$ is a finite field extension with $k'\subseteq k^{1/p^n}$, as in Example \ref{exmpweilrestr}. The group $G = \mathrm R(\mathbf{GL}_{p^n})$ is (standard) pseudo-reductive, for the Weil restriction $\mathrm R = \mathrm R_{k'/k}$. Note that $\mathcal D(\mathbf{GL}_{p^n}) = \mathbf{SL}_{p^n}$ and that $\mathcal D(G) = \mathrm R(\mathcal D(\mathbf{GL}_{p^n}))$ since $\mathrm R(\mathbf{SL}_{p^n})$ is a perfect group by \cite[Cor.\ A.7.11]{CGP15}.
Also, $\mathbf{GL}_{p^n}$ has a maximal torus $C'\simeq \mathbf G_{\mathrm m}^{p^n}$ (the diagonal matrices) and center $\mathrm Z_{\mathbf{GL}_{p^n}}\simeq\mathbf G_{\mathrm m}$. Similarly, consider the following maximal torus of $\mathbf{SL}_{p^n}$
$$T'\coloneqq C'\cap\mathbf{SL}_{p^n} = \ker(\mathrm{det}|_{C'})\simeq\mathbf G_{\mathrm m}^{p^n-1}$$
and $\mathrm Z_{\mathbf{SL}_{p^n}}\simeq\mu_{p^n}$. By \cite[Prop.\ A.5.15]{CGP15}, Weil restriction commutes with the formation of centers and Cartan subgroups, so $\mathrm R(C')$ is a Cartan subgroup of $G$, and there is an isomorphism
$$\frac{\mathrm Z_G}{\mathrm Z_{\mathcal D(G)}} = \frac{\mathrm Z_{\mathrm R(\mathbf{GL}_{p^n})}}{\mathrm Z_{\mathrm R(\mathbf{SL}_{p^n})}} = \frac{\mathrm R(\mathrm Z_{\mathbf{GL}_{p^n}})}{\mathrm R(\mathrm Z_{\mathbf{SL}_{p^n}})}\simeq\frac{\mathrm R(\mathbf G_{\mathrm m})}{\mathrm R(\mu_{p^n})}\xrightarrow{\;\;\;\sim\;\;\;}\mathbf G_{\mathrm m}$$
where the final map is the dashed map from Example \ref{exmpweilrestr}. We also have by Proposition \ref{proppreshapiro} (which says that $\mathrm R$ preserves smooth surjections) a similar chain of isomorphisms:
$$\frac{C}{C\cap\mathcal D(G)} = \frac{C}{\mathrm{im}(\phi)} = \frac{\mathrm R(C')}{\mathrm R(T')}\simeq\frac{\mathrm R(\mathbf G_{\mathrm m}^{p^n})}{\mathrm R(\mathbf G_{\mathrm m}^{p^n-1})}\simeq\mathrm R(\mathbf G_{\mathrm m})$$
Here, $\phi : \mathrm R(T')\hookrightarrow\mathrm R(C') = C$ is an inclusion with image exactly $\mathrm R(C')\cap\mathrm R(\mathbf{SL}_{p^n}) = C\cap\mathcal D(G)$.

\noindent
Finally, by comparing this situation to the calculation in Example \ref{exmpweilrestr}, we see that
\begin{center}
the diagram
\begin{tikzcd}[row sep = tiny]
& \dfrac{\mathrm Z_G}{\mathrm Z_{\mathcal D(G)}}\arrow[d, hook, shift left = 2, bend left=30, shorten <=-5pt, shorten >=-10pt]\\
\mathrm Z_G\arrow[ur, dashed]\arrow[r] & \dfrac{C}{C\cap\mathcal D(G)}
\end{tikzcd}
fits into the diagram
\begin{tikzcd}
\mathbf G_{\mathrm m}\arrow[r, "x\,\mapsto\, x^{p^n}"]\arrow[d, hook] & \mathbf G_{\mathrm m}\arrow[d, hook]\\
\mathrm R(\mathbf G_{\mathrm m})\arrow[ur, dashed]\arrow[r, "x\,\mapsto\, x^{p^n}"] & \mathrm R(\mathbf G_{\mathrm m})
\end{tikzcd}
there.
\end{center}
Consequently, the cohomology map $\mathrm{H}^2(k,\mathrm Z_G/\mathrm Z_{\mathcal D(G)})\twoheadrightarrow\mathrm{H}^2(k, C/(C\cap\mathcal D(G)))$ is identified with $\mathrm{Br}(k)\twoheadrightarrow\mathrm{Br}(k')$, which is generally not an injection. For us, this fact comes into play in the important Example \ref{exmpabelmain}. 
\end{exmp}

%% file: partA-2.tex
\subsection{Generalized Standard Pseudo-reductive Groups}\label{ssectpsred2}
Before this subsection, recall that we assume ``reductive'' to in particular mean ``connected''. Our work above with the group $\mathrm R(\mathbf{GL}_{p^n})$ was facilitated in particular by our knowledge of its derived and Cartan subgroups, which we could infer from its construction through Weil restriction. This is generalized in the following procedure by which, starting with very well-behaved groups over an extension of $k$, we may produce (essentially all) pseudo-reductive groups over $k$:

\begin{exmp}\label{exmppsredpres}
Let $k'$ denote a nonzero finite reduced $k$-algebra, i.e. $k'=\prod_{i=1}^s k'_i$ for finitely many finite field extensions $k'_i/k$. 
Let $G'$ be a pseudo-reductive $k'$-group (i.e. every fiber $G'_{k_i}$ is a pseudo-reductive algebraic group; all related constructions can be performed fiber-wise). Then the Weil restriction $\mathrm R_{k'/k}(G')$ is a pseudo-reductive group over $k$; we will often simply write $\mathrm R\coloneqq\mathrm R_{k'/k}$. We consider a maximal torus $T'$ in $G'$ and the associated Cartan subgroup $C'\coloneqq Z_{G'}(T')$. Then $\mathrm R(C')$ is a Cartan subgroup of $\mathrm R(G')$.

Consider a commutative pseudo-reductive group $C$ over $k$. Suppose given a homomorphism $\phi : \mathrm R(C')\rightarrow C$, and a left action of $C$ on $\mathrm R(G')$ whose composition with $\phi$ is the usual action (of $\mathrm R(C')$ on $\mathrm R(G')$ by conjugation) and whose effect on the subgroup $\mathrm R(C')$ is trivial. If $j$ denotes the inclusion $\mathrm R(C')\hookrightarrow\mathrm R(G')$, then we define an inclusion ${\alpha : \mathrm R(C')\rightarrow\mathrm R(G')\rtimes C}$~(where the semidirect product is given by the above action) via $\alpha(t') = (j(t)^{-1},\phi(t))$. It is elementary~to~see that this realizes $\mathrm R(C')$ as a central subgroup (i.e. a subgroup contained in the center).

Finally, let $G\coloneqq \mathrm{coker}(\alpha) = (\mathrm R(G')\rtimes C)/\mathrm R(C')$. This is a pseudo-reductive group over $k$, as shown in \cite[Prop.\ 1.4.3]{CGP15}. The images of $\mathcal D(\mathrm R(G'))$ and $C$ inside $G$ are exactly $\mathcal D(G)$ and $\mathrm Z_G(T)$, respectively, where $T$ is the unique maximal torus of $G$ containing the image of the maximal torus of $\mathrm R(C')$. As the resulting map $C\hookrightarrow G$ is an inclusion, this construction makes $C$ into a Cartan subgroup of $G$, essentially ``replacing $\mathrm R(C')$ in $\mathrm R(G')$''.
\end{exmp}

The above construction depends on many choices -- in particular the choice of $C$, the map $\phi$ and the action of $C$ on $\mathrm R(G')$. A natural attempt to construct such a situation would be to have the obvious map $\mathrm R(C')\rightarrow\mathrm R(C'/Z_{G'})$ factor through $\phi$ and have $C$ act on $\mathrm R(G')$ by the factor map $C\rightarrow\mathrm R(C'/Z_{G'})$. This works, because the action of conjugation by $\mathrm R(C')$ always factors through $\mathrm R(C'/Z_{G'})$, and the action of $\mathrm R(C'/Z_{G'})$ is trivial on $\mathrm R(C')\subseteq\mathrm R(G')$ since $C'$ is commutative (it is a Cartan subgroup of a pseudo-reductive group).
 
We've stated the above construction without any conditions on $G'$ except pseudo-reductivity. A priori, every pseudo-reductive group $G$ can be constructed by the above method trivially, by taking $k'=k$ and $G'=G$. The structure theory of $G$ rests on our ability to construct it as above, while simultaneously requiring strong properties from $G'$.

\begin{defn}
If a pseudo-reductive group can be constructed using the above method, where the fibers of $G'$ over factor fields $k'_i$ of $k'$ are reductive groups and where the choices are given (as in the discussion post-example) by factoring $\mathrm R(C')\rightarrow\mathrm R(C'/Z_{G'})$, we say that it is a \textit{standard pseudo-reductive group}. Then in particular $C' = T'$ and we call $(G',k'/k,T',C)$ a presentation of $(G,T)$. If $G$ is reductive or commutative, then it's standard.

Moreover, if a noncommutative pseudo-reductive group is standard in the above sense, we may choose the fibers of $G'$ over all factors of $k'$ to be semisimple, absolutely simple and simply connected. Such a presentation $(G', k'/k, T', C)$ is uniquely determined by and functorial with respect to $(G, T)$. It is the \textit{standard presentation} of $(G, T)$ (cf.\ \cite[Prop.\ 4.1.4(iii), 5.2.2]{CGP15}).
\end{defn}

An affine algebraic group $G$ over $k$ is said to be \textit{pseudo-simple} if it is smooth, connected, noncommutative and all its proper normal subgroups are trivial. It is \textit{simple} if it's semisimple and pseudo-simple. It is \textit{absolutely (pseudo-)simple} if $G_{k_s}$ is (pseudo-)simple. 

The construction we described is ubiquitous, as it turns out by \cite[Prop.\ 5.1.1]{CGP15} that all pseudo-reductive groups are standard outside of some specific cases when $\mathrm{char}(k)\in\{2,3\}$. To incorporate these cases, we must first introduce another kind of absolutely pseudo-simple group, analogous to those considered above (we follow \cite[Def.\ 7.2.6 and Prop.\ 7.3.1]{CGP15}):

\begin{defn}\label{defgenstand}
Let $k$ be a field of characteristic $p\in\{2,3\}$ and $k'/k$ a nontrivial finite field extension with $k'\subseteq k^{1/p}$. Let $G'$ be any semisimple group over $k'$ that is absolutely simple and simply connected with an edge of multiplicity $p$ in its absolute Dynkin diagram. There is a unique nontrivial factorization of the Frobenius isogeny $G'\rightarrow (G')^{(p)}$, which is through a \textit{very special isogeny} $\pi' : G'\rightarrow\mathcal G'$ (see \cite[Def.\ 7.1.3]{CGP15}).

Consider a Levi subgroup $\mathcal G\subseteq\mathrm R(\mathcal G')$ and its preimage $G\subseteq\mathrm R(G')$ by $\mathrm R(\pi') : \mathrm R(G')\rightarrow\mathrm R(\mathcal G')$. If $G$ is smooth, then it is pseudo-reductive (\cite[Thm.\ 7.2.3]{CGP15}) and any algebraic group over $k$ isomorphic to it is called a \textit{basic exotic pseudo-reductive group}. By \cite[Prop.\ 8.1.1 and Cor.\ 8.1.3]{CGP15}, it is never standard, but always absolutely pseudo-simple.

Given a nonzero finite reduced $k$-algebra $k'=\prod_{i=1}^s k'_i$, let us call a $k'$-group $G'$ \textit{primitive} if each fiber $G'_{k'_i}$ is absolutely pseudo-simple and either semisimple simply connected or basic exotic pseudo-reductive. If a pseudo-reductive group $G$ over $k$ is commutative, or is noncommutative and can be constructed via the method of Example \ref{exmppsredpres}, where the $k'$-group $G'$ is primitive, then we say this group is a \textit{generalized standard pseudo-reductive group} and call $(G', k'/k, T', C)$ a \textit{generalized standard presentation} adapted to the pair $(G, T)$. Note that our definition here agrees with \cite[Def.\ 10.1.9]{CGP15} and not \cite[Def.\ 9.1.7]{CP15} (the difference is in the notion of ``primitive pairs''), as was more convenient for our purposes.
\end{defn}

Let $G$ be a pseudo-reductive group over $k$. If $\mathrm{char}(k) = 2$, then assume that $[k : k^2]\leq 2$ and that $G_{k_s}$ has a reduced root system (which is always true if $\mathrm{char}(k)\neq 2$). The group $G$ is generalized standard (and is standard when, in some generalized presentation of $G$, no fibers of $G'$ are basic exotic) by \cite[Thm.\ 10.2.1(2), Prop.\ 10.2.4]{CGP15}. We will recall the developments in the nonreduced case in the next subsection, and precisely formulate the required structure theorems there. Meanwhile, we focus on the uniqueness of generalized standard presentations, as it is crucial to our work with global representability of bands:

\begin{thm}\label{thmgenstandfunct}
Let $G$ be a noncommutative generalized standard pseudo-reductive group over $k$. The following properties hold:
\begin{enumerate}[(a)]
    \item Let $T\subseteq G$ be a maximal torus and $(G', k'/k, T', C)$ a generalized standard presentation adapted to the pair $(G, T)$. Given an isomorphism $f : G_0\rightarrow G$ with $T_0\coloneqq f^{-1}(T)$ and a generalized standard presentation $(G'_0, k'_0/k, T'_0, C_0)$ which is adapted to the pair $(G_0, T_0)$, the isomorphism defined by the following composition
    $$(\mathrm R_{k'_0/k}(G'_0)\rtimes C_0)/\mathrm R_{k'_0/k}(\mathrm Z_{G'_0}(T'_0))\cong G_0\xrightarrow{\;f\;} G\cong (\mathrm R_{k'/k}(G')\rtimes C)/\mathrm R_{k'/k}(\mathrm Z_{G'}(T'))$$
    is induced by the fixed isomorphism
    $C_0\cong\mathrm Z_{G_0}(T_0)\xrightarrow{\;f\;} \mathrm Z_G(T)\cong C$
    and a unique pair $(\alpha,\beta)$, where $\alpha : k'_0\rightarrow k'$ is a $k$-algebra isomorphism and $\beta : \alpha_*G'_0\rightarrow G'$ a $k'$-group isomorphism such that $\alpha_*T'_0 = \beta^{-1}(T')$.
    \item Consider the map $j : \mathrm R_{k'/k}(G')\rightarrow G$. Similarly to the above point, the triple $(G', k'/k, j)$~is uniquely functorial in isomorphisms of $G$ over $k$. The image of $j$ is exactly $\mathcal D(G)\subseteq G$. 
    \item A fixed triple $(G', k'/k, j)$ induces a bijective correspondence between the set of maximal tori $T'\subseteq G'$ and the set of maximal tori $T\subseteq G$, such that the $4$-tuples $(G', k'/k, T', \mathrm Z_G(T))$ are generalized standard presentations of $G$. Thus, in particular, the property of $G$ being generalized standard does not depend on a chosen $T$. 
\end{enumerate}
\end{thm}
\vspace{-7pt}

\begin{prf}
Points (a) and (c) are parts of \cite[Prop.\ 10.2.2 and 10.2.4]{CGP15}, while (b) is most clearly stated in \cite[Rem.\ 9.1.9 and Prop.\ 9.1.12(ii)]{CP15} (it also follows from \cite[Rem.\ 10.1.11 and Prop.\ 10.1.12(1)]{CGP15} by arguing as in the proof of our Proposition \ref{propsemiautpsred}).
\end{prf}

\begin{cor}[{\cite[Prop.\ 6.4.1]{Con12}}]\label{corconradetform}
Let $G$ and $G_0$ be algebraic groups over $k$ such that $G_{0,\,k_s}\simeq G_{k_s}$ (i.e. they are \textit{\'etale $k$-forms}). If there exists a nonzero finite reduced $k$-algebra $k'$ and a primitive $k'$-group $G'$ such that $G\simeq\mathrm R_{k'/k}(G')$, then there exists a nonzero finite reduced $k$-algebra $k'_0$ and a primitive $k'_0$-group $G'_0$ such that $\smash{G_0\simeq\mathrm R_{k'_0/k}(G'_0)}$.
\end{cor}
\begin{prf}
Under the assumptions of the theorem, $\smash{j : \mathrm R_{k'/k}(G')\xrightarrow{\sim} G}$ is the generalized standard presentation of $G$ (adapted to any maximal torus $T\subseteq G$; here $\phi : \mathrm R_{k'/k}(C')\rightarrow C$ from Example \ref{exmppsredpres} is an isomorphism). By uniqueness of generalized standard presentations (up to a choice of $T$) and the unique functoriality of $j$, the isomorphism $j_{k_s} : \mathrm R_{(k'\otimes_k k_s)/k_x}(G')\xrightarrow{\sim} G_{k_s}\simeq G_{0,\,k_s}$ descends to a generalized standard presentation of the same form for $G_0$.
\end{prf}

%% file: partA-3.tex
\vspace{-8pt}
\subsection{Reducing to the Generalized Standard Case}\label{ssectpsred3}
The classification of pseudo-reductive groups becomes far more complicated in characteristic $2$, with many (families of) examples of groups which are not generalized standard. A partial structure theory is provided in \cite[\textsection 9.2]{CP15} (with a slightly wider definition of ``generalized standard pseudo-reductive groups'' than ours), but as our main applications are in the case of local and global fields, for which $[k : k^p] = p$, the earlier theory developed in \cite{CGP15} suffices. In particular, we suppose that $\smash{[k : k^p] = p}$ when $\mathrm{char}(k) = p\in\{2,3\}$, which allows for strong reduction statements:

\begin{prop}\label{propredbasexpsred}
Assume $\mathrm{char}(k) = p\in\{2,3\}$ and $[k : k^p] = p$. The map $\mathrm R(\pi)|_G : G\rightarrow\mathcal G$ from the definition of a basic exotic pseudo-reductive group $G$ is a surjective map of smooth groups, onto a semisimple, simply connected and absolutely simple group. Moreover:
\begin{enumerate}[(a)]
    \item The induced maps $G(k)\rightarrow\mathcal G(k)$ and $\mathrm H^1(k, G)\rightarrow\mathrm H^1(k, \mathcal G)$ are bijections.
    \item For an analogously defined $G_0\rightarrow\mathcal G_0$, any isomorphism $G_0\simeq G$ descends uniquely to an isomorphism $\mathcal G_0\simeq\mathcal G$. The resulting map $\mathrm{Isom}(G_0, G)\rightarrow\mathrm{Isom}(\mathcal G_0,\mathcal G)$ is a bijection.
    \item The restriction $\mathrm R(\pi)|_T : T\rightarrow\mathcal T$ to a maximal torus $T\subseteq G$ is an isogeny onto a maximal torus $\mathcal T\subseteq\mathcal G$. This defines a bijection between the maximal tori of $G$ and of $\mathcal G$.
\end{enumerate}
\end{prop}
\begin{prf}
The map $\mathrm R(\pi)|_G$ is a surjection by \cite[Prop.\ 7.3.1(c)]{CGP15}. To prove that $\mathcal G$ is semisimple, simply connected and absolutely simple, note that $\mathcal G_{k'}\hookrightarrow\mathrm R(\mathcal G')_{k'}\twoheadrightarrow\mathcal G'$ is an isomorphism (as $\mathcal G_{\overline k}\cong\smash{\big(\mathrm R(\mathcal G')_{\overline k}\big)}^{red}\cong\mathcal G'_{\overline k}$ by Proposition \ref{propweilrestrunip}, since $\mathcal G\subseteq\mathrm R(\mathcal G')$ is Levi and $\mathcal G'$ is reductive).
Now it suffices to see that the group $\mathcal G'$ is semisimple and absolutely simple, which follows from the isogeny $G'\twoheadrightarrow\mathcal G'$, and it is simply connected by \cite[Prop.\ 7.1.5]{CGP15} (this is possible, since the mentioned isogeny is never central).
Next, as the nontrivial extension $k'/k$ in the definition of basic exotic groups is contained in $k^{1/p}$, we have $k' = k^{1/p}$. The statements (a)-(c) are then shown in \cite[Props.\ 7.3.3(1), 7.3.5(1) and Cor.\ 7.3.4]{CGP15}, respectively.
\end{prf}

When $\mathrm{char}(k) = 2$ and $[k : k^2] = 2$, the failure of an arbitrary pseudo-reductive group to be generalized standard is explained by the existence of the following groups:
\vspace{-2pt}

\begin{defn}
Let $G$ be a pseudo-reductive group over a field $k$ of characteristic $p = 2$, with $[k : k^2] = 2$. Recall that the \textit{radical} $\mathscr R_k(G)$ of $G$ is the unique largest smooth connected solvable normal subgroup of $G$ defined over $k$. We say that $G$ is a \textit{basic non-reduced pseudo-simple group} if $G$ is absolutely pseudo-simple (that is, the quotient $G_{\overline k}/\mathscr R_{\overline k}(G_{\overline k})$ is simple; see \cite[Lem.\ 3.1.2]{CGP15}), it has a non-reduced root system and there exists a quadratic extension $K/k$ such that $\mathscr R_K(G_K)_{\overline k} = \mathscr R_{\overline k}(G_{\overline k})\subseteq G_{\overline k}$; see \cite[Def.\ 10.1.2]{CGP15}.

More generally, if there are a nonzero finite reduced $k$-algebra $k'=\prod_{i=1}^s k'_i$ and a $k'$-group $G'$ such that each fiber $G'_{k'_i}$ is a basic non-reduced pseudo-simple group, and if $G\simeq\mathrm R(G')$, then we say that $G$ is a \textit{totally non-reduced pseudo-reductive group} (\cite[Def.\ 10.1.1, Prop.\ 10.1.4]{CGP15}).
\end{defn}

\begin{prop}\label{propredbasnonred}
With the above terminology, the following properties hold:
\begin{enumerate}[(a)]
    \item If $G$ is totally non-reduced over $k$, then the pair $(k'/k,G')$ is determined uniquely up to unique isomorphism. Thus it is in particular functorial with respect to isomorphisms~of~$G$.\;
    \item If $G$ is basic non-reduced pseudo-simple over $k$, then it is determined up to isomorphism by the dimension $n$ of its maximal tori. In its definition, $K = k^{1/2}$ and $G^{ss}_K\coloneqq G_K/\mathscr R_K(G_K)$ is isomorphic to the simple group $\mathrm{Sp}_{2n,\,K}$. Moreover, the homomorphism ${G\rightarrow\mathrm R_{K/k}(G^{ss}_K)\eqqcolon\mathcal G}$ induces bijections $G(k)\cong\mathcal G(k)$, and also $\mathrm H^1(k,G) = 1$.
\end{enumerate}
\end{prop}
\begin{prf}
Here (a) is \cite[Prop.\ 10.1.4(2)]{CGP15}. For (b), see \cite[Thms.\ 2.3.6 and 2.3.8]{Con12}.
\end{prf}

\begin{thm}\label{thmpsredstrmain}
Let $G$ be a pseudo-reductive group over $k$, assume $[k : k^2]\leq 2$ if $\mathrm{char}(k) = 2$. Then $G$ can be written as a product $G_1\times G_2$ of pseudo-reductive groups, uniquely functorially in isomorphisms of $G$, such that $G_1$ is generalized standard and $G_2$ is totally non-reduced.
\end{thm}
\begin{prf}
This is \cite[Thm.\ 10.2.1]{CGP15}.
\end{prf}

\begin{cor}\label{corstrpsredetforms}
Let $G$ be a pseudo-reductive group over a local or global field $k$. Then there exists a surjection $\widetilde G\twoheadrightarrow G$ of pseudo-reductive groups over $k$ with smooth central kernel, where $\widetilde G$ is a (split) extension of a commutative affine algebraic group $C$ by a smooth~affine~algebraic group $D$ such that:
\begin{itemize}
    \item Any \'etale $k$-form $D'$ of $D$ satisfies $\mathrm{H}^1(k,D') = 1$
    \item Let $A$ be a commutative pseudo-reductive group over $k$. For any \'etale $k$-form $\widetilde G'$ of $\widetilde G\times A$, there exists a unique subgroup $D'$ of $\widetilde G'$ such that any isomorphism $\widetilde G'_{k_s}\simeq (\widetilde G\times A)_{k_s}$ of algebraic groups over $k_s$ restricts to an isomorphism $D'_{k_s}\simeq D_{k_s}$
\end{itemize}
\end{cor}
\begin{prf}
Write $G = G_1\times G_2$ as in the preceding theorem and consider the generalized standard presentation $G_1\cong (\mathrm R_{k'/k}(G')\rtimes C)/\mathrm R_{k'/k}(\mathrm Z_{G'}(T'))$ (with $G' = 0$ if $G_1$ is commutative). We~take $\widetilde G\coloneqq G_2\times (\mathrm R_{k'/k}(G')\rtimes C) = D\rtimes C$ for $D\coloneqq G_2\times\mathrm R_{k'/k}(G')$. Suppose given a commutative pseudo-reductive group $A$. By Theorem \ref{thmgenstandfunct}(b), the map $\mathrm R_{k'/k}(G')\rightarrow(\widetilde G/G_2)\times A$ is uniquely functorial in isomorphisms of $(\widetilde G/G_2)\times A$. As generalized standard presentations commute with base change by separable extensions of $k$, and the product $\widetilde G\times A = (A\times\widetilde G/G_2)\times G_2$ is uniquely functorial in isomorphisms of $\widetilde G\times A$, we conclude the second property in the statement.

For the first property, note that any \'etale $k$-form $D'$ of $D$ is of the form $G'_2\times\mathrm R_{k''/k}(G'')$ by Theorem \ref{thmpsredstrmain} and by Corollary \ref{corconradetform}, where $G'_2$ is totally non-reduced, $k''$ is a nonzero finite reduced $k$-algebra and the $k''$-group $G''$ is primitive (in the sense of Definition \ref{defgenstand}). By Propositions \ref{propredbasexpsred} and \ref{propredbasnonred} (and Shapiro's lemma: Proposition \ref{propshapiro}), we reduce to saying that $\mathrm H^1(k, \mathcal G) = 1$ holds for every semisimple simply connected group $\mathcal G$ over a finite extension $k''_i$ of $k$. This is well-known to hold over local and global fields (cf.\ \cite[Thm.\ 5.1.1(i)]{Con12}).
\end{prf}

%% file: partB-0.tex
\subsection{A Reminder on Weil Restrictions}\label{ssectweilrestr}
Fix two affine Noetherian schemes $S = \Spec(B)$, $S' = \Spec(B')$ and a map $S'\rightarrow S$ which is finite and faithfully flat. The faithfulness assumption is not strictly necessary, but it simplifies some of the following statements.

If $Y'$ is an (arbitrary) disjoint union of quasi-projective schemes over $S'$, then its pushforward as a sheaf to $S$ is representable by a scheme over $S$, the \textit{Weil restriction} $\mathrm R(Y')\coloneqq \mathrm R_{S'/S}(Y')$ (see \cite[\textsection7.6, Thm.\ 4]{BLR90}). In other words, there is a canonical identification
\begin{equation}\label{eqweilrdef}
\mathrm{Mor}_S(X, \mathrm R(Y')) = \mathrm{Mor}_{S'}(X_{S'}, Y')
\end{equation}
for $S$-schemes $X$. Many useful properties of Weil restrictions are developed in \cite[\textsection A.5]{CGP15} (where $S'/S$ is merely assumed to be finite flat) and we recall some of them here.

\begin{prop}\label{propweiladjmorph}
Let $X,X'$ be disjoint unions of quasi-projective schemes over, respectively, $B,B'$. The unit and counit maps $\iota_X : X\rightarrow\mathrm R(X_{B'})$ and $q_{X'} : \mathrm R(X')_{B'}\rightarrow X'$ are~given~by:
\begin{itemize}
    \item $\iota_{X,A} = X\!\left(A\xrightarrow{\;a\,\mapsto\, a\otimes 1\;} A\otimes_B B'\right) : X(A)\longrightarrow X(A\otimes_B B')$ for $B$-algebras $A$
    \item $q_{X',A'} = X'\!\left(A'\otimes_B B'\xrightarrow{\;a'\otimes c'\,\mapsto\, c'a'\;} A'\right) : X'(A'\otimes_B B')\longrightarrow X'(A')$ for $B'$-algebras $A'$
\end{itemize}
Moreover, $\iota_X$ is a monomorphism and, if $X$ is of finite type over $B$, a closed immersion.
\end{prop}
\begin{prf}
See (the proof of) \cite[Prop.\ A.5.7]{CGP15}.
\end{prf}

\begin{cor}\label{corweiladjmorph}
Let $H$ be a quasi-projective group scheme over $S'$. Then $\mathrm R(H')$ admits a canonical structure of a quasi-projective group scheme over $S$. Moreover, \eqref{eqweilrdef} restricts to an equality
$$\mathrm{Hom}_S(G, \mathrm R(H')) = \mathrm{Hom}_{S'}(G_{S'}, H')$$
for every group scheme $G$ over $S$.
\end{cor}
\begin{prf}
The group structure comes from the functoriality of $\mathrm R$, while the quasi-projectivity is \cite[Prop.\ A.5.8]{CGP15}. For the equality, it suffices to note that the unit and counit are both group homomorphisms.
\end{prf}

\begin{prop}\label{proplocalgquasiproj}
If $G$ is a locally algebraic group over a field $k$, then it is a disjoint union of quasi-projective schemes over $k$.
\end{prop}
\begin{prf}
Every component $X$ of $G$ becomes isomorphic over a finite field extension $k'/k$ to some finite disjoint union of copies of $G^0_{k'}$ (see the proof of \cite[VI$_{\mathrm{A}}$, Cor.\ 2.4.1]{SGA3}), so $X_{k'}$ is quasi-projective over $k'$ by \cite[Thm.\ 8.43]{Mil17}. Therefore $X\hookrightarrow\mathrm R_{k'/k}(X_{k'})$ is a closed immersion into a quasi-projective scheme over $k$ by Proposition \ref{propweiladjmorph} and \cite[Prop.\ A.5.8]{CGP15}.
\end{prf}

In particular, an algebraic group $G'$ over $k'$ is quasi-projective when $k'$ is a product of fields, hence Weil restrictions of (locally) algebraic groups along finite field extensions do exist.

\begin{prop}\label{propweilrestrunip}
Let $G'$ be a smooth group scheme of finite type over some nonzero finite reduced $k$-algebra $k'$. Write $k' = \prod k'_i$ for finitely many finite field extensions $k'_i/k$. If we take $S = \Spec(k)$ and $S' = \Spec(k')$, then $\mathrm R(G') = \prod\mathrm R_{k'_i/k}(G'_{k'_i})$.

The formation of the Weil restriction commutes with taking centers and Cartan subgroups (of $G'$ and $\mathrm R(G')$). The counit map $q : \mathrm R(G')_{k'}\rightarrow G'$ is smooth and surjective. Moreover, if the factor fields $k'_i$ are purely inseparable over $k$, then $\ker(q)$ has connected unipotent fibers.
\end{prop}
\begin{prf}
This is shown in \cite[Props.\ A.5.11 and A.5.15]{CGP15}.
\end{prf}

\begin{prop}\label{proppreshapiro}
Let $0\rightarrow H'\rightarrow G'\rightarrow Q'\rightarrow 0$ be a short exact sequence of quasi-projective group schemes over $S'$ (meaning in particular that $G'\rightarrow Q'$ is faithfully flat). The sequence
$$0\rightarrow\mathrm R(H')\rightarrow\mathrm R(G')\rightarrow\mathrm R(Q')$$
is exact, and the last map is faithfully flat in the following two situations:
\begin{itemize}
    \item $S'/S$ is (finite) \'etale
    \item $H',G',Q'$ are smooth and $S$ is a field
\end{itemize}
In general, if $G'$ is a smooth (locally finite-type) group scheme over $S'$, then so is $\mathrm R(G')$ over~$S$.
\end{prop}
\begin{prf}
Left-exactness is clear since pushforwards are right-adjoint to base change. Now suppose $S'/S$ is finite \'etale. It suffices to show faithful flatness \'etale-locally on $S$, so we may assume (by \cite[Lem.\ 04HN]{Stacks}) that $S'$ is simply a non-empty product of copies of $S$. Then the map $\mathrm R(G') = (G')^n\longrightarrow (Q')^n = \mathrm R(Q')$ is faithfully flat.

In the other case, by smoothness of $H'$, we have that the map $\mathrm R(G')\rightarrow\mathrm R(Q')$ is surjective by \cite[Cor.\ A.5.4(1)]{CGP15}. The groups $\mathrm R(G'),\mathrm R(Q')$ are smooth by \cite[Prop.\ A.5.2(4)]{CGP15} (and the same proof works for locally finite-type $G'$). Finally, a surjective homomorphism of smooth algebraic groups over a field is always faithfully flat (see \cite[Exmp.\ A.1.12]{CGP15}).
\end{prf}

The following proposition is a general formulation of the statement usually called Shapiro's lemma (cf.\ \cite[Lem.\ 4.1.6]{Con12}).

\begin{prop}\label{propshapiro}
Let $G'$ be a $S'$-group scheme which is a disjoint union of quasi-projectives. Consider $n\in\mathbf Z$ and suppose that $n\leq 1$ if $G'$ is not commutative. There is a functorial map
$$\mathrm{H}^n(S, \mathrm R(G'))\rightarrow\mathrm{H}^n(S', G')$$
of Abelian groups (resp. pointed sets if $G'$ is not commutative) which is an isomorphism if the group scheme $G'$ is smooth over $S'$ or if $S'/S$ is \'etale.
\end{prop}
\begin{prf}
Suppose first that $G'$ is commutative. The desired map then coincides with the maps appearing in the Leray spectral sequence for $\pi : S'\rightarrow S$:
$$\mathrm{H}^n(S, \pi_*G')\rightarrow\mathrm{H}^n(S', G')$$
If $\pi$ is \'etale, then $\pi_*$ is exact (for all abelian fppf sheaves) by the same argument as in the previous proposition, so higher direct images vanish. Otherwise, if $G'$ is smooth, then so is $\mathrm R(G')$ and we may pass to the corresponding statement in \'etale cohomology. However,~higher~direct images of finite pushforwards between \'etale sites again vanish (see \cite[Prop.\ 03QP]{Stacks}).

The proof is similar in the nonabelian case (we again pass to the \'etale sites $S_\et$ and $S'_\et$ when the group $G'$ is smooth), however we replace the arguments with higher direct images by the following general lemma.
\end{prf}

\begin{lem}
Consider a surjective map of schemes $\pi : S'\rightarrow S$. Suppose given a group sheaf $\mathcal G'$ on $S'_\et$ (resp.\ $S'_\fppf$). There is a natural map of \'etale (resp.\ fppf) cohomology sets
$$\mathrm{H}^1(S, \pi_*\mathcal G')\rightarrow\mathrm{H}^1(S', \mathcal G')$$
which is an isomorphism when the map $\pi$ is finite (resp.\ finite \'etale).
\end{lem}
\begin{prf}
Given a sheaf torsor $\mathcal P$ on $S$ representing some class $[\mathcal P]\in\mathrm{H}^1(S, \pi_*\mathcal G')$, we find for it a trivializing \'etale (resp.\ fppf) cover $\mathcal U = (U_i)$ of $S$ and a corresponding cocycle:
$$g_{ij}\in (\pi_*\mathcal G')(U_i\times_S U_j) = \mathcal G'(U_{i,S'}\times_{S'} U_{j,S'})$$
The gluing data $((U_{i,S'}),(g_{ij}))$ defines a sheaf torsor $\mathcal P'$ of $\mathcal G'$ on $S'$. This construction behaves well with respect to refinements and with respect to taking cohomologous cocycles, hence the class $[\mathcal P']\in\mathrm{H}^1(S', \mathcal G')$ is independent of choices. This defines the map in the statement.

We claim that, in our situation, an inverse map is given by taking a class $[\mathcal P']$ to the class $[\pi_*\mathcal P']$ represented by pushforwards. It suffices only to prove that $\pi_*\mathcal P'$ is locally nonempty on $S_\et$ (resp.\ $S_\fppf$). Indeed, after finding a cover $(U_i)$ of $S$ on which $\pi_*\mathcal P'$ has points, it will follow that the action of $\pi_*\mathcal G'$ on $\pi_*\mathcal P'$ defines a sheaf torsor. The two operations are mutually inverse up to isomorphism (since the descent data on $(U_i)$ and $(U_{i,S'})$ agree) and the inverse map on cohomology sets is thus in particular well defined. 

In the case of $S_\et$, local nonemptiness holds since $(\pi_*\mathcal P')_{\overline{s}} = \prod \mathcal P'_{\overline{s}'}$ for the finite map $\pi$ and any geometric point $\overline{s}\rightarrow S$ (\cite[Prop.\ 03QP]{Stacks}), where the finite product is taken over all geometric points $\overline{s}'\rightarrow S'$ lying over $\overline{s}$, at least one of which exists by surjectivity of $\pi$. In the case of $S_\fppf$, we use that $\pi$ is finite \'etale and thus (\cite[Lem.\ 04HN]{Stacks}) every point $s\in S$ has an \'etale neighborhood $V$ with $\pi^{-1}V\cong\coprod V'_i$ and each $V'_i\rightarrow V$ an isomorphism. If $\mathcal U'_i$ is an fppf cover of $V'_i$ such that $\mathcal P'(\mathcal U'_i)\neq\varnothing$, we may take an fppf cover $\mathcal U$ of $V$ which is a common refinement of each $\mathcal U'_i$ and then $(\pi_*\mathcal P')(\mathcal U)\neq\varnothing$.
\end{prf}

As opposed to smooth groups, Weil restrictions of non-smooth algebraic groups generally behave quite badly. For instance, we will now show that the Weil restriction of an infinitesimal group can have positive dimension. The following example will also become very important in Examples \ref{exmppsredh2notinj} and \ref{exmpabelmain}.

\begin{exmp}\label{exmpweilrestr}
Suppose $k'/k$ is a finite field extension with $k'\subseteq k^{1/p^n}$, so purely inseparable. Take the Kummer exact sequence $0\rightarrow\mu_{p^n}\rightarrow\mathbf G_{\mathrm m}\rightarrow\mathbf G_{\mathrm m}\rightarrow 0$. By the left-exactness of Weil restrictions, there is a sequence of commutative group sheaves for $\pi : \Spec(k')\rightarrow\Spec(k)$
\begin{center}\begin{tikzcd}
0\arrow[r] & \mathrm R(\mu_{p^n,\,k'})\arrow[r] & \mathrm R(\mathbf G_{\mathrm m,\,k'})\arrow[r] & \mathrm R(\mathbf G_{\mathrm m,\,k'})\arrow[r] & \mathrm R^1\pi_*(\mu_{p^n,\,k'})\arrow[r] & 0
\end{tikzcd}\end{center}
which is exact since $\mathrm R^1\pi_*G = 0$ for any smooth commutative algebraic group $G$ over $k'$. Indeed,
$$\mathrm H^1(A\otimes_k k',G)\cong\mathrm H^1_\et(A\otimes_k k',G)\cong\mathrm H^1_\et(A,\mathrm R(G))\cong\mathrm H^1(A,\mathrm R(G))$$
for any $k$-algebra $A$ by the above lemma (in which there is no Noetherianity assumption) and the fppf sheafification of the resulting presheaf $A\rightsquigarrow\mathrm H^1(A,\mathrm R(G))$ is trivial.

The operation of taking the largest multiplicative subgroup (resp. largest unipotent quotient) $G^m$ (resp.\ $G^u$) of a commutative affine algebraic group $G$ is an exact functor, as~can~be~checked over the algebraic closure $\overline{k}$ (over which the multiplicative-unipotent decomposition splits; see the beginning of Subsection \ref{ssectgenanis}). The quotient $\mathrm R(\mathbf G_{\mathrm m,\,k'})/\mathbf G_{\mathrm m}$ is unipotent by Proposition \ref{propweilrestrunip}, so $\mathrm R(\mathbf G_{\mathrm m,\,k'})^m = \mathbf G_{\mathrm m}$, which shows exactness of rows and columns in the commutative diagram
\begin{center}\begin{tikzcd}
0\arrow[r] & \mu_{p^n}\arrow[r]\arrow[d, hook] & \mathbf G_{\mathrm m}\arrow[r, "x\,\mapsto\, x^{p^n}"]\arrow[d, hook] & \mathbf G_{\mathrm m}\arrow[d, hook]\arrow[r] & 0\\
0\arrow[r] & \mathrm R(\mu_{p^n,\,k'})\arrow[r]\arrow[d, two heads] & \mathrm R(\mathbf G_{\mathrm m,\,k'})\arrow[ur, dashed, two heads]\arrow[r, "x\,\mapsto\, x^{p^n}"]\arrow[d, two heads] & \mathrm R(\mathbf G_{\mathrm m,\,k'})\arrow[d, two heads]\arrow[r] & \mathrm R^1\pi_*(\mu_{p^n,\,k'})\arrow[d, equal]\arrow[r] & 0\\
0\arrow[r] & U_{\mu_{p^n}}\arrow[r, "\sim"] & U_{\mathbf G_{\mathrm m}}\arrow[r, "0"] & U_{\mathbf G_{\mathrm m}}\arrow[r, "\sim"] & \mathrm R^1\pi_*(\mu_{p^n,\,k'})\arrow[r] & 0
\end{tikzcd}\end{center}
in which each column is the multiplicative-unipotent decomposition of the middle group. The map $\mathrm R(\mathbf G_{\mathrm m,\,k'})\rightarrow \mathrm R(\mathbf G_{\mathrm m,\,k'})$ in the middle is, for $k$-algebras $A$, given by
$$
(A\otimes_k k')^\times\xrightarrow{\;x\,\mapsto\,x^{p^n}\;}(A\otimes_k k)^\times\subseteq(A\otimes_k k')^\times$$
and thus factors through $\mathbf G_{\mathrm m}$, as shown in the diagram (if $k' = k^{1/p^n}$, then the dashed arrow is the ``norm map'' $\mathrm N_{k'/k}$). It follows that all $4$ groups in the bottom row are abstractly isomorphic, and in particular smooth connected unipotent of dimension $[k':k]-1$, as is true for $U_{\mathbf G_{\mathrm m}}$.

By Shapiro's lemma and Hilbert's theorem 90, we get from one of the middle columns the following exact sequence in cohomology (note that $\mathrm{H}^2(k, U_{\mathbf G_{\mathrm m}}) = 0$ since $U_{\mathbf G_{\mathrm m}}$ is unipotent):
$$0\longrightarrow\mathrm{H}^1(k, U_{\mathbf G_{\mathrm m}})\longrightarrow\mathrm{Br}(k)\xrightarrow{\;\mathrm{Br}(k\,\hookrightarrow\, k')\;}\mathrm{Br}(k')\longrightarrow 0$$
If, for example, $k$ and $k'$ are local fields, then $\ker(\mathrm{Br}(k\hookrightarrow k'))\simeq\mathbf Z/p^n$ by \cite[Thm.\ 8.9]{Har20}. Then $\mathrm{H}^1(k, U_{\mu_{p^n}})\simeq\mathbf Z/p^n$ and $\mathrm{H}^1(k, \mathrm R^1\pi_*(\mu_{p^n,\,k'}))\simeq\mathbf Z/p^n$, so we may also conclude that these unipotent groups are not split.
\end{exmp}

%% file: partB-1.tex
\subsection{Vanishing Theorems in Cohomology}\label{ssectvanthm}
We begin with a statement which, in the case of a smooth group $G$, is a simple application of Shapiro's lemma discussed above, and in general requires more work:

\begin{lem}\label{lemskewvanish}
Let $k$ be a field and consider the inclusion $\overline k\hookrightarrow R\coloneqq\overline{k}\otimes_k\overline{k}$ via the first factor. 
If $G$ is a locally algebraic group over $\overline k$, then $\mathrm H^1(R,G) = 1$.
\end{lem}

\begin{rem}
We note that the choice of algebraic closure $\overline{k}$ induces a bijection
$$\Gamma = \Gal(k_s/k) = \mathrm{Aut}(\overline{k}/k)\longrightarrow \Spec(\overline{k}\otimes_k\overline{k})$$
defined by taking $c : \overline{k}\rightarrow\overline{k}$ to the kernel ideal of the $k$-algebra map $u_c = (\mathrm{id},c) : \overline{k}\otimes_k\overline{k}\rightarrow\overline{k}$, because each $k$-algebra map from $\overline{k}\otimes_k\overline{k}$ to a field $K$ factors, up to isomorphism of $K$, through a unique map of the form $u_c$. Moreover, we claim that this bijection is in fact a homeomorphism between the Krull and Zariski topologies: Indeed, for a finite separable extension $K/k$ embedded into $k_s$, a clopen coset $c\cdot\Gal(k_s/K)\subseteq\Gamma$ maps onto the clopen subscheme given by the quotient $1\otimes c : \overline k\otimes_k\overline k\twoheadrightarrow \overline k\otimes_K\overline k$, which determines a correspondence between the bases of topology. Finally, each local ring of $\Spec(\overline{k}\otimes_k\overline{k})$ is of the form $\overline k\otimes_{k_s}\overline k$.   
\end{rem}

\begin{prf}[of Lemma \ref{lemskewvanish}]
Suppose that the lemma is known when $k = k_s$. For a general field $k$, given a torsor $T$ of $G$ on $\Spec(R)$ this implies, by the remark above, that $T$ trivializes on each local ring of $R$, hence on some open cover (by a limiting argument; see \cite[Thm.\ 2.1]{Mar07} for the case of torsors of nonabelian group schemes locally of finite presentation).~Any~open~cover of $\Spec(R)$ is refined by a disjoint open cover since $\Spec(R)$ is compact and totally disconnected, hence $T$ trivializes globally.

We may thus suppose that $k = k_s$. When $k$ is perfect, the proof is done (as $\overline{k}\otimes_{\overline k} \overline{k} = \overline{k}$), so we let $k$ be imperfect with $\mathrm{char}(k) = p > 0$. Then $R = \overline k\otimes_k\overline k$ is a (non-Noetherian) local ring with residue field $R/\mathfrak m = \overline k$. Each element of $\mathfrak m$ is nilpotent and thus $R_{\mathrm{red}} = \overline k$. Furthermore, the Frobenius endomorphism is surjective on $R$, as can be checked on elements of the form $a\otimes b\in R$.
By \cite[VII$_\mathrm{A}$, Prop.\ 8.3]{SGA3}, there is an infinitesimal normal subgroup scheme $I$ of $G$ such that $G/I$ is smooth. We may thus separately consider the cases of $G$ smooth and $G$ infinitesimal: If $G$ is smooth, then $\mathrm H^1(R,G) = 1$ can be checked in the \'etale topology and it holds because~the small \'etale site of $R$ is trivial, because the same is true for $R_{\mathrm{red}}$. (Alternatively, one may argue without the above remark that $R = \varinjlim(K\otimes_k\overline k)$ over finite extensions $K/k$ and use that
$$\mathrm H^1(K\otimes_k\overline{k},\,G^K) =\mathrm H^1\big(\overline{k},\,\mathrm R_{(K\otimes_k\overline{k})/\overline{k}}(G^K)\big) = 1$$
by Shapiro's lemma, where $\{G^K\}_K$ is a system of $(K\otimes_k\overline{k})$-forms of $G$ for each $K$.)

It remains to consider the case where $G$ is an infinitesimal $\overline k$-group scheme. We may further assume that $G$ is of height $1$ via filtering it canonically by such groups (this step is not strictly necessary, but it simplifies further notation); then $G\simeq\Spec(A)$ as a $\overline k$-scheme for
$$A =\overline k[a_1,\ldots,a_r]/(a_1^p,\ldots,a_r^p)$$
by \cite[Prop.\ 11.28]{Mil17}.
Let $T$ be a $G$-torsor over $R$. Then $T_{\overline k}\simeq G$ as a $\overline k$-scheme, because $\mathrm H^1(\overline k,G) = 1$. For some finite $R$-algebra $B$ such that $T\simeq\Spec(B)$, it follows that $B/\mathfrak mB\simeq A$. 
We choose elements $c_i\in B$ mapping to $a_i\in A$ under this isomorphism.

Next, we claim that $B^p\subseteq R\subseteq B$. This property is clearly stable under arbitrary base change of $R$, and we first observe that it can be checked fppf-locally: Given a faithfully flat $R$-algebra $S$ and any $b\in B$ such that $(b\otimes 1)^p = 1\otimes s$ in $B\otimes S$ (which holds if $(B\otimes S)^p\subseteq S$), we have:
$$(b^p\otimes 1-1\otimes b^p)\otimes 1 = 0\quad\textrm{ in }(B\otimes B)\otimes S$$
Since $B\otimes B\rightarrow(B\otimes B)\otimes S$ is injective, we conclude $b^p\otimes 1 = 1\otimes b^p$ and, because $B$ is also faithfully flat over $R$, we conclude $b^p\in R$. 
Now, to prove the claim, it suffices to show that $(A_R)^p\subseteq R\subseteq A_R$, which is obvious since $A^p\subseteq\overline k\subseteq A$. We conclude that $c_i^p\in R$. However, $c_i^p$ is nilpotent (since we know that $c_i^p\in\mathfrak mB$) and therefore $c_i^p\in\mathfrak m$.

We may thus find $m_i\in\mathfrak m$ with $m_i^p = c_i^p$ and write $b_i\coloneqq c_i-m_i\in B$. The collection $\{b^J\}$, for $J\in\{0,\ldots,p-1\}^r$ and $b^J = b_1^{J_1}\ldots b_r^{J_r}$, generates $B$ as an $R$-module by Nakayama's lemma~since each $b^J$ maps to $a^J = a_1^{J_1}\ldots a_r^{J_r}$. As $b_i^p = 0$ for all $i$, there is a surjective $R$-algebra map
$$f : R[x_1,\ldots,x_r]/(x_1^p,\ldots,x_r^p)\longrightarrow B$$
defined by taking $x_i$ to $b_i$. To end the proof, we need to show that $R\hookrightarrow B$ admits a retraction; for this it suffices to show that $f$ is an isomorphism. This can again be checked fppf-locally~on~$R$, but $B$ is fppf-locally on $R$ a free module of rank $p^r$ (because $A_R$ is free). Thus $f$ is fppf-locally a surjective linear map between free modules of rank $p^r$, hence an isomorphism.
\end{prf}

The following lemma is simple and well-known in characteristic $0$ (where it holds for $n\geq 1$ since all unipotent groups are split). The analogous statement in the noncommutative case, for bands represented by unipotent groups, will follow (at least in the smooth connected case over a local or global function field) from Theorem \ref{thmmainabelh2}; see also Lemma \ref{lemuniph1surj}.

\begin{lem}\label{propunipfppfvanish}
Let $U$ be a commutative unipotent algebraic group over an arbitrary field $k$. Then $\mathrm H^n(k,U) = 0$ for $n\geq 2$.
\end{lem}
\begin{prf}
Every unipotent group admits a filtration by subgroups of $\mathbf G_{\mathrm a}$ and we may thus assume $U\subseteq\mathbf G_{\mathrm a}$. Then either $U = \mathbf G_{\mathrm a}$ or $\mathbf G_{\mathrm a}/U\simeq \mathbf G_{\mathrm a}$ by \cite[IV, \textsection 2, 1.1]{DG70}. We now use the fact that $\mathrm H^m(k, \mathbf G_{\mathrm a}) = 0$ for $m\geq 1$.
\end{prf}

If $\mathrm{char}(k) = 0$, then all algebraic groups are smooth so the above statement can be formulated in terms of Galois cohomology. In characteristic $p > 0$, Galois (i.e.\ \'etale) cohomology can still be useful, but must be considered separately from fppf cohomology. Let $k$ be an arbitrary field with $\mathrm{char}(k) = p$ and absolute Galois group $\Gamma = \Gal(k_s/k)$. The simple vanishing statements we prove below will be used in the main part of the paper:

Recall that a \textit{$p$-group} is an (abstract, not necessarily finite or commutative) group $A$ such that the order of each element $x\in A$ is a power of $p$. The $p$-cohomological dimension of $k$ is $\mathrm{cd}_p(k)\leq 1$ by \cite[II, Prop.\ 3]{Ser97}, which means that $\mathrm{H}^n(\Gamma, A) = 0$ for each $n\geq 2$ and any (commutative) $\Gamma$-module $A$ which is a $p$-group.

\begin{prop}\label{propunipvanish}
Let $A$ be a $\Gamma$-module and let $G$ be a unipotent algebraic group over $k$. If $A$ is a subquotient of $G(k_s)$, then $\mathrm{H}^n(\Gamma, A) = 0$ for $n\geq 2$.
\end{prop}
\begin{prf}
It suffices to show that $G(k_s)$ is a $p$-group, which follows from the fact that $G$ admits a filtration by subgroups of the $p$-torsion group $\mathbf G_{\mathrm a}$.
\end{prf}

Let $\Pi : k_\mathrm{fppf}\rightarrow k_\Et$ and $\pi : k_\mathrm{fppf}\rightarrow k_\et$ be the natural continuous maps of sites. For a locally algebraic commutative group $G$ over $k$, the sheaf $\mathrm{R}^n\pi_*(G)$ (which is a restriction of $\mathrm{R}^n\Pi_*(G)$) corresponds to the $\Gamma$-module $\mathrm{H}^n(k_s,G)$. In particular, we may write:
$$\mathrm H^m_\Et(k,\mathrm{R}^n\Pi_*(G)) = \mathrm H^m_\et(k,\mathrm{R}^n\pi_*(G)) = \mathrm{H}^m(\Gamma,\mathrm{H}^n(k_s,G))$$
When $G$ is smooth, then $\mathrm{H}^n(k_s,G) = 0$ for $n\geq 1$.

\begin{prop}\label{proph2h1coh}
Let $G$ be a commutative locally algebraic group over $k$. Then there exists a smooth commutative locally algebraic group $H$ over $k$ and a monomorphism $i : G\hookrightarrow H$. If $G$ is algebraic, we may choose $H$ algebraic and then $i$ is a closed immersion.

As a consequence, if $G$ is algebraic, then $\mathrm{H}^n(k_s,G) = 0$ for all $n\geq 2$. Moreover, the group $\mathrm{H}^1(k_s, G)$ is a $p$-group, and thus $\mathrm{H}^n(\Gamma, \mathrm{H}^1(k_s,G)) = 0$ for $n\geq 2$.
\end{prop}
\begin{prf}
A monomorphism from an algebraic group to a locally algebraic group is the immersion of a closed subgroup, and the quotient space is a scheme (a locally algebraic group when the subgroup is normal) by \cite[VI$_{\textrm{A}}$, Prop.\ 2.5.2 and Thm.\ 3.3.2]{SGA3}. By \cite[VII$_{\textrm{A}}$, Prop.\ 8.3]{SGA3}, there is an exact sequence $0\rightarrow F\rightarrow G\rightarrow Q\rightarrow 0$, where $F$ is infinitesimal and $Q$ is a smooth locally algebraic group scheme. By \cite[III, Thm.\ A.6]{Mil06}, there is an Abelian variety $A$ which admits a subgroup isomorphic to $F$. Then $G$ has a monomorphism into the locally algebraic quotient $H\coloneqq (G\times A)/\mathrm{diag}(F)$ which is smooth over $Q$ and thus over $k$.

For the second part, note that $\mathrm{H}^1(k_s, F)$ surjects onto $\mathrm{H}^1(k_s, G)$ because $\mathrm{H}^1(k_s, Q) = 0$, but $F$ is killed by some $p^m$, and so is the cohomology group $\mathrm{H}^1(k_s, F)$.
\end{prf}

\begin{prop}\label{propquotsumpts}
Suppose that $H,N$ are normal subgroups of an algebraic group $G$ over $k$, such that $G = H\cdot N$. Suppose that both $H\cap N$ and $G/N$ are commutative and consider the Abelian group:
$$A\coloneqq\dfrac{G(k_s)}{H(k_s)\cdot N(k_s)}$$
Then $A$ is a $p$-group and $\mathrm{H}^n(\Gamma, A) = 0$ for $n\geq 2$. Moreover, if $H\cap N$ is smooth, then $A = 0$.
\end{prop}
\begin{prf}
There is an injection $G(k_s)/N(k_s)\hookrightarrow (G/N)(k_s)$ which shows that the first group is commutative. Similarly for $H(k_s)/(H\cap N)(k_s)$, since $H/(H\cap N)\cong G/N$. We now consider the short exact sequence
\begin{center}\begin{tikzcd}
1\arrow[r] & H\cap N\arrow[r] & H\arrow[r] & \dfrac{H}{H\cap N}\arrow[r] & 1\end{tikzcd}\end{center}
associated to the Abelian normal subgroup $H\cap N$ of $H$. It gives an exact sequence of groups which makes up the top row of the following commutative diagram:
\begin{center}\begin{tikzcd}
1\arrow[r] & \dfrac{H(k_s)}{H(k_s)\cap N(k_s)}\arrow[d, "\alpha", pos = 0.4]\arrow[r] & \dfrac{H}{H\cap N}(k_s)\arrow[d, "{\rotatebox{90}{$\sim$}}", pos = 0.4]\arrow[r] & \mathrm H^1(k_s, H\cap N)\\
1\arrow[r] & \dfrac{G(k_s)}{N(k_s)}\arrow[r] & \dfrac{G}{N}(k_s)
\end{tikzcd}\end{center}
Then $A\simeq\coker\alpha$ is isomorphic to a subgroup of $\mathrm{H}^1(k_s, H\cap N)$ by the snake lemma. Now the previous proposition implies the claim.
\end{prf}

For completion, we also discuss bands on $k_\et$ (equivalently, Galois bands) represented locally by nilpotent $p$-groups, which generalize the case of commutative $p$-groups considered so far. By~a \textit{Galois band} $(A,\kappa)$, we mean the Galois-theoretic definition of a band (called a ``kernel'' in \cite{Spr66}) as a group $A$ equipped with a homomorphism
$$\kappa \;:\; \Gamma\longrightarrow\dfrac{\aut(A)}{A/\mathrm Z(A)}$$
admitting a continuous lift in a sense akin to Definition \ref{defcont} (this continuity condition is weaker than the one originally present in \cite{Spr66}). The $\mathrm H^2$ set of a Galois band is the set of equivalence classes of cocycles $(f,g)$, defined as usual (Remark \ref{remetbandh2def}).

\begin{prop}\label{propliftnilp}
Let $L = (A, \kappa)$ be a Galois band on $k$. If $A$ is a nilpotent $p$-group, then $\mathrm{H}^2(\Gamma, L) = \mathrm{N}^2(\Gamma, L) = 1$. 

This holds in particular if $A$ is a subquotient of $\Ub(k_s)$ for a unipotent group $\Ub$ over $k_s$.
\end{prop}
\begin{prf}
The last statement follows because $\Ub(k_s)$ is, for some $N$, a subgroup of the full group $\mathbf U_N(k_s)$ of unipotent $N\times N$ matrices, a $p$-group which is nilpotent by \cite[\textsection3.2, Exmp.\ 3]{Ser16}; and then because nilpotent $p$-groups are closed under taking subgroups and quotients. For the first part, we note that the commutative group $\mathrm Z(A)$ is a $p$-group, hence $\mathrm{H}^2(\Gamma, \mathrm Z(A)) = 0$. Thus $\mathrm{H}^2(\Gamma, L)$ has at most one element. It now suffices to prove that $L$ is representable.

Since $A$ is nilpotent, the center $\mathrm Z(A)$ is nontrivial whenever $A$ is (\cite[Ch.\ 3, Cor.\ 3.10]{Ser16}). By induction on the minimal length of a descending central series of $A$, we may assume that the theorem is proven for the induced \'etale band $\overline{L} = (A/\mathrm Z(A), \overline{\kappa})$. 
If we take a cocycle $(f,g)$ of $L$, then it descends to a cocycle $(\overline f,\overline g)$ which represents the unique (and neutral) class of $\overline{L}$. The cocycle $(\overline f,\overline g)$ is thus equivalent to a cocycle of the form $(\overline f',1)$ by a continuous assignment $\overline h : \Gamma\rightarrow A/\mathrm Z(A)$, that is:
$$\overline f'_s = \mathrm{int}(\overline h_s)\circ\overline f_s
\;\;\;\textrm{ and }\;\;\;
1 = \overline h_s\cdot\overline f_s(\overline h_t)\cdot\overline g_{s,t}\cdot\overline h_{st}^{-1}$$
We may lift $\overline h$ to a continuous assignment $h : \Gamma\rightarrow A$ and apply it to find a cocycle $(f',g')$ equivalent to $(f,g)$, such that $f'$ lifts $\overline f'$ and that $g' : \Gamma\times\Gamma\rightarrow A$ lands into $\mathrm Z(A)$. Because the group $\mathrm Z(A)$ is commutative (and $f'_sf'_t(f'_{st})^{-1} = \mathrm{int}(g'_{s,t})$), the lift $f'$ defines an action of the group $\Gamma$ on it. Since $\mathrm Z(A)$ is a $p$-group, the class $[g']\in\mathrm H^2(\Gamma,\mathrm Z(A))$ is trivial. Thus there exists $h' : \Gamma\rightarrow\mathrm Z(A)$ by which $(f',g')$ is equivalent to $(f',1)$.
\end{prf}


%% file: partB-2.tex
\subsection{Semisimple Bands over Global and Local Fields}\label{ssectdouai}
Suppose given a global or local field $k$ and a band $L$ on its \'etale or fppf site. It is a classical fact (see \cite[VIII]{Dou76}, where this is shown also for the fpqc topology) that, if $L$ is locally represented by a semisimple~group, then all elements in $\mathrm{H}^2(k, L)$ are neutral. Because any reductive band $L$ is always globally representable, it is equivalent to say that, given a semisimple group $G$ over $k$, both of the maps 
$$\mathrm{H}^1(\Gamma,G(k_s)/\mathrm Z_G(k_s))\rightarrow\mathrm{H}^2(\Gamma,\mathrm Z_G(k_s))
\;\;\textrm{ and }\;\;
\mathrm{H}^1(k,G/\mathrm Z_G)\rightarrow\mathrm{H}^2(k,\mathrm Z_G)$$
are surjective.
Although these surjectivity results may be well-known in number theory, we will still briefly review their proofs in this subsection. One reason for this is that some intermediate steps appearing in it will be used (and generalized) in Subsection \ref{ssectgenanis} and another is that all of the aforementioned steps are, in our primary source \cite[\textsection 6.5]{PR94}, formulated only over number fields. While most of their proofs are straightforward to translate into positive characteristic, we would like to provide characteristic-free references (in Lemmas \ref{lemtorapprox}, \ref{lemcohdirsum}) and also to present a more elementary treatment which avoids some duality results implicit in \cite{PR94} (see below). For simplicity, we state all our results only in positive characteristic, resp.\ over nonarchimedean local fields. A slightly different proof of the global case, independent of characteristic, is given in \cite[VIII]{Dou76}. See also \cite[Prop.\ 4.5]{RosTN} and references therein.

Given a torus $T$ over $k$, we write $\mathrm X(T)\coloneqq\mathrm{Hom}(T_{k_s},\mathbf G_{\mathrm m,\,k_s})$ for its module of characters. For a fixed finite Galois extension $K/k$, this construction defines a coequivalence between the category of tori over $k$ split by $K/k$ and the category of finite free $\mathbf Z$-modules equipped with an action of $\Gamma\coloneqq\Gal(K/k)$. Note that the Cartier dual $\widehat{T}$ is an \'etale sheaf isomorphic over $K$ to the constant sheaf $\underline{\smash{\mathrm X(T)}}$.

\begin{prop}\label{propequivcondanistor}
Suppose given a torus $T$ over $k$. Then $\mathrm{Hom}(T,\mathbf G_{\mathrm m}) = 0$ holds if and only if $\mathrm{Hom}(\mathbf G_{\mathrm m},T) = 0$.
\end{prop}
\begin{prf}
In the notation of the previous paragraph, we have equalities $\mathrm X(T)^\Gamma = \mathrm{Hom}(T,\mathbf G_{\mathrm m})$ and $(\mathrm X(T)^0)^\Gamma = \mathrm{Hom}(\mathrm X(T),\mathbf Z)^\Gamma = \mathrm{Hom}(\mathbf G_{\mathrm m},T)$, since $\mathrm X(\mathbf G_{\mathrm m}) = \mathbf Z$ (with trivial $\Gamma$-action). Here we denote by $A^0\coloneqq\mathrm{Hom}(A,\mathbf Z)$ the dual of a $\Gamma$-module $A$ which is finite free over $\mathbf Z$. There is a canonical isomorphism $A\xrightarrow{\sim}(A^0)^0$ of $\Gamma$-modules, and therefore the proposition statement reduces to showing the general statement that $(A^0)^\Gamma\neq 0$ implies $A^\Gamma\neq 0$.

For this, suppose that there exists a nonzero functional $p : A\rightarrow\mathbf Z$ fixed by $\Gamma$. If $x\in A$ satisfies $p(x)\neq 0$, we define $\mathrm N(x)\coloneqq \sum_{s\in\Gamma}s.x\in A^\Gamma$. Then $p(\mathrm N(x)) = [K:k]\cdot p(x)\neq 0$, so in particular $\mathrm N(x)\neq 0$.
\end{prf}

A torus is called \textit{anisotropic} if it satisfies the properties in the above proposition. As an application of local Tate duality (for example, see \cite[Thm.\ 1.2.1]{RosTD}), $\mathrm{H}^2(k, T) = 0$ when $k$ is a local field and $T$ is anisotropic over $k$.

Our interest in this notion comes from the following two lemmas (which can be tracked to sources written by Kneser). The first one can be derived from the proof of \cite[Prop.\ 6.12]{PR94} using the abstract duality in Tate cohomology (see \cite[top of p.\ 303]{PR94}), however we offer an elementary proof:

\begin{lem}\label{lemanistor}
Let $k$ be a global field and $T$ a torus over $k$. If $T_{k_v}$ is anisotropic for some place $v$ of $k$, then $\mathrm{H}^1(k, \widehat{T})\rightarrow\mathrm{H}^1(k_v, \widehat{T})$ is injective. In particular, then $\Sh^2(T) = \Sh^1(\widehat{T})^* = 0$.
\end{lem}
\begin{prf}
The last statement follows from global Tate duality (for example, \cite[Thm.\ 1.2.10]{RosTD}). For the rest, take a finite Galois extension $K/k$ splitting $T$ and let $\Gamma\coloneqq\Gal(K/k)$. Fix some place $w\mid v$ of $K$ and consider the decomposition group $\Gamma_v\coloneqq\Gal(K_w/k_v)\hookrightarrow\Gamma$. We need to show that the restriction map
$$\mathrm{H}^1(k, \widehat{T}) = \mathrm{H}^1(\Gamma, \mathrm X(T))\longrightarrow\mathrm{H}^1(\Gamma_v, \mathrm X(T)) = \mathrm{H}^1(k_v, \widehat{T})$$
is injective. Here the equalities hold because $\mathrm{H}^1(\Gal(k_s/K), \mathbf Z^n) = \mathrm{Hom}_{cont}(\Gal(k_s/K), \mathbf Z^n) = 0$ and similarly for $K_w$.

Take a crossed homomorphism $a : \Gamma\rightarrow X(T)$ representing a class $[a]\in\mathrm{H}^1(\Gamma, \mathrm X(T))$. Suppose that $a|_{\Gamma_v}$ is a coboundary: that is, there exists $m\in\mathrm X(T)$ such that $(a - \mathrm dm)|_{\Gamma_v} = 0$, where $\mathrm dm(s) = s.m-m$. Since $\mathrm{H}^1(\Gamma, \mathrm X(T))$ is $N$-torsion for $N = [K:k]$, we have that also $N\cdot a = \mathrm dm_0$ for some $m_0\in\mathrm X(T)$. Then the following equality holds
$$t.m_0-m_0 = N\cdot a(t) = t.(Nm)-(Nm)$$
for all $t\in\Gamma_v$. This implies that $m_0-Nm$ lies in $\mathrm X(T)^{\Gamma_v}$, and hence $m_0 = Nm$ because $T_{k_v}$ is anisotropic. We conclude that $N\cdot a = N\cdot \mathrm dm$, and finally $a = \mathrm dm$ since $\mathrm X(T)$ is torsion-free. Thus $[a] = 0$.
\end{prf}

\begin{lem}\label{lemlocexanis}
Let $k$ be a (nonarchimedean) local field and $G$ a semisimple group over $k$. Then $G$ admits a maximal torus $T$ which is anisotropic over $k$.
\end{lem}
\begin{prf}
This is \cite[Thm.\ 6.21]{PR94}. The proof works in any characteristic.
\end{prf}

The proof of the main statement can easily be done for local fields, but there is an additional lemma (\cite[\textsection 7.2, Cor.\ 3]{PR94}) necessary for the global field case, which will allows us to choose a maximal torus with the necessary properties:

\begin{lem}\label{lemtorapprox}
Let $G$ be a reductive group over a global field $k$. Let $S$ be a finite set of places of $k$ and choose, for every $v\in S$, a maximal torus $T_v$ of $G_{k_v}$. Then $G$ admits a maximal torus $T$ such that $T_{k_v}$ is $G(k_v)$-conjugate to $T_v$, for all $v\in S$. In particular, the tori $T_{k_v}$ and $T_v$~are then abstractly isomorphic for any $v\in S$.
\end{lem}
\begin{prf}
Let $T_0$ be a fixed maximal torus of $G$. Then $X\coloneqq G/\mathrm N_G(T_0)$ is the ``variety of tori'' associated with $G$: That is, for any $k$-algebra $R$, the set $X(R)$ is canonically identified with the set of maximal tori in $G_R$ (which follows since any two maximal tori are rationally conjugated over $k_s$). It is a smooth connected $k$-variety which is known to be rational (that is, birational to some $\mathbf A^N$) by \cite[XII, Cor.\ 1.10]{SGA3}.

In particular, it follows that $X$ satisfies weak approximation with respect to $S$. That is, the image of the map $X(k)\rightarrow X(k_s) = \prod_{v\in S}X(k_v)$ is dense in the product topology.
On the other hand, the orbits $G(k_v).[T_v]$ are open in $X(k_v)$ since the smooth surjective map $G_{k_v}\rightarrow X_{k_v}$ given by the action on $[T_v]$ induces an open map on $k_v$-points (\cite[Prop.\ 4.3]{CesTC}). The intersection $X(k)\cap\prod_{v\in S}(G(k_v).[T_v])$ is thus nonempty, which gives the desired torus $T$.
\end{prf}

\begin{lem}\label{lemcohdirsum}
Let $G$ be an algebraic group over a global field $k$. For $i\in\mathbf Z$, suppose that one of the following two cases is satisfied:
\begin{itemize}
    \item $i = 1$ and $G$ is smooth and connected
    \item $i\geq 2$ and $G$ is commutative
\end{itemize}
Then, given some class $x\in\mathrm{H}^i(k, G)$ (resp. $x\in\mathrm{H}^i(\Gamma, G(k_s))$), its local images $x_v\in\mathrm{H}^i(k_v, G)$ (resp. $x_v\in\mathrm{H}^i(\Gamma_v, G(k_{v,s}))$) are trivial for almost all $v$.
\end{lem}
\begin{prf}
The fppf case is \cite[Thm.\ 2.18]{CesPT}. For the \'etale case, recall \cite[Lem.\ C.4.1]{CGP15} by which there exists a unique smooth subgroup $G'$ of $G$ such that $G'(K) = G(K)$ for all separable field extensions $K/k$. Then $\mathrm{H}^i(\Gamma, G(k_s)) = \mathrm{H}^i(\Gamma, G'(k_s)) = \mathrm{H}^i(k, G')$ (and the same for $k_v$, since $k_v$ and $k_{v,s}$ are separable over $k$), which reduces the proof to the fppf case of $G'$.
\end{prf}

As announced, the following theorem encompasses both the \'etale and fppf version of the~same statement about semisimple bands:

\begin{thm}\label{thmsurjh1h2ss}
Let $k$ be a global or local field (of positive characteristic) and $G$ a semisimple group over $k$. The maps $\mathrm{H}^1(\Gamma,G(k_s)/\mathrm Z_G(k_s))\rightarrow\mathrm{H}^2(\Gamma,\mathrm Z_G(k_s))$ and $\mathrm{H}^1(k,G/\mathrm Z_G)\rightarrow\mathrm{H}^2(k,\mathrm Z_G)$ are both surjective.
\end{thm}
\begin{prf}
We will only prove this statement for the first map (as it is used later). The proof for the second map is identical. Finally, note that $\mathrm{H}^2(k,T) = \mathrm{H}^2(\Gamma,T(k_s))$ for $T$ smooth (in particular, for a torus).

\textit{The local case:} Take an anisotropic maximal torus $T$ in $G$, which exists by Lemma~\ref{lemlocexanis}. The composition $\mathrm{H}^1(\Gamma,T(k_s)/\mathrm Z_G(k_s))\rightarrow\mathrm{H}^1(\Gamma,G(k_s)/\mathrm Z_G(k_s))\rightarrow\mathrm{H}^2(\Gamma,\mathrm Z_G(k_s))$ is surjective by the exact sequence
$$\mathrm{H}^1(\Gamma,T(k_s)/\mathrm Z_G(k_s))\longrightarrow\mathrm{H}^2(\Gamma,\mathrm Z_G(k_s))\longrightarrow\mathrm{H}^2(\Gamma,T(k_s)) = \mathrm{H}^2(k,T) = 0$$
so in particular we conclude that the map $\mathrm{H}^1(\Gamma,G(k_s)/\mathrm Z_G(k_s))\rightarrow\mathrm{H}^2(\Gamma,\mathrm Z_G(k_s))$ is surjective.

\textit{The global case:} Fix an element $x\in\mathrm{H}^2(\Gamma,\mathrm Z_G(k_s))$. Its local images $x_v\in\mathrm{H}^2(\Gamma_v,\mathrm Z_G(k_{v,s}))$ are $0$ for almost all $v$ by Lemma \ref{lemcohdirsum}. Take a nonempty finite set $S$ of places including all $v$ such that $x_v\neq 0$, then fix an anisotropic maximal torus $T_v$ in $G_{k_v}$ for each $v\in S$. Now Lemma \ref{lemtorapprox} supplies a maximal torus $T$ in $G$ such that $T_{k_v}\simeq T_v$ whenever $v\in S$. Consider the following commutative diagram with exact rows:
\begin{center}\begin{tikzcd}
\mathrm{H}^1(\Gamma,T(k_s)/\mathrm Z_G(k_s))\arrow[r]\arrow[d] & \mathrm{H}^2(\Gamma,\mathrm Z_G(k_s))\arrow[r]\arrow[d] & \mathrm{H}^2(k,T)\arrow[d]\\
\prod_v\mathrm{H}^1(\Gamma_v,T(k_{v,s})/\mathrm Z_G(k_{v,s}))\arrow[r] & \prod_v\mathrm{H}^2(\Gamma_v,\mathrm Z_G(k_{v,s}))\arrow[r] & \prod_v\mathrm{H}^2(k_v,T)
\end{tikzcd}\end{center}
We want to show that $\mathrm{im}(x)\in\mathrm{H}^2(k, T)$ is $0$. By the previous case and our choice of $T$, we know that $\mathrm{im}(x)_v = \mathrm{im}(x_v) = 0$ for all $v$. Thus $\mathrm{im}(x)\in\Sh^2(T)$, but finally, we have $\Sh^2(T) = 0$ by Lemma \ref{lemanistor} since $T_{k_v}$ is anisotropic for at least one $v$ (here $S$ is nonempty).
\end{prf}

%% file: partC-0.tex
\subsection{The $\mathcal{UPic}_X$ complex of $\lqf$-sheaves}

Let $k$ be a field and $X$ a $k$-scheme. We define the following complex of $\Gamma$-modules (following \cite{BvH09}, \cite{BvH12}; see also \cite[\textsection 3]{DH22})
$$\mathrm{UPic}(\Xb) = \left[\frac{k_s(\Xb)^\times}{k_s^\times}\rightarrow\mathrm{Div}(\Xb)\right]$$
where such notation will always denote a complex in degrees $0$ and $1$. In \cite[Cor.\ 2.20]{BvH09}, it is shown that, when $k$ is a global or local field (so $\mathrm{H}^3(k,\mathbf G_\mathrm{m}) = 0$), for smooth and geometrically integral $X$ there is a natural isomorphism:
$$\mathbf{H}^2(\Gamma,\mathrm{UPic}(\Xb))\xrightarrow{\;\;\;\sim\;\;\;} \Br_\mathrm{a}(X) = \frac{\ker(\Br(X)\rightarrow\Br(\Xb))}{\Br(k)}$$

Our goal in the following pages is to generalize this statement to the fppf topology, using sheaves defined on an appropriately small fppf site:

\begin{defn}
Recall that a morphism of schemes $f : X\rightarrow S$ is called \textit{locally quasi-finite} if it is locally of finite type and, for every $s\in S$, the fiber $X_s$ is discrete as a topological space. This definition agrees with \cite[Lem.\ 06RT]{Stacks}. It is a general fact (see \cite[Lem.\ 0572]{Stacks}) that, for an arbitrary scheme $X$, every fppf covering $(X_i\rightarrow X)_{i\in I}$ can be refined by an fppf covering in which all morphisms are locally quasi-finite.

Let $S_\lqf$ denote the full subcategory of $\mathrm{Sch}/S$ having as objects all the scheme morphisms $X\rightarrow S$ which are locally quasi-finite. We equip it with the fppf topology and call it the \textit{(small) lqf-site} of $S$. It is indeed a site by the proposition below (even in the sense of \cite[II, \textsection1]{Mil80}, which is slightly stronger than in general).

The above remark on refinement of coverings shows that the cohomology of an fppf sheaf $\mathcal F$ over $S$ can be calculated on its restriction to the small lqf-site $S_\lqf$ (cf. \cite[III, Prop.\ 3.1]{Mil80}). That is, for any $i\in\mathbf Z$, the canonical map
$$\mathrm{H}^i(S_\lqf,\mathrm{id}_{S,*}\mathcal F)\xrightarrow{\;\;\sim\;\;}\mathrm{H}^i(S,\mathcal F)$$
is an isomorphism, where $\mathrm{id}_{S,*}$ denotes the obvious pushforward of sheaves between the~two~sites. We will thus simply write $\mathrm{H}^i(S,\mathcal G)$ for the cohomology of any sheaf $\mathcal G$ on the site $S_\lqf$.
\end{defn}

\begin{prop}
Locally quasi-finite morphisms are closed under composition and arbitrary pullbacks. Given morphisms $f : X\rightarrow Y$ and $g : Y\rightarrow S$, if $g\circ f$ is locally quasi-finite, then so is $f$. Fiber products (over any base in $S_\lqf$) exist in $S_\lqf$ and agree with those in $\mathrm{Sch}/S$.
\end{prop}
\begin{prf}
The first statement is \cite[Lem.\ 01TL, 01TM]{Stacks}. The second is \cite[Lem.\ 03WR]{Stacks}, and the last statement is a direct consequence of the first two.
\end{prf}

\begin{exmp}
Given a field $K$, the site $\Spec(K)_\lqf$ consists of discrete schemes $X$ locally of finite type over $K$. As each $X$ must admit an affine basis, we see that it is an arbitrary disjoint union of one-point schemes of finite type over $K$, that is, of spectra of finite local $K$-algebras.
\end{exmp}

Our reason for introducing the lqf-site is the following crucial lemma:

\begin{lem}\label{lemr1lqf}
Let $K$ be a field and $\Spec(K)\rightarrow S$ a scheme morphism. It induces a continuous map between small lqf-sites $\pi : \Spec(K)_\lqf\rightarrow S_\lqf$ such that $\mathrm{R}^1\pi_*\mathbf G_{\mathrm m,K} = 0$.
\end{lem}
\begin{prf}
The sheaf $\mathrm{R}^1\pi_*\mathbf G_{\mathrm m,K}$ is the sheafification of the functor $X\rightsquigarrow\mathrm{H}^1(X_K,\mathbf G_{\mathrm m})$, for $X\in S_\lqf$. The scheme $X_K$ is, by the above example, a disjoint union of spectra of local rings. It is then well-known that $\mathrm{H}^1(X_K,\mathbf G_{\mathrm m}) = \mathrm{Pic}(X_K) = 0$.
\end{prf}

Now, let $k$ be a field and $X$ a scheme of finite type over $k$. For simplicity, we suppose that $X = \Spec(R)$ is affine and integral (more generally, an analogous construction can be done~when $X$ is just reduced and with finitely many irreducible components; see \cite[II,~Prop.\ 1.2]{Gro68} or \cite[II, Ex. 3.9]{Mil80}). Let $K\coloneqq \mathrm{Frac}(R)$ be the field of fractions of $X$. We will also consider the two obvious maps $j : \Spec(K)\rightarrow X$ and $\pi : X\rightarrow\Spec(k)$.

There is a natural map of fppf sheaves $\alpha :\mathbf G_{\mathrm m, X}\rightarrow j_*\mathbf G_{\mathrm m, K}$ on $X$. We define a sheaf $\mathcal{Div}_X$ as its fppf cokernel, giving an exact sequence
\begin{equation}\label{eqfppfvslqf}
    0\rightarrow\mathcal Q\rightarrow j_*\mathbf G_{\mathrm m, K}\rightarrow\mathcal{Div}_X\rightarrow 0
\end{equation}
where $\mathcal Q\coloneqq \mathrm{im}(\alpha)$. We will study the following exact sequence of sheaves on $\Spec(k)_\lqf$:
$$0\rightarrow\pi_*\mathcal Q\rightarrow \pi_*j_*\mathbf G_{\mathrm m, K}\rightarrow\pi_*\mathcal{Div}_X\rightarrow \mathrm{R}^1\pi_*\mathcal Q\longrightarrow \mathrm{R}^1\pi_*(j_*\mathbf G_{\mathrm m, K}) = 0$$
The vanishing on the right comes from the inclusion of sheaves $\mathrm{R}^1\pi_*(j_*\mathbf G_{\mathrm m, K})\hookrightarrow\mathrm{R}^1(\pi\circ j)_*\mathbf G_{\mathrm m, K}$ and an application of Lemma \ref{lemr1lqf} to the morphism $\pi\circ j$.

\begin{rem}
Although we defined the sequence \eqref{eqfppfvslqf} for all $Y\in X_\fppf$, the expected geometric meaning of sections of these sheaves (of rational functions, resp.\ divisors) is missing~for~general non-open morphisms $Y\rightarrow X$. Still, the morphism $X_A\rightarrow X$ is open for each $k$-algebra $A$.

A more general notion of ``relative meromorphic functions'' (resp.\ ``relative Cartier divisors'') of $X$ over $k$ was developed in \cite[IV$_4$, \textsection20.6, resp.\ \textsection21.15]{EGA} and they form a functor on $\mathrm{Sch}/k$ which is easily seen to be a (big) fppf sheaf. It can be shown that the restriction of this sheaf to $\Spec(k)_\lqf$ agrees with our $\pi_*j_*\mathbf G_{\mathrm m, K}$ (resp.\ $\pi_*\mathcal{Div}_X$). This suffices for our needs and we will not use here the more general definitions.
\end{rem}

\begin{prop}
The natural morphisms $\pi_*\mathbf G_{\mathrm m,X}\rightarrow\pi_*\mathcal Q$ and $\mathrm{R}^1\pi_*\mathbf G_{\mathrm m,X}\rightarrow\mathrm{R}^1\pi_*\mathcal Q$ of fppf sheaves on $\Spec(k)_\lqf$ are isomorphisms.
\end{prop}
\begin{prf}
Let $\mathcal C$ denote the full subcategory of $X_\lqf$ consisting of flat morphisms $Y\rightarrow X$. Then each lqf cover of any given object $Y\in\mathcal C$ is also contained in $\mathcal C$ (however, $\mathcal C$ is in general not closed under fiber products with basis different from $X$, hence it is not a site in the sense of \cite{Mil80}). Under our assumptions (in particular, reducedness of $X$), the natural map $\mathcal O_X(X)=R\rightarrow K$ is injective. By flatness, the map $\mathbf G_{\mathrm m}(Y)\rightarrow j_*\mathbf G_{\mathrm m, K}(Y)$ is injective for every affine $Y\in\mathcal C$. It then follows that $\mathbf G_{\mathrm m}(Y) = \mathcal Q(Y)$ for all $Y\in\mathcal C$.

Since Čech and derived functor cohomology on $X_\lqf$ agree in degree $1$ (\cite[III, Cor.\ 2.10]{Mil80}), we immediately get that $\mathrm{H}^1(Y,\mathbf G_{\mathrm m}) = \mathrm{H}^1(Y,\mathcal Q)$ for all $Y\in\mathcal C$. It now suffices~to~note~that all $k$-schemes $Z$ are flat over $k$, and thus $X\times_k Z\in\mathcal C$ for all $Z\in\Spec(k)_\lqf$.
\end{prf}


\begin{defn}
We define the following complexes of sheaves on $\Spec(k)_\lqf$:
$$\mathcal{UPic}'_X\coloneqq\big[\pi_*j_*\mathbf G_{\mathrm m, K}\rightarrow \pi_*\mathcal{Div}_X\big]$$
$$\mathcal{UPic}_X\coloneqq\left[\frac{\pi_*j_*\mathbf G_{\mathrm m, K}}{\mathbf G_m}\rightarrow \pi_*\mathcal{Div}_X\right]$$
Note that there is a natural exact sequence:
\begin{equation}\label{equpic}
    0\rightarrow\mathbf G_m\rightarrow\mathcal{UPic}'_X\rightarrow\mathcal{UPic}_X\rightarrow 0
\end{equation}
\end{defn}

Recall that $\tau_{\leq 1}\mathrm{R}\pi_*\mathbf G_{\mathrm m, X}$ denotes the truncation in degree $1$ of the complex $\mathrm{R}\pi_*\mathbf G_{\mathrm m, X}$ in the derived category $\mathcal D(\Spec(k)_\lqf)$. That is, it is the unique (in the derived category) complex $D$ equipped with a map $D\rightarrow\mathrm{R}\pi_*\mathbf G_{\mathrm m, X}$ which induces the following isomorphisms of homology sheaves:
$$\mathcal H^i(D)\cong\left\{\begin{aligned}
    & \mathrm{R}^i\pi_*\mathbf G_{\mathrm m, X} & &\textrm{if }i\leq 1\\
    & 0 & &\textrm{if }i > 1
\end{aligned}\right.$$

\begin{prop}
There is a natural isomorphism $\tau_{\leq 1}\mathrm{R}\pi_*\mathbf G_{\mathrm m, X}\rightarrow\mathcal{UPic}'_X$ of complexes in the derived category $\mathcal D(\Spec(k)_\lqf)$.
\end{prop}
\begin{prf}
The canonical map of complexes
$\tau_{\leq 1}\mathrm{R}\pi_*\mathbf G_{\mathrm m, X}\longrightarrow\tau_{\leq 1}\mathrm{R}\pi_*\mathcal Q\cong\left[\pi_*j_*\mathbf G_{\mathrm m, K}\rightarrow \pi_*\mathcal{Div}_X\right]$
is a quasi-isomorphism by the previous proposition.
\end{prf}

Now \eqref{equpic} gives the following long exact sequence (cf. \cite[Prop.\ 2.19]{BvH09})
$$0\rightarrow\mathrm{Pic}(X)\rightarrow\mathbf{H}^1(k,\mathcal{UPic}_X)\rightarrow\mathrm{Br}(k)\rightarrow\mathrm{Br}'(X)\rightarrow\mathbf{H}^2(k,\mathcal{UPic}_X)\rightarrow\mathrm{H}^3(k,\mathbf G_{\mathrm m})$$
in which $\mathrm{Br}'(X)\coloneqq\ker\big(\mathrm{Br}(X)\rightarrow\mathrm{H}^0(k,\mathrm{R}^2\pi_*\mathbf G_{\mathrm m, X})\big)\;\subseteq\Br(X)$.

\begin{cor}\label{corupic}
If $k$ is a local or global field, then there is a canonical inclusion:
$$\mathbf{H}^2(k,\mathcal{UPic}_X)\xhookrightarrow{\;\;\;\;\;\;}\frac{\mathrm{Br}(X)}{\mathrm{Br}(k)}$$
\end{cor}

%% file: partC-1.tex
\subsection{Sheaves Associated to Torsors}
Let $X$ be an affine and integral scheme of finite type over $k$ and suppose given a left $H$-torsor $Y$ over $X$, for a smooth algebraic group $H$ over $k$. Let the maps $j : \Spec(K)\rightarrow X$ and $\pi : X\rightarrow\Spec(k)$ be as above. We also write $\widehat{H}\coloneqq\widehat{H}_\ab$, since $\mathcal{Hom}(H,\mathbf G_{\mathrm m}) = \mathcal{Hom}(H_\ab,\mathbf G_{\mathrm m}) = \widehat{H}_\ab$.

There are two important sheaves $\mathcal F$ and $\mathcal V$ which can be associated to the $H$-torsor $Y\rightarrow X$ on the (big) fppf site of $k$. 
To do this, we first name the surjective morphism:
$$q \;:\; Y\times_X Y\cong H\times Y\longrightarrow H$$
To define the sheaf $\mathcal F$, it suffices to give its points over each $k$-algebra $A$:
$$\mathcal F(A)\coloneqq\left\{
\begin{aligned}
    f : Y_A&\rightarrow\mathbf G_{\mathrm{m},A}\\
    \mathrm{Sch}/A-&\textrm{morphism} 
\end{aligned}
\left|\;
\begin{aligned}
    \exists\;\tilde{f} : H_A&\rightarrow\mathbf G_{\mathrm{m},A}\\
    \mathrm{Sch}/A-&\textrm{morphism} 
\end{aligned}\;\textrm{ s.t.}\!\!
\begin{tikzcd}
    Y_A\times_{X_A} Y_A\arrow[r, "f\times f"]\arrow[d, "q_A"] & \mathbf G_{\mathrm{m},A}\times_A \mathbf G_{\mathrm{m},A}\arrow[d, "\mathrm{id}\cdot\mathrm{inv}"]\\
    H_A\arrow[r,  "\tilde{f}"] & \mathbf G_{\mathrm{m},A}
\end{tikzcd}\!\!\right.\right\}$$
It is indeed a sheaf, as it can be written as an obvious pullback of sheaves (cf. \cite[Def.~2.5]{Don24}). The sheaf $\mathcal V$ is defined in a similar, if a bit more complicated way:
$$\mathcal V(A)\coloneqq\left\{
\begin{aligned}
    f : Y_A&\dashrightarrow\mathbf G_{\mathrm{m},A}\\
    \textrm{ratio}&\textrm{nal map}\\
    \textrm{ov}&\textrm{er }A
\end{aligned}
\left|\;
\begin{aligned}
    \exists\;&\varnothing\neq U\subseteq X\textrm{ open}\\
    f&\textrm{ is defined on }Y_{U_A}\\
    \exists\;&\;\,\tilde{f} : H_A\rightarrow\mathbf G_{\mathrm{m},A}\\
    &\!\mathrm{Sch}/A{-}\textrm{morphism} 
\end{aligned}\!\textrm{ s.t.}\!\!
\begin{tikzcd}
    Y_{U_A}\times_{U_A} Y_{U_A}\arrow[r, "f\times f"]\arrow[d, "q_A"] & \mathbf G_{\mathrm{m},A}\times_A \mathbf G_{\mathrm{m},A}\arrow[d, "\mathrm{id}\cdot\mathrm{inv}"]\\
    H_A\arrow[r,  "\tilde{f}"] & \mathbf G_{\mathrm{m},A}
\end{tikzcd}\!\!\right.\right\}$$
It is also a sheaf, since every fppf covering of an affine scheme can be refined by a finite (affine) fppf covering, and any finite intersection of open sets $U_i\subseteq X$ is again an open set. Although these sheaves are defined on the whole fppf site, we state our main result only over the lqf site (which is necessary to have surjectivity on the right of the second sequence below):

\begin{thm}\label{thmqiso}
There exists the following collection of exact sequences of sheaves on $\Spec(k)_\lqf$
\begin{equation}\label{eqqisomain}
\begin{tikzcd}[row sep = tiny]
0\arrow[r] & \pi_*\mathbf G_{\mathrm m,\,X}\arrow[r] & \mathcal F\arrow[r, "t"] & \widehat{H}\arrow[r, "w"] & \mathrm R^1\pi_*\mathbf G_{\mathrm m,\,X}\\ 
0\arrow[r] & \pi_*j_*\mathbf G_{\mathrm m,\,K}\arrow[r, "v"] & \mathcal V\arrow[r] & \widehat{H}\arrow[r] &  0\\
0\arrow[r] & \mathcal F\arrow[r, "u"] & \mathcal V\arrow[r] & \pi_*\mathcal{Div}_X
\end{tikzcd}
\end{equation}
which define a diagram of complexes of sheaves
$$\left[\begin{tikzcd}
    \mathcal F\arrow[d, "t"] \\ \widehat{H}
\end{tikzcd}\right]
\xleftarrow{\;\;\;\;\varepsilon\;\;\;\;}
\left[\begin{tikzcd}
    \mathcal F\oplus\pi_*j_*\mathbf G_{\mathrm m,\,K}\arrow[d, "u+v"] \\ \mathcal V
\end{tikzcd}\right]
\xrightarrow{\;\;\;\;\;\;\;\;\;\;}
\left[\begin{tikzcd}
    \pi_*j_*\mathbf G_{\mathrm m,\,K}\arrow[d] \\ \pi_*\mathcal{Div}_X
\end{tikzcd}\right] = \mathcal{UPic}'_X$$
such that the map $\varepsilon$ is a quasi-isomorphism. Moreover, the right map is a quasi-isomorphism if and only the sheaf morphism $\mathcal V\rightarrow\pi_*\mathcal{Div}_X$ in the third sequence is surjective.
\end{thm}

The second half of the statement follows formally from the exactness of the three sequences in the first half, as for this it is necessary and sufficient that the following diagram commutes
\begin{center}\begin{tikzcd}[row sep = tiny]
    \mathcal F\arrow[dd, "t"]\arrow[dr, "u"] & & \pi_*j_*\mathbf G_{\mathrm m,\,K}\arrow[dd]\arrow[dl, "v", swap] \\
    & \mathcal V\arrow[dl]\arrow[dr] \\
    \widehat{H} & & \pi_*\mathcal{Div}_X
\end{tikzcd}\end{center}
which will be obvious from the definitions given below. We split the proof of this theorem into three propositions:

\begin{prop}\label{proplqfexseq1}
There is an exact sequence as in the first row of \eqref{eqqisomain}.
\end{prop}
\begin{prf}
As $q$ is an epimorphism of schemes, $\tilde{f}$ is defined uniquely by $f$. It is also easily seen to be a homomorphism, which defines a map $t : \mathcal F\rightarrow\widehat{H}$. Its kernel is exactly the sheaf of maps $f$ which factor through $X$, which is $\pi_*\mathbf G_{\mathrm m,\,X}$.

We define a map $w : \widehat{H}\rightarrow\mathrm R^1\pi_*\mathbf G_{\mathrm m,\,X}$ by sending $\alpha\in\widehat{H}(A)$ 
to $\alpha_*[Y_A]\in\mathrm{H}^1(X_A,\mathbf G_{\mathrm m})$. Here, $[Y_A]\in\mathrm{H}^1(X_A,H)$ is the class of $[Y_A]$ as an $H_A$-torsor of $X_A$. In other words,\;\;
$$\alpha_*[Y_A] = [Z]\;\;\textrm{ for the contracted product }\;\; Z\coloneqq H_A{\setminus}\big(\mathbf G_{\mathrm m,\,X}\times_X Y\big)_A$$
where $H_A$ acts on $\big(\mathbf G_{\mathrm m,\,X}\times_X Y\big)_A$ by $h\cdot(c,y)=(c\cdot\alpha(h)^{-1},h\cdot y)$. The class $[Z]$ is trivial~if~and only if there exists a $\mathbf G_{\mathrm m,\, X_A}$-equivariant morphism $Z\rightarrow\mathbf G_{\mathrm m,\, X_A}$ of $X_A$-schemes. This is true if and only if $\alpha = \tilde{f}$ for some $f\in\mathcal F(A)$, in which case this morphism is $[c,y]\mapsto c\cdot f(y)$~(and conversely, we may define $f(y)$ as the image of $[1,y]$). In particular, after the sheafification of $A\rightsquigarrow\mathrm{H}^1(X_A,\mathbf G_{\mathrm m})$ to construct $\mathrm R^1\pi_*\mathbf G_{\mathrm m,\,X}$, it follows that $\ker(w) = \mathrm{im}(t)$.
\end{prf}

\begin{prop}
There is an exact sequence as in the third row of \eqref{eqqisomain}.
\end{prop}
\begin{prf}
The inclusion of $\mathcal F$ into $\mathcal V$ is clear, it corresponds to those $f\in\mathcal V(A)$ for which $U = X$. To construct the map $\mathcal V\rightarrow\pi_*\mathcal{Div}_X$, we want to show that $f$ naturally determines an element $D\in\mathcal{Div}_X(X_A)$. Consider the fppf covering $Y_A\rightarrow X_A$ and recall that $\mathcal{Div}_X$ is a sheaf on the (big) fppf site of $k$, not just the lqf site. By definition, $f$ is defined on some $(Y\times_X U)_A$, so it maps to a divisor $D\in\mathcal{Div}_X(Y_A)$ via the composition:
$$\mathbf G_{\mathrm m}((Y\times_X U)_A) \longrightarrow\mathbf G_{\mathrm m}((Y\times_X \Spec(K))_A)\cong j_*\mathbf G_{\mathrm m,\,K}(Y_A)\longrightarrow\mathcal{Div}_X(Y_A)$$
To prove that $D$ lies in $\mathcal{Div}_X(X_A)$, it thus suffices to show that $\mathrm{pr_1^*}D-\mathrm{pr_2^*}D = 0$ for the two projections from $(Y\times_X Y)_A$ to $Y_A$. This difference is exactly the divisor of the map
$$(f\circ\mathrm{pr}_1)\cdot\mathrm{inv}(f\circ\mathrm{pr}_2) \;:\; (Y_U\times_U Y_U)_A\longrightarrow\mathbf G_{\mathrm m,\,A}$$
from the definition of $\mathcal V$. But that map is equal to $\tilde{f}\circ q_A$, which is defined on $(Y\times_X Y)_A$. This shows the desired property. Moreover, it is clear that $D = 0$ if and only if we may take~$U$~to~be $X$, which proves exactness of the sequence in the statement.
\end{prf}

\begin{prop}\label{proplqfexseq2}
There is an exact sequence as in the second row of \eqref{eqqisomain}.
\end{prop}
\begin{prf}
Both maps are defined analogously to Proposition \ref{proplqfexseq1}. All that remains to prove~is~the surjectivity on the right. Since $\mathcal F\rightarrow\widehat{H}$ factors through $\mathcal V$, it suffices to show that the map
$$\mathcal V\longrightarrow\mathrm{coker}\!\left(\mathcal F\rightarrow\widehat{H}\right) = \mathrm{im}\!\left(\widehat{H}\xrightarrow{\;\;w\;\;}\mathrm R^1\pi_*\mathbf G_{\mathrm m,\,X}\right)$$
is a surjection of sheaves. Consider $\alpha\in\widehat{H}(A)$ and the image $\alpha_*[Y_A]\in\mathrm{H}^1(X_A, \mathbf G_{\mathrm m})$, for any quasi-finite $k$-algebra $A$.
The proof of Lemma \ref{lemr1lqf} shows that $\mathrm{H}^1(K\otimes_k A, \mathbf G_{\mathrm m}) = 0$ and thus any torsor representing $\alpha_*[Y_A]$ trivializes over $U_A$, for some open set $U\subseteq X$. We now find a preimage $f\in\mathcal V(A)$ of $\alpha_*[Y_A]$ exactly as in Proposition \ref{proplqfexseq1}.
\end{prf}

This completes the proof of Theorem \ref{thmqiso}.

Now, suppose that $X$ is a homogeneous space of a smooth connected affine group $G$ over $k$, with smooth geometric stabilizer $\Hb$ which then has its maximal abelian quotient $H_\ab$ uniquely defined over $k$. The main result of this section is about how the above construction relates to the complex $\mathcal{UPic}_X$:

\begin{cor}\label{corqiso}
There exist sheaves $\mathcal E$ and $\mathcal U$ on $\Spec(k)_\lqf$ and maps which fit into the following diagram of complexes of lqf-sheaves
$$\left[\begin{tikzcd}
    \mathcal E\arrow[d] \\ \widehat{H}_\ab
\end{tikzcd}\right]
\xleftarrow{\;\;\;\;\varepsilon\;\;\;\;}
\left[\begin{tikzcd}
    \mathcal E\oplus\dfrac{\pi_*j_*\mathbf G_{\mathrm m,\,K}}{\mathbf G_m}\arrow[d] \\ \mathcal U
\end{tikzcd}\right]
\xrightarrow{\;\;\;\;\;\;\;\;\;\;}
\left[\begin{tikzcd}
    \dfrac{\pi_*j_*\mathbf G_{\mathrm m,\,K}}{\mathbf G_m}\arrow[d] \\ \pi_*\mathcal{Div}_X
\end{tikzcd}\right] = \mathcal{UPic}_X$$
such that the map $\varepsilon$ is a quasi-isomorphism. Moreover, the right map is a quasi-isomorphism if and only the constructed sheaf morphism $\mathcal U\rightarrow\pi_*\mathcal{Div}_X$ is surjective.
\end{cor}

We will prove the corollary after the remark below, but before that we explain the main~idea. The proof is split into two steps:

First, if $X(k)\neq\varnothing$, then any $k$-point $x\in X(k)$ defines a map $G\rightarrow X$ which makes $G$ into an $H$-torsor of $X$, for a $k$-form $H$ of $\Hb$. Thus the theorem holds for $\mathcal E = \mathcal F/\mathbf G_{\mathrm m}$ and $\mathcal U = \mathcal V/\mathbf G_{\mathrm m}$. The general case (where $X(k)$ is possibly empty) will show that this construction is essentially independent of the chosen $x$.

In the general case, there exists only a finite extension $k'/k$ over which $X(k')\neq\varnothing$ and over which the first step thus holds. We will show that the $\Spec(k')_\lqf$-sheaves $\mathcal F/\mathbf G_{\mathrm m,\,k'}$ and $\mathcal V/\mathbf G_{\mathrm m,\,k'}$ descend to $\Spec(k)_\lqf$-sheaves $\mathcal E$ and $\mathcal U$, as do all the relevant maps, and that this descent is independent of the choice of point in $X(k')$. Then Corollary \ref{corqiso} will automatically hold over $k$, as the required exactness statements can all be checked locally.

\begin{rem}
In \cite[Thm.\ 4.10 and Thm.\ 5.8]{BvH12}, the following chain of quasi-isomorphisms between $\Gamma$-modules is proven (written here in the notation of loc.\ cit.)
$$\left[\begin{tikzcd}
    \widehat G(k_s)\arrow[d] \\ \widehat{H}(k_s)
\end{tikzcd}\right]
\cong
\left[\begin{tikzcd}
    \mathrm{Z}^1_\mathrm{alg}(\overline G, \mathcal O(\overline X)^\times)\arrow[d] \\ \mathrm{Pic}_G(\overline X)
\end{tikzcd}\right]
\xleftarrow{\;\;\;\;\sim\;\;\;\;}
\left[\begin{tikzcd}
    \mathrm{Z}^1_\mathrm{alg}(\overline G, \mathcal O(\overline X)^\times)\oplus\dfrac{k(\overline X)^\times}{k^\times}\arrow[d] \\ \mathrm{UPic}_G(\overline X)^1
\end{tikzcd}\right]
\xrightarrow{\;\;\;\;\sim\;\;\;\;}
\left[\begin{tikzcd}
    \dfrac{k(\overline X)^\times}{k^\times}\arrow[d] \\ \mathrm{Div}(\overline X)
\end{tikzcd}\right]$$
under the assumption $\mathrm{Pic}(\overline G)=0$, which is necessary only to show that the map on the right~is a quasi-isomorphism. We do not attempt to generalize this part of their statement. Moreover,~the connection with $\widehat G(k_s)$ is also possible only by an application of Rosenlicht's lemma, which~fails over nonreduced $X\in\Spec(k)_\lqf$, so we do not consider it.

Our sheaves $\mathcal E$ and $\mathcal U$ can also be shown to directly generalize the $\Gamma$-modules $\mathrm{Z}^1_\mathrm{alg}(\overline G, \mathcal O(\overline X)^\times)$ and $\mathrm{UPic}_G(\overline X)^1$ introduced in \cite[\textsection2]{BvH12}, respectively. However, we chose to define them in a different way, which avoids working with cocycles and linearizations of line bundles. Finally, it is important to remark that the $\Gamma$-module $k(\overline{X})^\times$ does not in general correspond to our sheaf $\pi_*j_*\mathbf G_{\mathrm m,\, K}$; this is true in our situation, but only because $X$ is geometrically integral (see the discussion below on descending $\pi_*j_*\mathbf G_{\mathrm m,\, K'}$ to $k$).
\end{rem}

To prove the corollary, fix $k'/k$ and $x\in X(k')$. Recall the construction of the Springer band in Subsection \ref{ssectspringer}, which is in particular nicely represented (cf. Corollary \ref{corsprbandissep}) and we can thus represent $L_X$ by a triple $(k'/k,\Hb,\varphi_H)$, where $\Hb\coloneqq\mathrm{Stab}_{G_{k'}}(x)$ and $\varphi_H$ is the outer-isomorphism represented by $\varphi_G\circ\mathrm{int}(g_x)^{-1}$ for any $g_x\in G(k'\otimes_k k')$ such that $\mathrm{pr}_1^*(x).g_x = \varphi_X^{-1}(\mathrm{pr}_2^*(x))$.

We will also write $\varphi_H$ for the well-defined descent datum $\mathrm{pr}_1^*(H_{\ab,\,k'})\rightarrow\mathrm{pr}_2^*(H_{\ab,\,k'})$. Now~the descent datum on the sheaf $\widehat{H}_\ab$ takes the form:
$$\varphi_{\widehat{H}}(\alpha)\coloneqq \varphi_{\mathbf G_{\mathrm m}}\circ\alpha\circ \varphi_H^{-1}$$
Similarly, we write $\varphi_{\widehat{X}}$ for the descent datum of the sheaf $\pi_*\mathbf G_{\mathrm m,\,X}$. It is also a good time to note that the morphism $r^x : G_{k'}\rightarrow X_{k'}$ given by $g\mapsto x.g$ does not descent to $k$, but instead the descent datum $\varphi_X$ commutes along $r^x$ with $\varphi_G\circ\ell_{g_x}^{-1}$ (cf. the proof of Proposition \ref{propspringertriple}), where $\ell$ denotes multiplication from the left on $G$.

Next, we show that $\mathcal F/\mathbf G_{\mathrm m,\,k'}$ descends to a sheaf on $k$. Informed by the first exact sequence of \eqref{eqqisomain}, we will initially attempt to construct a descent datum $\varphi_{\mathcal F}$ on $\mathcal F$ such that the following diagram commutes (and then see where this attempt fails):
\begin{center}\begin{tikzcd}
0\arrow[r] & \mathrm{pr}_1^*(\pi_*\mathbf G_{\mathrm m,\,X_{k'}})\arrow[d, "\varphi_{\widehat{X}}"]\arrow[r] & \mathrm{pr}_1^*\mathcal F\arrow[d, "\varphi_{\mathcal F}"]\arrow[r] & \mathrm{pr}_1^*\widehat{H}\arrow[d, "\varphi_{\widehat{H}}"]\\
0\arrow[r] & \mathrm{pr}_2^*(\pi_*\mathbf G_{\mathrm m,\,X_{k'}})\arrow[r] & \mathrm{pr}_2^*\mathcal F\arrow[r] & \mathrm{pr}_2^*\widehat{H}
\end{tikzcd}\end{center}
For the diagram to commute, we must define $\varphi_{\mathcal F}$ as:
\begin{equation}\label{eqdescF1}
\varphi_{\mathcal F}(f)\coloneqq\varphi_{\mathbf G_{\mathrm m}}\circ f\circ \big(\varphi_G\circ\ell_{g_x}^{-1}\big)^{-1}
\end{equation}
We now calculate the cochain condition on an arbitrary local section $f$ of $\mathrm{pr}_{12}^*\mathrm{pr}_1^*\mathcal F$ to get
\begin{equation}\label{eqdescF2}
\left((\mathrm{pr}_{13}^*\varphi_{\mathcal F})^{-1}\circ(\mathrm{pr}_{23}^*\varphi_{\mathcal F})\circ(\mathrm{pr}_{12}^*\varphi_{\mathcal F})\right)(f) 
    = f\circ\ell_{h_x} = \ell_{\tilde{f}(h_x)}\circ f
\end{equation}
for the element $h_x = \mathrm dg_x\in (\mathrm{pr}_{12}^*\mathrm{pr}_1^*H)(k'\otimes_k k'\otimes_k k')$ from Subsection \ref{ssectspringer}. The last equality here follows from the definition of $\mathcal F$, and the element $\tilde{f}(h_x)$ lies in $(\mathrm{pr}_{12}^*\mathrm{pr}_1^*\mathbf G_{\mathrm m,\,k'})(k'\otimes_k k'\otimes_k k')$. This shows that the map $\varphi_{\mathcal F}$ is not necessarily a descent datum on $\mathcal F$, however it does induce a descent datum on $\mathcal F/\mathbf G_{\mathrm m,\,k'}$, since then $[f] = [\ell_{\tilde{f}(h_x)}\circ f]$. Since descent of sheaves is always effective,~we~have~just constructed a sheaf $\mathcal E$ on $\Spec(k)_\lqf$.

We proceed analogously to prove that $\mathcal V/\mathbf G_{\mathrm m,\,k'}$ descends to a sheaf $\mathcal U$ on $\Spec(k)_\lqf$, with one caveat: It is not a priori clear that the $k'$-sheaf $\pi_*j_*\mathbf G_{\mathrm m,\, K'}$ descends to $k$ (here $\Spec(K')$ is the generic point of $X_{k'}$). However, as $X$ is a homogeneous space of a smooth~connected~group, it is in particular geometrically integral and thus $K' = K\otimes_k k'$, where $\Spec(K)$ is the generic point of $X$. Equivalently, for every nonempty open subscheme $U'$ of $X_{k'}$, there exists a nonempty open subscheme $U$ of $X$ such that $U_{k'}\subseteq U'$ in $X_{k'}$ (compare this with the definition of $\mathcal V$). This shows that $\pi_*j_*\mathbf G_{\mathrm m,\, K'} = (\pi_*j_*\mathbf G_{\mathrm m,\, K})_{k'}$. The descent datum of $\pi_*j_*\mathbf G_{\mathrm m,\, K'}$ is now applied in the same way as above to prove that $\mathcal V/\mathbf G_{\mathrm m,\,k'}$ descends to $k$.

Finally, it is immediate that all morphisms descend and commute as required by construction. This concludes the proof of the corollary.

\begin{rem}
These constructions are independent of choices of $x$ and $g_x$. Replacing $g_x$ by $g_x'$ amounts to a conjugation by an element of $H(k'\otimes k')$ which does not affect the Springer~band (and induces the identity map on $\widehat{H}$, $\mathcal E$, $\mathcal U$ and $X$). When replacing $x$ by $x'$ (which replaces $\varphi_H$ by some $\varphi_H'$), we may enlarge $k'$ to find $\tilde{g}\in G(k')$ such that $x = x'.\tilde{g}$. There is a commutative square $\mathrm{int}(\mathrm{pr}_2^*\tilde{g})\circ\varphi_H' = \varphi_H\circ\mathrm{int}(\mathrm{pr}_1^*\tilde{g})$ which represents a canonical isomorphism of the Springer band, and similarly for the other constructions.
\end{rem}

%% file: partC-2.tex
\subsection{Commutativity with the Poitou-Tate Map}\label{ssectptcomm}

Let $k$ be a global field and let $X$ be a homogeneous space of a smooth connected affine group $G$ over $k$, with smooth connected geometric stabilizer $\Hb$. We write $K$ and $K^v$, respectively, for the fields of rational functions on $X$ and $X_{k_v}$ over all places $v$ of $k$. The previous section defines sheaves $\mathcal U$ and $\mathcal U^v$, over $k$ and all completions $k_v$, respectively. These are related by maps $K\otimes_k k_v\rightarrow K^v$ and $\mathcal U_{k_v}\rightarrow\mathcal U^v$ for all $v$ (and, conversely to the above considered situation of finite extensions $k'/k$, these maps are in general not isomorphisms, even for $X = \Spec(k[T])$). Note that this is possible because $\mathcal U$ was defined on the fppf site of $k$ and not just the lqf site. There are in particular natural compositions of maps in hypercohomology
$$\mathbf{H}^2\!\left(k, \big[\pi_*j_*\mathbf G_{\mathrm m,\,K}\rightarrow\mathcal U\big]\right)\longrightarrow\mathbf{H}^2\!\left(k_v, \big[\pi_*j_*\mathbf G_{\mathrm m,\,K}\rightarrow\mathcal U\big]\right)\longrightarrow\mathbf{H}^2\!\left(k_v, \big[\pi_*j_*\mathbf G_{\mathrm m,\,K^v}\rightarrow\mathcal U^v\big]\right)$$
which we will simply denote by $A\mapsto A_v$. These observations implicitly underpin all calculations to follow and will not be repeated later.

Suppose that $X(k_v)\neq\varnothing$ for all places $v$ of $k$. The Springer band $L_X$ has its maximal Abelian quotient $H_\ab$ uniquely defined over $k$. In this subsection, we complete the proof of Lemma \ref{lemphslift} by showing, for any element $A\in\Sh^1(\widehat{H}_\ab)$, the following equality in $\mathbf Q/\mathbf Z$:
\begin{equation}\label{eqptvsbm}
    -\langle\xi_X^\ab, A\rangle_{PT} = BM_X(\phi(A))
\end{equation}
Here, $\xi_X^\ab$ is the Abelianization in $\mathrm{H}^2(k, H_\ab)$ of the Springer class $\xi_X = [h_x]$, and it is in $\Sh^2(H_\ab)$ when $X$ has points over all $k_v$. Next, $\langle-, -\rangle_{PT}$ denotes Rosengarten's Poitou-Tate pairing and $BM_X : \Be(X)\rightarrow\mathbf Q/\mathbf Z$ is the Brauer-Manin obstruction to the Hasse principle (see Definition \ref{defbm}). Finally, $\phi : \mathrm{H}^1(k,\widehat{H}_\ab)\rightarrow\Br(X)/\Br(k)$ is induced by the map $\widehat{H}_\ab[-1]\rightarrow\mathcal{UPic}_X$ in the derived category of $\Spec(k)_\lqf$, which was constructed in Corollary \ref{corqiso}. Since $\phi$ is constructed naturally over any global or local field (cf. Corollary \ref{corupic}), the image $\phi(A)$ indeed lies in $\Be(X)$. We start by making the right-hand side more explicit:

The exact sequence of sheaves
\begin{equation}\label{eqcechcompseq}
0\longrightarrow\mathbf G_{\mathrm m}\longrightarrow\pi_*j_*\mathbf G_{\mathrm m,\,K}\longrightarrow\mathcal U\longrightarrow\widehat{H}_\ab\longrightarrow\mathcal 0   
\end{equation}
induces a sequence in chain complexes of $\mathrm{Ab}$-valued sheaves on $\Spec(k)_\lqf$
$$0\longrightarrow\mathbf G_{\mathrm m}\longrightarrow\big[\pi_*j_*\mathbf G_{\mathrm m,\,K}\rightarrow\mathcal U\big]\longrightarrow\widehat{H}_\ab[-1]\longrightarrow\mathcal 0$$
which becomes an exact triangle in the derived category $\mathcal D(\Spec(k)_\lqf)$. In particular, $\widehat{H}_\ab[-1]$ is quasi-isomorphic to $\big[\pi_*j_*\mathbf G_{\mathrm m,\,K}/\mathbf G_{\mathrm m}\rightarrow\mathcal U\big]$. The same is true when we replace $k$ by any local completion $k_v$ (and $K, \mathcal U$ by $K^v, \mathcal U^v$).

\begin{lem}\label{lemcechcomp1}
Both rows of the following commutative diagram are exact:
\begin{center}\begin{tikzcd}[column sep = 25pt]
0\arrow[r] & \mathrm{H}^2(k, \mathbf G_{\mathrm m})\arrow[r]\arrow[d] & \mathbf{H}^2\!\left(k, \big[\pi_*j_*\mathbf G_{\mathrm m,\,K}\rightarrow\mathcal U\big]\right)\arrow[r]\arrow[d] & \mathrm{H}^1(k, \widehat{H}_\ab)\arrow[d]\arrow[r] & 0\\
0\arrow[r] & \prod_v\mathrm{H}^2(k_v, \mathbf G_{\mathrm m})\arrow[r] & \prod_v\mathbf{H}^2\!\left(k_v, \big[\pi_*j_*\mathbf G_{\mathrm m,\,K^v}\rightarrow\mathcal U^v\big]\right)\arrow[r] & \prod_v\mathrm{H}^1(k_v, \widehat{H}_\ab)\arrow[r] & 0
\end{tikzcd}\end{center}
If $\delta$ denotes the connecting homomorphism given by the snake lemma, then
$$\mathrm{im}(\delta)\subseteq\frac{\bigoplus_v\Br(k_v)}{\Br(k)}\;\;\;\textrm{ and }\;\;\;
\left(\sum\nolimits_v\mathrm{inv}_v\right)\circ\delta \;=\; BM_X\circ\phi|_{\Sh^1(\widehat{H}_\ab)}$$
where $\mathrm{inv}_v : \Br(k_v)\rightarrow\mathbf Q/\mathbf Z$ are the invariant maps (cf. Definition \ref{defbm}).
\end{lem}
\begin{prf}
The exactness in the statement is clear (recall that $\mathrm{H}^3(k,\mathbf G_{\mathrm m}) = 0$ and $\mathrm{H}^3(k_v,\mathbf G_{\mathrm m}) = 0$ by \cite[Ch.\ VII, \textsection 11.4]{CF67}), except for the leftmost map in both rows. For the rest, note~that~there is a natural commutative diagram of complexes of sheaves with exact rows
\begin{center}\begin{tikzcd}
0\arrow[r] & \mathbf G_{\mathrm m}\arrow[r]\arrow[d, equal] & \big[\pi_*j_*\mathbf G_{\mathrm m,\,K}\rightarrow\mathcal U\big]\arrow[r]\arrow[d] & \big[\pi_*j_*\mathbf G_{\mathrm m,\,K}/\mathbf G_{\mathrm m}\rightarrow\mathcal U\big]\arrow[d]\arrow[r] & 0\\
0\arrow[r] & \mathbf G_{\mathrm m}\arrow[r] & \mathcal{UPic}'_X\arrow[r] & \mathcal{UPic}_X\arrow[r] & 0
\end{tikzcd}\end{center}
and similarly over every local completion $k_v$, by Corollary \ref{corqiso}. Taking long exact sequences in hypercohomology, we construct a map between the diagram in the statement of the lemma and the diagram \eqref{eqdiagbm} in Definition \ref{defbm}, with identity maps on $\mathrm{H}^2(k,\mathbf G_{\mathrm m})$ and on $\mathrm{H}^2(k_v,\mathbf G_{\mathrm m})$. This shows exactness on the left of the first diagram, as well as the desired compatibility in the second part of the statement (analogous to the proof of \cite[Lem.\ 2.7]{Don24}).
\end{prf}

The main results of this subsection will be proven using Čech hypercohomology:
\begin{defn}\label{defnhypercech}
Recall that, given a complex $\mathcal F^\bullet$ of (pre)sheaves on $\Spec(k)_\lqf$ and a cover $\mathcal Z\rightarrow\Spec(k)$, we may consider the total complex of Abelian groups:
$$\mathbf{\check C}^n(\mathcal Z,\mathcal F^\bullet)\coloneqq\bigoplus\nolimits_{i+j=n}\mathrm{\check C}^j(\mathcal Z,\mathcal F^i)$$
We will need only the example $\mathcal F^\bullet = [\mathcal B\xrightarrow{\Psi}\mathcal C]$ and make the convention that $\mathrm d = \mathrm d^{\mathbf{\check C}}$ 
is given~by $\mathrm d(b,c) = (\mathrm db,\Psi(b)-\mathrm dc)$. This is consistent with the common convention: $\mathrm d^{\mathcal C[-1]} = -\mathrm d^{\mathcal C}$

We write $\mathbf{\check H}^n(\mathcal Z,\mathcal F^\bullet)$ for the $n$-th cohomology group of this complex. Since every cover of $k$ is refined by some finite extension $k'/k$, define the \textit{Čech (hyper)cohomology groups}
$$\mathbf{\check H}^n(k,\mathcal F^\bullet)\coloneqq\varinjlim\nolimits_{k\subseteq k'\subseteq\overline{k}}\mathbf{\check H}^n(k'/k,\mathcal F^\bullet)$$
where the limit is taken over all finite extensions $k'/k$ in some fixed algebraic closure $\overline{k}/k$. Equivalently, these are the cohomology groups of a fixed \textit{Čech complex}:
$$\mathbf{\check C}^n(k,\mathcal F^\bullet)\coloneqq\varinjlim\nolimits_{k\subseteq k'\subseteq\overline{k}}\mathbf{\check C}^n(k'/k,\mathcal F^\bullet)$$
By the usual arguments (\cite[Lem.\ 08BN]{Stacks}), there are natural maps $\mathbf{\check H}^n(k,\mathcal F^\bullet)\rightarrow\mathbf{H}^n(k,\mathcal F^\bullet)$.
\end{defn}

\begin{prop}\label{propc3presh}
Given a complex of (pre)sheaves
$0\longrightarrow\mathcal A\longrightarrow\mathcal B\longrightarrow\mathcal C\longrightarrow\mathcal D\longrightarrow 0$
on $\Spec(k)_\lqf$, consider the following complexes of Abelian groups for $n\in\mathbf Z$:
\begin{equation*}
0\longrightarrow \mathrm{\check C}^n(k, \mathcal A)\longrightarrow \mathrm{\check C}^n(k, \mathcal B)\longrightarrow \mathrm{\check C}^n(k, \mathcal C)\longrightarrow \mathrm{\check C}^n(k, \mathcal D)\longrightarrow 0\tag{$\ast_n$}
\end{equation*}
The following properties hold (over $k$, but analogously also over all $k_v$):
\begin{itemize}
    \item If $(\ast_n)$ is exact for all $n$, then there is a long exact sequence in Čech cohomology:
$$\cdots\longrightarrow \mathrm{\check H}^n(k, \mathcal A)\longrightarrow \mathbf{\check H}^n\!\left(k, \big[\mathcal B\rightarrow\mathcal C\big]\right)\longrightarrow \mathrm{\check H}^{n-1}(k, \mathcal D)\longrightarrow\mathrm{\check H}^{n+1}(k, \mathcal A)\longrightarrow\cdots$$
    \item $(\ast_n)$ is exact for all $n$ for our complex \eqref{eqcechcompseq}
\end{itemize}
\end{prop}
\begin{prf}
For the first point, assume that all $(\ast_n)$ are exact. Straight from the definition of Čech complexes, we see that their construction commutes with taking cones in the categories of chain complexes (of $\Spec(k)_\lqf$-sheaves and of Abelian groups, respectively):
$$\mathbf{\check C}^n\!\left(k, \big[\mathcal B\rightarrow\mathcal C\big]\right) = \mathbf{\check C}^n\!\left(k, \mathrm{cone}(\mathcal B\rightarrow\mathcal C)\right) = \mathrm{cone}\left(\mathrm{\check C}^n(k, \mathcal B)\rightarrow\mathrm{\check C}^n(k, \mathcal C)\right)$$
Then the desired long exact sequence is simply the hypercohomology sequence of the following short exact sequence of complexes of Abelian groups.
$$0\rightarrow\mathrm{\check C}^n(k, \mathcal A)\rightarrow\mathrm{cone}\left(\mathrm{\check C}^n(k, \mathcal B)\rightarrow\mathrm{\check C}^n(k, \mathcal C)\right)\rightarrow\mathrm{\check C}^n(k, \mathcal D)[-1]\rightarrow 0$$

For the second point, it suffices to check that the following two short sequences are exact
$$0\longrightarrow\mathbf G_{\mathrm m}(R)\longrightarrow\pi_*j_*\mathbf G_{\mathrm m,\,K}(R)\longrightarrow(\pi_*j_*\mathbf G_{\mathrm m,\,K}/\mathbf G_{\mathrm m})(R)\longrightarrow\mathcal 0$$
$$0\longrightarrow(\pi_*j_*\mathbf G_{\mathrm m,\,K}/\mathbf G_{\mathrm m})(R)\longrightarrow\mathcal U(R)\longrightarrow\widehat{H}_\ab(R)\longrightarrow\mathcal 0$$
for any quasi-finite $k'$-algebra $R$ (for some $k'$ such that $X(k')\neq\varnothing$, over which we may replace $\mathcal U$ by $\mathcal V$). It suffices to check surjectivity on the right, which follows by examining the proofs of Lemma \ref{lemr1lqf} and Proposition \ref{proplqfexseq2}, respectively. 
\end{prf}

\begin{cor}\label{corcechcomp1}
The following commutative diagram has exact rows
\begin{center}\begin{tikzcd}
0\arrow[r] & \mathrm{\check H}^2(k, \mathbf G_{\mathrm m})\arrow[r]\arrow[d, pos=0.4, "\rotatebox{90}{$\sim$}"] & \mathbf{\check H}^2\!\left(k, \big[\pi_*j_*\mathbf G_{\mathrm m,\,K}\rightarrow\mathcal U\big]\right)\arrow[r]\arrow[d, pos=0.4, "\rotatebox{90}{$\sim$}"] & \mathrm{\check H}^1(k, \widehat{H}_\ab)\arrow[r]\arrow[d, pos=0.4, "\rotatebox{90}{$\sim$}"] & 0\\
0\arrow[r] & \mathrm{H}^2(k, \mathbf G_{\mathrm m})\arrow[r] & \mathbf{H}^2\!\left(k, \big[\pi_*j_*\mathbf G_{\mathrm m,\,K}\rightarrow\mathcal U\big]\right)\arrow[r] & \mathrm{H}^1(k, \widehat{H}_\ab)\arrow[r] & 0
\end{tikzcd}\end{center}
and all three vertical maps are isomorphisms. The same holds when $k$ is replaced by any $k_v$.
\end{cor}
\begin{prf}
The bottom row is exact by Lemma \ref{lemcechcomp1}. The first and third vertical maps are isomorphisms by \cite[Prop.\ 2.9.6 and Prop.\ 2.9.9]{RosTD}, which in particular proves injectivity on the left of the top row. The preceding proposition now shows that the top row is exact (again, by loc.\ cit.\ $\mathrm{\check H}^3(K, \mathbf G_{\mathrm m})\cong\mathrm{H}^3(K, \mathbf G_{\mathrm m}) = 0$ for a local or global field $K$). Finally, the middle vertical map is an isomorphism by the $5$-lemma.
\end{prf}

In view of Lemma \ref{lemcechcomp1} and Corollary \ref{corcechcomp1}, to prove \eqref{eqptvsbm} it suffices to show that the snake lemma map in the following Čech cohomology diagram agrees with the map $-\langle\xi_X^\ab, -\rangle_{PT}$:
\begin{center}\begin{tikzcd}[column sep = 25pt]
0\arrow[r] & \mathrm{\check H}^2(k, \mathbf G_{\mathrm m})\arrow[r]\arrow[d] & \mathbf{\check H}^2\!\left(k, \big[\pi_*j_*\mathbf G_{\mathrm m,\,K}\rightarrow\mathcal U\big]\right)\arrow[r]\arrow[d] & \mathrm{\check H}^1(k, \widehat{H}_\ab)\arrow[d]\arrow[r] & 0\\
0\arrow[r] & \prod_v\mathrm{\check H}^2(k_v, \mathbf G_{\mathrm m})\arrow[r] & \prod_v\mathbf{\check H}^2\!\left(k_v, \big[\pi_*j_*\mathbf G_{\mathrm m,\,K^v}\rightarrow\mathcal U^v\big]\right)\arrow[r] & \prod_v\mathrm{\check H}^1(k_v, \widehat{H}_\ab)\arrow[r] & 0
\end{tikzcd}\end{center}
Explicitly, take an element $A\in\Sh^1(\widehat{H}_\ab)$. Then $A$ lifts to a class represented by some Čech cocycle $(b,\overline{c})\in\mathbf{\check Z}^2\!\left(k, \big[\pi_*j_*\mathbf G_{\mathrm m,\,K}\rightarrow\mathcal U\big]\right)$. Using that $A_v = 0$ for all $v$, we show in Remark \ref{remcechsetup} that there are $(b^v,\overline{c}^v)\in\mathbf{\check C}^1\!\left(k_v, \big[\pi_*j_*\mathbf G_{\mathrm m,\,K^v}\rightarrow\mathcal U^v\big]\right)$ such that $(b,\overline{c})_v - \mathrm d(b^v,\overline{c}^v) = (b_v-\mathrm db^v,0)$ lies in the subgroup $\mathrm{\check Z}^2(k_v, \mathbf G_{\mathrm m})$. The image of $A$ under the connecting homomorphism is the sum of the invariants $\mathrm{inv}_v([b_v-\mathrm db^v])$ in $\mathbf Q/\mathbf Z$, which is finite and independent of choices.

On the other hand, we now recall the construction of the global Poitou-Tate duality pairing
$$\langle -,-\rangle_{PT} : \Sh^2(H_\ab)\times\Sh^1(\widehat{H}_\ab)\rightarrow\mathbf Q/\mathbf Z$$
from \cite[\textsection 5.13]{RosTD} (up to reordering terms in the cup product and the resulting sign changes; cf.\  \cite[Rem.\ 3.4]{Don24} for a detailed explanation): Given classes $\xi\in\Sh^2(H_\ab)$ and ${A\in\Sh^1(\widehat{H}_\ab)}$, fix some representative Čech cocycles $h\in \mathrm{\check Z}^2(k,H_\ab)$ and $\alpha\in \mathrm{\check Z}^1(k,\widehat{H}_\ab)$. For all $v$, there exist cochains $\chi^v\in \mathrm{\check C}^1(k_v,H_\ab)$ 
with $\mathrm d\chi^v = h_v$. 
Moreover, since $\mathrm{\check H}^3(k,\mathbf G_\mathrm{m}) = 0$, there is a cochain $t\in\mathrm{\check C}^2(k,\mathbf G_\mathrm{m})$ such that $\mathrm dt = \alpha\!\smallsmile\! h$. Then in particular $\mathrm dt_v = -\mathrm d(\alpha_v\!\smallsmile\!\chi^v)$, which defines classes $[-(\alpha_v\!\smallsmile\!\chi^v)-t_v]\in\mathrm{\check H}^2(k_v,\mathbf G_\mathrm{m})$. The pairing $\langle \xi,A\rangle_{PT}$ is well-defined as the sum of the corresponding invariants in $\mathbf Q/\mathbf Z$, i.e. the resulting sum is finite and its value independent of all choices made in the construction.

Let an arbitrary class $A\in\Sh^1(\widehat{H}_\ab)$ be given. The remainder of this section will be devoted to showing that we can make all of the choices of $b,\overline{c},b^v,\overline{c}^v,h,\chi^v,t$ above so that the identity
\begin{equation}\label{eqptvsbm2}
    [(\alpha_v\!\smallsmile\!\chi^v)+t_v] = [b_v-\mathrm db^v]
\end{equation}
is satisfied in $\mathrm{\check H}^2(k_v, \mathbf G_{\mathrm m})$ for $\xi = \xi_X^\ab$, and for every place $v$. This will then immediately imply the desired equality \eqref{eqptvsbm}.

From now on, fix some finite field extension $k'/k$ and a point $x\in X(k')$. This choice determines a $G$-equivariant map $r^x : G_{k'}\rightarrow X_{k'}$ with fiber $\Hb\coloneqq (r^x)^{-1}(x)$. It also defines a short exact sequence
$$0\longrightarrow\mathbf G_{\mathrm m,\,k'}\longrightarrow\mathcal V\longrightarrow\mathcal U_{k'}\longrightarrow 0$$
by construction of $\mathcal U$. Since $\mathrm{H}^1(k'\otimes_k k',\mathbf G_{\mathrm m}) = 0$ (cf. the proof of Lemma \ref{lemr1lqf}), we may write any element of $\mathcal U(k'\otimes_k k')$ as $\overline{c}$ for some $c\in\mathcal V(k'\otimes_k k')$ (as usual in this paper, $k'\otimes_k k'$ is seen here as a $k'$-algebra via the first factor), not to be confused with its image $\widetilde{c}$ in $\widehat{H}_\ab(k'\otimes_k k')$.\quad\quad\quad

We write $\Psi$ for the map $\pi_*j_*\mathbf G_{\mathrm m,\,K}\rightarrow\mathcal U$. If $\Psi(b_0) = \overline{c}_0$ for some local sections $b_0,c_0$, that means that $b_0\circ r^x$ and $c_0$ differ by a constant (as locally defined scheme maps from~$G_{k'}$~to~$\mathbf G_{\mathrm m,\,k'}$). Last, we will simply write $g\coloneqq g_x$ and $h\coloneqq h_x$ for the elements defining, as in Proposition \ref{propcechsprclassdef}, the Springer class with respect to $x$ (and we do not distinguish between $h$ in $\Hb$ and its image in $H_\ab$). Recall that $h = \mathrm dg\coloneqq\mathrm{pr}_{12}^*g\cdot(\mathrm{pr}_{12}^*\varphi_G^{-1})(\mathrm{pr}_{23}^*g)\cdot\mathrm{pr}_{13}^*g^{-1}\in\Hb(k'\otimes_k k'\otimes_k k')$.

\begin{prop}\label{prophardclaim}
Suppose given $g'\in G(k')$ and a cocycle $(b,\overline{c})\in\mathbf{\check Z}^2\!\left(k, \big[\pi_*j_*\mathbf G_{\mathrm m,\,K}\rightarrow\mathcal U\big]\right)$. Then there is a cocycle of the form $(b',\overline{c\circ r_{g'}})$, where $c$ stays the same and ${r_{g'} : G_{k'\otimes_k k'}\rightarrow G_{k'\otimes_k k'}}$ denotes multiplication by $\mathrm{pr}^*_1(g')$ on the right. Moreover, the images of both cocycles define the same class in $\mathrm{\check H}^1(k, \widehat{H}_\ab)$, their image under the map $[b_0,\overline{c_0}]\mapsto[\widetilde{c_0}]$ in the above notation.
\end{prop}
\begin{prf}
It is clear that the images $\widetilde{c}$ and $\widetilde{c\circ r_{g'}}$ in $\widehat{H}_\ab(k'\otimes_k k')$ coincide, since
$$\widetilde{c}(h_0) = c(h_0g_0)-c(g_0) = c(h_0g_0g')-c(g_0g')$$
for any local sections $h_0,g_0$ by definition. The fact that $\mathrm d(b,\overline{c}) = 0$ means exactly that $\mathrm db = 0$ and $\Psi(b) = \mathrm d\overline{c}$. It remains thus to prove that there exists $b'\in\pi_*j_*\mathbf G_{\mathrm m,\,K}(k'\otimes_k k'\otimes_k k')$ such that $\mathrm db' = 0$ and $\Psi(b') = \mathrm d\overline{(c\circ r_{g'})}$.

Although $\mathcal V$ is only a sheaf over $k'$, not $k$, we may define in $\mathcal V(k'\otimes_k k'\otimes_k k')$ the following element (here, the transition map $\varphi_{\mathcal V}$ is analogous to \eqref{eqdescF1} from the definition of $\mathcal{E}$, $\mathcal{U}$)
\begin{align*}
    \mathrm dc \coloneqq\;&\mathrm{pr}^*_{12}c + \mathrm{pr}^*_{12}\varphi_{\mathcal V}^{-1}(\mathrm{pr}^*_{23}c) - \mathrm{pr}^*_{13}c\\
    =\;& \mathrm{pr}^*_{12}c + \big(\mathrm{pr}^*_{12}\varphi_{\mathbf G_{\mathrm m}}^{-1}\circ\mathrm{pr}^*_{23}c\circ\mathrm{pr}^*_{12}\ell_g\circ\mathrm{pr}^*_{12}\varphi_G^{-1}\big) - \mathrm{pr}^*_{13}c
\end{align*}
so that $\overline{\mathrm dc} = \mathrm d\overline{c}$. In particular, this implies that $b\circ r^x$ and $\mathrm dc$ differ by a constant.

We may similarly write $\overline{\mathrm d(\mathrm dc)} = \mathrm d(\mathrm d\overline{c}) = 0$, which proves that $\mathrm d(\mathrm dc)$ is a constant map (equal to some value in $\mathbf G_{\mathrm m}(k'\otimes_k k'\otimes_k k'\otimes_k k')$), although in general nonzero. In fact, when calculating it straight from the definition, all terms cancel out except the following (cf. \eqref{eqdescF2}):
\begin{align*}
    \mathrm d(\mathrm dc) &= (\mathrm{pr}^*_{12}\varphi_{\mathcal V}^{-1}\circ\mathrm{pr}^*_{23}\varphi_{\mathcal V}^{-1})(\mathrm{pr}^*_{34}c)-\mathrm{pr}^*_{13}\varphi_{\mathcal V}^{-1}(\mathrm{pr}^*_{34}c) \\
    &= \left(\big((\mathrm{pr}^*_{13}\varphi_{\mathcal V})^{-1}\circ(\mathrm{pr}^*_{23}\varphi_{\mathcal V})\circ(\mathrm{pr}^*_{12}\varphi_{\mathcal V})\big)^{-1}-\mathrm{id}\right)\!\big(\mathrm{pr}^*_{13}\varphi_{\mathcal V}^{-1}(\mathrm{pr}^*_{34}c)\big) = \big(\mathrm{pr}^*_{13}\varphi_{\widehat{H}}^{-1}(\mathrm{pr}^*_{34}\widetilde{c})\big)(h^{-1})
\end{align*}
Moreover, this shows that $\mathrm d(\mathrm dc) = \mathrm d(\mathrm d(c\circ r_{g'}))$, since $\widetilde{c} = \widetilde{c\circ r_{g'}}$. Therefore we may simply take the element $b'$ defined by the relation $b'\circ r^x - \mathrm d(c\circ r_{g'}) = b\circ r^x-\mathrm dc$, which makes sense since $\widetilde{\mathrm d(c\circ r_{g'})} = \mathrm d(\widetilde{c\circ r_{g'}}) = \mathrm d(\widetilde{c}) = 0$.
\end{prf}

We discuss briefly the implications of this statement: Our fixed class $A\in\Sh^1(\widehat{H}_\ab)$ lifts to a class represented by $(b,\overline{c})\in\mathbf{\check Z}^2\!\left(k, \big[\pi_*j_*\mathbf G_{\mathrm m,\,K}\rightarrow\mathcal U\big]\right)$. We claim that, given any fixed collection of points $x_1,\ldots,x_n\in X(k'\otimes_k k'\otimes_k k')$, we may choose (up to enlarging $k'/k$) this pair $(b,\overline{c})$ defined over $k'$ such that the values $b(x_j)\in\mathbf G_{\mathrm m}(k'\otimes_k k'\otimes_k k')$ are well-defined. Recall that $b$ is a rational map on $X_{k'\otimes_k k'\otimes_k k'}$, whose domain of definition includes the preimage of a nonempty open subscheme of $U^b\subseteq X$, and thus also the preimage of $(U^b)_{k'}\subseteq X_{k'}$. Similarly, for any choice of $c$, fix a nonempty open subscheme $U^c\subseteq X_{k'}$ on whose preimage(s) $c$ is defined.

To see this claim, note that the proof of the above proposition shows that $b\circ r^x$ is defined wherever $\mathrm dc$ is, so in particular wherever all three of $\mathrm{pr}^*_{12}c$, $\mathrm{pr}^*_{13}c$ and $\mathrm{pr}^*_{12}\varphi_{\mathcal V}^{-1}(\mathrm{pr}^*_{23}c)$ are. This reduces the problem to choosing $c$ which is defined on a finite number of prescribed points $y_1,\ldots,y_n\in G(k'\otimes_k k'\otimes_k k')$. The set-theoretic images of all these $y_i$, by the various projections from $G_{k'\otimes_k k'\otimes_k k'}$, lie in some finite set of closed points $\overline{z}_j\in G_{k'}$. By enlarging $k'/k$, we may assume that each closed point $\overline{z}_j$ corresponds to a $k$-point $z_j\in G(k)$. For any starting choice of $c$, the finite intersection $\bigcap_j z_j^{-1}U^c$ of open dense subschemes is a nonempty subscheme $T\subseteq G_{k'}$, and we may take $g'\in T(k')$ up to enlarging $k'$.

The above proposition then gives a pair $(b',c\circ r_{g'})$ with the desired property. We summarize the consequences of this claim in the following very long remark, before proving \eqref{eqptvsbm2}:

\begin{rem}\label{remcechsetup}
As above, take any $(b,\overline{c})\in\mathbf{\check Z}^2\!\left(k'/k, \big[\pi_*j_*\mathbf G_{\mathrm m,\,K}\rightarrow\mathcal U\big]\right)$ such that $[\,\widetilde{c}\,] = A$. We fix finite extensions $k'_v/k_v$, with $k'\subseteq k'_v$ for each $v$. Since $A_v = 0$, we may use Proposition~\ref{propc3presh} to choose $\overline{c}^v\in\mathcal U^v(k'_v)$ such that $-\mathrm d\widetilde{c}^v = \widetilde{c}_v$ (the minus sign here comes from the shift $\widehat{H}_\ab[-1]$, cf.\ Definition \ref{defnhypercech}). As before, $\overline{c}^v$ lifts to some $c^v\in\mathcal V(k'_v)$, for the same fixed sheaf $\mathcal V$ on $k'$. Again by Proposition \ref{propc3presh}, there is $b^v\in\pi_*j_*\mathbf G_{\mathrm m,\,K^v}(k'_v\otimes_{k_v} k'_v)$ such that $\Psi(b^v) = \overline{c}_v+\mathrm d\overline{c}^v$, so:
$$(b,\overline{c})_v-\mathrm d(b^v,\overline{c}^v) = (b_v,\overline{c}_v)-(\mathrm db^v,\Psi(b^v)-\mathrm d\overline{c}^v) = (b_v-\mathrm db^v,0)$$
Moreover, this must be the image of a constant cocycle in $\mathrm{Z}^2(k_v, \mathbf G_{\mathrm m})$, because:
$$\Psi(b_v-\mathrm db^v) = \Psi(b_v)-\mathrm d\Psi(b^v) = d\overline{c}_v-(d\overline{c}_v+\mathrm{d}^2\overline{c}^v) = 0$$
Hence, we may indeed take this $b^v$ to appear in \eqref{eqptvsbm2}. To get the value of the difference $b_v-\mathrm db^v$, we would like to evaluate it at some point in $X(k'_v\otimes_{k_v}k'_v\otimes_{k_v}k'_v)$. A natural candidate to consider is (some restriction of) the point $x$, and for this we introduce additional notation:

For sections $b_0\in\mathrm{\check C}^p(k'/k, \pi_*j_*\mathbf G_{\mathrm m,\,K}) = (\pi_*j_*\mathbf G_{\mathrm m,\,K})(k'{^{\,\otimes_k (p+1)}})$ and $x_0\in X(k'{^{\,\otimes_k (q+1)}})$, we let
$$b_0\smile x_0\coloneqq(\mathrm{pr}_{1,\ldots,p+1}^*b_0)\left(\mathrm{pr}_{1,p+1}^*\varphi_X^{-1}(\mathrm{pr}_{p+1,\ldots,p+q+1}^*x_0)\right)\in\mathbf G_{\mathrm m}(k'{^{\,\otimes_k (p+q+1)}}) = \mathrm{\check C}^{p+q}(k'/k, \mathbf G_{\mathrm m})$$
if this evaluation is well-defined. This notation mimics the usual cup product, and we may similarly define $c_0\smile g_0$ for $\mathcal V$ and $G_{k'}$, however in this situation we must use $\varphi_G\circ\ell_g^{-1}$ in place of $\varphi_G$ (this definition should be taken as is; we are not bothered by the fact that $\mathcal V$ is not defined over $k$, nor that this twisted ``descent datum'' might not satisfy the cocycle property, although this does mean that the formula for $\mathrm d(c_0\smile g_0)$ is slightly different than the one we give now, as we will see later). The advantage of this notation is that differentials satisfy intuitive formulas; for example, although an expression of the form ``$\mathrm dx_0$'' does not make sense, it still holds that:
$$\mathrm d(b_0\smile x_0) = \mathrm db_0\smile x_0+(-1)^p\sum\nolimits_{i=1}^{q+1}(-1)^{i-1}(\mathrm{pr}_{1,\ldots,p+1}^*b_0)\left(\mathrm{pr}_{1,m(i)}^*\varphi_X^{-1}(\mathrm{pr}_{p+1,\ldots,\widehat{p+i},\ldots,p+q+2}^*x_0)\right)$$
where $m(i)$ in $\mathrm{pr}_{1,m(i)}^*\varphi_X^{-1}$ equals $p+2$ if $i = 1$, and $p+1$ otherwise. Again, this formula makes sense only if all terms are well-defined (which is not immediate from the well-definedness of $b_0\smile x_0$, as some additional terms will need to be introduced, added and subtracted on the right side of the identity).

Continuing the remark, we explain the main setup for the rest of the section: For each $v$, we fix a point $x^v\in X(k_v)$ and choose (up to enlarging $k'_v/k_v$) some $g^v\in G(k'_v)$ such that $x_v.g^v = x^v$ holds for our fixed $x\in X(k')$. In the calculations to follow, we would like to use many ``values'' of the form $b\smile x$, $c\smile g$, $c\smile 1$, $\mathrm db^v\smile x_v$, $\mathrm dc^v\smile g^v$, $\mathrm dc^v\smile 1$, etc., as well as some evaluations appearing in intermediate calculations which are not of this form (as commented on above), such as e.g. $\mathrm dc^v(g_v)$. 
For some such value to be well-defined, it is equivalent to the ``condition'' that the domain of definition of $b$ (resp.\ $c$, $b^v$, $c^v$) includes a finite number of predetermined points, depending on the evaluation required. Now, having already fixed $x,g$ and all $x^v,g^v$, we make careful choices of $b,c$ and (for all $v$) $b^v,c^v$ using the ``claim'' following Proposition \ref{prophardclaim} to ensure that all the required values are indeed well-defined. 

Very importantly, because that claim only allows us to assume that finitely many arbitrary predetermined points lie in the domain of definition of $b$ (resp.\ $c$, $b^v$, $c^v$), we must ensure that, at the point of choosing $b$ (resp.\ $c$, $b^v$, $c^v$), we only impose finitely many conditions on its domain of definition. For this reason, we choose them in the following order:

First, choose $b,c$ such that the global expressions $b\smile x$, $c\smile g$, etc.\ are well-defined. Each of the conditions on $b$ adds additional conditions on $c$ (since $\Psi(b) = \mathrm d\overline{c}$; see the discussion~after Proposition \ref{prophardclaim}), but only finitely many. We cannot require that the local expressions $b_v\smile x^v$ or $c_v\smile g^v$ are defined for all $v$, as that would impose infinitely many conditions on $b,c$. 

Second, for each $v$ choose $b^v,c^v$ such that the local expressions $b^v\smile x_v$, $\mathrm dc^v\smile g^v$, $\mathrm dc^v(g_v)$, etc.\ are well-defined. Note that the situation is more subtle for $b^v$ than in the global case, as we cannot in general force $b^v\smile x^v$ to be well-defined: Indeed, the domain of $b^v$ depends on the domain of both $c$ and $c^v$ through the condition $\Psi(b^v) = \overline{c}_v+\mathrm d\overline{c}^v$ and we have already made our choice of $c$. However, $b^v\smile x_v = (c\smile g)_v+\mathrm dc^v\smile g_v+(\mathrm{const.})$ is well-defined as the input~of~$c$ rests only on the well-defined global value $c\smile g$; similarly for $\mathrm db^v\smile x_v$.

From this point onward, we assume that all the finitely many expressions appearing until the end of this section are well-defined. We may in particular suppose both $b_v\smile x_v$ and $\mathrm db^v\smile x_v$ well-defined in $\mathbf G_{\mathrm m}(k'_v\otimes_{k_v}k'_v\otimes_{k_v}k'_v)$, leading us to evaluate the constant cocycle $b_v-\mathrm db^v$ as:
$$(b_v\smile x_v)-(\mathrm db^v\smile x_v) = b_v(\mathrm{pr}_3^*x_v)-\mathrm db^v(\mathrm{pr}_3^*x_v) = (b_v-\mathrm db^v)(\mathrm{pr}_3^*x_v)\in\mathrm{\check Z}^2(k'_v/k_v, \mathbf G_{\mathrm m})$$
The class of this element is the right-hand side of \eqref{eqptvsbm2}. Note that the notation here is slightly different from most of this work, as we write $\mathrm{pr}_3^*$ instead of $\mathrm{pr}_{13}^*\mathrm{pr}_2^*$ for brevity.
\end{rem}

Let $x,g,h,b,c,x^v,g^v,b^v,c^v$ be as in the remark above. We will calculate the left side of \eqref{eqptvsbm2}:

\begin{prop}
The element $\chi^v$ defined by the following expression
$$\chi^v\coloneqq g_v\cdot \varphi_G^{-1}(\mathrm{pr}_2^*g^v)\cdot(\mathrm{pr}_1^*g^v)^{-1}\in G(k'_v\otimes_{k_v} k'_v)$$
lies in $\Hb(k'_v\otimes_{k_v} k'_v)$, which maps to $\mathrm{\check C}^1(k_v, H_\ab)$. Furthermore, $d\chi^v = h_v$ in $\mathrm{\check Z}^2(k_v, H_\ab)$.
\end{prop}
\begin{prf}
The first claim is equivalent to checking that $\mathrm{pr}_1^*x_v.\chi^v = \mathrm{pr}_1^*x_v$. Recall that $x^v\in X(k_v)$, so $\mathrm{pr}_1^*x^v = \varphi_X^{-1}(\mathrm{pr}_2^*x^v)\in X(k'_v\otimes_{k_v} k'_v)$, and that $x^v = x_v.g^v$. We calculate:
\begin{align*}
\mathrm{pr}_1^*x_v.\chi^v &= \varphi_X^{-1}(\mathrm{pr}_2^*x_v)\cdot\varphi_G^{-1} (\mathrm{pr}_2^*g^v)\cdot(\mathrm{pr}_1^*g^v)^{-1} = \varphi_X^{-1}(\mathrm{pr}_2^*(x_vg^v))\cdot(\mathrm{pr}_1^*g^v)^{-1}\\
&= \varphi_X^{-1}(\mathrm{pr}_2^*x^v)\cdot(\mathrm{pr}_1^*g^v)^{-1} = \mathrm{pr}_1^*x^v\cdot(\mathrm{pr}_1^*g^v)^{-1} = \mathrm{pr}_1^*x_v
\end{align*}

For the second part, recall that $\varphi_H = \varphi_G\circ\mathrm{int}(g)^{-1}$. We write:
\begin{align*}
(\mathrm{pr}_{12}^*\varphi_H^{-1})(\mathrm{pr}_{23}^*\chi^v)\cdot\mathrm{pr}_{12}^*\chi^v\cdot(\mathrm{pr}_{13}^*\chi^v)^{-1} &= \mathrm{pr}_{12}^*g_v\cdot\mathrm{pr}_{12}^*\varphi_G^{-1}\left(\mathrm{pr}_{23}^*g_v\cdot \mathrm{pr}_{23}^*\varphi_G^{-1}(\mathrm{pr}_3^*g^v)\cdot(\mathrm{pr}_2^*g^v)^{-1}\right)\\
&\hspace{-150pt} \cdot\;\mathrm{pr}_{12}^*g_v^{-1} \cdot\left(\mathrm{pr}_{12}^*g_v\cdot \mathrm{pr}_{12}^*\varphi_G^{-1}(\mathrm{pr}_2^*g^v)\cdot(\mathrm{pr}_1^*g^v)^{-1}\right)\cdot\left(\mathrm{pr}_1^*g^v\cdot \mathrm{pr}_{13}^*\varphi_G^{-1}(\mathrm{pr}_3^*g^v)^{-1}\cdot\mathrm{pr}_{13}^*g_v^{-1}\right)\\
&= \mathrm{pr}_{12}^*g_v\cdot\mathrm{pr}_{12}^*\varphi_G^{-1}(\mathrm{pr}_{23}^*g_v)\cdot\mathrm{pr}_{13}^*g_v^{-1} = h
\end{align*}
Finally, $\mathrm d\chi^v = \mathrm{pr}_{12}^*\chi^v\cdot(\mathrm{pr}_{12}^*\varphi_H^{-1})(\mathrm{pr}_{23}^*\chi^v)\cdot(\mathrm{pr}_{13}^*\chi^v)^{-1}$, but as we only need equality in the Abelian quotient $H_\ab$, we are free to reorder these terms to agree with the image of $h$.
\end{prf}

The difficulty in proving the next proposition lies, of course, in the fact that the map $c$ is not necessarily additive. We also need to be careful with the ``descent datum'' $\varphi_{\mathcal V}$ associated with $\mathcal V$, as its ``cocycle rule'' yields additional terms in the equations below, which nevertheless all cancel out (as suggested by parentheses in the definition of $t$):

\begin{prop}
Let $1^2\in G(k'\otimes_k k')$ denote the neutral element. We define a cochain
$$t\coloneqq (c\smile 1^2-c\smile g)+(b\smile x)$$
in $\mathrm{\check C}^2(k, \mathbf G_{\mathrm m})$. Then $\mathrm dt = \widetilde{c}\smile h$ holds in $\mathrm{\check Z}^3(k, \mathbf G_{\mathrm m})$, where $\widetilde c$ is the image of $c$ under $\mathcal U\rightarrow\widehat{H}_\ab$.
\end{prop}
\begin{prf}
Let $g_0$ denote some element of $G(k'\otimes_k k')$. Then by definition:
$$c\smile g_0 = \mathrm{pr}_{12}^*c\left(\mathrm{pr}_{12}^*(\varphi_G\circ\ell_g^{-1})^{-1}(\mathrm{pr}_{23}^*g_0)\right) = \mathrm{pr}_{12}^*c\left(\mathrm{pr}_{12}^*g\cdot \mathrm{pr}_{12}^*\varphi_G^{-1}(\mathrm{pr}_{23}^*g_0)\right)$$
Reordering terms in the definition of $\mathrm d(c\smile g_0)$, we get
\begin{align*}
\hspace{30pt}&\hspace{-30pt}
\begin{aligned}
    \mathrm d(c\smile g_0) + \mathrm{pr}_{13}^*c\left(\mathrm{pr}_{13}^*g\cdot \mathrm{pr}_{13}^*\varphi_G^{-1}(\mathrm{pr}_{34}^*g_0)\right) &- \mathrm{pr}_{12}^*c\left(\mathrm{pr}_{12}^*g\cdot \mathrm{pr}_{12}^*\varphi_G^{-1}(\mathrm{pr}_{24}^*g_0)\right) \\
    &+ \mathrm{pr}_{12}^*c\left(\mathrm{pr}_{12}^*g\cdot \mathrm{pr}_{12}^*\varphi_G^{-1}(\mathrm{pr}_{23}^*g_0)\right)
\end{aligned}\\
&= \mathrm{pr}_{12}^*\varphi_{\mathbf G_{\mathrm m}}^{-1}\left(\mathrm{pr}_{23}^*c\left(\mathrm{pr}_{23}^*g\cdot \mathrm{pr}_{23}^*\varphi_G^{-1}(\mathrm{pr}_{34}^*g_0)\right)\right) \\
&= \left(\mathrm{pr}_{12}^*\varphi_{\mathbf G_{\mathrm m}}^{-1}\circ\mathrm{pr}_{23}^*c\circ\mathrm{pr}_{12}^*(\varphi_G\circ\ell_g^{-1})\right)\left(\mathrm{pr}_{12}^*g\cdot\mathrm{pr}_{12}^*\varphi_G^{-1}\left(\mathrm{pr}_{23}^*g\cdot \mathrm{pr}_{23}^*\varphi_G^{-1}(\mathrm{pr}_{34}^*g_0)\right)\right) \\
&= \left(\mathrm{pr}_{12}^*\varphi_{\mathcal V}^{-1}(\mathrm{pr}_{23}^*c)\right)\left(\mathrm{pr}_{12}^*g\cdot\mathrm{pr}_{12}^*\varphi_G^{-1}(\mathrm{pr}_{23}^*g)\cdot \mathrm{pr}_{12}^*\varphi_G^{-1}\left( \mathrm{pr}_{23}^*\varphi_G^{-1}(\mathrm{pr}_{34}^*g_0)\right)\right) \\
&= \left(\mathrm{pr}_{12}^*\varphi_{\mathcal V}^{-1}(\mathrm{pr}_{23}^*c)\right)\left(\mathrm{pr}_{123}^*h\cdot \mathrm{pr}_{13}^*g\cdot\mathrm{pr}_{13}^*\varphi_G^{-1}(\mathrm{pr}_{34}^*g_0)\right) \\
&= \left(\mathrm{pr}_{12}^*\varphi_{\widehat{H}}^{-1}(\mathrm{pr}_{23}^*\widetilde{c})\right)(\mathrm{pr}_{123}^*h) +  \left(\mathrm{pr}_{12}^*\varphi_{\mathcal V}^{-1}(\mathrm{pr}_{23}^*c)\right)\left(\mathrm{pr}_{13}^*g\cdot\mathrm{pr}_{13}^*\varphi_G^{-1}(\mathrm{pr}_{34}^*g_0)\right)
\end{align*}
or equivalently (this is just the usual formula for $\mathrm d(c\smile g_0) - \mathrm dc\smile g_0$, but an extra term with $h$ appears because $\varphi_G\circ\ell_g^{-1}$ does not satisfy the cocycle property):
\begin{align*}
\mathrm d(c\smile g_0) - \left(\mathrm{pr}_{12}^*\varphi_{\widehat{H}}^{-1}(\mathrm{pr}_{23}^*\widetilde{c})\right)&(\mathrm{pr}_{123}^*h) - \mathrm{pr}_{123}^*\mathrm dc\left(\mathrm{pr}_{13}^*g\cdot \mathrm{pr}_{13}^*\varphi_G^{-1}(\mathrm{pr}_{34}^*g_0)\right) =
\\\mathrm{pr}_{12}^*c\left(\mathrm{pr}_{12}^*g\cdot \mathrm{pr}_{12}^*\varphi_G^{-1}(\mathrm{pr}_{24}^*g_0)\right) -\mathrm{pr}_{12}^*c&\left(\mathrm{pr}_{12}^*g\cdot \mathrm{pr}_{12}^*\varphi_G^{-1}(\mathrm{pr}_{23}^*g_0)\right) - \mathrm{pr}_{12}^*c\left(\mathrm{pr}_{13}^*g\cdot \mathrm{pr}_{13}^*\varphi_G^{-1}(\mathrm{pr}_{34}^*g_0)\right)
\end{align*}

For $g_0 = 1^2$, this equality becomes
\begin{align*}
\mathrm d(c\smile 1^2) &- \left(\mathrm{pr}_{12}^*\varphi_{\widehat{H}}^{-1}(\mathrm{pr}_{23}^*\widetilde{c})\right)(\mathrm{pr}_{123}^*h) - \mathrm{pr}_{123}^*\mathrm dc\,(\mathrm{pr}_{13}^*g) = - \mathrm{pr}_{12}^*c\,(\mathrm{pr}_{13}^*g)
\end{align*}
while for $g_0 = g$, it becomes:
\begin{align*}
\mathrm d(c\smile g) &- \left(\mathrm{pr}_{12}^*\varphi_{\widehat{H}}^{-1}(\mathrm{pr}_{23}^*\widetilde{c})\right)(\mathrm{pr}_{123}^*h) - \mathrm{pr}_{123}^*\widetilde{\mathrm dc}\,(\mathrm{pr}_{134}^*h) - \mathrm{pr}_{123}^*\mathrm dc\,(\mathrm{pr}_{14}^*g) =\\
&\mathrm{pr}_{12}^*\widetilde{c}\,(\mathrm{pr}_{124}^*h-\mathrm{pr}_{123}^*h-\mathrm{pr}_{134}^*h) + \mathrm{pr}_{12}^*c\,(\mathrm{pr}_{14}^*g) -\mathrm{pr}_{12}^*c\,(\mathrm{pr}_{13}^*g) - \mathrm{pr}_{12}^*c\,(\mathrm{pr}_{14}^*g)
\end{align*}
Subtracting the two (and noting that $\widetilde{\mathrm dc} = \mathrm d\widetilde{c} = 0$), we get:
$$\mathrm d(c\smile 1^2-c\smile g) + \mathrm{pr}_{123}^*\mathrm dc\,(\mathrm{pr}_{14}^*g) - \mathrm{pr}_{123}^*\mathrm dc\,(\mathrm{pr}_{13}^*g) = \mathrm{pr}_{12}^*\widetilde{c}\,(\mathrm{pr}_{134}^*h-\mathrm{pr}_{124}^*h+\mathrm{pr}_{123}^*h)$$

On the other hand, we calculate (using that $\mathrm db = 0$)
\begin{align*}
    \mathrm d(b\smile x) = \mathrm d(b\smile x) - \mathrm db\smile x &= \mathrm{pr}_{123}^*b\left(\mathrm{pr}_{14}^*\varphi_X^{-1}(\mathrm{pr}_4^*x)\right)-\mathrm{pr}_{123}^*b\left(\mathrm{pr}_{13}^*\varphi_X^{-1}(\mathrm{pr}_3^*x)\right)\\
    &= \mathrm{pr}_{123}^*(b\circ r^x)\left(\mathrm{pr}_{14}^*g\right)-\mathrm{pr}_{123}^*(b\circ r^x)\left(\mathrm{pr}_{13}^*g\right)\\
    &= \mathrm{pr}_{123}^*\mathrm dc\left(\mathrm{pr}_{14}^*g\right)-\mathrm{pr}_{123}^*\mathrm dc\left(\mathrm{pr}_{13}^*g\right)
\end{align*}
where we have used that $b\circ r^x$ and $\mathrm dc$ differ by a constant (since $\Psi(b) = \mathrm d\overline{c}$). Finally, we have
$$\widetilde{c}\smile h = \mathrm{pr}_{12}^* \widetilde{c}\left(\mathrm{pr}_{12}^*\varphi_H^{-1}(\mathrm{pr}_{234}^*h)\right) = \mathrm{pr}_{12}^* \widetilde{c}\,(\mathrm dh) +\mathrm{pr}_{12}^*\widetilde{c}\,(\mathrm{pr}_{134}^*h-\mathrm{pr}_{124}^*h+\mathrm{pr}_{123}^*h)$$
and $\mathrm dh = 0$. Combining the last three identities gives that $\mathrm dt = \widetilde{c}\smile h$.
\end{prf}

Recall by Remark \ref{remcechsetup} that $[\,\widetilde c\,] = A\in\Sh^1(\widehat{H}_\ab)$ is the class we started with. Using the two previous propositions, we may rewrite the desired identity \eqref{eqptvsbm2} in the following form:
$$\big[(\widetilde{c}_v\smile\chi^v)+(c_v\smile 1^2_v - c_v\smile g_v)+(b_v\smile x_v)\big] = \big[(b_v\smile x_v)-(\mathrm db^v\smile x_v)\big]$$
However, this new identity is clearly implied by the following straightforward calculation, which ends this section and the proof of Lemma \ref{lemphslift}:

\begin{prop}
Let $1^1_v\in G(k'_v)$ denote the neutral element. The equality 
$$(\widetilde{c}_v\smile\chi^v)+(c_v\smile 1^2_v - c_v\smile g_v) = -(\mathrm db^v\smile x_v)+\mathrm d(b^v\smile x_v) + \mathrm d(\mathrm dc^v\smile g^v-\mathrm dc^v\smile 1^1_v)$$
holds in $\mathrm{\check Z}^2(k_v, \mathbf G_{\mathrm m})$.
\end{prop}
\begin{prf}
Let $g^v_0$ denote some element of $G(k'_v)$. Then by definition:
$$\mathrm dc^v\smile g^v_0 = \mathrm dc\left((\varphi_G\circ\ell_g^{-1})^{-1}(\mathrm{pr}_2^*g^v_0)\right) = \mathrm dc\left(g_v\cdot \varphi_G^{-1}(\mathrm{pr}_2^*g^v_0)\right)$$
As in the proof of the previous proposition, we calculate:
\begin{align*}
\mathrm d(\mathrm dc^v&\smile g^v_0) - \mathrm{pr}_{12}^*\mathrm dc^v\left(\mathrm{pr}_{12}^*g_v\cdot \mathrm{pr}_{12}^*\varphi_G^{-1}(\mathrm{pr}_2^*g^v_0)\right) = - \mathrm{pr}_{13}^*\mathrm dc^v\left(\mathrm{pr}_{13}^*g_v\cdot \mathrm{pr}_{13}^*\varphi_G^{-1}(\mathrm{pr}_3^*g^v_0)\right) \\
&\;\begin{aligned}    + \left(\mathrm{pr}_{12}^*\varphi_{\mathcal V}^{-1}(\mathrm{pr}_{23}^*\mathrm dc)\right)\!&\left(\mathrm{pr}_{12}^*g_v\cdot \mathrm{pr}_{12}^*\varphi_G^{-1}(\mathrm{pr}_{23}^*g_v)\cdot \mathrm{pr}_{13}^*\varphi_G^{-1}(\mathrm{pr}_3^*g^v_0)\right)
\\
= \mathrm{pr}_{13}^*\mathrm d\widetilde{c}^v(h_v) +\left(- \mathrm{pr}_{13}^*\mathrm dc^v + \mathrm{pr}_{12}^*\varphi_{\mathcal V}^{-1}(\mathrm{pr}_{23}^*\mathrm dc)\right)\!&\left(\mathrm{pr}_{12}^*g_v\cdot \mathrm{pr}_{12}^*\varphi_G^{-1}(\mathrm{pr}_{23}^*g_v)\cdot \mathrm{pr}_{13}^*\varphi_G^{-1}(\mathrm{pr}_3^*g^v_0)\right)
\end{aligned} \\
&= \mathrm{pr}_{13}^*\mathrm d\widetilde{c}^v(h_v) +\big(\mathrm d(\mathrm dc^v) - \mathrm{pr}_{12}^*\mathrm dc^v\big)\!\left(\mathrm{pr}_{12}^*g_v\cdot \mathrm{pr}_{12}^*\varphi_G^{-1}(\mathrm{pr}_{23}^*g_v)\cdot \mathrm{pr}_{13}^*\varphi_G^{-1}(\mathrm{pr}_3^*g^v_0)\right)
\end{align*}
Now, recall that $\mathrm d(\mathrm dc^v)$ is constantly equal to some value $C$, because $\widetilde{\mathrm d(\mathrm dc^v)} = \mathrm d(\mathrm d\widetilde{c}^v) = 0$. For $g^v_0 = 1^1_v$, we therefore have
\begin{align*}
\mathrm d(\mathrm dc^v\smile 1^1_v) - \mathrm{pr}_{13}^*\mathrm d\widetilde{c}^v(h_v) - C &= \mathrm{pr}_{12}^*\mathrm dc^v\left(\mathrm{pr}_{12}^*g_v\right) - \mathrm{pr}_{12}^*\mathrm dc^v\left(\mathrm{pr}_{12}^*g_v\cdot \mathrm{pr}_{12}^*\varphi_G^{-1}(\mathrm{pr}_{23}^*g_v)\right) \\
&= \mathrm{pr}_{12}^*\mathrm dc^v\left(\mathrm{pr}_{12}^*g_v\right) - \mathrm{pr}_{12}^*\mathrm dc^v\left(\mathrm{pr}_{13}^*g_v\right) - \mathrm{pr}_{12}^*\mathrm d\widetilde{c}^v(h_v)
\end{align*}
and similarly, when substituting $g^v_0 = g^v$, this equality becomes:
\begin{align*}
\mathrm d(\mathrm dc^v&\smile g^v) - \mathrm{pr}_{13}^*\mathrm d\widetilde{c}^v(h_v) - C\\ &= - \mathrm{pr}_{12}^*\mathrm d\widetilde{c}^v\left(\left(\mathrm{pr}_{12}^*g_v\cdot \mathrm{pr}_{12}^*\varphi_G^{-1}(\mathrm{pr}_{23}^*g_v)\cdot \mathrm{pr}_{13}^*\varphi_G^{-1}(\mathrm{pr}_3^*g^v)\right)\cdot\left(\mathrm{pr}_{12}^*g_v\cdot \mathrm{pr}_{12}^*\varphi_G^{-1}(\mathrm{pr}_2^*g^v)\right)^{-1}\right) \\
&= - \mathrm{pr}_{12}^*\mathrm d\widetilde{c}^v\left(\left(\mathrm{int}(\mathrm{pr}_{12}^*g_v)\circ \mathrm{pr}_{12}^*\varphi_G^{-1}\right)\!\left(\mathrm{pr}_{23}^*g_v\cdot \mathrm{pr}_{23}^*\varphi_G^{-1}(\mathrm{pr}_3^*g^v)\cdot(\mathrm{pr}_2^*g^v)^{-1}\right)\right) \\
&= - \mathrm{pr}_{12}^*\mathrm d\widetilde{c}^v\left(\mathrm{pr}_{12}^*\varphi_H^{-1}(\mathrm{pr}_{23}^*\chi^v)\right)
\end{align*}
Next, we compute the following expression, using that $\Psi(b^v) = \widetilde{c}_v+\mathrm d\widetilde{c}^v$:
\begin{align*}
-(\mathrm db^v\smile x_v) + \mathrm d(b^v\smile x_v) &= \mathrm{pr}_{12}^*b^v\left(\mathrm{pr}_{12}^*\varphi_X^{-1}(\mathrm{pr}_2^*x_v)\right)-\mathrm{pr}_{12}^*b^v\left(\mathrm{pr}_{13}^*\varphi_X^{-1}(\mathrm{pr}_3^*x_v)\right)\\
    &= \mathrm{pr}_{12}^*(b^v\circ r^x)\left(\mathrm{pr}_{12}^*g_v\right)-\mathrm{pr}_{12}^*(b^v\circ r^x)\left(\mathrm{pr}_{13}^*g_v\right)\\
    &= \mathrm{pr}_{12}^*(c_v+\mathrm dc^v)\left(\mathrm{pr}_{12}^*g_v\right)-\mathrm{pr}_{12}^*(c_v+\mathrm dc^v)\left(\mathrm{pr}_{13}^*g_v\right)
\end{align*}
Combining the last three identities and using that $d\widetilde{c}^v = -\widetilde{c}_v$, we conclude that the right-hand side of the equality in the proposition statement can be written as
\begin{align*}
&\mathrm{pr}_{12}^*\widetilde{c}_v\left(\mathrm{pr}_{12}^*\varphi_H^{-1}(\mathrm{pr}_{23}^*\chi^v)\right) + \mathrm{pr}_{12}^*c_v\left(\mathrm{pr}_{12}^*g_v\right)-\mathrm{pr}_{12}^*c_v\left(\mathrm{pr}_{13}^*g_v\right) - \mathrm{pr}_{12}^*\widetilde{c}_v(h_v)\\
=\;& \mathrm{pr}_{12}^*\widetilde{c}_v\left(\mathrm{pr}_{12}^*\varphi_H^{-1}(\mathrm{pr}_{23}^*\chi^v)\right) + \mathrm{pr}_{12}^*c_v\left(\mathrm{pr}_{12}^*g_v\right)-\mathrm{pr}_{12}^*c_v\left(\mathrm{pr}_{12}^*g_v\cdot \mathrm{pr}_{12}^*\varphi_G^{-1}(\mathrm{pr}_{23}^*g_v)\right)\\
=\;& (\widetilde{c}_v\smile\chi^v)+(c_v\smile 1^2_v - c_v\smile g_v)
\end{align*}
which is exactly the left-hand side.
\end{prf}